\documentclass[reqno,12pt]{amsart}
\usepackage{amsmath, amssymb}
\usepackage{amsthm, amsbsy}
\usepackage{graphicx}
\usepackage{float}
\usepackage{color}   
\usepackage{hyperref}
\hypersetup{
    colorlinks=true, 
    citecolor=blue,
    filecolor=black,
    urlcolor=black
    linktoc=all,     
    linkcolor=blue,  
}

\usepackage[ruled]{caption}

\usepackage[toc,page]{appendix}

\theoremstyle{plain}
\newtheorem{theorem}{\indent\bf Theorem}[section]
\newtheorem{lemma}[theorem]{\indent\bf Lemma}
\newtheorem{corollary}[theorem]{\indent\bf Corollary}
\newtheorem{proposition}[theorem]{\indent\bf Proposition}
\theoremstyle{definition}
\newtheorem{definition}[theorem]{\indent\bf Definition}
\newtheorem{remark}[theorem]{\indent\bf Remark}
\newtheorem{problem}[theorem]{\indent\bf Problem}

\begin{document}
\title[Logarithmic Newton polygons and polytopes]{Logarithmic Newton polygons and polytopes, \\ and the factorization of Dirichlet polynomials}

\author[N.C. Bonciocat]{Nicolae Ciprian Bonciocat}
\address{Simion Stoilow Institute of Mathematics of the Romanian Academy, Research Unit 7, P.O. Box 1-764,
Bucharest 014700, Romania}
\email{Nicolae.Bonciocat@imar.ro}


\dedicatory{Dedicated to Professor Mihai Cipu on the occasion of his retirement}

\begin{abstract}
To study a Dirichlet polynomial $f(s)=\frac{a_{m}}{m^{s}}+\cdots +\frac{a_{n}}{n^{s}}$ by regarding it as a multivariate polynomial via the canonical map $\phi$ sending $p_i^{-s}$ to an indeterminate $X_i$, with $p_i$ the $i$th prime number, requires knowing the prime factorizations of all the integers in the support of $f$. We devise several methods to study the factorization of Dirichlet polynomials over unique factorization domains that circumvent the use of $\phi$, and obtain irreducibility criteria that are analogous to the classical results of Sch\"onemann, Eisenstein, Dumas, St\"ackel, Ore and Weisner for polynomials, and to more recent results of Filaseta and Cavachi. Some of the proofs rely on logarithmic versions of the classical Newton polygons. Criteria that use two or more $p$-adic valuations by combining information from different logarithmic Newton polygons of $f$, as well as irreducibility conditions for Dirichlet polynomials that assume a prime or a prime power value are also obtained. We also find excluding intervals for the relative degrees of the factors of a Dirichlet polynomial, and upper bounds for the multiplicities of the irreducible factors, in particular square-free criteria, that use no derivatives.  Criteria of absolute irreducibility analogous to results of Ostrowski, Gao and Stepanov-Schmidt are finally provided in the multivariate case by using logarithmic Newton polytopes and logarithmic upper Newton polygons.
\end{abstract}

\maketitle

\vspace{-0.9cm}

{\small 
\tableofcontents
}
\section{Introduction}\label{Introducere}

A Dirichlet polynomial $f(s)$ with complex coefficients $a_{i}$ and exponents $\lambda _i$ is a sum of the form $\sum _{i\in S}a_ie^{-\lambda _is}$ with $S$ a finite set of natural numbers and $(\lambda _i)$ an increasing sequence of positive real numbers. Many fundamental results in analytic and multiplicative number theory that rely on their properties refer to the particular case when $\lambda _i=\log i$.  Dirichlet polynomials are thus partial sums of corresponding Dirichlet series, and they can be used to approximate most Zeta-functions and $L$-functions, and their powers as well, and also their functional equations (Ivic \cite{Ivic}, Titchmarsh \cite{Titchmarsh}, Bombieri and Friedlander \cite{Bombieri}). Using the change of indeterminate $s\mapsto -s$, one may define the Dirichlet polynomials as sums of the form $\sum _{i\in S}a_i\cdot i^s$, which is a notation used by Lov\'asz in \cite{Lovasz}, where such series are called finite Dirichlet series. Huxley \cite{Huxley} studied Dirichlet polynomials of the form $\sum_{n=N}^{2N}\frac{a(n)\chi(n)}{n^s}$ with $\chi$ a Dirichlet character and $s=\sigma +it$, $\sigma \geq 0$. In particular, the analytical properties of Dirichlet polynomials play an important role in the study of the Riemann zeta-function (Montgomery \cite{Montgomery}). The study of almost-periodic functions relies too on the properties of Dirichlet polynomials (Besicovitch \cite{Besicovitch}), and some factorization properties of Dirichlet polynomials with integer coefficients have applications in group theory (Lucchini, Damian, Detomi and Morini \cite{Detomi}, \cite{Damian1}, \cite{Damian2}, \cite{Damian3}, \cite{Lucchini}, Brown \cite{Brown}, Hironaka \cite{Hironaka}, and Patassini \cite{Patassini1}, \cite{Patassini2}, \cite{Patassini3}, for instance). Dirichlet polynomials have also applications in lattice theory (Shala \cite{Shala}). It is well-known that the ring of Dirichlet polynomials with coefficients in a unique factorization domain $R$ is isomorphic to a ring of polynomials in countably many indeterminates over $R$, and that the ring of Dirichlet series with complex coefficients is isomorphic to a ring of formal power series in countably many indeterminates (Cashwell and Everett \cite{Cashwell}).

In contrast to the intensive study of their analytical properties (see for instance Dickson \cite{Dickson}, Montgomery \cite{Montgomery2}, Halasz \cite{Halasz}, Huxley \cite{Huxley2}, Heath-Brown \cite{Heath-Brown}, Jutila \cite{Jutila}, Bourgain \cite{Bourgain}, Oliveira \cite{Oliveira}, Roy and Vatwani \cite{RoyVatwani}, Guth and Maynard \cite{GuthMaynard}, and the references therein), the study of the factorization properties of Dirichlet polynomials has received significantly less attention. In this respect, there seems to be a lack of effective results on the factorization properties of Dirichlet polynomials over unique factorization domains, such as irreducibility criteria, square-free criteria, or bounds for the number of their irreducible factors, or for their multiplicities, and also a lack of effective results on the factorization of multivariate Dirichlet polynomials over an arbitrary field. The aim of this paper is to provide a series of such results, some of them being analogous to classical results for polynomials, such as the irreducibility criteria of Sch\"onemann-Eisenstein \cite{Schonemann}, \cite{Eisenstein}, Dumas \cite{Dumas}, St\"ackel \cite{Stackel}, Ore \cite{Ore4}, Weisner \cite{Weisner}, and Ostrowski \cite{Ostrowski1}, \cite{Ostrowski2}, or to more recent results of Filaseta \cite{Filaseta1}, Cavachi \cite{Cavachi} and Gao \cite{Gao}, for instance.

Besides being interesting on their own, a motivation to study Dirichlet polynomials over arbitrary unique factorization domains $R$ comes from the fact that they offer an additional framework to study multivariate polynomials, with which they are intimately linked. 
When $R$ is the ring of rational integers, another strong motivation to find and to study irreducible Dirichlet polynomials comes from their deep connections with arithmetic and number theory. For instance, we may hope that, under certain hypotheses, some of them might assume prime values infinitely often when $s$ runs in the set of negative integers. In this regard, one might easily formulate a conjecture similar to that of Bouniakovsky \cite{Bouniakovsky}:
\medskip

{\bf Conjecture 1.} {\it A primitive, irreducible Dirichlet polynomial $f(s)$ with positive leading coefficient and such that the set of values $f(\mathbb{Z}_{-})$ has no common divisor greater than $1$, assumes prime values infinitely often.}
\medskip 

However, such a conjecture should be considered with extreme caution, since two of its simplest cases, namely when $f_{1}(s)=-1+\frac{1}{2^{s}}$ and $f_{2}(s)=1+\frac{1}{2^{s}}$, refer to Mersenne primes (primes of the form $2^{p}-1$ with $p$ prime) and to Fermat primes (primes of the form $2^{2^{n}}+1$). While there is hope that there are infinitely many Mersenne primes, heuristics lead to a lot of skepticism on the existence of infinitely many Fermat primes. In this respect, such a conjecture might necessitate some additional assumptions on $f$. 
One might also consider an analogue for Dirichlet polynomials of Schinzel's Hypothesis H \cite{Schinzel}:
\medskip

{\bf Conjecture 2.} {\it For every nonconstant irreducible Dirichlet polynomials $f_1(s),\dots ,f_k(s)$ with positive leading coefficients, one of the following conditions holds:

1) There are infinitely many negative integers $n$ such that all of $f_1(n),\dots ,f_k(n)$ are simultaneously prime numbers;

2) There exists an integer $m>1$ that depends on $f_1,\dots ,f_k$ which divides $f_1(n)\cdots f_k(n)$ for each negative integer $n$ (i.e. there is a prime number $p$ such that for every negative integer $n$ there exists an index $i$ such that $f_i(n)$ is divisible by $p$).
}
\smallskip

We mention that to avoid considering negative integers $n$ in these conjectures, we may use the change of indeterminate $s\mapsto -s$, and define as in \cite{Lovasz} the Dirichlet polynomials as sums of the form $a_mm^s+\cdots +a_nn^s$.
\smallskip

Some fundamental results concerning Schinzel's Hypothesis H have been recently proved in various settings. Skorobogatov and Sofos \cite{Skorobogatov} proved that almost all polynomials of any fixed degree satisfy Schinzel's Hypothesis H. Bodin, D\`ebes and Najib \cite{Bodin1} proved that Schinzel's Hypothesis H holds if one replaces the ring $\mathbb{Z}$ of rational integers by the ring $R_m=R[x_1,\dots ,x_m]$ with $m\geq 1$ and $R$ is a unique factorization domain with quotient field $K$ in which the product formula holds, and in case $K$ has positive characteristic $p$, one has $K^p\neq K$ (and an extension of this result for multivariate polynomials). Improvements on a coprime version of Schinzel Hypothesis have been obtained by Bodin, D\`ebes, K\"onig and Najib in \cite{Bodin2}. We expect that some of the techniques used to prove these results might be adapted to the case of Dirichlet polynomials, to study the truthfulness of the two conjectures above. 

For a third conjecture in the same vein, we state here the simplest form of an analogue for Dirichlet polynomials of Hilbert's irreducibility theorem:
\smallskip

{\bf Conjecture 3.}\ {\em Let $f(s,t)$ be a bivariate irreducible Dirichlet polynomial with rational coefficients. Then $f$ remains irreducible for infinitely many integer specializations of $s$.
}

If we denote by $p_{1},p_{2},\dots $ the rational primes and use the map $\phi$ that sends $p_{i}^{-s}$ to an indeterminate $X_{i}$, a term of a Dirichlet polynomial $f(s)$ of the form $\frac{a_{n}}{n^{s}}$ with $n$ having the canonical decomposition $n=p_{i_{1}}^{k_{1}}\cdots p_{i_{r}}^{k_{r}}$, will be identified with the monomial $a_{n}X_{i_{1}}^{k_{1}}\cdots X_{i_{r}}^{k_{r}}$. Thus, one can regard $f$ as a multivariate polynomial, say $F(X_{i_{1}},\dots X_{i_{t}})$, and rephrase Conjecture 1, for instance, in terms of $F$, by asking whether $F$ assumes prime values for infinitely many specializations of the form $(X_{i_1},\dots ,X_{i_t})=(p_{i_1}^d,\dots ,p_{i_t}^d)$, where the primes $p_{i_1},\dots ,p_{i_t}$ are fixed and $d$ runs in the set of positive integers. Conversely, every multivariate polynomial $F(X_{1},\dots ,X_{n})$ can be transformed into a Dirichlet polynomial if we replace in $F$ the indeterminates $X_{i}$ by $p_{i}^{-s}$.

Despite this intimate connection between Dirichlet polynomials and multivariate polynomials, there is no easy way to study them simultaneously, the main obstruction coming from arithmetic. The reason is that, in general, there is not enough ``transparency" neither in factoring the integers in the support of $f$ to see what monomials will $F$ have, nor in computing and writing in ascending order the integers obtained by replacing the indeterminates of $F$ by $p_{1}^{-s},p_{2}^{-s},\dots $ to find $f$. In this respect, given a Dirichlet polynomial $f(s)=\frac{a_{m}}{m^{s}}+\cdots +\frac{a_{n}}{n^{s}}$, to effectively use $\phi$ in order to find the associated multivariate polynomial $F$, requires knowing the prime factorizations of all the indices $i$ with $a_i\neq 0$. Without knowing these prime factorizations, we cannot even find the number of indeterminates of $F$ and its degree. To overcome this obstruction, we will devise some methods to study the factorization of Dirichlet polynomials over unique factorization domains that essentially circumvent the use of $\phi$, or reduce the number of required prime factorizations to a minimum. This will be our main goal. 

The paper is structured as follows. Section \ref{DefNot} contains notations, definitions, and some basic facts on the irreducibility of Dirichlet polynomials, mostly of arithmetical nature, followed by an analogue for Dirichlet polynomials of the Sch\"onemann-Eisenstein irreducibility criterion. Section \ref{NewtonLogPol} presents the construction of the Newton log-polygon, a logarithmic version of the classical Newton polygon, which is most suitable for studying Dirichlet polynomials, and also provides results that are analogous to Dumas's irreducibility criterion for polynomials \cite{Dumas}. Using Newton log-polygon method, we also obtain irreducibility criteria for linear combinations of relatively prime Dirichlet polynomials of the form $f(s)+p^kg(s)$ with $p$ prime, that are analogous to Cavachi's criterion for polynomials \cite{Cavachi}. Other irreducibility criteria for linear combinations of Dirichlet polynomials, analogous to Sch\"onemann's criterion for polynomials and its variations, will be given in Appendix A. In Section \ref{RelativeDegree} we will adapt an elegant method of Filaseta that was crucial for proving the irreducibility of all but finitely many Bessel polynomials \cite{Filaseta1}, and obtain explicit intervals in which no relative degrees of the factors of some Dirichlet polynomials can lie.
We also obtain conditions that allow a Dirichlet polynomial to remain irreducible after multiplication of some of its coefficients by arbitrarily chosen elements belonging to the ring of coefficients. In Section \ref{moreprimes} we obtain a series of irreducibility conditions that use two or more $p$-adic valuations, by combining information from two or more Newton log-polygons of a given Dirichlet polynomial with respect to different primes. In Section \ref{primevalues} we provide several irreducibility criteria for Dirichlet polynomials that assume a prime value, or a prime power value, that are analogous to the irreducibility criteria of St\"ackel, Ore, and Weisner for polynomials. These results rely on Gelfond's inequality for the product of the heights of the factors of a multivariate polynomial, and also on some suitable estimates for the prime counting function $\pi(n)$. In Section \ref{squarefree} we will first show that in some cases the multiplicities of the irreducible factors of a Dirichlet polynomial $f$ are bounded by expressions that depend only on information about the indices $i$ in the support of $f$, not on the values of the coefficients. We will then devise a general method to study the multiplicities of the irreducible factors of Dirichlet polynomials, that requires no derivatives. This method holds over arbitrary unique factorization domains $R$, that do not necessarily include the logarithms of positive rational integers, required to define $f'(s)$. In particular, we will obtain a square-free criterion that requires no derivatives, and which reduces in case $R$ has positive characteristic to asking certain matrices associated to $f$ to have full rank. Conditions for two Dirichlet polynomials $f$ and $g$ to have a common factor will be then given in terms of the rank of some matrices associated to $f$ and $g$, which are analogous to the Sylvester matrix of two univariate polynomials. Section \ref{logpolytope} is devoted to the study of the factorization of multivariate Dirichlet polynomials over an arbitrary field $K$, which requires using logarithmic versions of the classical Newton polytopes. Here we obtain a logarithmic analogue of Ostrowski's Theorem on the decomposability of the Newton polytope of a product of two multivariate polynomials \cite{Ostrowski1}, \cite{Ostrowski2}, and applications to the study of the absolute irreducibility of multivariate Dirichlet polynomials, that rely on the concept of {\it log-integrally indecomposable} polytopes. We also provide a criterion of log-integrally indecomposability of polytopes, which is a logarithmic analogue of Gao's criterion for integrally indecomposability of polytopes \cite{Gao}. In Section \ref{Stepanov} we use a logarithmic analogue for bivariate Dirichlet polynomials of the {\it upper} Newton polygon for bivariate polynomials \cite{Gao}, and obtain irreducibility conditions, analogous to those in the Stepanov-Schmidt criterion for bivariate polynomials, that are expressed in terms of the degrees of the coefficients with respect to one of the indeterminates. Several examples will be given in the last section of the paper. 

\section{Definitions, notations, some basic facts, and an analogue of the \\ Sch\"onemann-Eisenstein irreducibility criterion}\label{DefNot}

A Dirichlet polynomial $f(s)$ with coefficients in a unique factorization domain $R$ is a finite sum of the form $\frac{a_{1}}{i_{1}^{s}}+\frac{a_{2}}{i_{2}^{s}}+\cdots +\frac{a_{n}}{i_{n}^{s}}$, with $i_{1}<i_{2}<\dots <i_{n}$ positive integers, $a_1,\dots,a_n\in R$, $a_{1}\cdots a_{n}\neq 0$, and $s$ regarded as an indeterminate. A simpler notation for $f$ that we will mostly use is $\frac{a_{m}}{m^{s}}+\cdots +\frac{a_{n}}{n^{s}}$, with $m,n$ arbitrary integers with $1\leq m\leq n$, and $a_ma_n\neq 0$.
One can obviously add and multiply in a natural way two such objects, of finite arbitrary length, to obtain other finite sums of this type. As is known, unlike in the polynomial case, here we will use for multiplication the Dirichlet product. Thus, if $f(s)=\frac{a_{m}}{m^{s}}+\cdots +\frac{a_{n}}{n^{s}}$ and $g(s)=\frac{b_{u}}{u^{s}}+\cdots +\frac{b_{v}}{v^{s}}$, then the product $f\cdot g(s)$ will be the Dirichlet polynomial $\frac{c_{mu}}{(mu)^{s}}+\cdots +\frac{c_{nv}}{(nv)^{s}}$, with coefficients 
\[
c_{k}=\sum\limits _{i\cdot j=k}a_{i}b_{j},\quad k=mu,\dots ,nv.
\]
One may easily check that with respect to the usual addition and to the multiplication defined above, the set $DP(R)[s]$ of Dirichlet polynomials with coefficients in $R$ is a commutative ring with unity. Using the map $\phi$ described before, the ring $DP(R)[s]$ of Dirichlet polynomials is easily seen to be isomorphic to a ring of polynomials in countably many indeterminates over $R$, thus being a unique factorization domain. 

We will proceed now with some definitions and notations.  

\begin{definition}\label{def1}
Let $R$ be a unique factorization domain, $DP(R)[s]$ the ring of Dirichlet polynomial with coefficients in $R$, and let $f\in DP(R)[s]$. The {\it degree} of a nonzero term of $f$ of the form $\frac{a_{i}}{i^{s}}$ will be by definition $i$, so for a Dirichlet polynomial $f(s)=\frac{a_{m}}{m^{s}}+\cdots +\frac{a_{n}}{n^{s}}$ with $a_{m}a_{n}\neq 0$ we will refer to $m$ as the {\it min-degree} of $f$ and we will denote it $\deg _{{\rm min}}f$, while $n$ will be called the {\it degree} of $f$, and will be denoted by $\deg f$. Also, by convention, the zero Dirichlet polynomial will have degree zero, and if $f\neq 0$, the rational number $\frac{n}{m}\geq 1$ will be referred to as the {\it relative degree} of $f$.
Moreover,  $a_{n}$ will be referred to as the {\it leading coefficient} of $f$, $a_m$ as the {\it min-degree coefficient} of $f$, $a_1$ as the {\it constant term} of $f$, and if $a_n=1$, $f$ will be called {\it monic}. The {\it support} of $f$, denoted by $Supp(f)$, is the set of indices $i$ with $a_i\neq 0$. 

We mention here that, unless $a_1\neq 0$, we will try to avoid writing $f$ as $\frac{a_1}{1^s}+\cdots +\frac{a_{n}}{n^{s}}$, since many factorization problems rely on knowing $\deg _{{\rm min}}f$, as we will see later.

A Dirichlet polynomial $f$ is called {\it primitive} if its coefficients are coprime. An analogue of Gauss's lemma holds in this case too, so a product of primitive Dirichlet polynomials is easily seen to be primitive too.

A Dirichlet polynomial $f$ is called {\it algebraically primitive}
if the integers in the support of $f$ are relatively prime. If $Supp(f)=\{i_1,\dots ,i_k\}$ and $\gcd (i_1,\dots ,i_k)=d>1$, then the Dirichlet polynomial obtained from $f$ by dividing each of the indices $i_1,\dots ,i_k$ by $d$ will be called the {\it algebraically primitive part of} $f$. We note here that the product of two algebraically primitive Dirichlet polynomials is also algebraically primitive.

A nonconstant Dirichlet polynomial $f$ is called {\it irreducible} if it cannot be expressed as a product of two nonconstant Dirichlet polynomials $g,h\in DP(R)[s]$, otherwise being called {\it reducible}. We note that in this definition we don't ask an irreducible Dirichlet polynomial to necessarily be primitive. Thus, just like in the case of polynomials, a non-primitive irreducible $f$ will be {\it irreducible over} $Q(R)$, the quotient field of $R$, while a primitive irreducible $f$ will be referred to as {\it irreducible over} $R$. 

Let $P_{f}=\{ p_{1},\dots ,p_{r}\} $ be the set of all the prime numbers that appear in at least one of the the prime factorizations of the indices $i\in Supp(f)$.  We will refer to the prime numbers $p\in P_{f}$ as the {\it relevant} primes of $f$.

A Dirichlet polynomial $f$ is {\it square-free} if it has no repeated irreducible factors (including irreducible elements in $R$). Thus, a primitive Dirichlet polynomial is square-free if it has no repeated nonconstant irreducible factors. More generally, if $k\geq 2$ is an integer, we will say that (a primitive) $f$ is {\it $k$-power-free} if the multiplicities of its (nonconstant) irreducible factors are at most $k-1$.
\end{definition}

Throughout the paper, proving the irreducibility of a Dirichlet polynomial $f$ without explicitly assuming that $f$ is primitive, will mean that $f$ is irreducible over $Q(R)$, the quotient field of the ring $R$ of coefficients. Moreover, in some of our results, where there is no risk of confusion, to simplify the statements, we will no longer specify the ring $R$ of coefficients.

The simplest examples of infinite families of irreducible Dirichlet polynomials are easily obtained directly from Definition \ref{def1}:
\begin{proposition}\label{prop1}
All Dirichlet polynomials of prime degree are irreducible.
\end{proposition}
We note here that unlike univariate polynomials, where for instance, $X^{2}+1$ is irreducible over $\mathbb{Z}$, being reducible over $\mathbb{Z}[i]$,  a Dirichlet polynomial of prime degree will remain irreducible even if we embed its ring of coefficients into a larger one, with enriched factorization properties. This is the simplest example of an {\it absolutely irreducible} Dirichlet polynomial, meaning that it remains irreducible over an algebraic closure of the quotient field of $R$.

Among the Dirichlet polynomials $f(s)=\frac{a_{m}}{m^{s}}+\cdots +\frac{a_{n}}{n^{s}}$ of composite degree $n$, or having the min-degree $m>1$, one may easily identify a class of irreducible ones, consisting of those having nonzero terms of degree sufficiently close to $n$ or $m$. In this respect we have:
\begin{proposition}\label{prop2}
Let $f$ be a Dirichlet polynomial of composite degree $n$, and let $p_{n}$ be the smallest prime divisor of $n$. If $f$ has a nonzero coefficient $a_{i}$ with $n-p_{n}<i<n$, then $f$ is irreducible.
\end{proposition}
\begin{proof}
\ Let us assume to the contrary that $f(s)=g(s)\cdot h(s)$ for two nonconstant Dirichlet polynomials $g(s)=\frac{b_{d_{1}}}{d_{1}^{s}}+\cdots +\frac{b_{d_{2}}}{d_{2}^{s}}$ and $h(s)=\frac{c_{e_{1}}}{e_{1}^{s}}+\cdots +\frac{c_{e_{2}}}{e_{2}^{s}}$, say, with $d_2>1$, $e_2>1$, $b_{d_{2}}\neq 0$, $c_{e_{2}}\neq 0$, and $d_{2}\cdot e_{2}=n$. By the multiplication rule, the coefficient $a_{i}$ is given by
\[
a_{i}=\sum\limits_{j\cdot k=i}b_{j}c_{k},
\] 
and in this sum we must have $b_{j}c_{k}\neq 0$ for at least one pair $(j,k)$ with $j\cdot k=i$, as $a_{i}\neq 0$. Consider such a pair $(j,k)$. Since $d_{2}\cdot e_{2}=n$ and $i<n$, at least one of the inequalities $j\leq d_{2}$ and $k\leq e_{2}$ must be a strict one, so without loss of generality we will assume that $j<d_{2}$. Since $e_2$ is a divisor of $n$ and $e_{2}\neq 1$, we must have $e_2\geq p_n$. This and the fact that $j<d_{2}$ allow us to successively deduce that
\[
i=j\cdot k\leq (d_{2}-1)k\leq (d_{2}-1)e_{2}=n-e_{2}\leq n-p_{n},
\]
which contradicts our assumption that $n-p_{n}<i$, and completes the proof. 
\end{proof}
\begin{proposition}\label{prop3}
Let $f(s)=\frac{a_{m}}{m^{s}}+\cdots +\frac{a_{n}}{n^{s}}$ be a Dirichlet polynomial with $m>1$ and $a_ma_n\neq 0$, and let $p_{m}$ and $p_{n}$ be the smallest prime divisors of $m$ and $n$, respectively. If $m>n/p_{n}$ and $f$ has a nonzero coefficient $a_{i}$ with $m<i<m+p_{m}$, then $f$ is irreducible.
\end{proposition}
\begin{proof}\ We may obviously assume that $n$ is composite. We will use the same notations as in the proof of Proposition \ref{prop2}. So let us assume again that $f$ is reducible. We first observe that our condition $m>n/p_{n}$ prevents $d_{1}$ and $e_{1}$ to be equal to $1$. Indeed, if $e_{1}=1$, say, then $d_{1}$ must be equal to $m$, so 
\begin{equation}\label{d2sim}
d_{2}\geq m.
\end{equation}
On the other hand, as $g$ and $h$ are assumed to be nonconstant factors of $f$, $e_{2}$ must be greater than $1$, so $d_2$ must be a proper divisor of $n$, and hence $d_{2}\leq n/p_{n}$, which by (\ref{d2sim}) yields $m\leq n/p_{n}$, a contradiction. In a similar way one can prove that  $d_1>1$.

We next see that in the sum $\sum_{j\cdot k=i}b_{j}c_{k}$ defining $a_{i}$, for any pair $(j,k)$ with $b_{j}c_{k}\neq 0$ and $j\cdot k=i$, at least one of the inequalities $j\geq d_{1}$ and $k\geq e_{1}$ must be a strict one, as $i>m=d_1e_1$. Without loss of generality we may assume that $j>d_{1}$. This allows us to deduce that
\[
i=j\cdot k\geq (d_{1}+1)k\geq (d_{1}+1)e_{1}=m+e_{1}\geq m+p_{m},
\]
as $e_{1}\mid m$ and $e_{1}\neq 1$ (by the argument above). This contradicts our assumption that $i<m+p_{m}$, and completes the proof. 
\end{proof}

In particular, from Propositions \ref{prop1}, \ref{prop2} and \ref{prop3} we obtain: 
\begin{corollary}\label{coro1}
Let  $f(s)=\frac{a_{m}}{m^{s}}+\cdots +\frac{a_{n}}{n^{s}}$ be a Dirichlet polynomial with $a_{m}a_{n}\neq 0$. If $a_{n-1}\neq 0$, or, if $m>n/2$ and $a_{m+1}\neq 0$, then $f$ is irreducible. 
\end{corollary}

These results also show that in the case of Dirichlet polynomials with integer coefficients, Tur\'an's conjecture on the irreducibility of neighboring polynomials is trivially settled in the affirmative, for if $f(s)$ has no nonzero terms $\frac{a_i}{i^s}$ with $n-p_n<i<n$, we may pick such an index $i$, and the Dirichlet polynomial $f(s)+\frac{1}{i^s}$ will be irreducible. Thus, for a nontrivial analogue of Tur\'an's conjecture, one may for instance consider the following one (with possible refinements that take into account Proposition \ref{prop3}).
\smallskip

{\bf Conjecture 4.} {\em There is an absolute constant $C>0$ such that, if $f(s)=\sum\limits_{i=1}^{n}\frac{a_{i}}{i^{s}}$ is a Dirichlet polynomial of degree $n$ with integer coefficients, and $p_n$ is the smallest prime factor of $n$, then there is a Dirichlet polynomial $g(s)=\sum\limits _{i=1}^{n-p_n}\frac{b_{i}}{i^{s}}$ with integer coefficients such that $\sum\limits_{i=1}^{n-p_n}|b_i|\leq C$ and $f(s)+g(s)$ is irreducible over $\mathbb{Q}$.}
\smallskip

Unlike univariate polynomials, a Dirichlet polynomial $f(s)=\frac{a_{m}}{m^{s}}+\cdots +\frac{a_{n}}{n^{s}}$ is subject to more restrictive conditions on the possible values of the degrees and relative degrees of its factors, and these inherent restrictions are imposed by the prime factorizations of $m$ and $n$. So $m$ and $n$ alone might dictate much of the factorization properties of $f$, sometimes diminishing the role of the coefficients $a_{i}$. To see this, we need to introduce some arithmetical objects, which will be used here and in some of the following sections.
\begin{definition}\label{def3}
For a pair of fixed, arbitrarily chosen integers $(m,n)$ with $1\leq m<n$, and any integer $k>0$ we define
\begin{eqnarray*}
S_{rd}(m,n) & := & \left\{ \frac{d}{c}\ :\ d\mid n,\ c\mid m,\  \ {\rm and}\ \ 1<\frac{d}{c}<\frac{n}{m}\right\} ,\\
S^{k}_{rd}(m,n) & := &\left\{ \frac{d}{c}\ :\ d\mid n,\ c\mid m,\  \ {\rm and}\ \ 1<\frac{d}{c}\leq \sqrt[k+1]{\frac{n}{m}}\thinspace \right\},\\
\rho (m,n) & := & \max\left\{ \frac{d}{c}\ :\ d\mid n,\ c\mid m,\  \ {\rm and}\ \ \frac{d}{c}\leq \sqrt{\frac{n}{m}}\thinspace \right\}, \quad {\rm and} \\
\delta (m,n) & := & \thinspace \min\left\{ \frac{d}{c}\ :\ d\mid n,\ c\mid m,\ \  {\rm and}\ \ d>c>0\right\} .
\end{eqnarray*}
We may regard $S_{rd}(m,n)$ as the set of all the possible relative degrees of the non-trivial proper factors of an arbitrary algebraically primitive Dirichlet polynomial $f(s)=\frac{a_{m}}{m^{s}}+\cdots +\frac{a_{n}}{n^{s}}$ with $a_{m}a_{n}\neq 0$, and $S^{k}_{rd}(m,n)$ as the set of all such possible relative degrees that do not exceed $\sqrt[k+1]{\frac{n}{m}}$. We will call $\rho (m,n)$ the {\it rational square root} of the pair $(m,n)$, and $\delta (m,n)$ the {\it rational floor} of the pair $(m,n)$.
\end{definition}
\begin{remark}\label{somearithm}
i) We note that $\rho (m,n)\geq 1$, as $c=d=1$ satisfy trivially this definition, and also that if $\rho (m,n)>1$, then $\delta (m,n)\leq\rho (m,n)$. It should be stressed that $\rho (m,n)$ and $\delta (m,n)$ depend on both $m$ and $n$, not only on their ratio $n/m$. To see this, take for instance $(m_{1},n_{1})=(3,7)$ and $(m_{2},n_{2})=(15,35)$. Even if $7/3=35/15$, we have $\rho (3,7)=1$ and $\delta (3,7)=7/3$, while $\rho (15,35)=\delta (15,35)=7/5$. 

ii) Let $c$ and $d$ be positive divisors of $m$ and $n$, respectively, and let $c'=\frac{m}{c}$ and $d'=\frac{n}{d}$ be their complementary divisors. It is easy to check that $1<\frac{d}{c}\leq \sqrt{\frac{n}{m}}$ if and only if $\sqrt{\frac{n}{m}}\leq \frac{d'}{c'}<\frac{n}{m}$. This allows us to conclude that
\begin{equation}\label{arithmirred}
S_{rd}(m,n)=\emptyset \quad \Leftrightarrow\quad S^{1}_{rd}(m,n)=\emptyset \quad \Leftrightarrow\quad \rho (m,n)=1.
\end{equation}

iii) Let $c$ and $d$ be positive divisors of $m$ and $n$, respectively. For $\frac{d}{c}$ to belong to $S^{1}_{rd}(m,n)$, $d$ must be a proper divisor of $n$, for otherwise $n$ would be forced to satisfy the inequality $n\leq \frac{c^2}{m}$, which cannot hold, as $n>m\geq c$.

iv) Although it is a trivial fact, we will also mention here that given an algebraically primitive Dirichlet polynomial $f(s)=\frac{a_{m}}{m^{s}}+\cdots +\frac{a_{n}}{n^{s}}$ with $a_{m}a_{n}\neq 0$, and two positive divisors $c$ and $d$ of $m$ and $n$ respectively, with $c<d$, for $d/c$ to be the relative degree of a hypothetical non-trivial factor $g$ of $f$, we also have to ask $d/c<n/m$ (with equality possible if one drops the condition that $f$ is algebraically primitive) and in this case the complementary factor of $g$ will have relative degree $(nc)/(md)$, which will also belong to $S_{rd}(m,n)$.
Analyzing the sets $S_{rd}(m,n)$ and $S^{k}_{rd}(m,n)$ might provide us valuable information on the canonical decomposition of $f$, obtained without even looking at its coefficients.

v) Rather surprisingly, $\rho(m,n)$ also appears in the study of the irreducibility of integer polynomials $f(X)$ whose roots lie in the Apollonius circle defined as the locus of points $P$ with $d(P,(b,0))=\rho(|f(a)|,|f(b)|)\cdot d(P,(a,0))$, and $a,b$ integers with $|f(b)|>|f(a)|$ (see \cite{BoncioIndagationes}).
\end{remark}

If one of the equivalent conditions in (\ref{arithmirred}) is satisfied, then an algebraically primitive Dirichlet polynomial $f(s)=\frac{a_{m}}{m^{s}}+\cdots +\frac{a_{n}}{n^{s}}$ with $a_{m}a_{n}\neq 0$ must be irreducible, as seen in the following simple result. 
\begin{proposition}\label{rationalfloor}
Let $f(s)=\frac{a_{m}}{m^{s}}+\cdots +\frac{a_{n}}{n^{s}}$ be an algebraically primitive Dirichlet polynomial with $a_{m}a_{n}\neq 0$. If $\rho (m,n)=1$, then $f$ must be irreducible.
\end{proposition}
\begin{proof}\ Assume to the contrary that we can write $f$ as a product of two nonconstant Dirichlet polynomials, say $g(s)=\frac{\alpha _{c_{1}}}{c_{1}^{s}}+\cdots +\frac{\alpha _{d_{1}}}{d_{1}^{s}}$ and $h(s)=\frac{\beta _{c_{2}}}{c_{2}^{s}}+\cdots +\frac{\beta _{d_{2}}}{d_{2}^{s}}$. Since
\[
\frac{n}{m}=\frac{d_{1}}{c_{1}}\cdot \frac{d_{2}}{c_{2}},
\]
one of the relative degrees of the factors $g$ and $h$, say $d_{1}/c_{1}$, must be less than or equal to $\sqrt{n/m}$, so $d_{1}/c_{1}\leq \rho (m,n)=1$. This forces $c_{1}=d_{1}>1$, so $g$ must have a single term, which obviously cannot hold, as $f$ was assumed to be algebraically primitive. Therefore $f$ must be irreducible, and this completes the proof. \end{proof}
We notice here that the conclusion in the statement of Proposition \ref{rationalfloor} fails if one drops the condition that $f$ is algebraically primitive. 
Reasoning in a similar way, one can actually prove the following more general result.
\begin{proposition}\label{radicalkplus1}
Let $f(s)=\frac{a_{m}}{m^{s}}+\cdots +\frac{a_{n}}{n^{s}}$ be an algebraically primitive Dirichlet polynomial with $a_{m}a_{n}\neq 0$. If $S^{k}_{rd}(m,n)=\emptyset$, then $f$ is a product of at most $k$ irreducible factors.
\end{proposition}

\begin{definition}\label{def4}
If $\rho (m,n)=1$, then an algebraically primitive Dirichlet polynomial $f(s)=\frac{a_{m}}{m^{s}}+\cdots +\frac{a_{n}}{n^{s}}$ with $a_{m}a_{n}\neq 0$ will be called {\it arithmetically irreducible}.
\end{definition}
Even if the proofs so far have been quite simple, at this point we may already state the following seemingly difficult problem:
\begin{problem}\label{pb1}
Find all the pairs of integers $n>m\geq 1$ with  $\rho (m,n)=1$.
\end{problem}
As we shall see in the following result, our condition $\rho (m,n)=1$ forces $m$ and $n$ to be ``rather close". 
\begin{proposition}\label{closeenough}
Let $m$ and $n$ be integers with $1\leq m<n$, let $p$ be the smallest prime factor of $n$, and $q$ the largest divisor of $m$ smaller than $p$. If $\rho (m,n)=1$, then 
$1<\frac{n}{m}<\frac{p^2}{q^2}$. In particular, if $\rho (m,n)=1$ and $n$ is even, then $1<\frac{n}{m}<4$. 
\end{proposition}
\begin{proof}\ If we assume to the contrary that $\frac{n}{m}\geq\frac{p^2}{q^2}$, then $\frac{p}{q}\leq \sqrt{\frac{n}{m}}$, which further implies that $\rho (m,n)\geq\frac{p}{q}>1$, a contradiction. For an even $n$ we have $p=2$, so $q$ is equal to $1$. 
\end{proof}

Our hopes to solve Problem \ref{pb1} in its entire generality are somehow dashed by the fact that, even if we fix the primes dividing $m$ and $n$ and let their multiplicities vary, computing $\rho(m,n)$ will require testing inequations between products of prime powers, which can only be done on a case by case basis. However, we present in the following result some simple cases where one can easily prove that $\rho(m,n)$ is equal to $1$.
\begin{proposition}\label{corofloor} 
$\rho (m,n)=1$ in each one of the following cases:

\noindent i) \ \thinspace $n$ is a prime number greater than $m$;

\noindent ii) \thinspace $m=p^kq$, $n=p^{k+1}$ with $p$ a prime number and integers $k,q$ with $k\geq 0$ and $0<q<p$;

\noindent iii) $m=p^{k}$, $n=p^{k+1}$ for some prime number $p$ and some integer $k\geq 0$;

\noindent iv) \thinspace $m=p^{k}$, $n=p^{k}q$ for some prime numbers $p,q$ with $q<p$ and some integer $k\geq 1$.
\end{proposition}
\begin{proof}\ i) In our first case, as $n=p$ with $p$ a prime greater than $m$, any positive quotient $\frac{d}{c}$ in the definition of 
$\rho (m,n)$ either has the form $\frac{1}{c}$ with $c\mid m$, and 
hence is at most $1$, or is equal to $\frac{p}{c}$, which obviously 
exceeds $\sqrt{\frac{p}{m}}$. Therefore $\rho (m,n)$ must be equal to $1$. 

ii) In this case any positive quotient $\frac{d}{c}$
with $d\mid n$ and $c\mid m$ has the form $\frac{p^i}{r}$
with $i$ an integer satisfying $-k\leq i\leq k+1$, and $r$ a divisor of $q$. For $i\leq 0$ these quotients will be at most $1$, while for positive $i$ all the corresponding quotients will exceed $\sqrt{\frac{p}{q}}$, since $\frac{p}{r}>\sqrt{\frac{p}{q}}$.

iii) Here we observe that if $m=p^{k}$ and $n=p^{k+1}$, then any 
positive quotient $d/c$ with $c\mid m$ and $d\mid n$ 
has the form $p^{i}$ with $i$ an integer satisfying $-k\leq i\leq k+1$, 
and the maximal such power of $p$ that is less than or equal to 
$\sqrt{p}$ is $1$, so in this case too we have $\rho (m,n)=1$. 

iv) In this case, if $m=p^{k}$ and $n=p^{k}q$, any positive quotient 
$d/c$ with $c\mid m$ and $d\mid n$ 
has the form $p^{i}q^{j}$ with $i$ an integer satisfying 
$-k\leq i\leq k$ and $j\in \{ 0,1\} $. For $j=0$, 
no such quotient other than 1 belongs to the interval [$1,\sqrt{q}$], 
since $p>\sqrt{q}$. Finally, we observe that for $j=1$ no 
integer $i$ can satisfy the condition  $1< p^{i}q< \sqrt{q}$, since $p>q$. 
\end{proof}

As we shall see later in Theorems \ref{TeoremaCaLaFilaseta1} and \ref{TeoremaCaLaFilaseta2}, when $\rho (m,n)>1$, both $\rho (m,n)$ and $\delta (m,n)$ play an important role in some irreducibility tests that rely on $p$-adic information about the coefficients of Dirichlet polynomials, where $p$ is a prime element of the ring of coefficients.

We will proceed now with irreducibility conditions that use information about the prime factorization of the coefficients of a Dirichlet polynomial. For polynomials, the first such irreducibility conditions appeared in the Sch\"onemann-Eisenstein irreducibility criterion (see Cox \cite{Cox} for the history of this criterion). Over the time, this criterion was generalized by various authors, of which we will only mention K\" onigsberger \cite{Konigsberger}, Netto \cite{Netto}, Bauer \cite{Bauer}, Perron \cite{Perron}, Kurshach \cite{Kurschak}, Ore \cite{Ore1}, \cite{Ore2}, \cite{Ore3}, Rella \cite{Rella}, MacLane \cite{MacLane}, and in more recent years Panaitopol and \c Stef\u anescu \cite{PanaitopolStefanescu1}, \cite{PanaitopolStefanescu2}, Mott \cite{Mott}, Brown \cite{RBrown}, Bush and Hajir \cite{BushHajir}, Weintraub \cite{Weintraub} and Jakhar \cite{Jakhar1}, \cite{Jakhar2}.
A natural question is whether for Dirichlet polynomials there exists an analogue of the Sch\"onemann-Eisenstein criterion. The answer is affirmative, and the simplest such result that we can prove by ignoring the conditions in Propositions \ref{prop2} and \ref{prop3} is the following one.
\begin{theorem}\label{naiveEisenstein}
Let $f(s)=\frac{a_{m}}{m^{s}}+\cdots +\frac{a_{n}}{n^{s}}$ be an algebraically primitive Dirichlet polynomial with coefficients in a unique factorization domain $R$, with $a_{m}a_{n}\neq 0$. If for a prime element $p$ of $R$ we either have 

i) \ $p\mid a_{i}$ for each $i=m,\dots ,n-1$, $p\nmid a_{n}$ and $p^{2}\nmid a_{m}$, \ or

ii) $p\mid a_{i}$ for each $i=m+1,\dots ,n$, $p\nmid a_{m}$ and $p^{2}\nmid a_{n}$,

\noindent then $f$ is irreducible over $Q(R)$, the quotient field of $R$.
\end{theorem}
\begin{proof}\ i) Assume to the contrary that $f(s)=g(s)h(s)$ with $g(s)=\frac{b_{r}}{r^{s}}+\cdots +\frac{b_{t}}{t^{s}}$ and $h(s)=\frac{c_{u}}{u^{s}}+\cdots +\frac{c_{v}}{v^{s}}$, so $b_{r}c_{u}=a_{m}$ and $b_{t}c_{v}=a_{n}$. Besides, since $f$ is algebraically primitive, we must have $r<t$ and $u<v$. Since $p^{2}\nmid a_{m}$, only one of $b_{r}$ and $c_{u}$ is divisible by $p$, say $p\mid b_{r}$ and $p\nmid c_{u}$. On the other hand, as $p\nmid a_{n}$, none of $b_{t}$ and $c_{v}$ is divisible by $p$, so it makes sense to consider the least index $k$ such that $p\nmid b_{k}$. This index $k$ should therefore satisfy the inequalities $r<k\leq t$. Let us consider now the coefficient $a_{k\cdot u}$. By the multiplication rule we have
\[
a_{k\cdot u}=\sum\limits _{i\cdot j=k\cdot u}b_{i}c_{j}.
\]
We observe that for a term $b_{i}c_{j}$ in this sum, we must have $i\geq r$, $j\geq u$ and $i\cdot j=k\cdot u$, so $i$ cannot exceed $k$. Therefore, except for $b_{k}c_{u}$, which is not divisible by $p$, in all the remaining terms $b_{i}c_{j}$ (if any) the index $i$ is less than $k$, so $b_{i}$ must be divisible by $p$. This shows that $p\nmid a_{k\cdot u}$, which contradicts our hypothesis that $p\mid a_{i}$ for each $i=m,\dots ,n-1$, as $ku<kv\leq tv=n$;

ii) In this case, since $p^{2}\nmid a_{n}$, only one of $b_{t}$ and $c_{v}$ is divisible by $p$, say $p\mid b_{t}$ and $p\nmid c_{v}$. On the other hand, as $p\nmid a_{m}$, none of $b_{r}$ and $c_{u}$ is divisible by $p$, so it makes sense to consider this time the largest index $k$ such that $p\nmid b_{k}$, which will satisfy the inequalities $r\leq k<t$. We will consider this time the coefficient $a_{k\cdot v}$, given by
\[
a_{k\cdot v}=\sum\limits _{i\cdot j=k\cdot v}b_{i}c_{j}.
\]
For a term $b_{i}c_{j}$ in the above sum, we must have $i\leq t$, $j\leq v$ and $i\cdot j=k\cdot v$, so $i\geq k$. Therefore, except for $b_{k}c_{v}$, which is not divisible by $p$, in all the remaining terms $b_{i}c_{j}$ (if any) the index $i$ must exceed $k$, so $b_{i}$ must be divisible by $p$. This shows that $p\nmid a_{k\cdot v}$, which contradicts our hypothesis that $p\mid a_{i}$ for each $i=m+1,\dots ,n$, as $kv>ku\geq ru=m$.
This completes the proof of the theorem. 
\end{proof}

As we shall see in the following section, Theorem \ref{naiveEisenstein} is a corollary of a more general result obtained by using logarithmic Newton polygons, namely Theorem \ref{calaDumas}, which is an analogue for Dirichlet polynomials of the classical irreducibility criterion of Dumas for polynomials. 

By combining Theorem \ref{naiveEisenstein} and Propositions \ref{prop2} and \ref{prop3}, one can find similar irreducibility criteria that avoid imposing divisibility conditions to terms $a_{i}$ with $i$ close to $m$ and $n$.

\section{Logarithmic Newton polygons and an analogue of \\ Dumas's irreducibility criterion} \label{NewtonLogPol}

In this section we will construct a logarithmic version of the Newton polygon, which is an analogue for algebraically primitive Dirichlet polynomials of the classical Newton polygon for polynomials, and we will refer to it as the {\it logarithmic Newton polygon}, or briefly the {\it Newton log-polygon}. 
First of all, we will fix an arbitrarily chosen real number $\mathfrak{b}$ greater than $1$, which will be the base our logarithms will be considered to (say $\mathfrak{b}=2$, or the base $e$ of natural logarithms, or $10$). For a given base $\mathfrak{b}$, we will call {\it log-integral} any point having coordinates $(\log _{\mathfrak{b}}x,\ y)$ with $x$ a positive integer, and $y$ an integer. As we shall see later, the base $\mathfrak{b}$ will be in some sense irrelevant, so unless for some reason there is a need to use a particular base, to keep notations as simple as possible, we choose to work with natural logarithms. It is however worth mentioning that the construction of the Newton log-polygon below can be done using any base $\mathfrak{b}>1$ for our logarithms.

Now let $p$ be a fixed, prime element of $R$, and let $f(s)=\frac{A_{m}}{m^{s}}+\cdots +\frac{A_{n}}{n^{s}}$ be a Dirichlet polynomial with coefficients in $R$, $A_{m}A_{n}\neq 0$, $1\leq m<n$. Let us represent the nonzero coefficients of $f$ as $A_{i}=a_{i}p^{\alpha _{i}}$ with $a_{i}\in R$ and $p\nmid a_{i}$. We will sometimes denote $\alpha _i$ by $\nu_p(A_i)$. To each nonzero coefficient $a_{i}p^{\alpha _{i}}$ we will associate a point in the plane with coordinates $(\log i,\alpha _{i})$. The lower convex hull of the points $(\log m,\alpha _{m}),\dots ,(\log n,\alpha _{n})$ is an analogue for $f$ of the classical Newton polygon, and is constructed as follows: 
Let $i_{1}=m$, $P_{1}=(\log i_{1},\alpha _{i_{1}})=(\log m,\alpha _{m})$, and let $P_{2}=(\log i_{2},\alpha _{i_{2}})$ with $i_{2}$ the largest integer with the property that there are no points $(\log i,\alpha _{i})$ lying below the line passing through $P_{1}$ and $P_{2}$. Then let $P_{3}=(\log i_{3},\alpha _{i_{3}})$, with $i_{3}$ the largest integer such that there are no points $(\log i,\alpha _{i})$ lying below the line passing through $P_{2}$ and $P_{3}$, and so on. The last line segment built in this way will be $P_{r-1}P_{r}$, with $P_{r}=(\log n,\alpha _{n})$. The broken line $P_{1}P_{2}\dots P_{r}$ thus constructed will be called the {\it Newton log-polygon of} $f$ with respect to the prime element $p$. We will refer to the log-integral points $P_{1},\dots ,P_{r}$ as the {\it vertices} of the Newton log-polygon of $f$, and the Newton log-polygon of $f$ is the lower convex hull of its vertices $P_{1},\dots ,P_{r}$.
The line segments $P_{l}P_{l+1}$ will be called {\it edges}. An edge $P_{l}P_{l+1}$ might pass through some log-integral points, other than its endpoints $P_{l}$ and $P_{l+1}$. These points have coordinates of the form $(\log i, x_{i})$ with $i\in \mathbb{Z_{+}}$ and $x_{i}\in \mathbb{N}$, with $x_{i}$ not necessarily equal to $\alpha _{i}$. Let us denote all these intermediate points by $Q_{1},\dots ,Q_{k}$, say. Then $P_{l}Q_{1}$, $Q_{1}Q_{2}$, $\dots $, $Q_{k-1}Q_{k}$ and $Q_{k}P_{l+1}$ will be called {\it segments} of the Newton log-polygon. Thus, a segment will have no log-integral points other than its endpoints, and an edge will consist of a number of segments, possibly just one, in which case the edge itself will be called a segment.  Besides, any two edges must have different slopes, while two segments that belong to the same edge will obviously have the same slope. The vectors $\overline{P_{i}P_{i+1}}$ will be called {\it vectors} of the edges of the Newton log-polygon, and {\it the vector system} of the Newton log-polygon will be the union of all the vectors of its edges.

Recall that $\frac{n}{m}$ is called the {\it relative degree} of $f$. Since there is no risc of confusion, for an edge (segment) $AB$ of the Newton log-polygon of $f$ having endpoints $A=(\log i, x_{i})$ and $B=(\log j, x_{j})$ with $i<j$, the rational number $\frac{j}{i}>1$ will be called the {\it relative degree} of the edge (segment) $AB$, with $\log \frac{j}{i}$ being the {\it width} of $AB$. Thus, the relative degree of an edge is the product of the relative degrees of all of its segments, while the relative degree of $f$ is the product of the relative degrees of all the edges appearing in the Newton log-polygon of $f$.

One can also use this construction for Dirichlet polynomials $f$ that are not algebraically primitive, but doing so will only result in a translation on the $x$-axis of the Newton log-polygon of the algebraically primitive part of $f$ by some log-integral value. 

The reader may naturally wonder how difficult will be to plot the lower convex hull of some set of log-integral points. Fortunately the Newton log-polygon is in some sense a virtual construction, that requires no precision at all when plotting log-integral points in the plane, not even approximations of the logarithms that we use. The reason is that, actually, the construction of the vertices $P_{1},\dots ,P_{r}$ depends only on some diophantine inequalities, and the identification of all the intermediate log-integral points lying on the edges $P_{1}P_{2},\dots ,P_{r-1}P_{r}$ only requires solving some diophantine equations. Indeed, if we choose two indices $i,j$ with $i<j$, we see that a log-integral point $(\log k,\nu_{p}(A_{k}))$ lies over or on the line passing through the log-integral points $(\log i,\nu_{p}(A_{i}))$ and $(\log j,\nu_{p}(A_{j}))$ if and only if
\begin{equation}\label{deasupra}
j^{\nu _{p}(A_{k})-\nu _{p}(A_{i})}\geq k^{\nu _{p}(A_{j})-\nu _{p}(A_{i})}\cdot i^{\nu _{p}(A_{k})-\nu _{p}(A_{j})},
\end{equation}
while a log-integral point $(\log k,y)$ with $y$ and $k$ integers with $i<k<j$, lies on this line if and only if we have the equality
\begin{equation}\label{exactpe}
j^{y-\nu _{p}(A_{i})}=k^{\nu _{p}(A_{j})-\nu _{p}(A_{i})}\cdot i^{y-\nu _{p}(A_{j})}.
\end{equation}
The Newton log-polygon will help us visualize and gain geometric insight on some factorization properties, which are arithmetic in nature, and often reduce to conditions such as (\ref{deasupra}) and (\ref{exactpe}) that do not depend on the base of the logarithms that we use.

The total number of log-integral points lying on a line segment with endpoints that are log-integral too will be given later in Lemma \ref{puncte}, which will be fundamental in establishing effective irreducibility conditions for our Dirichlet polynomials. 

Let us consider an example for the case that $p=\mathfrak{b}=2$. So let us take for instance $f(s)=\frac{2^{2}}{4^{s}}+\frac{2^{2}}{6^{s}}+\frac{2}{8^{s}}+\frac{1}{9^{s}}+\frac{2^{2}}{ 10^{s}}+\frac{1}{12^{s}}+\frac{2}{15^{s}}$, and observe that $f=g\cdot h$ with 
$g(s)=\frac{2}{2^{s}}+\frac{1}{3^{s}}+\frac{1}{4^{s}}+\frac{2}{5^{s}}$ and $h(s)=\frac{2}{ 2^{s}}+\frac{1}{3^{s}}$. Now let us plot in Figure 1 the Newton log-polygon of $f$, and in Figure 2 the Newton log-polygons of its factors $g$ and $h$. We note that the point $Q_{1}=(\log _{2}6,1)$ is a log-integral point that belongs to the line segment $P_{1}P_{2}$.
We also observe that the edges of the Newton log-polygon of $f$ are obtained by translates of the edges of the Newton polygons of $g$ and $h$, so their slopes and lengths are not altered, and the translates of the edges with the same slope (which are the edges with negative slope $\log _{2}\frac{2}{3}$ in the Newton log-polygons of $g$ and $h$) are ``glued" together in $Q_{1}$. As we shall see immediately, this is not just a coincidence.

\begin{center}
\hspace{2.5cm}
\setlength{\unitlength}{8mm}
\begin{picture}(12,6)
\linethickness{0.15mm}

\multiput(0,0)(0,2){3}{\line(1,0){9}}

\put(0,0){\vector(1,0){10}}
\put(0,0){\vector(0,1){5}}

\thicklines

\put(0,4){\line(3,-2){6}} 
\put(6,0){\line(1,0){1.75}}
\put(7.65,0){\line(2,3){1.3}}  
\linethickness{0.1mm}

\put(3,0){\line(0,1){4}}
\put(4.75,0){\line(0,1){4}}
\put(6,0){\line(0,1){4}}
\put(6.8,0){\line(0,1){4}}
\put(7.65,0){\line(0,1){4}}
\put(8.97,0){\line(0,1){4}}

{\tiny 
\put(-0.5,1.85){$1$}
\put(-0.5,3.85){$2$}

\put(-0.75,-0.5){$(\log _{2}4,0)$}
\put(2.4,-0.5){$\log _{2}6$}
\put(4,-0.5){$\log _{2}8$}
\put(5.5,-0.5){$\log _{2}9$}
\put(7.15,-0.5){$\log _{2}12$}
\put(8.65,-0.5){$\log _{2}15$}

\put(0.25,4.25){$P_{1}$}
\put(3.25,2.25){$Q_{1}$}
\put(6.15,0.25){$P_{2}$}
\put(7.15,0.25){$P_{3}$}
\put(8.35,2.25){$P_{4}$}

}
\put(0,4){\circle{0.08}}
\put(0,4){\circle{0.12}}

\put(3,4){\circle{0.08}}
\put(3,4){\circle{0.12}}

\put(4.75,2){\circle{0.08}}
\put(4.75,2){\circle{0.12}}

\put(6,0){\circle{0.08}}
\put(6,0){\circle{0.12}}

\put(6.8,4){\circle{0.08}}
\put(6.8,4){\circle{0.12}}

\put(7.65,0){\circle{0.08}}
\put(7.65,0){\circle{0.12}}

\put(8.97,2){\circle{0.08}}
\put(8.97,2){\circle{0.12}}

\end{picture}
\end{center}
\bigskip

{\small
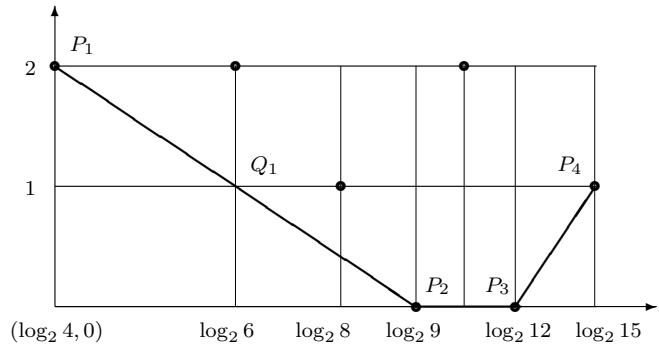
{\bf Figure 1.} The Newton log-polygon for $f(s)=\frac{2^{2}}{4^{s}}+\frac{2^{2}}{6^{s}}+\frac{2}{8^{s}}+\frac{1}{9^{s}}+\frac{2^{2}}{ 10^{s}}+\frac{1}{12^{s}}+\frac{2}{15^{s}}$ with respect to the prime number $p=2$ 
}

\begin{center}
\hspace{1cm}
\setlength{\unitlength}{8mm}
\begin{picture}(12,4)
\linethickness{0.15mm}

\multiput(0,0)(0,2){2}{\line(1,0){6}}

\multiput(8,0)(0,2){2}{\line(1,0){3}}

\put(0,0){\vector(1,0){7}}
\put(0,0){\vector(0,1){3}}

\put(8,0){\vector(1,0){4}}
\put(8,0){\vector(0,1){3}}

\thicklines

\put(0,2){\line(3,-2){3}} 
\put(3,0){\line(1,0){1.7}}
\put(4.67,0){\line(2,3){1.3}}  

\put(8,2){\line(3,-2){3}} 
\linethickness{0.1mm}

\put(3,0){\line(0,1){2}}
\put(4.67,0){\line(0,1){2}}
\put(5.97,0){\line(0,1){2}}

\put(11,0){\line(0,1){2}}
{\tiny 
\put(-0.5,1.85){$1$}

\put(7.5,1.85){$1$}

\put(-0.75,-0.5){$(\log _{2}2,0)$}
\put(2.4,-0.5){$\log _{2}3$}
\put(4.1,-0.5){$\log _{2}4$}
\put(5.4,-0.5){$\log _{2}5$}

\put(7.3,-0.5){$(\log _{2}2,0)$}
\put(10.6,-0.5){$\log _{2}3$}

}
\put(0,2){\circle{0.08}}
\put(0,2){\circle{0.12}}

\put(3,0){\circle{0.08}}
\put(3,0){\circle{0.12}}

\put(4.67,0){\circle{0.08}}
\put(4.67,0){\circle{0.12}}

\put(5.98,2){\circle{0.08}}
\put(5.98,2){\circle{0.12}}

\put(8,2){\circle{0.08}}
\put(8,2){\circle{0.12}}

\put(11,0){\circle{0.08}}
\put(11,0){\circle{0.12}}

\end{picture}
\end{center}
\bigskip

{\small
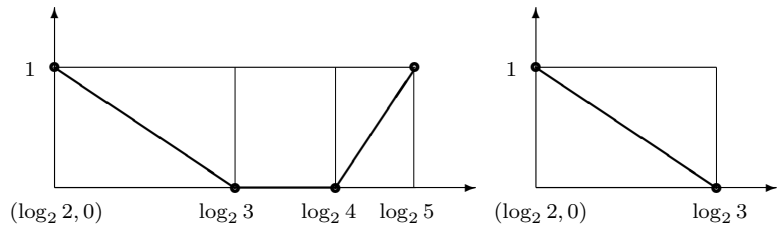
{\bf Figure 2.} The Newton log-polygons for $g(s)=\frac{2}{2^{s}}+\frac{1}{3^{s}}+\frac{1}{4^{s}}+\frac{2}{5^{s}}$ and $h(s)=\frac{2}{ 2^{s}}+\frac{1}{3^{s}}$
 with respect to the prime number $p=2$ 
}
\medskip

The following result, which will be fundamental to the theory of Newton log-polygons, is an analogue for Dirichlet polynomials of the famous Theorem of Dumas \cite{Dumas} for polynomials (see also Prasolov \cite[Th. 2.2.1]{Prasolov} for a short proof of Dumas's Theorem).

\begin{theorem}\label{DumDir}
Let $p\in R$ be a prime element, $g,h$ nonconstant algebraically primitive Dirichlet polynomials with coefficients in $R$, and let $f=g\cdot h$. Then the Newton log-polygon of $f$ with respect to $p$ is obtained by using translates of the edges of the Newton log-polygons of $g$ and $h$ with respect to $p$, using precisely one translate for each edge, so as to form a polygonal line with increasing slopes, when considered from left to the right.
\end{theorem}

In other words, this theorem says that the vector system of the Newton log-polygon of $f$ is obtained by considering the union of the vector systems of the Newton log-polygons of $g$ and $h$, repetitions allowed, and then replacing any two vectors having the same orientation (slope) by their sum.

\begin{proof}\ We will adapt here some of the notations and ideas in the proof of Dumas's Theorem given in \cite{Prasolov}, and the proof will be given in full detail. 

Let $f(s)=\sum\limits _{i=n'}^{n}\frac{a_{i}p^{\alpha _{i}}}{i^{s}}$, $g(s)=\sum\limits _{j=d'}^{d}\frac{b_{j}p^{\beta _{j}}}{j^{s}}$ and $h(s)=\sum\limits _{k=n'/d'}^{n/d}\frac{c_{k}p^{\gamma _{k}}}{k^{s}}$ with $a_{i},b_{j}$ and $c_{k}$ not divisible by $p$. Let us consider an edge of the Newton log-polygon of $f$, say $P_{l}P_{l+1}$ (that may consist of more than a single segment), and let us assume that the points $P_{l}$ and $P_{l+1}$ have coordinates $(\log i_{m},\alpha _{i_{m}})$, and $(\log i_{M},\alpha _{i_{M}})$, respectively, with $i_{m}<i_{M}$. The slope of $P_{l}P_{l+1}$ is
\[
S=\frac{\alpha _{i_{M}}-\alpha _{i_{m}}}{\log i_{M}-\log i_{m}}.
\]
Let now $\alpha _{i_{M}}-\alpha _{i_{m}}=A$ and $\log i_{M}-\log i_{m}=I$, so $S=\frac{A}{I}$. The edge $P_{l}P_{l+1}$ of the Newton log-polygon of $f$ belongs to the line $Iy -Ax=F$, where
\[
F=I\alpha _{i_{M}}-A\log i_{M}=I\alpha _{i_{m}}-A\log i_{m}.
\] 
According to the definition, all the points $(\log i,\alpha _{i})$, $i=n',\dots ,n$, either lie on this line, or above it, so $I\alpha _{i}-A\log i\geq F$ for all $i$, with a strict inequality for $i<i_{m}$ and $i>i_{M}$, and we have equality in the endpoints corresponding to $i_{m}$ and $i_{M}$ (and for the points $(\log i, \alpha _{i})$ with $i_{m}<i<i_{M}$, if any, both cases ``$>$" and ``=" being possible).

To each term $\frac{a_{i}p^{\alpha _{i}}}{i^{s}}$ we will associate a real number
\[
w\Bigl(\frac{a_{i}p^{\alpha _{i}}}{i^{s}}\Bigr):=I\alpha _{i}-A\log i,
\]
called its {\it weight} with respect to the edge $P_{l}P_{l+1}$. We may then regard $i_{m}$ and $i_{M}$ as the least and the largest $i$ such that the corresponding term $\frac{a_{i}p^{\alpha _{i}}}{i^{s}}$ of $f$ has minimum weight (which is $F$) with respect to the edge $P_{l}P_{l+1}$, which shows that $i_{m}$ and $i_{M}$ are in this way uniquely determined.
For the Dirichlet polynomial $g$ we will define now 
\[
G:=\underset{d'\leq j\leq d}{\rm min}\{I\beta _{j}-A\log j\} 
\]
and we will also define $j_{m}$ and $j_{M}$ as the least and the largest index, respectively, such that
\begin{equation}\label{GDirichlet}
G=I\beta _{j_{m}}-A\log j_{m}=I\beta _{j_{M}}-A\log j_{M}.
\end{equation}  
We notice that $j_{m}\leq j_{M}$ (while  $i_{m}<i_{M}$, here we don't necessarily have a strict inequality). Similarly, for $h$ we define
\[
H:=\underset{n'/d'\leq k\leq n/d}{\rm min}\{I\gamma _{k}-A\log k\} 
\]
and then we denote by $k_{m}$ and $k_{M}$ the least and the largest index, respectively, such that
\begin{equation}\label{HDirichlet}
H=I\gamma _{k_{m}}-A\log k_{m}=I\gamma _{k_{M}}-A\log k_{M}.
\end{equation}  
Here too, we have $k_{m}\leq k_{M}$ (not necessarily a strict inequality). By the multiplication rule for Dirichlet series, we have
\begin{equation}\label{regulaDirichlet}
\frac{a_{j_{m}k_{m}}p^{\alpha _{j_{m}k_{m}}}}{(j_{m}k_{m})^{s}}=\sum\limits _{j\cdot k=j_{m}\cdot k_{m}}\frac{b_{j}p^{\beta _{j}}}{j^{s}}\cdot \frac{c_{k}p^{\gamma _{k}}}{k^{s}}.
\end{equation}  
We observe that for two terms that are multiplied in (\ref{regulaDirichlet}) we have
\begin{eqnarray*}
w\Bigl(\frac{b_{j}p^{\beta _{j}}}{j^{s}}\cdot \frac{c_{k}p^{\gamma _{k}}}{k^{s}}\Bigr) & = & w\Bigl(\frac{b_{j}c_{k}p^{\beta _{j}+\gamma _{k}}}{(jk)^{s}}\Bigr)=I(\beta _{j}+\gamma _{k})-A\log (jk)\\
& = & I(\beta _{j}+\gamma _{k})-A(\log j+\log k)=
w\Bigl(\frac{b_{j}p^{\beta _{j}}}{j^{s}}\Bigr)+w\Bigl(\frac{c_{k}p^{\gamma _{k}}}{k^{s}}\Bigr),
\end{eqnarray*}
so the weight of the summand with $j=j_{m}$ and $k=k_{m}$ in (\ref{regulaDirichlet}) is $G+H$, while the weights of the rest of the summands all exceed $G+H$, since for these ones we either have $j<j_{m}$, or $k<k_{m}$. Indeed, let us assume that $j<j_{m}$ and $k>k_{m}$, for instance. Then
$w(\frac{b_{j}p^{\beta _{j}}}{j^{s}})>G$ and $w(\frac{c_{k}p^{\gamma _{k}}}{k^{s}})\geq H$, so
\[
w\Bigl(\frac{b_{j}p^{\beta _{j}}}{j^{s}}\cdot \frac{c_{k}p^{\gamma _{k}}}{k^{s}}\Bigr) >G+H
\]
(and similarly for $k<k_{m}$). Besides, we notice that when $j\cdot k$ is constant, the weight of a product $\frac{b_{j}p^{\beta _{j}}}{j^{s}}\cdot \frac{c_{k}p^{\gamma _{k}}}{k^{s}}$ regarded as a function on $\beta _{j}+\gamma _{k}$ is a strictly increasing function, as $I>0$. Therefore, since in (\ref{regulaDirichlet}) we have $j\cdot k=j_{m}\cdot k_{m}$, hence $j\cdot k$ is constant, the sum $\beta _{j}+\gamma _{k}$ attains its minimum for $j=j_{m}$ and $k=k_{m}$, and is strictly larger for any other pair $(j,k)$. Thanks to (\ref{regulaDirichlet}) we see now that
\begin{equation}\label{exactDirichlet}
w\Bigl(\frac{a_{j_{m}k_{m}}p^{\alpha _{j_{m}k_{m}}}}{(j_{m}k_{m})^{s}}\Bigr)=G+H.
\end{equation}
Moreover, reasoning similarly with (\ref{regulaDirichlet}) replaced by
\begin{equation}\label{regulageneralaDirichlet}
\frac{a_{i}p^{\alpha _{i}}}{i^{s}}=\sum\limits _{j\cdot k=i}\frac{b_{j}p^{\beta _{j}}}{j^{s}}\cdot \frac{c_{k}p^{\gamma _{k}}}{k^{s}},
\end{equation}  
we will prove that
\begin{eqnarray}
w\Bigl(\frac{a_{i}p^{\alpha _{i}}}{i^{s}}\Bigr) & > & G+H\qquad
{\rm for}\  i<j_{m}\cdot k_{m},\label{caz1Dirichlet}\\
w\Bigl(\frac{a_{i}p^{\alpha _{i}}}{i^{s}}\Bigr) & \geq & G+H\qquad
{\rm for}\  i\geq j_{m}\cdot k_{m}.\label{caz2Dirichlet}
\end{eqnarray}
Indeed, in view of (\ref{regulageneralaDirichlet}) we see that
\begin{equation}\label{minimDirichlet}
\alpha _{i}\geq \underset{j\cdot k=i}{\rm min}(\beta _{j}+\gamma _{k}),
\end{equation}
and let us assume that this minimum is attained for a pair $(j_{0},k_{0})$ with $j_{0}\cdot k_{0}=i<j_{m}\cdot k_{m}$. Then at least one of the inequalities $j_{0}<j_{m}$ and $k_{0}<k_{m}$ must hold, say $j_{0}<j_{m}$. In this case we have $w(\frac{b_{j_{0}}p^{\beta _{j_{0}}}}{j_{0}^{s}})>G$ and $w(\frac{c_{k_{0}}p^{\gamma _{k_{0}}}}{k_{0}^{s}})\geq H$, so $w(\frac{b_{j_{0}}p^{\beta _{j_{0}}}}{j_{0}^{s}}\cdot \frac{c_{k_{0}}p^{\gamma _{k_{0}}}}{k_{0}^{s}})>G+H$.
According to (\ref{minimDirichlet}) we then deduce that
\begin{eqnarray*}
w\Bigl(\frac{a_{i}p^{\alpha _{i}}}{i^{s}}\Bigr) & = & I\alpha _{i}-A\log i\geq
I(\beta _{j_{0}}+\gamma _{k_{0}})-A\log (j_{0}\cdot k_{0})\\
& = &
 w\Bigl(\frac{b_{j_{0}}c_{k_{0}}p^{\beta _{j_{0}}+\gamma _{k_{0}}}}{(j_{0}\cdot k_{0})^{s}}\Bigr)>G+H.
\end{eqnarray*}
Assume now that the minimum in (\ref{minimDirichlet}) is attained for a pair $(j_{0},k_{0})$ with $j_{0}\cdot k_{0}=i\geq j_{m}\cdot k_{m}$. In this case we might have $j_{0}=j_{m}$ and $k_{0}=k_{m}$, so we only get
\[
w\Bigl(\frac{a_{i}p^{\alpha _{i}}}{i^{s}}\Bigr) =I\alpha _{i}-A\log i\geq w\Bigl(\frac{b_{j_{0}}c_{k_{0}}p^{\beta _{j_{0}}+\gamma _{k_{0}}}}{i^{s}}\Bigr)\geq G+H.
\]
Summarizing, by (\ref{exactDirichlet}), (\ref{caz1Dirichlet}) and (\ref{caz2Dirichlet}), we must have
\[
G+H=F\qquad {\rm and}\qquad j_{m}\cdot k_{m}=i_{m}.
\]
Similarly, we also obtain $j_{M}\cdot k_{M}=i_{M}$. Taking the logaritm, one obtains
\begin{equation}\label{sumaDirichlet}
\log i_{M}-\log i_{m}=(\log j_{M}-\log j_{m})+(\log k_{M}-\log k_{m}).
\end{equation}
In particular, we see that at least one of the numbers $\log j_{M}-\log j_{m}$ and $\log k_{M}-\log k_{m}$ must be nonzero (hence positive). If both $\log j_{M}-\log j_{m}$ and $\log k_{M}-\log k_{m}$ are nonzero, then the geometric segment with endpoints $(\log j_{m},\beta _{j_{m}})$ and $(\log j_{M},\beta _{j_{M}})$ must be an edge of the Newton log-polygon of $g$, and the geometric segment with endpoints $(\log k_{m},\gamma _{k_{m}})$ and $(\log k_{M},\gamma _{k_{M}})$ must be an edge of the Newton log-polygon of $h$, and moreover, the slopes of these two edges will be the same, namely $S=\frac{A}{I}$, since from (\ref{GDirichlet}) and (\ref{HDirichlet}) one obtains
\begin{equation}\label{pantaDirichlet}
\frac{\beta _{j_{M}}-\beta _{j_{m}}}{\log j_{M}-\log j_{m}}=\frac{A}{I}=\frac{\gamma _{k_{M}}-\gamma _{k_{m}}}{\log k_{M}-\log k_{m}}.
\end{equation}
Indeed, the fact that the geometric segment with endpoints $(\log j_{m},\beta _{j_{m}})$ and $(\log j_{M},\beta _{j_{M}})$ is an edge of the Newton log-polygon of $g$ follows since, according to the definition of $G$ and of the indices $j_{m}$ and $j_{M}$, for $j<j_{m}$ and $j>j_{M}$ we have $I\beta _{j}-A\log j>G$, which shows that there are no points $(\log j,\beta _{j})$ lying under the line determined by this geometric segment. Moreover, the slope of the geometric segment with endpoints $(\log j_{m},\beta _{j_{m}})$ and $(\log j_{M},\beta _{j_{M}})$ is precisely $\frac{A}{I}$. Absolutely similar, one proves that the geometric segment with endpoints $(\log k_{m},\gamma _{k_{m}})$ and $(\log k_{M},\gamma _{k_{M}})$ is an edge of the Newton log-polygon of $h$. Moreover, if one of the numbers $\log j_{M}-\log j_{m}$ and $\log k_{M}-\log k_{m}$ is zero, say $\log k_{M}-\log k_{m}=0$, in relation (\ref{pantaDirichlet}) only the first equality will make sense, but this will still be sufficient to argue in the same way that the segment with endpoints $(\log j_{m},\beta _{j_{m}})$ and $(\log j_{M},\beta _{j_{M}})$ is an edge of the Newton log-polygon of $g$. Thus the conclusion remains valid for $g$ in this case too, the only difference being the fact that in this case there no longer exists an edge of slope $\frac{A}{I}$ in the Newton log-polygon of $h$.

Relation (\ref{sumaDirichlet}) is an equality of lengths of projections on the $x$-axis of some line segments having the same slope (according to (\ref{pantaDirichlet})), so this relation also shows us that the sum of the lengths of the edges with slope $S$ in the Newton log-polygons of $g$ and $h$ is precisely the length of the edge with slope $S$ in the Newton log-polygon of $f$. If one of the numbers $\log j_{M}-\log j_{m}$ and $\log k_{M}-\log k_{m}$ is zero, then the Newton log-poligon of one of $g$ and $h$ has an edge with slope $S$ of the same length as the edge with slope $S$ in the Newton log-polygon of $f$, while the Newton log-polygon of the other factor will not contain any edge with slope $S$. Thus, the vector of the edge with slope $S$ in the Newton log-polygon of $f$ is the sum of the vectors of the edges with slope $S$ in the Newton log-polygons of $g$ and $h$. Besides, the fact that $\log n=\log d+\log n/d$ and $\log n'=\log d'+\log n'/d'$ combined with (\ref{sumaDirichlet}) and (\ref{pantaDirichlet}) shows us that if one of the factors $g$ and $h$ has an edge with slope $S$ in its Newton log-polygon, then $S$ will necessarily appear among the slopes of the edges in the Newton log-polygon of $f$.
Indeed, let us denote by $Lf_{1},\dots ,Lf_{r}$ the lengths of the projections on the $x$-axis of the edges of the Newton log-polygon of $f$, and with $S_{1},\dots ,S_{r}$ the slopes of these edges, respectively. Similarly, let $Lg_{1},\dots ,Lg_{r}$ and $Lh_{1},\dots ,Lh_{r}$ be the lengths of the projections on the $x$-axis of the edges with slopes $S_{1},\dots ,S_{r}$, respectively, in the Newton log-polygons of $g$ and $h$ (some of these lengths possibly being zero). We then have
\begin{equation}\label{egalitDirichlet}
Lf_{i}=Lg_{i}+Lh_{i}, \quad i=1,\dots ,r,
\end{equation}
where for a fixed $i$ we have $Lg_{i}\geq 0$ and $Lh_{i}\geq 0$, with at least one of $Lg_{i}$ and $Lh_{i}$ nonzero. 
Denote now
\[
W_{f}:=\sum\limits _{i=1}^{r}Lf_{i},\quad W_{g}:=\sum\limits _{i=1}^{r}Lg_{i},\quad 
W_{h}:=\sum\limits _{i=1}^{r}Lh_{i}.
\]
We have $W_{f}=\log n-\log n'$, and since the Newton log-polygons of $g$ and $h$ might in principle contain edges of slopes different from $S_{1},\dots ,S_{r}$, we deduce that $W_{g}\leq \log d-\log d'$ and $W_{h}\leq \log n-\log n'-(\log d-\log d')$. On the other hand by (\ref{sumaDirichlet}) one obtains by summation that $W_{f}=W_{g}+W_{h}$, which forces the equalities $W_{g}=\log d-\log d'$ and $W_{h}=\log n-\log n'-(\log d-\log d')$. This shows that in fact, the Newton log-polygons of $g$ and $h$ cannot have edges with slopes other than $S_{1},\dots ,S_{r}$. This completes the proof of the theorem. 
\end{proof}

In particular, this theorem shows that each nonconstant factor of a Dirichlet polynomial $f$ must contribute at least one edge to the Newton log-polygon of $f$. Therefore, as in the case of polynomials we have:
\begin{corollary}\label{coroDumas}
If with respect to a prime element $p$ of $R$, the Newton log-polygon of an algebraically primitive Dirichlet polynomial $f$ consists of a single edge, and this edge contains no log-integral points other than its endpoints, then $f$ must be irreducible.
\end{corollary}
Here it is crucial to make distinction between edges and segments. It is not sufficient to ask the Newton log-polygon of $f$ to have a single edge, for an edge might in principle contain multiple segments coming from the Newton log-polygons of two or more of the nonconstant factors of $f$. That's why in Corollary \ref{coroDumas} we had to ask the edge to contain no log-integral points other than its endpoints.
In order to obtain an effective instance of this corollary, we will need the following lemma, that counts the log-integral points on a line segment having endpoints that are log-integral too, thus allowing one to see when an edge in the Newton log-polygon contains a single segment.

\begin{lemma}\label{puncte}
Let $x_{1},x_{2}$ be positive integers with $x_{1}<x_{2}$, and $p_{1},\dots ,p_{k}$ the distinct prime factors of $x_{1}\cdot x_{2}$. Let also $y_{1},y_{2}$ be two integers with $y_{1}\neq y_{2}$. Then the log-integral points $(\log x,y)$ lying on the line segment with endpoints $P=(\log x_{1},y_{1})$ and $Q=(\log x_{2},y_{2})$ are
\[
\left( \bigl(1-\frac{i}{\delta}\bigr)\log x_{1}+\frac{i}{\delta }\log x_{2}, \ \bigl(1-\frac{i}{\delta }\bigr)y_{1}+\frac{i}{\delta }y_{2}\right) ,\quad i=0,\dots ,\delta ,
\]
with $\delta =\gcd (y_{2}-y_{1},\nu _{p_{1}}(x_{2})-\nu _{p_{1}}(x_{1}),\dots ,\nu _{p_{k}}(x_{2})-\nu _{p_{k}}(x_{1}))$. Thus, the number of line segments lying on $PQ$ that contain no log-integral points other than their endpoints is precisely $\delta $.
\end{lemma}

\begin{proof}\ Without loss of generality we will assume that $y_{1}<y_{2}$. The line passing through $P$ and $Q$ has equation
\begin{equation}\label{dreapta}
Y=\frac{y_{2}-y_{1}}{\log x_{2}-\log x_{1}}\cdot  X+\frac{y_{1}\log x_{2}-y_{2}\log x_{1}}{\log x_{2}-\log x_{1}}.
\end{equation}
Suppose that there exists an integer $x$ with  $x_{1}<x<x_{2}$ and an integer $y$ with $y_{1}<y<y_{2}$ such that the log-integral point $(\log x,y)$ lies on the line with equation (\ref{dreapta}). Then $x$ and $y$ must satisfy the relation
\begin{equation}\label{Pedreapta}
x_{2}^{y-y_{1}}=x_{1}^{y-y_{2}}\cdot x^{y_{2}-y_{1}}.
\end{equation}
Let us consider all the prime numbers that divide at least one of $x_{1}$ and $x_{2}$, say $p_{1},\dots ,p_{k}$. In view of (\ref{Pedreapta}), the prime factors of $x$ are forced to belong to the set $\{ p_{1},\dots ,p_{k}\} $. We may therefore assume without loss of generality that
\[
x=p_{1}^{\alpha _{1}}\cdots p_{k}^{\alpha _{k}},\ \ x_{1}=p_{1}^{\beta _{1}}\cdots p_{k}^{\beta _{k}},\ \ 
x_{2}=p_{1}^{\gamma _{1}}\cdots p_{k}^{\gamma _{k}},
\]
with multiplicities $\alpha _{i}\geq 0$, $\beta _{i}\geq 0$ and $\gamma _{i}\geq 0$, such that for any given $i$, at most one of $\alpha _{i}$, $\beta _{i}$ and $\gamma _{i}$ might be zero, so any prime $p_{1},\dots ,p_{k}$ might be missing in at most one of the prime decompositions of $x$, $x_{1}$ and $x_{2}$. Now, equating the multiplicities in (\ref{Pedreapta}), we deduce that
\[
\gamma _{i}(y-y_{1})=\beta _{i}(y-y_{2})+\alpha _{i}(y_{2}-y_{1}),\quad {\rm for} \ i=1,\dots ,k,
\]
or equivalently, that
\begin{equation}\label{multiplicit}
(\gamma _{i}-\beta _{i})y=(\gamma _{i}-\beta _{i})y_{1}+(\alpha _{i}-\beta _{i})(y_{2}-y_{1}),\quad {\rm for} \ i=1,\dots ,k.
\end{equation}
We will analyse the possible values for $\alpha _{i}$ in (\ref{multiplicit}). We first observe that if for an index $i$ we have $\gamma _{i}=\beta _{i}$, by (\ref{multiplicit}) we must also have $\alpha _{i}=\beta _{i}$, so in this case we have only one possible value for $\alpha _{i}$. As $x_{1}<x_{2}$, we cannot have $\gamma _{i}=\beta _{i}$ for all $i$, so 
there exists a non-empty subset of $\{ p_{1},\dots ,p_{k}\} $, say $\{ p_{1},\dots ,p_{l}\} $ after relabelling, with $\gamma _{1}\neq \beta _{1},\dots ,\gamma _{l}\neq \beta _{l}$. This will allow us to divide by $\gamma _{1}-\beta _{1}$, $\dots $, $\gamma _{l}-\beta _{l}$ in the first $l$ equations in (\ref{multiplicit}), respectively, to obtain
\begin{equation}\label{yul}
y=y_{1}+\frac{\alpha _{i}-\beta _{i}}{\gamma _{i}-\beta _{i}}\cdot (y_{2}-y_{1}),\quad {\rm for}\ i=1,\dots ,l.
\end{equation}
By (\ref{yul}), the rational number $\frac{\alpha _{i}-\beta _{i}}{\gamma _{i}-\beta _{i}}$ must belong to the interval $(0,1)$, and moreover, it must be the same for each $i$. If we write it in lowest terms, say
\[
\frac{\alpha _{i}-\beta _{i}}{\gamma _{i}-\beta _{i}}=\frac{a}{b}\quad {\rm with}\ \ 0<a<b \ \ {\rm and}\ \gcd (a,b)=1,
\] 
we see that the multiplicities $\alpha _{i}$ must be of the form
\begin{equation}\label{liniar}
\alpha _{i}=\beta _{i}+\frac{a}{b}(\gamma _{i}-\beta _{i})\quad {\rm for}\ i=1,\dots ,l.
\end{equation}
For these multiplicities to be integers, $b$ must divide
$\gamma _{i}-\beta _{i}$ for each $i$. On the other hand, $y$ must be an integer too, so by (\ref{yul}) $b$ must also divide $y_{2}-y_{1}$. This forces $b$ to be a divisor of 
\begin{equation}\label{delta}
\delta:=\gcd (y_{2}-y_{1},\gamma _{1}-\beta _{1},\dots ,\gamma _{l}-\beta _{l}).
\end{equation}
Thus, the solutions $(x,y)$ of (\ref{Pedreapta}) with $x_{1}<x<x_{2}$ and $y_{1}<y<y_{2}$ are in a one-to-one correspondence with the rational numbers $\frac{a}{b}$ with $0<a<b$, $\gcd (a,b)=1$ and $b\mid \delta $. The number of these rationals is obviously equal to $\delta -1$, as they can only be obtained by cancelling the common factors of the numerators and denominators in each of the quotients $\frac{d}{\delta }$ with $d=1,\dots ,\delta -1$. We recall that if for an index $j$ we have $\gamma _{j}=\beta _{j}$, then $\alpha _{j}$ will also be equal to $\beta _{j}$, thus satisfying (\ref{liniar}) trivially, and we may obviously add $\gamma _{j}-\beta _{j}$ in the list of terms in the $\gcd $ in (\ref{delta}) without affecting the value of $\delta $. The number of log-integral points lying on the line segment $PQ$ (including $P$ and $Q$) is thus equal to
\[
\gcd (y_{2}-y_{1},\gamma _{1}-\beta _{1},\dots ,\gamma _{k}-\beta _{k})+1,
\]
and their coordinates are
\[
\left( \log p_{1}^{\beta _{1}+\frac{i}{\delta }(\gamma _{1}-\beta _{1})}\cdots p_{k}^{\beta _{k}+\frac{i}{\delta }(\gamma _{k}-\beta _{k})}, \ y_{1}+\frac{i}{\delta }\cdot (y_{2}-y_{1})\right) ,\quad i=0,\dots ,\delta ,
\]
or, equivalently,
\[
\left( \log (x_{1}^{\frac{\delta -i}{\delta }}x_{2}^{\frac{i}{\delta }}), \ y_{1}\frac{\delta -i}{\delta }+y_{2}\frac{i}{\delta }\right) ,\quad i=0,\dots ,\delta , 
\]
with $\delta =\gcd (y_{2}-y_{1},\gamma _{1}-\beta _{1},\dots ,\gamma _{k}-\beta _{k})$.
This completes the proof of the lemma. 
\end{proof}

In \cite{Dumas}, Dumas proved his famous irreducibility criterion for polynomials, imposing explicit conditions for the Newton polygon to consist of a single edge having no lattice points other than its endpoints:
\medskip

{\bf  Theorem (Dumas, 1906)} \ {\em  
Let $f(X)=a_{0}+a_{1}X+\cdots +a_{n}X^{n}$ be a polynomial with coefficients in a unique factorization domain $R$, and let
$p$ be a prime element of $R$. If
\smallskip

i) $\frac{\nu _{p}(a_{i})}{i}>\frac{\nu _{p}(a_{n})}{n}$ for $i=1,\dots ,n-1$,

ii) $\nu _{p}(a_{0})=0$,

iii) $\gcd(\nu _{p}(a_{n}),n)=1$,
\smallskip

\noindent then $f$ is irreducible over $Q(R)$, the quotient field of $R$.}
\medskip

We can now prove the following effective instance of Corollary \ref{coroDumas}, which may be regarded as the analogous for Dirichlet polynomials of Dumas's Theorem. 

\begin{theorem}\label{calaDumas}
Let $f(s)=\frac{a_{m}}{m^{s}}+\cdots +\frac{a_{n}}{n^{s}}$ be an algebraically primitive Dirichlet polynomial with coefficients in a unique factorization domain $R$, $a_{m}a_{n}\neq 0$, let
$p$ be a prime element of $R$, and assume that $q_{1},\dots ,q_{k}$ are all the prime factors of $m\cdot n$. If
\smallskip

i) \ $\nu _{p}(a_{n})\neq \nu _{p}(a_{m})$,

ii) $\left ( \frac{n}{i}\right ) ^{\nu _{p}(a_{i})-\nu _{p}(a_{m})}>\left ( \frac{m}{i}\right ) ^{\nu _{p}(a_{i})-\nu _{p}(a_{n})}$ \ for $m<i<n$,

iii) $\gcd (\nu _{p}(a_{n})-\nu _{p}(a_{m}),\nu _{q_{1}}(n)-\nu _{q_{1}}(m),\dots ,\nu _{q_{k}}(n)-\nu _{q_{k}}(m))=1$,
\smallskip

\noindent then $f$ is irreducible over $Q(R)$, the quotient field of $R$.
\end{theorem}

\begin{proof}\ We will prove that the Newton log-polygon of $f$ consists of a single edge joining the points $P=(\log m,\nu _{p}(a_{m}))$ and $Q=(\log n,\nu _{p}(a_{n}))$. To do so, we first note that the line passing through $P$ and $Q$ has equation
\[
y=\frac{\nu _{p}(a_{n})-\nu _{p}(a_{m})}{\log n-\log m}\cdot  x+\frac{\nu _{p}(a_{m})\log n-\nu _{p}(a_{n})\log m}{\log n-\log m}.
\]
One may easily check now that condition \emph{ii)} forces all the points $(\log i,\nu _{p}(a_{i}))$ with $m<i<n$ to lie above this line. To prove that the edge $PQ$ is actually a segment, we need to check that the only log-integral points lying on $PQ$ are its endpoints $P$ and $Q$. By Lemma \ref{puncte} with $x_{1}=m$, $x_{2}=n$, $y_{1}=\nu _{p}(a_{m})$, $y_{2}=\nu _{p}(a_{n})$, the total number of log-integral points on the line segment $PQ$ is $1+\gcd (\nu _{p}(a_{n})-\nu _{p}(a_{m}),\nu _{p_{1}}(n)-\nu _{p_{1}}(m),\dots ,\nu _{p_{k}}(n)-\nu _{p_{k}}(m))$,
which by \emph{iii)} is equal to $2$, so $PQ$ is indeed a segment, hence by Corollary (\ref{coroDumas}), $f$ must be irreducible.
\end{proof}

Theorem \ref{naiveEisenstein} is now easily seen to be the simplest instance of Theorem \ref{calaDumas}. In the first case of Theorem \ref{naiveEisenstein}, $\left ( \frac{n}{i}\right ) ^{\nu _{p}(a_{i})-\nu _{p}(a_{m})}\geq 1$ and $\left ( \frac{m}{i}\right ) ^{\nu _{p}(a_{i})-\nu _{p}(a_{n})}<1$, while in the second case $\left ( \frac{n}{i}\right ) ^{\nu _{p}(a_{i})-\nu _{p}(a_{m})}> 1$ and $\left ( \frac{m}{i}\right ) ^{\nu _{p}(a_{i})-\nu _{p}(a_{n})}\leq 1$, so condition ii) in Theorem \ref{calaDumas} is satisfied.

Two simple instances of Theorem \ref{calaDumas} are obtained when $m=1$, like in the following two corollaries, corresponding to the cases that $p\nmid a_{1}$, $p\mid a_{n}$, and $p\mid a_{1}$, $p\nmid a_{n}$, respectively.

\begin{corollary}\label{calaDumasSimplu1}
Let $f(s)=\frac{a_{1}}{1^s}+\frac{a_{2}}{2^{s}}+\cdots +\frac{a_{n}}{n^{s}}$ be a Dirichlet polynomial with coefficients in a unique factorization domain $R$, with $a_{1}a_{n}\neq 0$, let
$p$ be a prime element of $R$, and assume that $q_{1},\dots ,q_{k}$ are all the prime factors of $n$. If
\smallskip

i) \ $p\nmid a_{1}$, $p\mid a_{n}$,

ii) \thinspace $n^{\nu _{p}(a_{i})}>i^{\nu _{p}(a_{n})}$ for $1<i<n$,

iii) $\gcd (\nu _{p}(a_{n}),\nu _{q_{1}}(n),\dots ,\nu _{q_{k}}(n))=1$,
\smallskip

\noindent then $f$ is irreducible over $Q(R)$.
\end{corollary}

\begin{corollary}\label{calaDumasSimplu2}
Let $f(s)=\frac{a_{1}}{1^s}+\frac{a_{2}}{2^{s}}+\cdots +\frac{a_{n}}{n^{s}}$ be a Dirichlet polynomial with coefficients in a unique factorization domain $R$, with $a_{1}a_{n}\neq 0$, let
$p$ be a prime element of $R$, and assume that $q_{1},\dots ,q_{k}$ are all the prime factors of $n$. If
\smallskip

i) \ $p\mid a_{1}$, $p\nmid a_{n}$,

ii) \thinspace $n^{\nu _{p}(a_{1})-\nu _{p}(a_{i})}<i^{\nu _{p}(a_{1})}$ for $1<i<n$,

iii) $\gcd (\nu _{p}(a_{1}),\nu _{q_{1}}(n),\dots ,\nu _{q_{k}}(n))=1$,
\smallskip

\noindent then $f$ is irreducible over $Q(R)$.
\end{corollary}
\begin{remark}\label{totalnroffactors}
Theorem \ref{DumDir} combined with Lemma \ref{puncte} allows us to find an upper bound for the total number $n_f$ of irreducible factors of $f$ (counting multiplicities), by counting all the segments on all the edges $P_1Q_1,\dots ,P_{\ell}Q_{\ell}$ in the Newton log-polygon of $f$. Thus, we have $n_f\leq\delta _1+\cdots +\delta _{\ell}$, where for each $i$, $\delta _i$ is obtained by applying Lemma \ref{puncte} to the edge $P_iQ_i$.
\end{remark}
For Dirichlet polynomials having only two nonzero terms we obtain from Theorem \ref{calaDumas} the following simple irreducibility conditions.
\begin{corollary}\label{CelMaiSimpluCorolar}
Let $f(s)=\frac{a_{m}}{m^{s}}+\frac{a_{n}}{n^{s}}$ be a Dirichlet polynomial with coefficients in a unique factorization domain $R$, with $a_{m}a_{n}\neq 0$, $\gcd(m,n)=1$, and assume that $a_ma_n$ has a prime factor of multiplicity one. Then $f$ is irreducible over $Q(R)$.
\end{corollary}

\begin{proof}\ Note that $f$ is algebraically primitive, as $m$ and $n$ are coprime. Since $f$ has only two nonzero terms, condition ii) in Theorem \ref{calaDumas} is trivially satisfied, and our hypothesis that $\nu _{p}(a_{m}a_{n})=1$ for some prime element $p$ of $R$ shows that we must have $\nu _{p}(a_{n})-\nu _{p}(a_{m})=\pm 1$, so conditions i) and iii) in Theorem \ref{calaDumas} are satisfied too.
\end{proof}
For Dirichlet polynomials with integer coefficients we can generalize Theorem \ref{calaDumas} by using a change of indeterminate, as follows.

\begin{theorem}\label{calaDumasRationalCoefficients}
Let $f(s)=\frac{a_{m}}{m^{s}}+\cdots +\frac{a_{n}}{n^{s}}$ be an algebraically primitive Dirichlet polynomial with integer coefficients, with $a_{m}a_{n}\neq 0$, $t\geq 0$ an integer,
$p$ a prime number, and assume that $q_{1},\dots ,q_{k}$ are all the prime factors of $m\cdot n$. If
\smallskip

i) \ $\nu _{p}(a_{n}n^t)\neq \nu _{p}(a_{m}m^t)$,

ii) $\left ( \frac{n}{i}\right ) ^{\nu _{p}(a_{i}i^t)-\nu _{p}(a_{m}m^t)}>\left ( \frac{m}{i}\right ) ^{\nu _{p}(a_{i}i^t)-\nu _{p}(a_{n}n^t)}$ \ for $m<i<n$,

iii) $\gcd (\nu _{p}(a_{n}n^t)-\nu _{p}(a_{m}m^t),\nu _{q_{1}}(n)-\nu _{q_{1}}(m),\dots ,\nu _{q_{k}}(n)-\nu _{q_{k}}(m))=1$,
\smallskip

\noindent then $f$ is irreducible over $\mathbb{Q}$.
\end{theorem}

\begin{proof}\ Consider the Dirichlet polynomial $f(s-t)=\frac{a_{m}\cdot m^t}{m^{s}}+\cdots +\frac{a_{n}\cdot n^t}{n^{s}}$, which has integer coefficients, is algebraically primitive, and according to i), ii) and iii), satisfies the hypotheses of Theorem \ref{calaDumas}, and hence is irreducible over $\mathbb{Q}$. Assume to the contrary that $f(s)=g(s)h(s)$ with $g,h$ nonconstant Dirichlet polynomials with integer coefficients, say $g(s)=\frac{b_{u}}{u^{s}}+\cdots +\frac{b_{v}}{v^{s}}$ and $h(s)=\frac{c_{l}}{l^{s}}+\cdots +\frac{c_{w}}{w^{s}}$, with $a_i=\sum_{i=j\cdot k}b_jc_k$ for each $i\in \{m,\dots ,n\}$. Multiplication by $i^t$ yields $a_ii^t=\sum_{i=j\cdot k}(b_jj^t)(c_kk^t)$ for each $i$, which shows that $f(s-t)=g(s-t)h(s-t)$, implying that $f(s-t)$ is reducible, a contradiction.
\end{proof}
We notice that for Dirichlet polynomials 
with $|a_m|=|a_n|$ we cannot apply Theorem \ref{calaDumas}, no matter what primes $p$ we consider. For such Dirichlet polynomials $f$, we may at least try to test the conditions i), ii) and iii) in Theorem \ref{calaDumasRationalCoefficients} for some of the relevant primes of $f$, as in the following result.

\begin{corollary}\label{CorolarCaLaDumasRationalCoefficients}
Let $f(s)=\frac{a_{m}}{m^{s}}+\cdots +\frac{a_{n}}{n^{s}}$ be an algebraically primitive Dirichlet polynomial with integer coefficients, with $|a_{m}|=|a_{n}|\neq 0$, and assume that $q_{1},\dots ,q_{k}$ are all the prime factors of $m\cdot n$. Let $p\in \{q_1,\dots ,q_k\} $ be a prime number for which $\nu_{p}(m)<\nu_{p}(n)$.
If
\[
\nu _p(ia_i)\geq \nu _p(na_n) \ \ for\ \ m<i<n \ \ and\  \ \gcd (\nu _{q_{1}}(n)-\nu _{q_{1}}(m),\dots ,\nu _{q_{k}}(n)-\nu _{q_{k}}(m))=1,
\]
\noindent then $f$ is irreducible over $\mathbb{Q}$.
\end{corollary}
\begin{proof} \ First of all we observe that since $m<n$, there exists at least one prime $p$ among $q_1,\dots ,q_k$, such that $\nu _{p}(m)<\nu _{p}(n)$, so if we choose $t=1$, condition i) in Theorem \ref{calaDumasRationalCoefficients} will hold for our prime $p$, as $|a_m|=|a_n|$. We then observe that condition iii) in Theorem \ref{calaDumasRationalCoefficients} reduces to $\gcd (\nu _{q_{1}}(n)-\nu _{q_{1}}(m),\dots ,\nu _{q_{k}}(n)-\nu _{q_{k}}(m))=1$, since for $t=1$ we have
\[
\nu _{p}(a_{n}n^t)-\nu _{p}(a_{m}m^t)=\nu _{p}(na_{n})-\nu _{p}(ma_{m})=\nu _{p}(n)-\nu _{p}(m),
\]
which is one of the terms 
$\nu _{q_{1}}(n)-\nu _{q_{1}}(m),\dots ,\nu _{q_{k}}(n)-\nu _{q_{k}}(m)$. Finally, we observe that for $m<i<n$ we have $\frac{n}{i}>1$ and $\frac{m}{i}<1$, and our assumption that $\nu _p(ia_i)\geq \nu _p(na_n)$ together with the fact that $\nu _p(ma_m)<\nu _p(na_n)$ imply that 
\[
\nu _p(ia_i)-\nu _p(ma_m)>\nu _p(ia_i)-\nu _p(na_n)\geq 0.
\]
Therefore for $t=1$ condition ii) in Theorem \ref{calaDumasRationalCoefficients} is satisfied too.
\end{proof}

We will also include in this section two lemmas that will be used to derive irreducibility conditions for linear combinations of the form $f(s)+p^kg(s)$, with $f,g$ Dirichlet polynomials of coprime degrees, $p$ a prime element of $R$, and $k$ a positive integer, and whose proofs rely on Theorem \ref{DumDir} and Lemma \ref{puncte}. We mention here that linear combinations of relatively prime polynomials have been studied in both univariate and multivariate cases by Cavachi, V\^aj\^aitu and Zaharescu \cite{Cavachi}, \cite{CVZ1}, \cite{CVZ2}, Zhang, Yuan and Zhou \cite{Zhang}, and also in \cite{BoncioIJNT}, \cite{BoncioPMD2} and \cite{BoncioResMat}, and the same techniques apply to compositions and multiplicative convolutions of polynomials \cite{BoncioMathNach}, \cite{BoncioAnStUnivOvid}, \cite{BoncioBullSSMR}, \cite{BoncioMonatshefte}, \cite{BoncioAMH} and \cite{BoncioPMD1}.

\begin{lemma}
\label{UniquePositiveSlope}
Let $f,g$ be two Dirichlet polynomials with coefficients in a unique factorization domain $R$, with $\deg g=n$ and $\deg f=m$, $m<n$, and let $p_1,\dots ,p_r$ be the distinct prime factors of $m\cdot n$. Let also $p$ be a prime element of $R$ that divides none of the leading coefficients of $f$ and $g$, and let $k$ be any positive integer with 
\begin{equation}\label{ConditiaDeSegment}
\gcd(k,\nu _{p_1}(n)-\nu _{p_1}(m),\dots ,\nu _{p_r}(n)-\nu _{p_r}(m))=1.
\end{equation}
If $f(s)+p^{k}g(s)$ is algebraically primitive and can be written as a product of two nonconstant Dirichlet polynomials with coefficients in $R$, say $f_{1}(s)$ and $f_{2}(s)$, then one of the leading coefficients of $f_{1}$ and $f_{2}$ must be divisible by $p^{k}$, and the other must be coprime to $p$.
\end{lemma}
\begin{proof}
Let $f(s)=\frac{a_{1}}{1^s}+\cdots +\frac{a_{m}}{m^s}$ and 
$g(s)=\frac{b_{1}}{1^s}+\cdots +\frac{b_{n}}{n^s}$ , $a_{m}b_{n}\neq 0$, 
and let us write $f(s)+p^{k}g(s)=\frac{c_{1}}{1^s}+\cdots +\frac{c_{n}}{n^s}$, 
where $c_{i}=a_{i}+p^{k}b_{i}$ for $i=1,\dots ,m$, while 
$c_{i}=p^{k}b_{i}$ for $i=m+1,\dots ,n$. Our 
assumption that $p\nmid a_{m}b_{n}$ implies that $c_{m}$ is not 
divisible by $p$, $c_{m+1},\dots , c_{n}$ are all 
divisible by $p^{k}$, and moreover, $p^{k+1}\nmid c_{n}$. Thus, 
in the Newton log-polygon of $f+p^{k}g$ with respect to the prime element $p$, the rightmost edge will join the points $(\log m,0)$ 
and $(\log n,k)$, that are labelled in Figure 3 below by $Q_{\ell-1}$ and $Q_{\ell}$, 
respectively. 

\begin{center}

\setlength{\unitlength}{5mm}
\begin{picture}(12,7)
\linethickness{0.15mm}

\put(12,0.3){\circle{0.01}}

\put(12,0.6){\circle{0.01}}

\put(12,0.9){\circle{0.01}}

\put(12,1.2){\circle{0.01}}

\put(12,1.5){\circle{0.01}}

\put(12,1.8){\circle{0.01}}

\put(12,2.1){\circle{0.01}}

\put(12,2.4){\circle{0.01}}

\put(12,2.7){\circle{0.01}}

\put(12,3){\circle{0.01}}

\put(12,3.3){\circle{0.01}}

\put(12,3.6){\circle{0.01}}

\put(12,3.9){\circle{0.01}}

\put(12,4.2){\circle{0.01}}

\put(12,4.5){\circle{0.01}}

\thicklines  

\put(4.8,0){\line(2,1){7.2}}   

\linethickness{0.15mm}

\put(0,0){\vector(1,0){14}}
\put(0,0){\vector(0,1){6}}

\put(4.8,0){\circle{0.08}}
\put(4.8,0){\circle{0.12}}
\put(12,0){\circle{0.08}}
\put(12,0){\circle{0.12}}
\put(12,3.6){\circle{0.08}}
\put(12,3.6){\circle{0.12}}

{\small 
\put(12.3,3.4){$Q_{\ell}=(\log n,k)$}

\put(3.7,-0.75){$Q_{\ell-1}=(\log m,0)$}

\put(10.6,-0.75){$(\log n,0)$}

}

\end{picture}
\end{center}
\bigskip

{\small
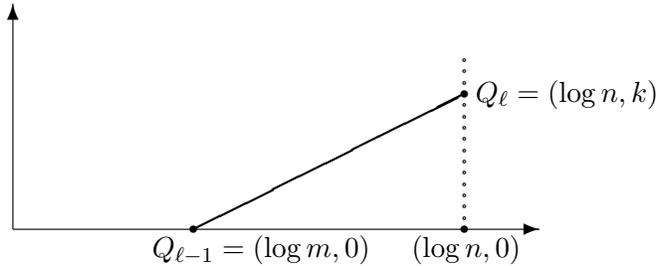
{\bf Figure 3.} The rightmost segment in the Newton log-polygon 
of $f(s)+p^{k}g(s)$ with respect to $p$.}
\medskip

Moreover, since $k$ satisfies (\ref{ConditiaDeSegment}), we deduce by Lemma \ref{puncte} that $Q_{\ell-1}Q_{\ell}$ is in fact a segment of the Newton log-polygon, as it contains no log-integral points other than 
its endpoints $Q_{\ell-1}$ and $Q_{\ell}$. 
Since the sequence of the slopes of the edges of the Newton log-polygon, when 
considered from left to the right, is a strictly increasing one, we 
deduce that $Q_{\ell-1}Q_{\ell}$ is the only segment with positive slope. 
On the other hand, if $f+p^{k}g$ is the product of two 
non-constant Dirichlet polynomials $f_{1}$ and $f_{2}$ with integer coefficients, 
then each of the factors $f_{1}$ and $f_{2}$ must have at least 
one coefficient which is not divisible by $p$, as $p\nmid c_m$. If we assume now 
that each of the leading coefficients of $f_{1}$ and $f_{2}$ is 
divisible by $p$, then each of the Newton log-polygons of $f_{1}$ and 
$f_{2}$ with respect to $p$ will contain at least one segment 
with positive slope, which by Theorem \ref{DumDir} would produce at least 
two segments with positive slopes in the Newton log-polygon of $f+p^{k}g$, 
a contradiction. Therefore, one of the leading coefficients of 
$f_{1}$ and $f_{2}$ must be divisible by $p^{k}$, while the other 
must be coprime to $p$. This completes the proof of our lemma. 
\end{proof}
We mention that Lemma \ref{UniquePositiveSlope} is the logarithmic analogue of Lemma 1.4 in \cite{BoncioIJNT}, which provides information on the leading coefficients of the factors of integer polynomials of the form $f(X)+p^kg(X)$, with $\deg f<\deg g$ and $k$ coprime to $\deg g-\deg f$.

As an application of Lemma \ref{UniquePositiveSlope} we have the following irreducibility criterion for linear combinations of Dirichlet polynomials.
\begin{theorem}
\label{FplusP^kGPositiveSlope}
Let $f$ and $g$ be two Dirichlet polynomials with coefficients in a unique factorization domain $R$, of coprime degrees $m$ and $n$, respectively, with $m<n$, and let $p_1,\dots ,p_r$ be the distinct prime factors of $m\cdot n$. Let also $p$ be a prime element of $R$ that divides none of the leading coefficients of $f$ and $g$, and let $k$ be any positive integer with 
\[
\gcd(k,\nu _{p_1}(n)-\nu _{p_1}(m),\dots ,\nu _{p_r}(n)-\nu _{p_r}(m))=1.
\]
Then $f(s)+p^{k}g(s)$ is irreducible over $Q(R)$.
\end{theorem}
\begin{proof}
We first observe that the Dirichlet polynomial $f(s)+p^kg(s)$ is algebraically primitive, as $\gcd(m,n)=1$ and $p$ divides none of the leading coefficients of $f$ and $g$. Let us assume to the contrary that $f(s)+p^{k}g(s)=f_1(s)f_2(s)$ with $f_1,f_2$ nonconstant Dirichlet polynomials with coefficients in $R$, say $f_1(s)=\frac{a_{c_1}}{c_1^s}+\cdots +\frac{a_{d_1}}{d_1^s}$ and $f_2(s)=\frac{b_{c_2}}{c_2^s}+\cdots +\frac{b_{d_2}}{d_2^s}$ with $d_1\geq 2$, $d_2\geq 2$ and $d_1d_2=n$. By Lemma \ref{UniquePositiveSlope}, one of the leading coefficients of $f_{1}$ and $f_{2}$ must be divisible by $p^{k}$, and the other must be coprime to $p$. Without loss of generality we may assume that $p^k\mid a_{d_1}$ and $p\nmid b_{d_2}$. Let us denote by $\overline{f},\overline{f_1}$ and $\overline{f_2}$ the Dirichlet polynomials obtained by reducing modulo $p$ the coefficients of $f,f_1$ and $f_2$, respectively. By reducing modulo $p$ the equality $f(s)+p^{k}g(s)=f_1(s)f_2(s)$, we obtain $\overline{f}=\overline{f_1}\cdot\overline{f_2}$. Since the leading coefficient of $f$ is not divisible by $p$, we have $\deg \overline{f}=\deg f=m$, and since $p\nmid b_{d_2}$, we have $\deg \overline{f_2}=\deg f_2=d_2$. If we denote $\deg \overline{f_1}$ by $d_1'$, say, we have $m=d_1'd_2$, while $n=d_1d_2$, so $d_2$ must be a common divisor of $m$ and $n$. Since $m$ and $n$ are coprime, $d_2$ must be equal to $1$, which is a contradiction. This completes the proof.
\end{proof}

Cavachi \cite{Cavachi} proved the following irreducibility criterion for linear combinations of relatively prime polynomials.
\medskip

{\bf Theorem} (Cavachi) {\em Let $f(X)$ and $g(X)$ be relatively prime polynomials with integer coefficients, with $\deg f<\deg g$. Then $f(X)+pg(X)$ is irreducible over $\mathbb{Q}$ for all but finitely many prime numbers $p$.}
\medskip

In particular, for $k=1$ we obtain from Theorem \ref{FplusP^kGPositiveSlope} the following analogue for Dirichlet polynomials of Cavachi's Theorem.
\begin{corollary}
\label{Coro1FplusP^kGPositiveSlope}
Let $f(s)$ and $g(s)$ be two Dirichlet polynomials with integer coefficients, of coprime degrees $m$ and $n$, respectively, with $m<n$.
Then $f(s)+pg(s)$ is irreducible over $\mathbb{Q}$ for all but finitely many prime numbers $p$.
\end{corollary}
\begin{proof}
Since $k=1$, condition (\ref{ConditiaDeSegment}) is trivially satisfied, and the irreducibility of $f(s)+pg(s)$ is ensured for every prime $p$ that divides none of the leading coefficients of $f$ and $g$.
\end{proof}
Another simple instance of Theorem \ref{FplusP^kGPositiveSlope} is the following irreducibility criterion.
\begin{corollary}
\label{Coro2FplusP^kGPositiveSlope}
Let $f(s)$ and $g(s)$ be two Dirichlet polynomials with integer coefficients, of coprime squarefree degrees $m$ and $n$, respectively, with $m<n$, and let $p$ be a prime number that divides none of the leading coefficients of $f$ and $g$. 
Then $f(s)+p^kg(s)$ is irreducible over $\mathbb{Q}$ for all positive integers $k$.
\end{corollary}
\begin{proof}
Since $m$ and $n$ are coprime and squarefree, $\nu _{p_i}(n)-\nu _{p_i}(m)=\pm 1$, so condition (\ref{ConditiaDeSegment}) is trivially satisfied for every $k$.
\end{proof}

Our following result is similar to Lemma \ref{UniquePositiveSlope}, and refers to the min-degree coefficients of the factors of $f+p^kg$, instead of their leading coefficients.

\begin{lemma}
\label{UniqueNegativeSlope}
Let $f,g$ be two Dirichlet polynomials with coefficients in a unique factorization domain $R$, with $\deg _{{\rm min}}f=m$, $\deg _{{\rm min}}g=n$, $m>n$, and let $p_1,\dots ,p_r$ be the distinct prime factors of $m\cdot n$. Let also $p$ be a prime element of $R$ that divides none of the min-degree coefficients of $f$ and $g$, and let $k$ be any positive integer with 
\begin{equation}\label{ConditiaDeSegment2}
\gcd(k,\nu _{p_1}(n)-\nu _{p_1}(m),\dots ,\nu _{p_r}(n)-\nu _{p_r}(m))=1.
\end{equation}
If $f(s)+p^{k}g(s)$ is algebraically primitive and can be written as a product of two nonconstant Dirichlet polynomials with coefficients in $R$, say $f_{1}(s)$ and $f_{2}(s)$, then one of the min-degree coefficients of $f_{1}$ and $f_{2}$ must be divisible by $p^{k}$.
\end{lemma}
\begin{proof}
Let $f(s)=\frac{a_{m}}{m^s}+\cdots +\frac{a_{M}}{M^s}$ and 
$g(s)=\frac{b_{n}}{n^s}+\cdots +\frac{b_{N}}{N^s}$ , $a_{m}a_{M}b_{n}b_{N}\neq 0$, 
and let us write $f(s)+p^{k}g(s)=\frac{c_{n}}{n^s}+\cdots +\frac{c_{u}}{u^s}$, 
where $c_{i}=p^{k}b_{i}$ for $i=n,\dots ,m-1$, while 
$c_{i}=a_k+p^{k}b_{i}$ for $i=m,\dots ,u$. Our 
assumption that $p\nmid a_{m}b_{n}$ implies that $c_{m}$ is not 
divisible by $p$, $c_{n},\dots , c_{m-1}$ are all 
divisible by $p^{k}$, and moreover, $p^{k+1}\nmid c_{n}$. Thus, 
in the Newton log-polygon of $f+p^{k}g$ with respect to the prime element $p$, the leftmost edge will join the points $(\log n,k)$ 
and $(\log m,0)$, that are labelled in Figure 4 below by $Q_{\ell-1}$ and $Q_{\ell}$, respectively.

\begin{center}

\setlength{\unitlength}{5mm}
\begin{picture}(12,7)
\linethickness{0.15mm}

\put(3.7,0.3){\circle{0.01}}

\put(3.7,0.6){\circle{0.01}}

\put(3.7,0.9){\circle{0.01}}

\put(3.7,1.2){\circle{0.01}}

\put(3.7,1.5){\circle{0.01}}

\put(3.7,1.8){\circle{0.01}}

\put(3.7,2.1){\circle{0.01}}

\put(3.7,2.4){\circle{0.01}}

\put(3.7,2.7){\circle{0.01}}

\put(3.7,3){\circle{0.01}}

\put(3.7,3.3){\circle{0.01}}

\put(3.7,3.6){\circle{0.01}}

\put(3.7,3.9){\circle{0.01}}

\put(3.7,4.2){\circle{0.01}}

\put(3.7,4.5){\circle{0.01}}

\thicklines  

\put(3.7,3.6){\line(2,-1){7.2}}   

\linethickness{0.15mm}

\put(0,0){\vector(1,0){15}}
\put(0,0){\vector(0,1){6}}

\put(3.7,0){\circle{0.08}}
\put(3.7,0){\circle{0.12}}
\put(3.7,3.6){\circle{0.08}}
\put(3.7,3.6){\circle{0.12}}
\put(10.9,0){\circle{0.08}}
\put(10.9,0){\circle{0.12}}

{\small 
\put(8.2,-0.75){$Q_{\ell}=(\log m,0)$}

\put(4.1,3.8){$Q_{\ell-1}=(\log n,k)$}

\put(2.2,-0.75){$(\log n,0)$}

}

\end{picture}
\end{center}
\bigskip

{\small
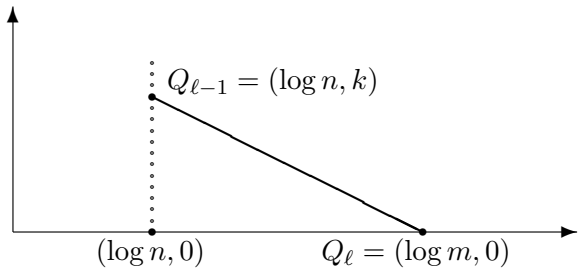
{\bf Figure 4.} The leftmost segment in the Newton log-polygon of $f(s)+p^{k}g(s)$ with respect to $p$.}
\medskip

Since $k$ satisfies (\ref{ConditiaDeSegment}), we see by Lemma \ref{puncte} that the edge $Q_{\ell-1}Q_{\ell}$ is in fact a segment of the Newton log-polygon, as it contains no log-integral points other than 
its endpoints $Q_{\ell-1}$ and $Q_{\ell}$, and moreover, it is the only segment with negative slope. 
Note that if $f+p^{k}g=f_1f_2$ with $f_1$ and $f_2$ non-constant Dirichlet polynomials with integer coefficients, then each of the factors $f_{1}$ and $f_{2}$ must have at least one coefficient not divisible by $p$, as $p\nmid c_m$. If we assume now 
that each of the min-degree coefficients of $f_{1}$ and $f_{2}$ is 
divisible by $p$, then each of the Newton log-polygons of $f_{1}$ and 
$f_{2}$ with respect to $p$ will contain at least one segment 
with negative slope, which by Theorem \ref{DumDir} will produce at least 
two segments with negative slopes in the Newton log-polygon of $f+p^{k}g$, which is a contradiction. Therefore, one of the min-degree coefficients of 
$f_{1}$ and $f_{2}$ must be divisible by $p^{k}$, and the other 
must be coprime to $p$. This completes the proof. 
\end{proof}

As an application of Lemma \ref{UniqueNegativeSlope} we obtain the following irreducibility criterion for linear combinations of Dirichlet polynomials, in terms of their min-degrees and their min-degree coefficients.
\begin{theorem}
\label{FplusP^kGNegativeSlope}
Let $f$ and $g$ be Dirichlet polynomials with coefficients in a unique factorization domain $R$, of coprime min-degrees $m$ and $n$, respectively, with $m>n>1$, and assume that $\max\{\deg f,\deg g\}\leq 2n$. Let $p_1,\dots ,p_r$ be the distinct prime factors of $m\cdot n$, let $p$ be a prime element of $R$ that divides none of the min-degree coefficients of $f$ and $g$, and let $k$ be any positive integer with 
\[
\gcd(k,\nu _{p_1}(n)-\nu _{p_1}(m),\dots ,\nu _{p_r}(n)-\nu _{p_r}(m))=1.
\]
Then $f(s)+p^{k}g(s)$ is irreducible over $Q(R)$.
\end{theorem}
\begin{proof}
We notice that $f(s)+p^kg(s)$ is algebraically primitive, as $\gcd(m,n)=1$ and $p$ divides none of the min-degree coefficients of $f$ and $g$.  Let us assume to the contrary that $f(s)+p^{k}g(s)=f_1(s)f_2(s)$ with $f_1,f_2$ nonconstant Dirichlet polynomials with coefficients in $R$, say $f_1(s)=\frac{a_{c_1}}{c_1^s}+\cdots +\frac{a_{d_1}}{d_1^s}$ and $f_2(s)=\frac{b_{c_2}}{c_2^s}+\cdots +\frac{b_{d_2}}{d_2^s}$ with $d_1\geq 2$, $d_2\geq 2$ and $c_1c_2=n$. By Lemma \ref{UniqueNegativeSlope}, one of the min-degree coefficients of $f_{1}$ and $f_{2}$ must be divisible by $p^{k}$, and the other must be coprime to $p$, say $p^k\mid a_{c_1}$ and $p\nmid b_{c_2}$. As before, let us denote by $\overline{f},\overline{f_1}$ and $\overline{f_2}$ the Dirichlet polynomials obtained by reducing modulo $p$ the coefficients of $f,f_1$ and $f_2$, respectively. By reducing modulo $p$ the equality $f(s)+p^{k}g(s)=f_1(s)f_2(s)$, we obtain $\overline{f}=\overline{f_1}\cdot\overline{f_2}$. Since the min-degree coefficient of $f$ is not divisible by $p$, we have $\deg_{\rm min} \overline{f}=\deg_{\rm min} f=m$, and since $p\nmid b_{c_2}$, we have $\deg_{\rm min} \overline{f_2}=\deg_{\rm min} f_2=c_2$. If we denote $\deg_{\rm min} \overline{f_1}$ by $c_1'$, say, we have $m=c_1'c_2$, while $n=c_1c_2$, so $c_2$ must be a common divisor of $m$ and $n$. Since $m$ and $n$ are coprime, $c_2$ must be equal to $1$, so $c_1$ must be equal to $n$. This in turn forces $d_1$ to be greater than $n$, as $f+p^kg$ cannot have a factor of the form $\frac{a_n}{n^s}$, being algebraically primitive. We thus have $d_2=\frac{\deg (f+p^kg)}{d_1}\leq \frac{2n}{d_1}<2$, so $d_2$ is forced to be equal to $1$, a contradiction. 
\end{proof}
In particular, we obtain for $k=1$ the following corollary for Dirichlet polynomials with integer coefficients.
\begin{corollary}
\label{Coro1FplusP^kGNegativeSlope}
Let $f$ and $g$ be two Dirichlet polynomials with integer coefficients, of coprime min-degrees $m$ and $n$, respectively, with $m>n>1$, and assume that $\max\{\deg f,\deg g\}\leq 2n$. Then $f(s)+pg(s)$ is irreducible over $\mathbb{Q}$ for all but finitely many prime numbers $p$.
\end{corollary}
\begin{proof}
Since $k=1$, condition (\ref{ConditiaDeSegment2}) is obviously satisfied, and the irreducibility of $f(s)+pg(s)$ is now guaranteed for every $p$ that divides none of the min-degree coefficients of $f$ and $g$.
\end{proof}

Another immediate consequence of Theorem \ref{FplusP^kGNegativeSlope} is the following irreducibility criterion.
\begin{corollary}
\label{Coro2FplusP^kGNegativeSlope}
Let $f$ and $g$ be two Dirichlet polynomials with integer coefficients, of coprime squarefree min-degrees $m$ and $n$, respectively, with $m>n>1$ and $\max\{\deg f,\deg g\}\leq 2n$. Let $p$ be a prime number that divides none of the min-degree coefficients of $f$ and $g$. 
Then $f(s)+p^kg(s)$ is irreducible over $\mathbb{Q}$ for all positive integers $k$.
\end{corollary}
\begin{proof}
Since $m$ and $n$ are coprime and squarefree, $\nu _{p_i}(n)-\nu _{p_i}(m)=\pm 1$, so condition (\ref{ConditiaDeSegment2}) is obviously satisfied for every $k$.
\end{proof}

In the case when the degrees of $f$ and $g$ are not relatively prime, it is considerably more difficult to draw a conclusion on the irreducibility of $f+p^kg$, but for small values of $k$, we can at least find some information on the relative degrees of its possible factors, as we will see in the following section in Proposition \ref{comblinear1} and Proposition \ref{comblinear2}. Other irreducibility criteria for linear combinations of two Dirichlet polynomials, analogous to Sch\"onemann's criterion for polynomials and its variations will be given in Appendix A. 

We end this section with the following conjecture, in which we drop the conditions on the degrees of $f$ and $g$ in Corollary \ref{Coro1FplusP^kGPositiveSlope}, and instead ask $f$ and $g$ to be relatively prime.
\smallskip

{\bf Conjecture 5.} {\em Let $f$ and $g$ be relatively prime Dirichlet polynomials with integer coefficients. Then $f(s)+pg(s)$ is irreducible over $\mathbb{Q}$ for all but finitely many prime numbers $p$.}

\section{The relative degrees of the factors of a Dirichlet polynomial}\label{RelativeDegree}
When we study the canonical decomposition of a Dirichlet polynomial, it is useful to have criteria to exclude some of the possible values of the relative degrees of its hypothetical factors. 
For polynomials, an important such result was obtained by Filaseta \cite{Filaseta1}, and this result was crucial in proving the irreducibility of all but finitely many Bessel polynomials (see also Filaseta and Trifonov \cite{FilasetaTrifonov}, and Filaseta, Finch and Leidy \cite{Filaseta-Finch-Leidy}, where this result was used to prove the irreducibility of some classes of generalized Laguerre polynomials):
\medskip

{\bf Lemma}\ (Filaseta, \cite[Lemma 2]{Filaseta1}) {\em \ Let $k$ and $l$ be integers with $k>l\geq 0$. Suppose $g(x)=\sum _{j=0}^{n}b_{j}x^{j}\in \mathbb{Z}[x]$ and $p$ is a prime such that $p\nmid b_{n}$, $p\mid b_{j}$ for all $j\in \{ 0,1,\dots ,n-l-1\} $, and the rightmost edge of the Newton polygon for $g(x)$ with respect to $p$ has slope $<1/k$. Then for any integers $a_{0},a_{1},\dots ,a_{n}$ with $|a_{0}|=|a_{n}|=1$, the polynomial $f(x)=\sum _{j=0}^{n}a_{j}b_{j}x^{j}\in \mathbb{Z}[x]$ cannot have a factor with degree in the interval $[l+1,k]$. }
\smallskip

Here the Newton polygon of $\sum _{j=0}^{n}b_{j}x^{j}$ is the lower convex hull of the set of points in the extended plane $\{ (0,\nu_p(b_n)),(1,\nu _p(b_{n-1})),\dots,(n,\nu _p(b_0))\} $.
\medskip

For Dirichlet polynomials we will first prove the following analogous result.

\begin{lemma}\label{calaFilaseta1}
Let $f(s)=\frac{a_{m}}{m^{s}}+\cdots +\frac{a_{n}}{n^{s}}$ be an algebraically primitive Dirichlet polynomial with coefficients in a unique factorization domain $R$, $a_{m}a_{n}\neq 0$, let $p$ be a prime element of $R$, $c_{1},c_{2}$ divisors of $m$ and $d_{1},d_{2}$ divisors of $n$ with $1<\frac{d_{1}}{c_{1}}\leq\frac{d_{2}}{c_{2}}<\frac{n}{m}$. Suppose that $p\nmid a_{i}$ for some index $i<\frac{md_{1}}{c_{1}}$, $p\mid a_{j}$ for all $j\in \{ \frac{md_{1}}{c_{1}},\dots ,n\} $, and the rightmost edge in the Newton log-polygon of $f$ with respect to $p$ has slope $<1/\log \frac{d_{2}}{c_{2}}$. Then for any nonzero elements $b_{m},\dots ,b_{n}\in R$ with $p\nmid b_{i}b_{n}$, the Dirichlet polynomial $g(s)=\frac{a_{m}b_{m}}{m^{s}}+\cdots +\frac{a_{n}b_{n}}{n^{s}}$ cannot have factors whose relative degrees belong to $[\frac{d_{1}}{c_{1}},\frac{d_{2}}{c_{2}}]\cup [\frac{nc_{2}}{md_{2}},\frac{nc_{1}}{md_{1}}]$.
\end{lemma}
\begin{proof}\ We will first consider the case that $b_{j}=1$ for all $j$, so $g(s)=f(s)$. The conditions of the lemma imply that the Newton log-polygon of $f$ contains at least one edge with positive slope, as $p\nmid a_{i}$ and $p\mid a_{n}$, and moreover, at least one edge with slope $\leq 0$, if $i>m$. Besides, all the edges in the Newton log-polygon of $f$ with respect to $p$ have slopes less than $1/\log \frac{d_{2}}{c_{2}}$. Consider now an edge $PQ$ with positive slope, and assume that $(\log x_{1},y_{1})$ and $(\log x_{2},y_{2})$ are two distinct log-integral points on such an edge, with $x_{1}<x_{2}$ and $y_{1}<y_{2}$, say. Then the slope $S$ of this edge satisfies
\[
\frac{1}{\log x_{2}-\log x_{1}}\leq S=\frac{y_{2}-y_{1}}{\log x_{2}-\log x_{1}}<\frac{1}{\log \frac{d_{2}}{c_{2}}},
\]
so 
\begin{equation}\label{ecart}
\frac{x_{2}}{x_{1}}>\frac{d_{2}}{c_{2}}.
\end{equation}
Let us assume to the contrary that $f$ has a factor $f_{1}$ with relative degree $\frac{d}{c}$ in the interval $[\frac{d_{1}}{c_{1}},\frac{d_{2}}{c_{2}}]$. In particular, in view of (\ref{ecart}) we have $\frac{d}{c}\leq \frac{d_{2}}{c_{2}}<\frac{x_{2}}{x_{1}}$, so the relative degree of $f_1$ must be smaller than the relative degree of any segment that belongs to an edge with positive slope. Thus, since all the translates of the edges in the Newton log-polygon of $f_{1}$ with respect to $p$ become either edges or parts of edges in the Newton log-polygon of $f$, we conclude that all these translates must be actually found on edges with slope $\leq 0$ of the Newton log-polygon of $f$. If there exist no edges with slope $\leq 0$ (for instance if $i=m$ and $p\mid a_{j}$ for all $j>m$), then such a factor $f_{1}$ cannot exist. So let us assume that there exist edges with slope $\leq 0$, and among them let $PQ$ be the rightmost one. Thus the sum $L$ of the lengths of the projections on the $x$-axis of all the edges with slopes $\leq 0$ must be at least $\log \frac{d}{c}$. On the other hand, since $p\mid a_{j}$ for all $j\in \{ \frac{md_{1}}{c_{1}},\dots ,n\} $, the $x$-coordinate of $Q$ cannot exceed $\log (\frac{md_{1}}{c_{1}}-1)$, so $L$ satisfies
\[
L\leq \log \left( \frac{md_{1}}{c_{1}}-1\right) -\log m<\log \frac{d_{1}}{c_{1}}\leq \log \frac{d}{c},
\]  
a contradiction. Therefore $f$ cannot have a factor of relative degree $\frac{d}{c}$ in the interval $[\frac{d_{1}}{c_{1}},\frac{d_{2}}{c_{2}}]$. Besides, $g$ cannot have a factor $h$ with relative degree in the interval $[\frac{nc_{2}}{md_{2}},\frac{nc_{1}}{md_{1}}]$, for otherwise the relative degree of its complementary factor $\frac{g}{h}$ would lie in the interval $[\frac{d_{1}}{c_{1}},\frac{d_{2}}{c_{2}}]$.

We will consider now the general case that $b_{m},\dots ,b_{n}$ are arbitrary nonzero elements of $R$ with $p\nmid b_{i}b_{n}$. First of all, we note that $g$ is algebraically primitive too. Now, the fact that $b_{i}$ and $b_{n}$ are not divisible by $p$ shows that the log-integral point $(\log i,0)$ and the rightmost endpoint in the Newton log-polygon of $f$ with respect to $p$ will not be affected by passing from $f$ to $g$, so their positions will remain the same in the Newton log-polygon of $g$. Since all the edges in the Newton log-polygon of $f$ lie above or on the line that contains its rightmost edge $P_{r}Q_{r}$, and $\nu _{p}(a_{k}b_{k})\geq \nu _{p}(a_{k})$ for all $k\in \{ m,\dots ,n\} $, all the edges of the Newton log-polygon of $g$ too will lie above or on the line containing $P_{r}Q_{r}$. Moreover, since the two Newton log-polygons have the same rightmost endpoint $Q_{r}$, we also deduce that the slope of the rightmost edge in the Newton log-polygon of $g$ is less than or equal to the slope of $P_{r}Q_{r}$.
This can also be seen by observing that
\[
\max _{k<n}\frac{\nu _{p}(a_{n})-\nu _{p}(a_{k}b_{k})}{\log n-\log k}\leq \max _{k<n}\frac{\nu _{p}(a_{n})-\nu _{p}(a_{k})}{\log n-\log k},
\]
with the left and the right sides being the slopes of the rightmost edges in the Newton log-polygons of $g$ and $f$, respectively. Therefore $g$ satisfies the same conditions that are imposed on $f$ in the statement of the lemma, so by the first part of the proof we conclude that $g$ too cannot have a factor of relative degree belonging to $[\frac{d_{1}}{c_{1}},\frac{d_{2}}{c_{2}}]\cup [\frac{nc_{2}}{md_{2}},\frac{nc_{1}}{md_{1}}]$. 
\end{proof}

Let us recall the definition of the {\it rational square root} and of the {\it rational floor} of a pair $(m,n)$:
\begin{eqnarray*}
\rho (m,n) & := & \max\left\{ \frac{d}{c}\ :\ d\mid n,\ c\mid m,\  \ {\rm and}\ \ \frac{d}{c}\leq \sqrt{\frac{n}{m}}\thinspace \right\}, \\
\delta (m,n) & := & \thinspace \min\left\{ \frac{d}{c}\ :\ d\mid n,\ c\mid m,\ \  {\rm and}\ \ d>c>0\right\} ,
\end{eqnarray*}
and the fact that, according to Proposition \ref{rationalfloor}, if $\rho (m,n)=1$, then any algebraically primitive Dirichlet polynomial $f(s)=\frac{a_{m}}{m^{s}}+\cdots +\frac{a_{n}}{n^{s}}$ with $a_{m}a_{n}\neq 0$ must be irreducible.

As an immediate consequence of Lemma \ref{calaFilaseta1}, one obtains the following irreducibility conditions for the case that $\rho (m,n)>1$ (which implies that $1<\delta (m,n)\leq\rho (m,n)$).

\begin{theorem}\label{TeoremaCaLaFilaseta1}
Let $f(s)=\frac{a_{m}}{m^{s}}+\cdots +\frac{a_{n}}{n^{s}}$ be an algebraically primitive Dirichlet polynomial with coefficients in a unique factorization domain $R$, $a_{m}a_{n}\neq 0$, and let $p$ be a prime element of $R$. If $p\nmid a_{k}$ for some index $k<m\delta(m,n)$, $p\mid a_{j}$ for all $j\geq m\delta (m,n)$, and 
\begin{equation}\label{PantaMicaDeTot}
\rho (m,n)<\left( \frac{n}{i}\right) ^{\frac{1}{\nu _{p}(a_{n})-\nu _{p}(a_{i})}}\ for\ all\ i<n\ with\ \nu _{p}(a_{i})<\nu _{p}(a_{n}),
\end{equation}
then for any nonzero elements $b_{m},\dots ,b_{n}\in R$ with $p\nmid b_{k}b_{n}$, the Dirichlet polynomial $g(s)=\frac{a_{m}b_{m}}{m^{s}}+\cdots +\frac{a_{n}b_{n}}{n^{s}}$ is irreducible over $Q(R)$.
\end{theorem}
\begin{proof}\  If $\rho (m,n)=1$, then $f$ and $g$ are irreducible by Proposition \ref{rationalfloor}, hence we may assume that $\rho (m,n)>1$, so 1/$\log \rho (m,n)$ makes sense. Note that $k$ is among the indices $i$ with $\nu _{p}(a_{i})<\nu _{p}(a_{n})$, so the set of indices $i<n$ with $\nu _{p}(a_{i})<\nu _{p}(a_{n})$ in (\ref{PantaMicaDeTot}) is nonempty.
We observe that condition (\ref{PantaMicaDeTot}) forces the slope $S$ of the rightmost edge in the Newton log-polygon of $f$ with respect to $p$ to satisfy the inequality
\[
S=\max_{i<n}\frac{\nu _{p}(a_{n})-\nu _{p}(a_{i})}{\log n-\log i}<\frac{1}{\log \rho (m,n)}.
\] 

According to Lemma \ref{calaFilaseta1} with $\frac{d_1}{c_1}=\delta (m,n)$, $\frac{d_2}{c_2}=\rho (m,n)$, and $i=k$, our assumptions that $p\nmid a_{k}$ and $p\mid a_{j}$ for all $j>m\delta (m,n)$, together with the fact that the rightmost edge in the Newton log-polygon of $f$ has slope smaller than $1/\log \rho (m,n)$, will prevent $g$ to have factors with relative degree in the interval $[\delta (m,n),\rho (m,n)]$. 
On the other hand, as $g$ is algebraically primitive, if it would be reducible, it should obviously have a nonconstant factor, say $h=\frac{e_{c}}{c^{s}}+\cdots +\frac{e_{d}}{d^{s}}$ with $d>c$, and whose relative degree $\frac{d}{c}$ belongs to the interval $[\delta (m,n),\rho (m,n)]$. This completes the proof.
\end{proof}
In particular, for $m=1$ one obtains the following simpler instance of Lemma \ref{calaFilaseta1}.
\begin{lemma}\label{calaFilaseta1m=1}
Let $f(s)=\frac{a_{1}}{1^s}+\frac{a_2}{2^s}+\cdots +\frac{a_{n}}{n^{s}}$ be a Dirichlet polynomial with coefficients in a unique factorization domain $R$, $a_{1}a_{n}\neq 0$, $p$ a prime element of $R$, and $d_{1},d_{2}$ divisors of $n$ with $1<d_1\leq d_2<n$. Suppose that $p\nmid a_{i}$ for some index $i<d_1$, $p\mid a_{j}$ for all $j\geq d_1$, and the rightmost edge in the Newton log-polygon of $f$ with respect to $p$ has slope $<1/\log  d_2$. Then for any nonzero elements $b_{1},\dots ,b_{n}\in R$ with $p\nmid b_{i}b_{n}$, the Dirichlet polynomial $g(s)=\frac{a_{1}b_{1}}{1^s}+\frac{a_2b_2}{2^s}+\cdots +\frac{a_{n}b_{n}}{n^{s}}$ cannot have factors whose degrees belong to $[d_1,d_2]\cup [\frac{n}{d_{2}},\frac{n}{d_{1}}]$.
\end{lemma}

A slightly weaker instance of Theorem \ref{TeoremaCaLaFilaseta1} is the following result.

\begin{theorem}\label{coroFilaseta1}
Let $f(s)=\frac{a_{m}}{m^{s}}+\cdots +\frac{a_{n}}{n^{s}}$ be an algebraically primitive Dirichlet polynomial with coefficients in a unique factorization domain $R$, $a_{m}a_{n}\neq 0$, let $p$ be a prime element of $R$, and assume that $p\nmid a_{m}$ and $p\mid a_{j}$ for all $j>m$. 
If
\begin{equation}\label{pantamica}
\rho (m,n)<\left( \frac{n}{i}\right) ^{\frac{1}{\nu _{p}(a_{n})-\nu _{p}(a_{i})}}\ for\ all\ i<n\ with\ \nu _{p}(a_{i})<\nu _{p}(a_{n}),
\end{equation}
then for any nonzero elements $b_{m},\dots ,b_{n}\in R$ with $p\nmid b_{m}b_{n}$, the Dirichlet polynomial $g(s)=\frac{a_{m}b_{m}}{m^{s}}+\cdots +\frac{a_{n}b_{n}}{n^{s}}$ is irreducible over $Q(R)$.
\end{theorem}

\begin{remark}\label{rho-Eisenstein1}
We mention here that Theorem \ref{naiveEisenstein} ii) is a special case of Theorem \ref{coroFilaseta1}, for if $p\mid a_{i}$ for each $i=m+1,\dots ,n$, $p\nmid a_{m}$ and $p^{2}\nmid a_{n}$, then condition (\ref{pantamica}) reduces to $\rho (m,n)<\frac{n}{m}$, which is obviously satisfied.
\end{remark}
In particular, Lemma \ref{calaFilaseta1} provides information on the factors for some classes of linear combinations of Dirichlet polynomials. We will only present here the simplest such result for linear combinations of the form $f(s)+p^{k}g(s)$, but one may easily prove more general statements, with less restrictive conditions on the coefficients and on the min-degrees of the Dirichlet polynomial that we are studying.

\begin{proposition}\label{comblinear1} 
Let $f(s)=\frac{a_{1}}{1^s}+\frac{a_{2}}{2^{s}}+\cdots +\frac{a_{t}}{t^{s}}$ and $g(s)=\frac{b_{1}}{1^s}+\frac{b_{2}}{2^{s}}+\cdots +\frac{b_{n}}{n^{s}}$ be two Dirichlet polynomials with coefficients in a unique factorization domain $R$, with $a_{1}a_{t}b_{1}b_{n}\neq 0$. Assume that $t<d\leq D$, with $d,D$ divisors of $n$, $D=\max\{ \delta :\delta \mid n,\ \delta\leq \sqrt{n}\} $, and let $p$ be a prime element of $R$ with $p\nmid a_{1}a_{t}b_{n}$. Then for any positive integer $k$ with
\begin{equation}\label{inegk}
k<\log _{D}\frac{n}{t},
\end{equation}
the Dirichlet polynomial $f(s)+p^{k}g(s)$ cannot have factors with degrees in $[d,D]\cup [\frac{n}{D},\frac{n}{d}]$.
\end{proposition}
\begin{proof}\ First of all, we observe that $D<\frac{n}{t}$, as $t<D\leq \sqrt{n}\leq \frac{n}{D}$, so the set of positive integers $k$ satisfying (\ref{inegk}) is nonempty. We also observe that $f(s)+p^{k}g(s)$ is algebraically primitive, as its constant term $a_{1}+p^{k}b_{1}$ is nonzero (since $p\nmid a_{1}$). Note that, according to the definition of $\rho (m,n)$, in this case we have $D=\rho (1,n)$.  Now, if we write $f(s)+p^{k}g(s)=\sum _{i=1}^{n}\frac{c_{i}}{i^{s}}$, with $c_{1}c_{n}\ne 0$, we see that $p\nmid c_{t}$, $p^{k}\mid c_{i}$ for $i\geq t+1$, and $p^{k+1}\nmid c_{n}$, as $p\nmid a_{t}b_{n}$. Therefore the rightmost edge in the Newton log-polygon of $f(s)+p^{k}g(s)$ with respect to $p$ has endpoints $(\log t,0)$ and $(\log n,k)$, so has slope $S$ satisfying
\[
S=\frac{k}{\log \frac{n}{t}}<\frac{\log _{D}\frac{n}{t}}{\log \frac{n}{t}}=\frac{1}{\log D}.
\]
By Lemma \ref{calaFilaseta1m=1} with $d_1=d$, $d_2=D$ and $i=t$, we deduce that $f(s)+p^{k}g(s)$ cannot have a factor with degree in the union of intervals $[d,D]\cup [\frac{n}{D},\frac{n}{d}]$.
\end{proof}

Lemma \ref{calaFilaseta1}, Theorem \ref{TeoremaCaLaFilaseta1}, Theorem \ref{coroFilaseta1} and Proposition \ref{comblinear1} have some ``mirrored" results that rely on information on the slope of the leftmost edge in the Newton log-polygon.

\begin{lemma}\label{calaFilaseta2}
Let $f(s)=\frac{a_{m}}{m^{s}}+\cdots +\frac{a_{n}}{n^{s}}$ be an algebraically primitive Dirichlet polynomial with coefficients in a unique factorization domain $R$, $a_{m}a_{n}\neq 0$, let $p$ be a prime element of $R$, $c_{1},c_{2}$ divisors of $m$ and $d_{1},d_{2}$ divisors of $n$ with $1<\frac{d_{1}}{c_{1}}\leq\frac{d_{2}}{c_{2}}<\frac{n}{m}$. Suppose that $p\mid a_{j}$ for all $j\in \{ m,\dots ,\frac{nc_{1}}{d_{1}}\} $, $p\nmid a_{i}$ for some index $i>\frac{nc_{1}}{d_{1}}$, and the leftmost edge in the Newton log-polygon of $f$ with respect to $p$ has slope $>-1/\log \frac{d_{2}}{c_{2}}$. Then for any nonzero elements $b_{m},\dots ,b_{n}\in R$ with $p\nmid b_{i}b_{m}$, the Dirichlet polynomial $g(s)=\frac{a_{m}b_{m}}{m^{s}}+\cdots +\frac{a_{n}b_{n}}{n^{s}}$ cannot have factors whose relative degrees belong to $[\frac{d_{1}}{c_{1}},\frac{d_{2}}{c_{2}}]\cup [\frac{nc_{2}}{md_{2}},\frac{nc_{1}}{md_{1}}]$.
\end{lemma}
\begin{proof}\ As before, we will first consider the case that $b_{j}=1$ for all $j$, so $g(s)=f(s)$. Our hypotheses imply that the Newton log-polygon of $f$ contains at least one edge with negative slope, as $p\nmid a_{i}$ and $p\mid a_{m}$, and moreover, at least one edge with slope $\geq 0$, if $i<n$. Besides, all the edges in the Newton log-polygon of $f$ with respect to $p$ have now slopes greater than $-1/\log \frac{d_{2}}{c_{2}}$. Let us consider now an edge $PQ$ with negative slope, and assume that $(\log x_{1},y_{1})$ and $(\log x_{2},y_{2})$ are two distinct log-integral points on such an edge, with $x_{1}<x_{2}$ and $y_{1}>y_{2}$, say. Then the slope $S$ of this edge satisfies
\[
\frac{-1}{\log x_{2}-\log x_{1}}\geq S=\frac{y_{2}-y_{1}}{\log x_{2}-\log x_{1}}>\frac{-1}{\log \frac{d_{2}}{c_{2}}},
\]
so as in Lemma \ref{calaFilaseta1} we must have
\begin{equation}\label{ecartNou}
\frac{x_{2}}{x_{1}}>\frac{d_{2}}{c_{2}}.
\end{equation}
We will assume to the contrary that $f$ has a factor $f_{1}$ with relative degree $\frac{d}{c}$ in the interval $[\frac{d_{1}}{c_{1}},\frac{d_{2}}{c_{2}}]$.
In this case (\ref{ecartNou}) forces all the translates of the edges in the Newton log-polygon of $f_{1}$ with respect to $p$ to belong to the edges with slope $\geq 0$ of the Newton log-polygon of $f$. If there exist no edges with slope $\geq 0$ (for instance if $i=n$ and $p\mid a_{j}$ for all $j<n$), then such a factor $f_{1}$ cannot exist. So let us assume that there exist edges with slope $\geq 0$, and among them let now $PQ$ be the leftmost one. Thus the sum $L$ of the lengths of the projections on the $x$-axis of all the edges with slopes $\geq 0$ must be at least $\log \frac{d}{c}$. On the other hand, since $p\mid a_{j}$ for all $j\in \{ m,\dots ,\frac{nc_{1}}{d_{1}}\} $, the $x$-coordinate of $P$ cannot be less than $\log (\frac{nc_{1}}{d_{1}}+1)$, so $L$ must satisfy
\[
L\leq \log n-\log \left( \frac{nc_{1}}{d_{1}}+1\right) <\log \frac{d_{1}}{c_{1}}\leq \log \frac{d}{c},
\]  
a contradiction. Therefore $f$ cannot have a factor $f_1$ of relative degree $\frac{d}{c}$ in the interval $[\frac{d_{1}}{c_{1}},\frac{d_{2}}{c_{2}}]$. Besides, $f$ can neither have a factor $f_1$ of relative degree $\frac{d}{c}$ in the interval $[\frac{nc_{2}}{md_{2}},\frac{nc_{1}}{md_{1}}]$, for otherwise $\frac{f}{f_1}$ would have relative degree in the interval $[\frac{d_{1}}{c_{1}},\frac{d_{2}}{c_{2}}]$.

For the general case that $b_{m},\dots ,b_{n}$ are arbitrary nonzero elements of $R$ with $p\nmid b_{i}b_{m}$, we first note that $g$ is algebraically primitive too, and then we observe that since $b_{i}$ and $b_{m}$ are not divisible by $p$, the log-integral point $(\log i,0)$ and the leftmost endpoint $(\log m,\nu _p(a_m))$ in the Newton log-polygon of $f$ with respect to $p$ will not be affected by passing from $f$ to $g$, so their positions will remain the same in the Newton log-polygon of $g$. Since all the edges in the Newton log-polygon of $f$ lie above or on the line that contains its leftmost edge $P_{l}Q_{l}$, and $\nu _{p}(a_{k}b_{k})\geq \nu _{p}(a_{k})$ for all $k\in \{ m,\dots ,n\} $, the edges of the Newton log-polygon of $g$ too will lie above or on the line containing $P_{l}Q_{l}$. Moreover, since the two Newton log-polygons have the same leftmost endpoint $P_{l}$, we also deduce that the slope of the leftmost edge in the Newton log-polygon of $g$ is greater than or equal to the slope of $P_{l}Q_{l}$.
This can also be seen by observing that
\[
\underset{k>m}{{\rm min}}\frac{\nu _{p}(a_{k}b_{k})-\nu _{p}(a_{m})}{\log k-\log m}\geq \underset{k>m}{{\rm min}}\frac{\nu _{p}(a_{k})-\nu _{p}(a_{m})}{\log k-\log m},
\]
with the left and the right sides being the slopes of the leftmost edges in the Newton log-polygons of $g$ and $f$, respectively. Thus $g$ satisfies the same conditions that are imposed on $f$ in the statement of the lemma, so by the first part of the proof we conclude again that $g$ cannot have a factor of relative degree $\frac{d}{c}$ belonging to $[\frac{d_{1}}{c_{1}},\frac{d_{2}}{c_{2}}]\cup [\frac{nc_{2}}{md_{2}},\frac{nc_{1}}{md_{1}}]$. 
\end{proof}

The mirrored result of Theorem \ref{TeoremaCaLaFilaseta1} is the following application of Lemma \ref{calaFilaseta2}.

\begin{theorem}\label{TeoremaCaLaFilaseta2}
Let $f(s)=\frac{a_{m}}{m^{s}}+\cdots +\frac{a_{n}}{n^{s}}$ be an algebraically primitive Dirichlet polynomial with coefficients in a unique factorization domain $R$, $a_{m}a_{n}\neq 0$, let $p$ be a prime element of $R$, and assume that $p\nmid a_{k}$ for some index $k>\frac{n}{\delta (m,n)}$ and $p\mid a_{j}$ for all $j\leq \frac{n}{\delta (m,n)}$. 
If
\begin{equation}\label{pantamare}
\rho (m,n)<\left( \frac{i}{m}\right) ^{\frac{1}{\nu _{p}(a_{m})-\nu _{p}(a_{i})}}\ for\ all\ i>m\ with\ \nu _{p}(a_{i})<\nu _{p}(a_{m}),
\end{equation}
then for any nonzero elements $b_{m},\dots ,b_{n}\in R$ with $p\nmid b_{m}b_{n}$, the Dirichlet polynomial $g(s)=\frac{a_{m}b_{m}}{m^{s}}+\cdots +\frac{a_{n}b_{n}}{n^{s}}$ is irreducible over $Q(R)$.
\end{theorem}
\begin{proof}\  As in the case of Theorem \ref{TeoremaCaLaFilaseta1}, we may assume that $\rho (m,n)>1$, so 1/$\log \rho (m,n)$ makes sense. We observe now that $k$ is among the indices $i>m$ with $\nu _{p}(a_{i})<\nu _{p}(a_{m})$, so the set of indices $i>m$ with $\nu _{p}(a_{i})<\nu _{p}(a_{m})$ in (\ref{pantamare}) is nonempty. Next, we observe that by (\ref{pantamare}), the slope $S$ of the leftmost edge in the Newton log-polygon of $f$ with respect to $p$ satisfies
\[
S=\underset{i>m}{{\rm min}}\thinspace\frac{\nu _{p}(a_{i})-\nu _{p}(a_{m})}{\log i-\log m}>\frac{-1}{\log \rho (m,n)}.
\] 
By Lemma \ref{calaFilaseta2} with $\frac{d_1}{c_1}=\delta (m,n)$, $\frac{d_2}{c_2}=\rho (m,n)$, and $i=k$, our assumptions that $p\nmid a_{k}$ for some index $k>\frac{n}{\delta (m,n)}$ and $p\mid a_{j}$ for all $j\leq \frac{n}{\delta (m,n)}$, together with the fact that the leftmost edge in the Newton log-polygon of $f$ has slope $>-1/\log \rho (m,n)$, will prevent $g$ to have factors with relative degree in the interval $[\delta (m,n),\rho (m,n)]$. On the other hand, as $g$ is algebraically primitive, if it would be reducible, it should obviously have a nonconstant factor whose relative degree lies in this interval. This completes the proof.
\end{proof}

In particular, for $m=1$ one obtains the following simpler instance of Lemma \ref{calaFilaseta2}.
\begin{lemma}\label{calaFilaseta2m=1}
Let $f(s)=\frac{a_{1}}{1^s}+\frac{a_2}{2^s}\cdots +\frac{a_{n}}{n^{s}}$ be a Dirichlet polynomial with coefficients in a unique factorization domain $R$, $a_{1}a_{n}\neq 0$, let $p$ be a prime element of $R$, and $d_{1},d_{2}$ divisors of $n$ with $1<d_1\leq d_2<n$. Suppose that $p\mid a_{j}$ for all $j\in \{ 1,\dots ,\frac{n}{d_{1}}\} $, $p\nmid a_{i}$ for some index $i>\frac{n}{d_{1}}$, and the leftmost edge in the Newton log-polygon of $f$ with respect to $p$ has slope $>-1/\log d_2$. Then for any nonzero elements $b_{1},\dots ,b_{n}\in R$ with $p\nmid b_{1}b_{i}$, the Dirichlet polynomial $g(s)=\frac{a_{1}b_{1}}{1^s}+\frac{a_2b_2}{2^s}\cdots +\frac{a_{n}b_{n}}{n^{s}}$ cannot have factors whose degrees belong to $[d_1,d_2]\cup [\frac{n}{d_{2}},\frac{n}{d_{1}}]$.
\end{lemma}
A slightly weaker instance of Theorem \ref{TeoremaCaLaFilaseta2} is the following result that mirrors Theorem \ref{coroFilaseta1}.
\begin{theorem}\label{coroFilaseta2}
Let $f(s)=\frac{a_{m}}{m^{s}}+\cdots +\frac{a_{n}}{n^{s}}$ be an algebraically primitive Dirichlet polynomial with coefficients in a unique factorization domain $R$, $a_{m}a_{n}\neq 0$, let $p$ be a prime element of $R$, and assume that $p\nmid a_{n}$ and $p\mid a_{j}$ for all $j<n$. 
If
\begin{equation}\label{pantamareOLD}
\rho (m,n)<\left( \frac{i}{m}\right) ^{\frac{1}{\nu _{p}(a_{m})-\nu _{p}(a_{i})}}\ for\ all\ i>m\ with\ \nu _{p}(a_{i})<\nu _{p}(a_{m}),
\end{equation}
then for any nonzero elements $b_{m},\dots ,b_{n}\in R$ with $p\nmid b_{m}b_{n}$, the Dirichlet polynomial $g(s)=\frac{a_{m}b_{m}}{m^{s}}+\cdots +\frac{a_{n}b_{n}}{n^{s}}$ is irreducible over $Q(R)$.
\end{theorem}

\begin{remark}\label{rho-Eisenstein2}
We note that Theorem \ref{naiveEisenstein} i) is a special case of Theorem \ref{coroFilaseta2}, for if $p\mid a_{i}$ for each $i=m,\dots ,n-1$, $p\nmid a_{n}$ and $p^{2}\nmid a_{m}$, then condition (\ref{pantamareOLD}) reduces to $\rho (m,n)<\frac{n}{m}$, which obviously holds.
\end{remark}
The final result in this section is the following one, that mirrors Proposition \ref{comblinear1}:
\begin{proposition}\label{comblinear2} 
Let $f(s)=\frac{a_{t}}{t^{s}}+\cdots +\frac{a_{n}}{n^{s}}$ and $g(s)=\frac{b_{1}}{1^s}+\frac{b_{2}}{2^{s}}+\cdots +\frac{b_{n}}{n^{s}}$ be two Dirichlet polynomials with coefficients in a unique factorization domain $R$, with $a_{t}a_{n}b_{1}b_{n}\neq 0$. Assume that $t>\frac{n}{d}\geq \frac{n}{D}$, with $d,D$ divisors of $n$, $D=\max\{ \delta:\delta\mid n,\ \delta\leq \sqrt{n}\} $, and let $p$ be a prime element of $R$ with $p\nmid a_{t}a_{n}b_{1}$. Then for any positive integer $k$ with
\begin{equation}\label{inegk2}
k<\log _{D}t,
\end{equation}
the Dirichlet polynomial $f(s)+p^{k}g(s)$ cannot have factors with degrees in $[d,D]\cup [\frac{n}{D},\frac{n}{d}]$.
\end{proposition}
\begin{proof}\ Here we first observe that since $D\leq \sqrt{n}$, we have $D<t$, as $t>\frac{n}{d}\geq \frac{n}{D}\geq D$, so the set of positive integers $k$ satisfying (\ref{inegk2}) is nonempty. We also observe that $f(s)+p^{k}g(s)$ is algebraically primitive, as its constant term $p^{k}b_{1}$ is nonzero.  If we write now $f(s)+p^{k}g(s)=\sum _{i=1}^{n}\frac{c_{i}}{i^{s}}$, with $c_{1}c_{n}\ne 0$ ($c_{n}=a_n+p^kb_n\neq 0$ as $p\nmid a_{n}$), we see that $p\nmid c_{t}$, $p^{k}\mid c_{i}$ for $i\leq t-1$, and $p^{k+1}\nmid c_{1}$, as $p\nmid a_{t}b_{1}$. Therefore the leftmost edge in the Newton log-polygon of $f(s)+p^{k}g(s)$ with respect to $p$ has endpoints $(0,k)$ and $(\log t,0)$, so has slope $S$ satisfying
\[
S=\frac{-k}{\log t}>\frac{-\log _{D}t}{\log t}=\frac{-1}{\log D}.
\]
The conclusion follows now by Lemma \ref{calaFilaseta2m=1} with $d_1=d$, $d_2=D$ and $i=t$.
\end{proof}

\section{Irreducibility criteria that use two or more $p$-adic valuations}\label{moreprimes}

Using a Newton polygon method, we provided in \cite{BoncioCommAlg} several irreducibility criteria that use simultaneously information on the divisibility of the coefficients of a polynomial with respect to two or more prime numbers. Two such results are the following ones.
\medskip

{\bf Theorem A.} \ {\em
Let $f(X)\in\mathbb{Z}[X]$ be a polynomial of degree $n$, let $k\geq 2$, and let $p_{1},\dots ,p_{k}$ be pairwise distinct prime numbers. For $i=1,\dots , k$ let us denote by $w_{i,1},\dots ,w_{i,n_{i}}$ the widths of all the segments of  the Newton polygon of $f$ with respect to $p_{i}$, and by $\mathcal{S}_{p_{i}}$ the set of all the integers in the interval  $(0,\lfloor \frac{n}{2}\rfloor ]$ that may be written as a linear combination of $w_{i,1},\dots ,w_{i,n_{i}}$ with coefficients $0$ or $1$. If $\mathcal{S}_{p_{1}}\cap\cdots \cap\mathcal{S}_{p_{k}}=\emptyset$,
then $f$ is irreducible over $\mathbb{Q}$.
}
\medskip

We mention here that the integers $w_{i,1},\dots ,w_{i,n_{i}}$ are not necessarily distinct. 

\medskip

{\bf Theorem B.}\ {\em 
Let $f(X)\in\mathbb{Z}[X]$ be a polynomial of degree $n$, let $k\geq 2$, let
$p_{1},\dots ,p_{k}$ be pair-wise distinct prime numbers, and for $i=1,\dots ,k$ let $d_{p_{i}}$ denote the greatest common divisor of the widths of the segments of  the Newton polygon of $f$ with respect to $p_{i}$. Then the degree of any factor of $f$ is divisible by $n/\gcd(\frac{n}{d_{p_{1}}},\dots ,\frac{n}{d_{p_{k}}})$. In particular, if $\gcd(\frac{n}{d_{p_{1}}},\dots ,\frac{n}{d_{p_{k}}})=1$,
then $f$ is irreducible over $\mathbb{Q}$.
}
\medskip

An immediate application of Theorem \ref{calaDumas} is the following analogue of Theorem A, that allows one to combine information on the Newton log-polygons of a Dirichlet polynomial $f$ with respect to different prime elements in $R$ in order to study its irreducibility. By contrast to its simple proof, this result will provide us with a lot of unexpected corollaries, as we shall see later on. Before stating it, we will recall the definition of $S^{1}_{rd}(m,n)$:
\begin{equation*}
S^{1}_{rd}(m,n) := \left\{ \frac{d}{c}\ :\ d\mid n,\ c\mid m,\  \ {\rm and}\ \ 1<\frac{d}{c}\leq \sqrt{\frac{n}{m}}\thinspace \right\},
\end{equation*}
and the fact that for a segment $AB$ of the Newton log-polygon of $f$ with endpoints $A=(\log i, x_{i})$ and $B=(\log j, x_{j})$ with $i<j$, the rational number $\frac{j}{i}>1$ is called the {\it relative degree} of $AB$.
\begin{theorem}\label{Spuri}
Let $f(s)=\frac{a_{m}}{m^{s}}+\cdots +\frac{a_{n}}{n^{s}}$ be an algebraically primitive Dirichlet polynomial with coefficients in a unique factorization domain $R$, $a_{m}a_{n}\neq 0$, and $p_{1},\dots ,p_{k}$ prime elements in $R$. For each $i=1,\dots , k$ let $q_{i,1},\dots ,q_{i,n_{i}}$ be the relative degrees of all the segments of the Newton log-polygon of $f$ with respect to $p_{i}$, repetitions allowed, and define
\[
\mathcal{S}_{p_{i}}=S^{1}_{rd}(m,n) \cap \left\{ q_{i,1}^{j_{1}}\cdots q_{i,n_{i}}^{j_{n_{i}}}\ :\ j_{1},\dots ,j_{n_{i}}\in \{0,1\}\right\} .
\]
If $\mathcal{S}_{p_{1}}\cap\cdots \cap\mathcal{S}_{p_{k}}=\emptyset$,
then $f$ is irreducible over $Q(R)$.
\end{theorem}
Here the relative degrees $q_{i,1},\dots ,q_{i,n_{i}}$ of the segments are not necessarily distinct, and are considered with ``their multiplicities'', since an edge may consist of more than a single segment, and all the segments on an edge have the same relative degree. We also recall the fact that $q_{i,1},\dots ,q_{i,n_{i}}$ are rational numbers greater than $1$, that satisfy the relation $q_{i,1}\cdots q_{i,n_{i}}=\frac{n}{m}$ for each $i=1,\dots ,k$.

\begin{proof}\ We first note that $1\not\in \mathcal{S}_{p_{i}}$, since $1\not\in S^{1}_{rd}(m,n)$, so in the definition of $\mathcal{S}_{p_{i}}$ we may assume that at least one of the exponents $j_{1},\dots ,j_{n_{i}}$ is nonzero. Let us assume to the contrary that $f$ is reducible, and let $g(s)=\frac{b_{c}}{c^{s}}+\cdots +\frac{b_{d}}{d^{s}}$ be a nonconstant factor of $f$ with minimum relative degree $\frac{d}{c}$, so $c\mid m$, $d\mid n$ and $\frac{d}{c}\in S^{1}_{rd}(m,n)$. Let us fix now an integer $i\in \{1,\dots ,k\}$, and let us look at the Newton log-polygon of $f$ with respect to $p_{i}$. In view of Theorem \ref{DumDir}, $\log d-\log c$ must be the sum of some of the widths $\log q_{i,1},\dots , \log q_{i,n_{i}}$ of the segments in the Newton log-polygon of $f$ with respect to $p_i$ (in such a sum, each $\log q_{i,j}$ may enter at most once, but two such terms entering the sum, say $\log q_{i,j_{1}}$ and $\log q_{i,j_{2}}$, may be equal, if they are the widths of two segments lying on the same edge). In other words, $\frac{d}{c}$ must be of the form
\[
\frac{d}{c}=q_{i,1}^{j_{1}}\cdots q_{i,n_{i}}^{j_{n_{i}}},
\]
for some exponents $j_{1},\dots ,j_{n_{i}}\in \{0,1\}$, not all zero, so $\frac{d}{c}$ must belong to $\mathcal{S}_{p_{i}}$. Since $i$ was chosen arbitrarily, $\frac{d}{c}$ must belong to each $\mathcal{S}_{p_{i}}$, $i=1,\dots ,k$, which obviously cannot hold, as $\mathcal{S}_{p_{1}}\cap\cdots \cap\mathcal{S}_{p_{k}}=\emptyset$, according to our hypotheses. Therefore $f$ must be irreducible over $Q(R)$, and this completes the proof of the theorem. \end{proof}

Another immediate application of Theorem \ref{calaDumas} that allows one to combine information on the Newton log-polygons of $f$ with respect to different prime elements in $R$, is the following analogue of Theorem B.

\begin{theorem}\label{gcduri}
Let $f$ be an algebraically primitive Dirichlet polynomial with coefficients in a unique factorization domain $R$, and let $p_{1},\dots ,p_{k}$ be prime elements in $R$. For each $i=1,\dots , k$ let $q_{i,1},\dots ,q_{i,n_{i}}$ be the relative degrees of all the segments of the Newton log-polygon of $f$ with respect to $p_{i}$, and let $d_{i}$ be the greatest common divisor of the numerators of the rationals $q_{i,1},\dots ,q_{i,n_{i}}$ written in lowest terms. If ${\rm lcm}(d_{1},\dots ,d_{k})=\deg f$,
then $f$ is irreducible over $Q(R)$.
\end{theorem}
\begin{proof}\ Assume that $f(s)=\frac{a_{m}}{m^{s}}+\cdots +\frac{a_{n}}{n^{s}}$, and let $g(s)=\frac{b_{c}}{c^{s}}+\cdots +\frac{b_{d}}{d^{s}}$ be a nonconstant factor of $f$ of minimum degree. Obviously $c\mid m$ and $d\mid n$.
If we fix an integer $i\in \{1,\dots ,k\}$ and look at the Newton log-polygon of $f$ with respect to $p_{i}$, then by Theorem \ref{DumDir}, $\log d-\log c$ must be the sum of some of the widths $\log q_{i,1},\dots , \log q_{i,n_{i}}$ of the segments, so we deduce as in the proof of Theorem \ref{Spuri} that
\[
\frac{d}{c}=q_{i,1}^{j_{1}}\cdots q_{i,n_{i}}^{j_{n_{i}}},
\]
for some exponents $j_{1},\dots ,j_{n_{i}}\in \{0,1\}$, not all zero. On the other hand, since $q_{i,1},\dots q_{i,n_{i}}$ are written in lowest terms, $d_{i}$ must be coprime to all the denominators of $q_{i,1},\dots q_{i,n_{i}}$, so $d_{i}$ must divide $d$. Since this must hold for all $i=1,\dots ,k$, $d$ must be a common multiple of $d_{1},\dots ,d_{k}$, and hence a multiple of ${\rm lcm}(d_{1},\dots ,d_{k})$. Therefore, if ${\rm lcm}(d_{1},\dots ,d_{k})=n$, $d$ must be equal to $n$, forcing $f$ to be irreducible. This completes the proof. \end{proof}

In particular, for Dirichlet polynomials $f(s)=\frac{a_{m}}{m^{s}}+\cdots +\frac{a_{n}}{n^{s}}$ with $\gcd (m,n)=1$ we have the following result, which allows us to extend Theorem \ref{calaDumas} by using divisibility conditions for the coefficients of $f$ with respect to any number of prime elements in $R$.
\begin{theorem}\label{DumasMultiPrimes}
Let $f(s)=\frac{a_{m}}{m^{s}}+\cdots +\frac{a_{n}}{n^{s}}$ be an algebraically primitive Dirichlet polynomial with coefficients in a unique factorization domain $R$, with $a_{m}a_{n}\neq 0$ and $\gcd (m,n)=1$, and let
$p_{1},\dots ,p_{t}$ be prime elements of $R$. Let $q_{1},\dots ,q_{k}$ be all the prime factors of $m\cdot n$. If
\smallskip

i) \ \thinspace $\nu _{p_{j}}(a_{m})\neq \nu _{p_{j}}(a_{n})$, for $j=1,\dots ,t$,

ii) $\left ( \frac{n}{i}\right ) ^{\nu _{p_{j}}(a_{i})-\nu _{p_{j}}(a_{m})}>\left ( \frac{m}{i}\right ) ^{\nu _{p_{j}}(a_{i})-\nu _{p_{j}}(a_{n})}$ for $m<i<n$, $j=1,\dots ,t$,

iii) $\gcd (\nu _{p_{1}}(a_{n})-\nu _{p_{1}}(a_{m}),\dots ,\nu _{p_{t}}(a_{n})-\nu _{p_{t}}(a_{m}),\nu _{q_{1}}(n)-\nu _{q_{1}}(m),\dots , \nu _{q_{k}}(n)-\nu _{q_{k}}(m))=1$,
\smallskip

\noindent then $f$ is irreducible over $Q(R)$.
\end{theorem}
\begin{proof}\ Note first that $f$ is algebraically primitive, as $m$ and $n$ are coprime. Conditions i) and ii) show that for each $j=1,\dots ,t$ the Newton log-polygon of $f$ with respect to $p_{j}$ consists of a single edge $P_{j}Q_{j}$, say. Let us fix an index $j\in \{ 1,\dots ,t\}$ and let $g(s)=\frac{b_{c}}{c^{s}}+\cdots +\frac{b_{d}}{d^{s}}$ be a nonconstant factor of $f$. Obviously $c\mid m$ and $d\mid n$. By Theorem \ref{DumDir}, $\log d-\log c$ must be the sum of some of the widths of the segments that the edge $P_{j}Q_{j}$ is composed of, so by Lemma \ref{puncte}, it must be a multiple of $\frac{1}{\delta _{j}}(\log n-\log m)$, say
\[
\log d-\log c=\frac{M_{j}}{\delta _{j}}(\log n-\log m)
\]
with $\delta _{j}=\gcd (\nu _{p_{j}}(a_{n})-\nu _{p_{j}}(a_{m}),\nu _{q_{1}}(n)-\nu _{q_{1}}(m),\dots ,\nu _{q_{k}}(n)-\nu _{q_{k}}(m))$, and $M_{j}$ an integer with $0<M_{j}\leq \delta _{j}$. This equality leads us to
\begin{equation}\label{Mjuri}
\frac{d}{c}=\left( \frac{n}{m}\right) ^{\frac{M_{j}}{\delta _{j}}}\quad {\rm for}\  j=1,\dots ,t.
\end{equation}
We shall prove now that $M_{j}=\delta _{j}$ for each $j$. Let us assume to the contrary that this is false. Note that $\frac{M_j}{\delta _j}$ are all rational numbers, which by (\ref{Mjuri}) must all be equal, so we may write
\[
\frac{M_{1}}{\delta _{1}}=\cdots =\frac{M_{t}}{\delta _{t}}=\frac{a}{b},
\]
with $0<a<b$ and $\gcd (a,b)=1$. This yields $a\delta _{j}=bM_{j}$ for each $j$, which, since $a$ and $b$ are coprime, forces $b$ to divide $\delta _{j}$ for each $j$. On the other hand, by condition iii) we see that $\delta _{1},\dots ,\delta _{t}$ must be coprime, so $b=1$, which leaves no possibility for $a$. Therefore $M_{j}=\delta _{j}$ for each $j$, so by (\ref{Mjuri}) $\frac{d}{c}$ and $\frac{n}{m}$ must be equal too. Since $d\mid n$, $c\mid m$, and $m$ and $n$ are coprime, we finally see that $d$ must be equal to $n$ (and $c$ must be equal to $m$), which proves that $f$ has no nonconstant factors other than itself, thus being irreducible. 
\end{proof}

In what follows, we will give a series of explicit irreducibility criteria that use divisibility conditions with respect to two prime elements $p,q \in R$ for the case that each of the Newton log-polygons with respect to $p$ and $q$ has at most two edges. Figure 5 below displays all the relevant such combinations of shapes for the two Newton log-polygons, depending on the sign of the slopes of the two edges. Here we had to exclude the four cases obtained by flipping horizontally Figures 5.g, 5.h, 5.i and 5.k, as they would contain a rightmost horizontal edge. Since we will be only interested in edges that contain a single segment, such a horizontal edge would have endpoints $(\log (n-1),0)$ and $(\log n,0)$, thus forcing $a_{n-1}\neq 0$ and hence the irreducibility of $f$ by Corollary \ref{coro1}. Note that 
the situation analysed in Theorem \ref{DumasMultiPrimes} corresponds to Figure 5.n. For the $13$ remaining cases we will provide only the simplest possible irreducibility conditions, where inequalities like the one in condition ii) in Theorem \ref{DumasMultiPrimes} are automatically satisfied, thus keeping the statements as close as possible to the one of the Sch\"onemann-Eisenstein criterion for polynomials. By also considering these inequalities, one may immediately state the corresponding $13$ results in full generality, but with more involved statements. For the proof of the following 13 results we actually don't need the full strength of Theorem \ref{Spuri}, in the sense that we will not have to compute $S^{1}_{rd}(m,n)$. Instead, to exclude some possible values of the relative degrees of the hypothetical factors of $f$, it will suffice to invoke Theorem \ref{DumDir} and Lemma \ref{puncte}.
\begin{center}

\setlength{\unitlength}{2mm}
\begin{picture}(58,7)
\linethickness{0.15mm}

{\small \put(6,3){$a$}}
\setlength{\unitlength}{2mm}

\linethickness{0.15mm}

\put(12,0){\line(0,1){3.3}}   
\thicklines

\put(0,4){\line(5,-3){6.6666}}   

\put(6.6666,0){\line(3,1){5.3334}}   

\linethickness{0.15mm}

\put(0,1.5){\circle{0.12}}
\put(0,1.5){\circle{0.08}}
\put(0.5,1.25){\circle{0.08}}
\put(1,1){\circle{0.08}}
\put(1.5,0.75){\circle{0.08}}
\put(2,0.5){\circle{0.08}}
\put(2.5,0.25){\circle{0.08}}
\put(3,0){\circle{0.12}}
\put(3,0){\circle{0.08}}

\put(3.5,0.15){\circle{0.08}}
\put(4,0.3){\circle{0.08}}
\put(4.5,0.45){\circle{0.08}}
\put(5,0.6){\circle{0.08}}
\put(5.5,0.75){\circle{0.08}}
\put(6,0.9){\circle{0.08}}
\put(6.5,1.05){\circle{0.08}}
\put(7,1.2){\circle{0.08}}
\put(7.5,1.35){\circle{0.08}}
\put(8,1.5){\circle{0.08}}
\put(8.5,1.65){\circle{0.08}}
\put(9,1.8){\circle{0.08}}
\put(9.5,1.95){\circle{0.08}}
\put(10,2.1){\circle{0.08}}
\put(10.5,2.25){\circle{0.08}}
\put(11,2.4){\circle{0.08}}
\put(11.5,2.55){\circle{0.08}}
\put(12,2.7){\circle{0.08}}

\put(0,0){\vector(1,0){13.5}}
\put(0,0){\vector(0,1){5.5}}

\put(6.6666,0){\circle{0.12}}

\put(0,4){\circle{0.12}}
\put(0,4){\circle{0.08}}

\put(12,1.7778){\circle{0.12}}
\put(12,1.7778){\circle{0.08}}


{\small \put(21,3){$b$}}

\setlength{\unitlength}{2mm}

\linethickness{0.15mm}

\put(27,0){\line(0,1){4}}   
\thicklines

\put(15,4){\line(2,-1){6}}   
\put(21,1){\line(6,-1){6}}   
\linethickness{0.15mm}

\put(15,4){\circle{0.08}}
\put(15,4){\circle{0.12}}
\put(21,1){\circle{0.08}}
\put(21,1){\circle{0.12}}
\put(27,0){\circle{0.08}}
\put(27,0){\circle{0.12}}

\put(15,0){\circle{0.12}}
\put(15,0){\circle{0.08}}
\put(15.5,0.05){\circle{0.08}}
\put(16,0.1){\circle{0.08}}
\put(16.5,0.15){\circle{0.08}}
\put(17,0.2){\circle{0.08}}
\put(17.5,0.25){\circle{0.08}}
\put(18,0.3){\circle{0.08}}
\put(18.5,0.35){\circle{0.08}}
\put(19,0.4){\circle{0.08}}
\put(19.5,0.45){\circle{0.08}}
\put(20,0.5){\circle{0.08}}
\put(20.5,0.55){\circle{0.08}}
\put(21,0.6){\circle{0.08}}
\put(21.5,0.65){\circle{0.08}}
\put(22,0.7){\circle{0.08}}
\put(22.5,0.75){\circle{0.08}}
\put(23,0.8){\circle{0.08}}
\put(23.5,0.85){\circle{0.08}}
\put(24,0.9){\circle{0.08}}
\put(24,0.9){\circle{0.12}}
\put(24.5,1.2){\circle{0.08}}
\put(25,1.6){\circle{0.08}}
\put(25.5,2){\circle{0.08}}
\put(26,2.4){\circle{0.08}}
\put(26.5,2.8){\circle{0.08}}
\put(27,3.2){\circle{0.08}}
\put(27,3.2){\circle{0.12}}

\put(15,0){\vector(1,0){13.5}}
\put(15,0){\vector(0,1){5.5}}


{\small \put(36,3){$c$}}
\linethickness{0.15mm}

\put(42,0){\line(0,1){5}}   
\thicklines

\put(30,0){\line(5,1){6}}   

\put(36,1.2){\line(2,1){6}}   
\linethickness{0.15mm}

\put(36,1.2){\circle{0.08}}
\put(36,1.2){\circle{0.12}}
\put(42,4.2){\circle{0.12}}

\put(30,0){\circle{0.12}}
\put(30,0){\circle{0.08}}
\put(30.5,0.05){\circle{0.08}}
\put(31,0.1){\circle{0.08}}
\put(31.5,0.15){\circle{0.08}}
\put(32,0.2){\circle{0.08}}
\put(32.5,0.25){\circle{0.08}}
\put(33,0.3){\circle{0.08}}
\put(33.5,0.35){\circle{0.08}}
\put(34,0.4){\circle{0.08}}
\put(34.5,0.45){\circle{0.08}}
\put(35,0.5){\circle{0.08}}
\put(35.5,0.55){\circle{0.08}}
\put(36,0.6){\circle{0.08}}
\put(36.5,0.65){\circle{0.08}}
\put(37,0.7){\circle{0.08}}
\put(37.5,0.75){\circle{0.08}}
\put(38,0.8){\circle{0.08}}
\put(38,0.8){\circle{0.12}}
\put(38.5,1){\circle{0.08}}
\put(39,1.2){\circle{0.08}}
\put(39.5,1.4){\circle{0.08}}
\put(40,1.6){\circle{0.08}}
\put(40.5,1.8){\circle{0.08}}
\put(41,2){\circle{0.08}}
\put(41.5,2.2){\circle{0.08}}
\put(42,2.4){\circle{0.08}}
\put(42,2.4){\circle{0.12}}

\put(30,0){\vector(1,0){13.5}}
\put(30,0){\vector(0,1){5.5}}

{\small \put(51,3){$d$}}

\setlength{\unitlength}{2mm}

\linethickness{0.075mm}

\put(57,0){\line(0,1){5}}   
\thicklines

\put(45,4.2){\line(2,-1){6}}   
\put(51,1.2){\line(5,-1){6}}

\put(51,1.2){\circle{0.08}}
\put(45,4.2){\circle{0.12}}

\put(57,0){\circle{0.12}}
\put(56.5,0.05){\circle{0.08}}
\put(56,0.1){\circle{0.08}}
\put(55.5,0.15){\circle{0.08}}
\put(55,0.2){\circle{0.08}}
\put(54.5,0.25){\circle{0.08}}
\put(54,0.3){\circle{0.08}}
\put(53.5,0.35){\circle{0.08}}
\put(53,0.4){\circle{0.08}}
\put(52.5,0.45){\circle{0.08}}
\put(52,0.5){\circle{0.08}}
\put(51.5,0.55){\circle{0.08}}
\put(51,0.6){\circle{0.08}}
\put(50.5,0.65){\circle{0.08}}
\put(50,0.7){\circle{0.08}}
\put(49.5,0.75){\circle{0.08}}
\put(49,0.8){\circle{0.08}}
\put(49,0.8){\circle{0.12}}
\put(48.5,1){\circle{0.08}}
\put(48,1.2){\circle{0.08}}
\put(47.5,1.4){\circle{0.08}}
\put(47,1.6){\circle{0.08}}
\put(46.5,1.8){\circle{0.08}}
\put(46,2){\circle{0.08}}
\put(45.5,2.2){\circle{0.08}}
\put(45,2.4){\circle{0.08}}
\put(45,2.4){\circle{0.12}}

\linethickness{0.15mm}

\put(45,0){\vector(1,0){13.5}}
\put(45,0){\vector(0,1){5.5}}

\end{picture}
\end{center}

\begin{center}
\setlength{\unitlength}{2mm}
\begin{picture}(58,7)
\linethickness{0.15mm}

{\small \put(6,3){$e$}}
\setlength{\unitlength}{2mm}

\linethickness{0.15mm}

\put(12,0){\line(0,1){5}}   

\thicklines
\put(0,0){\line(5,1){5}}   

\put(5,1){\line(2,1){7}}   

\linethickness{0.15mm}

\put(0,4){\circle{0.12}}
\put(0,4){\circle{0.08}}
\put(0.5,3.75){\circle{0.08}}
\put(1,3.5){\circle{0.08}}
\put(1.5,3.25){\circle{0.08}}
\put(2,3){\circle{0.08}}
\put(2.5,2.75){\circle{0.08}}
\put(3,2.5){\circle{0.08}}

\put(3.5,2.25){\circle{0.08}}
\put(4,2){\circle{0.08}}
\put(4.5,1.75){\circle{0.08}}
\put(5,1.5){\circle{0.08}}
\put(5.5,1.25){\circle{0.08}}
\put(6,1){\circle{0.08}}
\put(6.5,0.75){\circle{0.08}}
\put(7,0.5){\circle{0.08}}
\put(7.5,0.25){\circle{0.08}}
\put(8,0){\circle{0.08}}
\put(8,0){\circle{0.12}}

\put(8.5,0.25){\circle{0.08}}
\put(9,0.5){\circle{0.08}}
\put(9.5,0.75){\circle{0.08}}
\put(10,1){\circle{0.08}}
\put(10.5,1.25){\circle{0.08}}
\put(11,1.5){\circle{0.08}}
\put(11.5,1.75){\circle{0.08}}
\put(12,2){\circle{0.08}}
\put(12,2){\circle{0.12}}

\put(0,0){\vector(1,0){13.5}}
\put(0,0){\vector(0,1){5.5}}

\put(0,0){\circle{0.12}}
\put(0,0){\circle{0.08}}

\put(5,1){\circle{0.12}}
\put(5,1){\circle{0.08}}

\put(0,4){\circle{0.12}}
\put(0,4){\circle{0.08}}

\put(12,4.5){\circle{0.12}}
\put(12,4.5){\circle{0.08}}

{\small \put(21,3){$f$}}
\setlength{\unitlength}{2mm}

\linethickness{0.075mm}

\put(27,0){\line(0,1){4.8}}   
\thicklines

\put(15,4.5){\line(2,-1){7}}   

\put(22,1){\line(5,-1){5}}   

\linethickness{0.15mm}

\put(27,4){\circle{0.12}}
\put(27,4){\circle{0.08}}
\put(26.5,3.75){\circle{0.08}}
\put(26,3.5){\circle{0.08}}
\put(25.5,3.25){\circle{0.08}}
\put(25,3){\circle{0.08}}
\put(24.5,2.75){\circle{0.08}}
\put(24,2.5){\circle{0.08}}

\put(23.5,2.25){\circle{0.08}}
\put(23,2){\circle{0.08}}
\put(22.5,1.75){\circle{0.08}}
\put(22,1.5){\circle{0.08}}
\put(21.5,1.25){\circle{0.08}}
\put(21,1){\circle{0.08}}
\put(20.5,0.75){\circle{0.08}}
\put(20,0.5){\circle{0.08}}
\put(19.5,0.25){\circle{0.08}}
\put(19,0){\circle{0.08}}
\put(19,0){\circle{0.12}}

\put(18.5,0.25){\circle{0.08}}
\put(18,0.5){\circle{0.08}}
\put(17.5,0.75){\circle{0.08}}
\put(17,1){\circle{0.08}}
\put(16.5,1.25){\circle{0.08}}
\put(16,1.5){\circle{0.08}}
\put(15.5,1.75){\circle{0.08}}
\put(15,2){\circle{0.08}}
\put(15,2){\circle{0.12}}

\put(15,0){\vector(1,0){13.5}}
\put(15,0){\vector(0,1){5.5}}

\put(15,0){\circle{0.12}}
\put(15,0){\circle{0.08}}

\put(22,1){\circle{0.12}}
\put(22,1){\circle{0.08}}

\put(15,4.5){\circle{0.12}}
\put(15,4.5){\circle{0.08}}

\put(27,0){\circle{0.12}}
\put(27,0){\circle{0.08}}


{\small \put(36,3){$g$}}
\linethickness{0.15mm}

\put(42,0){\line(0,1){4.5}}   

\thicklines
\put(30,0){\line(1,0){4}}   

\put(34,0){\line(2,1){8}}   

\linethickness{0.15mm}

\put(30,4){\circle{0.12}}
\put(30,4){\circle{0.08}}
\put(30.5,3.75){\circle{0.08}}
\put(31,3.5){\circle{0.08}}
\put(31.5,3.25){\circle{0.08}}
\put(32,3){\circle{0.08}}
\put(32.5,2.75){\circle{0.08}}
\put(33,2.5){\circle{0.08}}

\put(33.5,2.25){\circle{0.08}}
\put(34,2){\circle{0.08}}
\put(34.5,1.75){\circle{0.08}}
\put(35,1.5){\circle{0.08}}
\put(35.5,1.25){\circle{0.08}}
\put(36,1){\circle{0.08}}
\put(36.5,0.75){\circle{0.08}}
\put(37,0.5){\circle{0.08}}
\put(37.5,0.25){\circle{0.08}}
\put(38,0){\circle{0.08}}

\put(38.5,0.25){\circle{0.08}}
\put(39,0.5){\circle{0.08}}
\put(39.5,0.75){\circle{0.08}}
\put(40,1){\circle{0.08}}
\put(40.5,1.25){\circle{0.08}}
\put(41,1.5){\circle{0.08}}
\put(41.5,1.75){\circle{0.08}}
\put(42,2){\circle{0.08}}
\put(42,2){\circle{0.12}}

\put(30,0){\vector(1,0){13.5}}
\put(30,0){\vector(0,1){5.5}}

\put(30,0){\circle{0.12}}
\put(30,0){\circle{0.08}}

\put(34,0){\circle{0.12}}
\put(34,0){\circle{0.08}}

\put(30,4){\circle{0.12}}
\put(30,4){\circle{0.08}}

\put(42,4){\circle{0.12}}
\put(42,4){\circle{0.08}}


{\small \put(51,3){$h$}}

\setlength{\unitlength}{2mm}

\linethickness{0.15mm}

\put(57,0){\line(0,1){4.7}}   

\thicklines

\put(45,0){\line(1,0){4}}   

\put(49,0){\line(2,1){8}}   

\linethickness{0.15mm}

\put(49,0){\circle{0.08}}
\put(49,0){\circle{0.12}}

\put(57,4){\circle{0.12}}
\put(57,4){\circle{0.08}}

\put(45,0){\circle{0.12}}
\put(45,0){\circle{0.08}}
\put(45.5,0.05){\circle{0.08}}
\put(46,0.1){\circle{0.08}}
\put(46.5,0.15){\circle{0.08}}
\put(47,0.2){\circle{0.08}}
\put(47.5,0.25){\circle{0.08}}
\put(48,0.3){\circle{0.08}}
\put(48.5,0.35){\circle{0.08}}
\put(49,0.4){\circle{0.08}}
\put(49.5,0.45){\circle{0.08}}
\put(50,0.5){\circle{0.08}}
\put(50.5,0.55){\circle{0.08}}
\put(51,0.6){\circle{0.08}}
\put(51.5,0.65){\circle{0.08}}
\put(52,0.7){\circle{0.08}}
\put(52.5,0.75){\circle{0.08}}
\put(53,0.8){\circle{0.08}}
\put(53,0.8){\circle{0.12}}
\put(53.5,1){\circle{0.08}}
\put(54,1.2){\circle{0.08}}
\put(54.5,1.4){\circle{0.08}}
\put(55,1.6){\circle{0.08}}
\put(55.5,1.8){\circle{0.08}}
\put(56,2){\circle{0.08}}
\put(56.5,2.2){\circle{0.08}}
\put(57,2.4){\circle{0.08}}
\put(57,2.4){\circle{0.12}}

\put(45,0){\vector(1,0){13.5}}
\put(45,0){\vector(0,1){5.5}}

\end{picture}
\end{center}

\begin{center}
\setlength{\unitlength}{2mm}
\begin{picture}(58,7)
\linethickness{0.15mm}

{\small \put(6,3){$i$}}

\setlength{\unitlength}{2mm}

\linethickness{0.15mm}

\put(12,0){\line(0,1){4.5}}   

\thicklines
\put(0,0){\line(1,0){4}}   

\put(4,0){\line(2,1){8}}   

\linethickness{0.15mm}

\put(0,4){\circle{0.12}}
\put(0,4){\circle{0.08}}
\put(0.5,3.75){\circle{0.08}}
\put(1,3.5){\circle{0.08}}
\put(1.5,3.25){\circle{0.08}}
\put(2,3){\circle{0.08}}
\put(2.5,2.75){\circle{0.08}}
\put(3,2.5){\circle{0.08}}

\put(3.5,2.25){\circle{0.08}}
\put(4,2){\circle{0.08}}
\put(4.5,1.75){\circle{0.08}}
\put(5,1.5){\circle{0.08}}
\put(5.5,1.25){\circle{0.08}}
\put(6,1){\circle{0.08}}
\put(6.5,0.75){\circle{0.08}}

\put(7,0.5){\circle{0.12}}
\put(7,0.5){\circle{0.08}}
\put(7.5,0.45){\circle{0.08}}
\put(8,0.4){\circle{0.08}}

\put(8.5,0.35){\circle{0.08}}
\put(9,0.3){\circle{0.08}}
\put(9.5,0.25){\circle{0.08}}
\put(10,0.2){\circle{0.08}}
\put(10.5,0.15){\circle{0.08}}
\put(11,0.1){\circle{0.08}}
\put(11.5,0.05){\circle{0.08}}
\put(12,0){\circle{0.08}}
\put(12,0){\circle{0.12}}

\put(0,0){\vector(1,0){13.5}}
\put(0,0){\vector(0,1){5.5}}

\put(0,0){\circle{0.12}}
\put(0,0){\circle{0.08}}

\put(4,0){\circle{0.12}}
\put(4,0){\circle{0.08}}

\put(0,4){\circle{0.12}}
\put(0,4){\circle{0.08}}

\put(12,4){\circle{0.12}}
\put(12,4){\circle{0.08}}


{\small \put(21,3){$j$}}

\setlength{\unitlength}{2mm}

\linethickness{0.15mm}

\put(27,0){\line(0,1){4}}   
\thicklines

\put(15,1){\line(6,1){12}}   

\linethickness{0.15mm}
\put(15,4){\circle{0.12}}
\put(15,4){\circle{0.08}}
\put(15.5,3.75){\circle{0.08}}
\put(16,3.5){\circle{0.08}}
\put(16.5,3.25){\circle{0.08}}
\put(17,3){\circle{0.08}}
\put(17.5,2.75){\circle{0.08}}
\put(18,2.5){\circle{0.08}}

\put(18.5,2.25){\circle{0.08}}
\put(19,2){\circle{0.08}}
\put(19.5,1.75){\circle{0.08}}
\put(20,1.5){\circle{0.08}}
\put(20.5,1.25){\circle{0.08}}
\put(21,1){\circle{0.08}}
\put(21.5,0.75){\circle{0.08}}
\put(22,0.5){\circle{0.08}}
\put(22.5,0.25){\circle{0.08}}
\put(23,0){\circle{0.08}}

\put(23.5,0.25){\circle{0.08}}
\put(24,0.5){\circle{0.08}}
\put(24.5,0.75){\circle{0.08}}
\put(25,1){\circle{0.08}}
\put(25.5,1.25){\circle{0.08}}
\put(26,1.5){\circle{0.08}}
\put(26.5,1.75){\circle{0.08}}
\put(27,2){\circle{0.08}}
\put(27,2){\circle{0.12}}

\put(15,1){\circle{0.12}}
\put(15,1){\circle{0.08}}

\put(27,3){\circle{0.12}}
\put(27,3){\circle{0.08}}

\put(15,4){\circle{0.12}}
\put(15,4){\circle{0.08}}

\put(27,4){\circle{0.12}}
\put(27,4){\circle{0.08}}

\put(15,0){\vector(1,0){13.5}}
\put(15,0){\vector(0,1){5.5}}


{\small \put(36,3){$k$}}
\setlength{\unitlength}{2mm}

\linethickness{0.15mm}

\put(42,0){\line(0,1){4}}   

\thicklines

\put(30,0){\line(1,0){6.6666}}   
\put(36.6666,0){\line(3,1){5.3334}}   
\linethickness{0.15mm}

\put(30,0.5){\circle{0.12}}
\put(30,0.5){\circle{0.08}}
\put(30.5,0.6){\circle{0.08}}
\put(31,0.7){\circle{0.08}}
\put(31.5,0.8){\circle{0.08}}
\put(32,0.9){\circle{0.08}}
\put(32.5,1){\circle{0.08}}
\put(33,1.1){\circle{0.08}}
\put(33.5,1.2){\circle{0.08}}
\put(34,1.3){\circle{0.08}}
\put(34.5,1.4){\circle{0.08}}
\put(35,1.5){\circle{0.08}}
\put(35.5,1.6){\circle{0.08}}
\put(36,1.7){\circle{0.08}}
\put(36.5,1.8){\circle{0.08}}
\put(37,1.9){\circle{0.08}}
\put(37.5,2){\circle{0.08}}
\put(38,2.1){\circle{0.08}}
\put(38.5,2.2){\circle{0.08}}
\put(39,2.3){\circle{0.08}}
\put(39.5,2.4){\circle{0.08}}
\put(40,2.5){\circle{0.08}}
\put(40.5,2.6){\circle{0.08}}
\put(41,2.7){\circle{0.08}}
\put(41.5,2.8){\circle{0.08}}
\put(42,2.9){\circle{0.08}}
\put(42,2.9){\circle{0.12}}

\put(30,0){\vector(1,0){13.5}}
\put(30,0){\vector(0,1){5.5}}

\put(36.6666,0){\circle{0.12}}

\put(42,1.7778){\circle{0.12}}
\put(42,1.7778){\circle{0.08}}


{\small \put(51,3.2){$l$}}
\setlength{\unitlength}{2mm}

\linethickness{0.15mm}

\put(57,0){\line(0,1){4}}  

\thicklines

\put(45,1){\line(6,1){12}}   

\linethickness{0.15mm}

\put(45,0){\circle{0.12}}
\put(45,0){\circle{0.08}}

\put(45,1){\circle{0.08}}
\put(45,1){\circle{0.12}}

\put(57,3){\circle{0.12}}
\put(57,3){\circle{0.08}}

\put(45,0){\circle{0.12}}
\put(45,0){\circle{0.08}}
\put(45.5,0.05){\circle{0.08}}
\put(46,0.1){\circle{0.08}}
\put(46.5,0.15){\circle{0.08}}
\put(47,0.2){\circle{0.08}}
\put(47.5,0.25){\circle{0.08}}
\put(48,0.3){\circle{0.08}}
\put(48.5,0.35){\circle{0.08}}
\put(49,0.4){\circle{0.08}}
\put(49.5,0.45){\circle{0.08}}
\put(50,0.5){\circle{0.08}}
\put(50.5,0.55){\circle{0.08}}
\put(51,0.6){\circle{0.08}}
\put(51.5,0.65){\circle{0.08}}
\put(52,0.7){\circle{0.08}}
\put(52.5,0.75){\circle{0.08}}
\put(53,0.8){\circle{0.08}}
\put(53,0.8){\circle{0.12}}
\put(53.5,1){\circle{0.08}}
\put(54,1.2){\circle{0.08}}
\put(54.5,1.4){\circle{0.08}}
\put(55,1.6){\circle{0.08}}
\put(55.5,1.8){\circle{0.08}}
\put(56,2){\circle{0.08}}
\put(56.5,2.2){\circle{0.08}}
\put(57,2.4){\circle{0.08}}
\put(57,2.4){\circle{0.12}}

\put(45,0){\vector(1,0){13.5}}
\put(45,0){\vector(0,1){5.5}}

\end{picture}
\end{center}

\begin{center}
\setlength{\unitlength}{2mm}
\begin{picture}(58,7)
\linethickness{0.15mm}

{\small \put(21,3){$m$}}

\setlength{\unitlength}{2mm}

\linethickness{0.15mm}

\put(27,0){\line(0,1){4.5}}   

\thicklines

\put(15,1){\line(6,1){12}}   

\linethickness{0.15mm}

\put(15,4){\circle{0.12}}
\put(15,4){\circle{0.08}}
\put(15.5,3.75){\circle{0.08}}
\put(16,3.5){\circle{0.08}}
\put(16.5,3.25){\circle{0.08}}
\put(17,3){\circle{0.08}}
\put(17.5,2.75){\circle{0.08}}
\put(18,2.5){\circle{0.08}}

\put(18.5,2.25){\circle{0.08}}
\put(19,2){\circle{0.08}}
\put(19.5,1.75){\circle{0.08}}
\put(20,1.5){\circle{0.08}}
\put(20.5,1.25){\circle{0.08}}
\put(21,1){\circle{0.08}}
\put(21.5,0.75){\circle{0.08}}

\put(22,0.5){\circle{0.12}}
\put(22,0.5){\circle{0.08}}
\put(22.5,0.45){\circle{0.08}}
\put(23,0.4){\circle{0.08}}

\put(23.5,0.35){\circle{0.08}}
\put(24,0.3){\circle{0.08}}
\put(24.5,0.25){\circle{0.08}}
\put(25,0.2){\circle{0.08}}
\put(25.5,0.15){\circle{0.08}}
\put(26,0.1){\circle{0.08}}
\put(26.5,0.05){\circle{0.08}}
\put(27,0){\circle{0.08}}
\put(27,0){\circle{0.12}}

\put(15,0){\vector(1,0){13.5}}
\put(15,0){\vector(0,1){5.5}}

\put(15,1){\circle{0.12}}
\put(15,1){\circle{0.08}}

\put(15,4){\circle{0.12}}
\put(15,4){\circle{0.08}}

\put(27,3){\circle{0.12}}
\put(27,3){\circle{0.08}}


{\small \put(36,3){$n$}}

\put(42,0){\line(0,1){4}}   
\thicklines

\put(30,3){\line(6,-1){12}}   
\linethickness{0.15mm}

\put(30,0.5){\circle{0.12}}
\put(30,0.5){\circle{0.08}}
\put(30.5,0.6){\circle{0.08}}
\put(31,0.7){\circle{0.08}}
\put(31.5,0.8){\circle{0.08}}
\put(32,0.9){\circle{0.08}}
\put(32.5,1){\circle{0.08}}
\put(33,1.1){\circle{0.08}}
\put(33.5,1.2){\circle{0.08}}
\put(34,1.3){\circle{0.08}}
\put(34.5,1.4){\circle{0.08}}
\put(35,1.5){\circle{0.08}}
\put(35.5,1.6){\circle{0.08}}
\put(36,1.7){\circle{0.08}}
\put(36.5,1.8){\circle{0.08}}
\put(37,1.9){\circle{0.08}}
\put(37.5,2){\circle{0.08}}
\put(38,2.1){\circle{0.08}}
\put(38.5,2.2){\circle{0.08}}
\put(39,2.3){\circle{0.08}}
\put(39.5,2.4){\circle{0.08}}
\put(40,2.5){\circle{0.08}}
\put(40.5,2.6){\circle{0.08}}
\put(41,2.7){\circle{0.08}}
\put(41.5,2.8){\circle{0.08}}
\put(42,2.9){\circle{0.08}}
\put(42,2.9){\circle{0.12}}

\put(30,0){\vector(1,0){13.5}}
\put(30,0){\vector(0,1){5.5}}

\put(30,3){\circle{0.12}}
\put(30,3){\circle{0.08}}

\put(42,1){\circle{0.12}}
\put(42,1){\circle{0.08}}

\end{picture}
\end{center}

\medskip

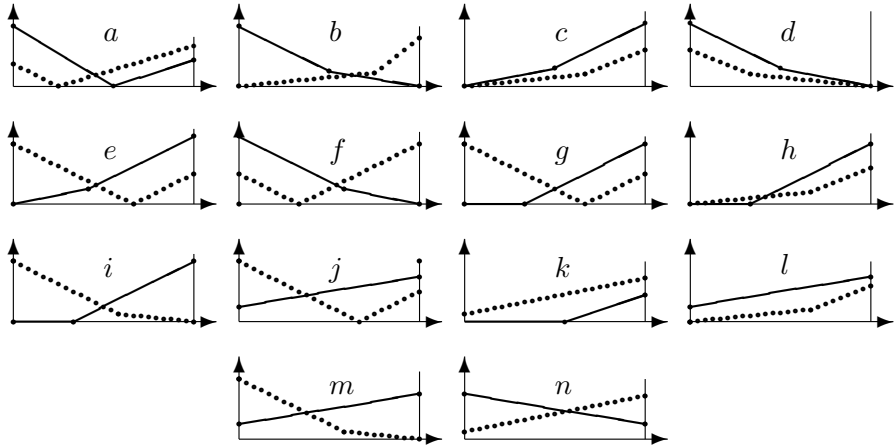
{\em {\small  {\bf Figure 5.}  Relevant combinations of shapes of Newton log-polygons of a Dirichlet polynomial with respect to a pair of prime elements, each one having at most two segments.
}
}
\medskip

Most of the irreducibility conditions for polynomials in the literature are asymmetric with respect to the usual and reverse ordering of the coefficients. Among the results in \cite{BoncioCommAlg}, we will also mention one providing irreducibility conditions which use two prime numbers and are invariant under reversing the order of the coefficients: 

\medskip

{\bf Theorem C.} (\cite[Corollary 1.14]{BoncioCommAlg})\ {\em Let $f(X)=a_{0}+a_{1}X+\cdots +a_{n}X^{n}\in \mathbb{Z}[X]$, 
$a_{0}a_{n}\neq 0$. If there exist two distinct indices $j,k\in \{1,\dots ,n-1\} $ 
such that $j+k\neq n$ and two distinct prime numbers $p$ and $q$ such that

i) \ $p\mid a_{i}$ \ for \ $i\neq j$, \ $p\nmid a_{j}$, \ $p^{2}\nmid a_{0}$ \ and \ $p^{2}\nmid a_{n}$,

ii) $q\mid a_{i}$ \ for \ $i\neq k$, \ $q\nmid a_{k}$, \ $q^{2}\nmid a_{0}$ \ and \ $q^{2}\nmid a_{n}$,

\noindent then $f$ is irreducible over $\mathbb{Q}$.}
\medskip

The following result, that uses symmetric irreducibility conditions too, may be considered as an analogue for Dirichlet polynomial of Theorem C, and corresponds to Figure 5.a. 

\begin{theorem}\label{npnp}
Let $f(s)=\frac{a_{m}}{m^{s}}+\cdots +\frac{a_{n}}{n^{s}}$ be an algebraically primitive Dirichlet polynomial with coefficients in a unique factorization domain $R$, with $a_{m}a_{n}\neq 0$ and let
$p,q$ be two prime elements of $R$. Assume there exist two different indices $j_{1},j_{2}\in \{ m+1,\dots ,n-1\} $ with $j_{1}j_{2}\neq mn$ such that:

i) \thinspace \thinspace $p\mid a_{i}$ for all $i\neq j_{1}$, $p\nmid a_{j_{1}}$, $p^{2}\nmid a_{m}$ and $p^{2}\nmid a_{n}$;

ii) \thinspace $q\mid a_{i}$ for all $i\neq j_{2}$, $q\nmid a_{j_{2}}$, $q^{2}\nmid a_{m}$ and $q^{2}\nmid a_{n}$;

\noindent Then $f$ is irreducible over $Q(R)$.
\end{theorem}
\begin{proof}\ Let us represent in Figure 6 the Newton log-polygons with respect to two prime elements $p$ and $q$, each one consisting of two edges having slopes with different signs, as in Figure 5.a. 
\begin{center}
\setlength{\unitlength}{4.65mm}
\begin{picture}(12,6)
\linethickness{0.075mm}

\put(12,0){\line(0,1){3.5}}   
\thicklines

\put(0,4){\line(5,-3){6.6666}}   

\put(6.6666,0){\line(3,1){5.3334}}

\put(0,3){\circle{0.12}}
\put(0,3){\circle{0.08}}
\put(0.25,2.75){\circle{0.08}}
\put(0.5,2.5){\circle{0.08}}
\put(0.75,2.25){\circle{0.08}}
\put(1,2){\circle{0.08}}
\put(1.25,1.75){\circle{0.08}}
\put(1.5,1.5){\circle{0.08}}
\put(1.75,1.25){\circle{0.08}}
\put(2,1){\circle{0.08}}
\put(2.25,0.75){\circle{0.08}}
\put(2.5,0.5){\circle{0.08}}
\put(2.75,0.25){\circle{0.08}}
\put(3,0){\circle{0.12}}
\put(3,0){\circle{0.08}}

\put(3,0){\circle{0.12}}
\put(3,0){\circle{0.08}}
\put(3.35,0.1){\circle{0.08}}
\put(3.7,0.2){\circle{0.08}}
\put(4.05,0.3){\circle{0.08}}
\put(4.4,0.4){\circle{0.08}}
\put(4.75,0.5){\circle{0.08}}
\put(5.1,0.6){\circle{0.08}}
\put(5.45,0.7){\circle{0.08}}
\put(5.8,0.8){\circle{0.08}}
\put(6.15,0.9){\circle{0.08}}
\put(6.5,1){\circle{0.08}}
\put(6.85,1.1){\circle{0.08}}
\put(7.2,1.2){\circle{0.08}}
\put(7.55,1.3){\circle{0.08}}
\put(7.9,1.4){\circle{0.08}}
\put(8.25,1.5){\circle{0.08}}
\put(8.6,1.6){\circle{0.08}}
\put(8.95,1.7){\circle{0.08}}
\put(9.3,1.8){\circle{0.08}}
\put(9.65,1.9){\circle{0.08}}
\put(10,2){\circle{0.08}}
\put(10.35,2.1){\circle{0.08}}
\put(10.7,2.2){\circle{0.08}}
\put(11.05,2.3){\circle{0.08}}
\put(11.4,2.4){\circle{0.08}}
\put(11.7,2.48){\circle{0.08}}
\put(12,2.56){\circle{0.08}}
\put(12,2.56){\circle{0.12}}

\linethickness{0.15mm}

\put(0,0){\vector(1,0){14}}
\put(0,0){\vector(0,1){5.5}}

\put(6.6666,0){\circle{0.12}}

\put(0,4){\circle{0.12}}
\put(0,4){\circle{0.08}}

\put(12,1.7778){\circle{0.12}}
\put(12,1.7778){\circle{0.08}}

{\tiny
\put(-6.25,4){$P_{1}=(\log m,\nu _{p}(a_{m}))$}

\put(-6.25,3){$P_{1}'=(\log m,\nu _{q}(a_{m}))$}

\put(6.3,-0.85){$P_{2}=(\log j_{1},0)$}

\put(1.25,-0.85){$P_{2}'=(\log j_{2},0)$}

\put(12.3,1.7778){$P_{3}=(\log n,\nu _{p}(a_{n}))$}

\put(12.3,2.7778){$P_{3}'=(\log n,\nu _{q}(a_{n}))$}

\put(-2,-0.85){$(\log m,0)$}

\put(11,-0.85){$(\log n,0)$}

\put(-10.9,-0.85){{\bf Figure 6.}}
}

\end{picture}
\end{center}
\bigskip

In our case $\nu _{p}(a_{m})=\nu _{p}(a_{n})=\nu _{q}(a_{m})=\nu _{q}(a_{n})=1$, so by Lemma \ref{puncte}, each one of the edges $P_{1}P_{2},P_{2}P_{3}$ and $P_{1}'P_{2}',P_{2}'P_{3}'$ contains a single segment. Let $g(s)=\frac{b_{c}}{c^{s}}+\cdots +\frac{b_{d}}{d^{s}}$ be a nonconstant factor of $f$, and let us assume that $d<n$. By using Theorem \ref{DumDir} for the prime elements $p$ and $q$, we deduce that the relative degree $\frac{d}{c}$ of $g$ must belong to the intersection
\[
\left\{ \frac{j_{1}}{m},\frac{n}{j_{1}}\right\} \bigcap \left\{ \frac{j_{2}}{m},\frac{n}{j_{2}}\right\} .
\]
On the other hand this intersection is empty, as $j_{1}\neq j_{2}$ and $j_{1}j_{2}\neq mn$, so $d$ must be equal to $n$, and $c$ to $m$, implying $g=f$. Therefore $f$ must be irreducible. 
\end{proof}

\begin{theorem}\label{nnpp}
Let $f(s)=\frac{a_{m}}{m^{s}}+\cdots +\frac{a_{n}}{n^{s}}$ be an algebraically primitive Dirichlet polynomial with coefficients in a unique factorization domain $R$, with $a_{m}a_{n}\neq 0$ and let
$p,q$ be two prime elements of $R$. Assume there exist indices $j_{1},j_{2}$ with $m<j_{1}<\sqrt {mn}<j_{2}<n$ and $j_{1}j_{2}\neq mn$ such that:

i) \ $p^{2}\mid a_{i}$ for $m\leq i<j_{1}$, $p\mid a_{i}$ for $j_{1}\leq i<n$,  $p\nmid a_{n}$, $p^{2}\nmid a_{j_{1}}$ and $p^{3}\nmid a_{m}$;

ii) \thinspace $q^{2}\mid a_{i}$ for $j_{2}<i\leq n$, $q\mid a_{i}$ for $m<i\leq j_{2}$,  $q\nmid a_{m}$, $q^{2}\nmid a_{j_{2}}$ and $q^{3}\nmid a_{n}$;

\noindent Then $f$ is irreducible over $Q(R)$.
\end{theorem}
\begin{proof}\ Let us represent now in Figure 7 the Newton log-polygons with respect to two prime elements $p$ and $q$, one consisting of two edges having negative slopes, and the other consisting of two edges having positive slopes, as in Figure 5.b. 
\begin{center}

\setlength{\unitlength}{4.65mm}
\begin{picture}(12,6)
\linethickness{0.075mm}

\put(12,0){\line(0,1){5}}   
\thicklines

\put(0,4){\line(2,-1){6}}   
\put(6,1){\line(6,-1){6}}   
\linethickness{0.15mm}

\put(0,4){\circle{0.08}}
\put(0,4){\circle{0.12}}
\put(6,1){\circle{0.08}}
\put(6,1){\circle{0.12}}
\put(12,0){\circle{0.08}}
\put(12,0){\circle{0.12}}

\put(0,0){\circle{0.12}}
\put(0,0){\circle{0.08}}
\put(0.25,0.025){\circle{0.08}}
\put(0.5,0.05){\circle{0.08}}
\put(0.75,0.075){\circle{0.08}}
\put(1,0.1){\circle{0.08}}
\put(1.25,0.125){\circle{0.08}}
\put(1.5,0.15){\circle{0.08}}
\put(1.75,0.175){\circle{0.08}}
\put(2,0.2){\circle{0.08}}
\put(2.25,0.225){\circle{0.08}}
\put(2.5,0.25){\circle{0.08}}
\put(2.75,0.275){\circle{0.08}}
\put(3,0.3){\circle{0.08}}
\put(3.25,0.325){\circle{0.08}}
\put(3.5,0.35){\circle{0.08}}
\put(3.75,0.375){\circle{0.08}}
\put(4,0.4){\circle{0.08}}
\put(4.25,0.425){\circle{0.08}}
\put(4.5,0.45){\circle{0.08}}
\put(4.75,0.475){\circle{0.08}}
\put(5,0.5){\circle{0.08}}
\put(5.25,0.525){\circle{0.08}}
\put(5.5,0.55){\circle{0.08}}
\put(5.75,0.575){\circle{0.08}}
\put(6,0.6){\circle{0.08}}
\put(6.25,0.625){\circle{0.08}}
\put(6.5,0.65){\circle{0.08}}
\put(6.75,0.675){\circle{0.08}}
\put(7,0.7){\circle{0.08}}
\put(7.25,0.725){\circle{0.08}}
\put(7.5,0.75){\circle{0.08}}
\put(7.75,0.775){\circle{0.08}}
\put(8,0.8){\circle{0.08}}
\put(8.25,0.825){\circle{0.08}}
\put(8.5,0.85){\circle{0.08}}
\put(8.75,0.875){\circle{0.08}}
\put(9,0.9){\circle{0.08}}
\put(9,0.9){\circle{0.12}}
\put(9.25,1){\circle{0.08}}
\put(9.5,1.2){\circle{0.08}}
\put(9.75,1.4){\circle{0.08}}
\put(10,1.6){\circle{0.08}}
\put(10.25,1.8){\circle{0.08}}
\put(10.5,2){\circle{0.08}}
\put(10.75,2.2){\circle{0.08}}
\put(11,2.4){\circle{0.08}}
\put(11.25,2.6){\circle{0.08}}
\put(11.5,2.8){\circle{0.08}}
\put(11.75,3){\circle{0.08}}
\put(12,3.2){\circle{0.08}}
\put(12,3.2){\circle{0.12}}

\put(0,0){\vector(1,0){14.5}}
\put(0,0){\vector(0,1){5.5}}

{\tiny
\put(-0.8,0.85){$P_{2}=(\log j_{1},\nu _{p}(a_{j_{1}}))$}

\put(9.85,0.85){$P_{2}'=(\log j_{2},\nu _{q}(a_{j_{2}}))$}

\put(12.3,3.2){$P_{3}'=(\log n,\nu _{q}(a_{n}))$}

\put(-6.2,3.85){$P_{1}=(\log m,\nu _{p}(a_{m}))$}

\put(-3,-0.85){$P_{1}'=(\log m,0)$}

\put(11.85,-0.85){$P_{3}=(\log n,0)$}

\put(-10.9,-0.85){{\bf Figure 7.}}
}

\end{picture}
\end{center}
\bigskip

In our particular case we have $\nu _{p}(a_{m})=\nu _{q}(a_{n})=2$ and $\nu _{p}(a_{j_{1}})=\nu _{q}(a_{j_{2}})=1$, so conditions i) and ii) together with Lemma \ref{puncte} show that each one of the edges $P_{1}P_{2}$, $P_{2}P_{3}$, $P_{1}'P_{2}'$ and $P_{2}'P_{3}'$ contains a single segment. The inequalities $m<j_{1}<\sqrt {mn}<j_{2}<n$ show that the slopes of $P_1P_2$ and $P_1'P_2'$ are smaller than the slopes of $P_{2}P_{3}$ and $P_{2}'P_{3}'$, respectively. In particular, $P_{1},P_{2}$ and $P_{3}$ are not collinear, and $P_{1}',P_{2}'$ and $P_{3}'$ too are not collinear.
As before, $f$ cannot have a nonconstant factor $g(s)=\frac{b_{c}}{c^{s}}+\cdots +\frac{b_{d}}{d^{s}}$ with $d<n$, since by Theorem \ref{DumDir}, $\frac{d}{c}$ would be forced to belong to the intersection
\[
\left\{ \frac{j_{1}}{m},\frac{n}{j_{1}}\right\} \bigcap \left\{ \frac{j_{2}}{m},\frac{n}{j_{2}}\right\} ,
\]
which is again empty, as $j_{1}\neq j_{2}$ and $j_{1}j_{2}\neq mn$. So again $f$ must be irreducible. \end{proof}

\begin{theorem}\label{pppp}
Let $f(s)=\frac{a_{m}}{m^{s}}+\cdots +\frac{a_{n}}{n^{s}}$ be an algebraically primitive Dirichlet polynomial with coefficients in a unique factorization domain $R$, with $a_{m}a_{n}\neq 0$ and let
$p,q$ be two prime elements of $R$. Assume there exist indices $j_{1}\neq j_{2}$ with $j_1,j_2\in (\sqrt {mn},n)$ 
such that:

i) \ $p\nmid a_{m}$, $p\mid a_{i}$ for $m<i\leq j_{1}$, $p^{2}\nmid a_{j_{1}}$,  $p^{2}\mid a_{i}$ for $j_{1}<i\leq n$ and $p^{3}\nmid a_{n}$;

ii) \thinspace $q\nmid a_{m}$, $q\mid a_{i}$ for $m<i\leq j_{2}$, $q^{2}\nmid a_{j_{2}}$, $q^{2}\mid a_{i}$ for $j_{2}<i\leq n$ and $q^{3}\nmid a_{n}$.

\noindent Then $f$ is irreducible over $Q(R)$.
\end{theorem}
\begin{proof}\ We represent now in Figure 8 the Newton log-polygons with respect to two prime elements $p$ and $q$, both consisting of two edges having positive slopes, as in Figure 5.c. 

\begin{center}
\setlength{\unitlength}{4.65mm}
\begin{picture}(12,6)
\linethickness{0.075mm}

\put(12,0){\line(0,1){5}}   
\thicklines

\put(0,0){\line(5,1){6}}   

\put(6,1.2){\line(2,1){6}}   
\put(6,1.2){\circle{0.08}}
\put(6,1.2){\circle{0.12}}
\put(12,4.2){\circle{0.12}}

\put(0,0){\circle{0.12}}
\put(0,0){\circle{0.08}}
\put(0.25,0.025){\circle{0.08}}
\put(0.5,0.05){\circle{0.08}}
\put(0.75,0.075){\circle{0.08}}
\put(1,0.1){\circle{0.08}}
\put(1.25,0.125){\circle{0.08}}
\put(1.5,0.15){\circle{0.08}}
\put(1.75,0.175){\circle{0.08}}
\put(2,0.2){\circle{0.08}}
\put(2.25,0.225){\circle{0.08}}
\put(2.5,0.25){\circle{0.08}}
\put(2.75,0.275){\circle{0.08}}
\put(3,0.3){\circle{0.08}}
\put(3.25,0.325){\circle{0.08}}
\put(3.5,0.35){\circle{0.08}}
\put(3.75,0.375){\circle{0.08}}
\put(4,0.4){\circle{0.08}}
\put(4.25,0.425){\circle{0.08}}
\put(4.5,0.45){\circle{0.08}}
\put(4.75,0.475){\circle{0.08}}
\put(5,0.5){\circle{0.08}}
\put(5.25,0.525){\circle{0.08}}
\put(5.5,0.55){\circle{0.08}}
\put(5.75,0.575){\circle{0.08}}
\put(6,0.6){\circle{0.08}}
\put(6.25,0.625){\circle{0.08}}
\put(6.5,0.65){\circle{0.08}}
\put(6.75,0.675){\circle{0.08}}
\put(7,0.7){\circle{0.08}}
\put(7.25,0.725){\circle{0.08}}
\put(7.5,0.75){\circle{0.08}}
\put(7.75,0.775){\circle{0.08}}
\put(8,0.8){\circle{0.08}}
\put(8,0.8){\circle{0.12}}
\put(8.25,0.9){\circle{0.08}}
\put(8.5,1){\circle{0.08}}
\put(8.75,1.1){\circle{0.08}}
\put(9,1.2){\circle{0.08}}
\put(9.25,1.3){\circle{0.08}}
\put(9.5,1.4){\circle{0.08}}
\put(9.75,1.5){\circle{0.08}}
\put(10,1.6){\circle{0.08}}
\put(10.25,1.7){\circle{0.08}}
\put(10.5,1.8){\circle{0.08}}
\put(10.75,1.9){\circle{0.08}}
\put(11,2){\circle{0.08}}
\put(11.25,2.1){\circle{0.08}}
\put(11.5,2.2){\circle{0.08}}
\put(11.75,2.3){\circle{0.08}}
\put(12,2.4){\circle{0.08}}
\put(12,2.4){\circle{0.12}}

\linethickness{0.15mm}

\put(0,0){\vector(1,0){14.9}}
\put(0,0){\vector(0,1){5.5}}

{\tiny
\put(0.5,1.7){$P_{2}=(\log j_{1},\nu _{p}(a_{j_{1}}))$}

\put(8.7,0.5){$P_{2}'=(\log j_{2},\nu _{q}(a_{j_{2}}))$}

\put(12.3,2.4){$P_{3}'=(\log n,\nu _{q}(a_{n}))$}

\put(5.8,4.2){$P_{3}=(\log n,\nu _{p}(a_{n}))$}

\put(-3,-0.85){$P_{1}=P_{1}'=(\log m,0)$}

\put(10.75,-0.85){$(\log n,0)$}

\put(-10.9,-0.85){{\bf Figure 8.}}
}

\end{picture}
\end{center}
\bigskip

Here $\nu _{p}(a_{n})=\nu _{q}(a_{n})=2$ and $\nu _{p}(a_{j_{1}})=\nu _{q}(a_{j_{2}})=1$, so by i), ii) and Lemma \ref{puncte}, each one of the edges $P_{1}P_{2}$, $P_{2}P_{3}$, $P_{1}'P_{2}'$ and $P_{2}'P_{3}'$ contains a single segment. Using the fact that both indices $j_1$ and $j_2$ belong to the interval $(\sqrt {mn},n)$, one may check that the slopes of $P_2P_3$ and $P_{2}'P_{3}'$ are greater than the slopes of $P_1P_2$ and $P_1'P_2'$, respectively. In particular, $P_{1},P_{2}$ and $P_{3}$ are not collinear, and $P_{1}',P_{2}'$ and $P_{3}'$ too are not collinear. Now, since $j_1\neq j_2$ and $j_1j_2>mn$, we conclude that the intersection
\[
\left\{ \frac{j_{1}}{m},\frac{n}{j_{1}}\right\} \bigcap \left\{ \frac{j_{2}}{m},\frac{n}{j_{2}}\right\} 
\]
is empty, thus forcing $f$ to be irreducible. \end{proof}

\begin{theorem}\label{nnnn}
Let $f(s)=\frac{a_{m}}{m^{s}}+\cdots +\frac{a_{n}}{n^{s}}$ be an algebraically primitive Dirichlet polynomial with coefficients in a unique factorization domain $R$, with $a_{m}a_{n}\neq 0$ and let
$p,q$ be two prime elements of $R$. Assume there exist indices $j_{1}\neq j_{2}$ with $j_{1}, j_{2}\in (m,\sqrt {mn})$ such that:

i) \thinspace $p^{2}\mid a_{i}$ for $m\leq i<j_{1}$, $p\mid a_{i}$ for $j_{1}\leq i<n$,  $p\nmid a_{n}$, $p^{2}\nmid a_{j_{1}}$ and $p^{3}\nmid a_{m}$;

ii) $q^{2}\mid a_{i}$ for $m\leq i<j_{2}$, $q\mid a_{i}$ for $j_{2}\leq i<n$,  $q\nmid a_{n}$, $q^{2}\nmid a_{j_{2}}$ and $q^{3}\nmid a_{m}$;

\noindent Then $f$ is irreducible over $Q(R)$.
\end{theorem}
\begin{proof}\ We represent now in Figure 9 the Newton log-polygons with respect to two prime elements $p$ and $q$, both consisting of two edges having negative slopes, as in Figure 5.d. 

\begin{center}
\setlength{\unitlength}{4.65mm}
\begin{picture}(12,6)
\linethickness{0.075mm}

\put(12,0){\line(0,1){5}}   
\thicklines

\put(0,4.2){\line(2,-1){6}}   
\put(6,1.2){\line(5,-1){6}}

\put(6,1.2){\circle{0.08}}
\put(6,1.2){\circle{0.12}}
\put(0,4.2){\circle{0.12}}
\put(0,4.2){\circle{0.08}}

\put(12,0){\circle{0.12}}
\put(12,0){\circle{0.08}}
\put(11.75,0.025){\circle{0.08}}
\put(11.5,0.05){\circle{0.08}}
\put(11.25,0.075){\circle{0.08}}
\put(11,0.1){\circle{0.08}}
\put(10.75,0.125){\circle{0.08}}
\put(10.5,0.15){\circle{0.08}}
\put(10.25,0.175){\circle{0.08}}
\put(10,0.2){\circle{0.08}}
\put(9.75,0.225){\circle{0.08}}
\put(9.5,0.25){\circle{0.08}}
\put(9.25,0.275){\circle{0.08}}
\put(9,0.3){\circle{0.08}}
\put(8.75,0.325){\circle{0.08}}
\put(8.5,0.35){\circle{0.08}}
\put(8.25,0.375){\circle{0.08}}
\put(8,0.4){\circle{0.08}}
\put(7.75,0.425){\circle{0.08}}
\put(7.5,0.45){\circle{0.08}}
\put(7.25,0.475){\circle{0.08}}
\put(7,0.5){\circle{0.08}}
\put(6.75,0.525){\circle{0.08}}
\put(6.5,0.55){\circle{0.08}}
\put(6.25,0.575){\circle{0.08}}
\put(6,0.6){\circle{0.08}}
\put(5.75,0.625){\circle{0.08}}
\put(5.5,0.65){\circle{0.08}}
\put(5.25,0.675){\circle{0.08}}
\put(5,0.7){\circle{0.08}}
\put(4.75,0.725){\circle{0.08}}
\put(4.5,0.75){\circle{0.08}}
\put(4.25,0.775){\circle{0.08}}
\put(4,0.8){\circle{0.08}}
\put(4,0.8){\circle{0.12}}
\put(3.75,0.9){\circle{0.08}}
\put(3.5,1){\circle{0.08}}
\put(3.25,1.1){\circle{0.08}}
\put(3,1.2){\circle{0.08}}
\put(2.75,1.3){\circle{0.08}}
\put(2.5,1.4){\circle{0.08}}
\put(2.25,1.5){\circle{0.08}}
\put(2,1.6){\circle{0.08}}
\put(1.75,1.7){\circle{0.08}}
\put(1.5,1.8){\circle{0.08}}
\put(1.25,1.9){\circle{0.08}}
\put(1,2){\circle{0.08}}
\put(0.75,2.1){\circle{0.08}}
\put(0.5,2.2){\circle{0.08}}
\put(0.25,2.3){\circle{0.08}}
\put(0,2.4){\circle{0.08}}
\put(0,2.4){\circle{0.12}}

\linethickness{0.15mm}

\put(0,0){\vector(1,0){14.9}}
\put(0,0){\vector(0,1){5.5}}

{\tiny
\put(6,1.6){$P_{2}=(\log j_{1},\nu _{p}(a_{j_{1}}))$}

\put(-2.72,0.55){$P_{2}'=(\log j_{2},\nu _{q}(a_{j_{2}}))$}

\put(-6.2,2.3){$P_{1}'=(\log m,\nu _{q}(a_{m}))$}

\put(-6.2,4.1){$P_{1}=(\log m,\nu _{p}(a_{m}))$}

\put(-1.2,-0.85){$(\log m,0)$}

\put(9.5,-0.85){$P_{3}=P_{3}'=(\log n,0)$}

\put(-10.9,-0.85){{\bf Figure 9.}}
}

\end{picture}
\end{center}
\bigskip

Here $\nu _{p}(a_{m})=\nu _{q}(a_{m})=2$ and $\nu _{p}(a_{j_{1}})=\nu _{q}(a_{j_{2}})=1$, so by i), ii) and Lemma \ref{puncte}, each one of the edges $P_{1}P_{2}$, $P_{2}P_{3}$, $P_{1}'P_{2}'$ and $P_{2}'P_{3}'$ contains a single segment. The fact that $j_{1}, j_{2}\in (m,\sqrt {mn})$ forces the slopes of $P_2P_3$ and $P_{2}'P_{3}'$ to be greater than the slopes of $P_1P_2$ and $P_1'P_2'$, respectively. In particular, $P_{1},P_{2}$ and $P_{3}$ are not collinear, and $P_{1}',P_{2}'$ and $P_{3}'$ too are not collinear. Since $j_1\neq j_2$ and $j_1j_2<mn$, the intersection
\[
\left\{ \frac{j_{1}}{m},\frac{n}{j_{1}}\right\} \bigcap \left\{ \frac{j_{2}}{m},\frac{n}{j_{2}}\right\}
\]
is empty, so $f$ must be irreducible. 
\end{proof}

\begin{theorem}\label{ppnp}
Let $f(s)=\frac{a_{m}}{m^{s}}+\cdots +\frac{a_{n}}{n^{s}}$ be an algebraically primitive Dirichlet polynomial with coefficients in a unique factorization domain $R$, with $a_{m}a_{n}\neq 0$ and let
$p,q$ be two prime elements of $R$. Assume there exist indices $j_{1}\neq j_{2}$ with $\sqrt {mn}<j_{1}<n$, $m<j_{2}<n$, $j_{1}j_{2}\neq mn$ such that:

i) \ $p\nmid a_{m}$, $p\mid a_{i}$ for $m<i\leq j_{1}$, $p^{2}\nmid a_{j_{1}}$,  $p^{2}\mid a_{i}$ for $j_{1}<i\leq n$ and $p^{3}\nmid a_{n}$;

ii) \thinspace $q\mid a_{i}$ for $i\neq j_{2}$, $q\nmid a_{j_{2}}$, $q^{2}\nmid a_{m}$ and  $q^{2}\nmid a_{n}$.

\noindent Then $f$ is irreducible over $Q(R)$.
\end{theorem}
\begin{proof}\ We represent now in Figure 10 the Newton log-polygons with respect to two prime elements $p$ and $q$, one consisting of two edges having positive slopes, and the other consisting of two edges having slopes of different signs, as in Figure 5.e. 

\begin{center}
\setlength{\unitlength}{4.65mm}
\begin{picture}(12,6)
\linethickness{0.075mm}

\put(12,0){\line(0,1){4.8}}   
\thicklines

\put(0,0){\line(5,1){5}}   

\put(5,1){\line(2,1){7}}   

\linethickness{0.15mm}

\put(0,4){\circle{0.12}}
\put(0,4){\circle{0.08}}
\put(0.25,3.875){\circle{0.08}}
\put(0.5,3.75){\circle{0.08}}
\put(0.75,3.625){\circle{0.08}}
\put(1,3.5){\circle{0.08}}
\put(1.25,3.375){\circle{0.08}}
\put(1.5,3.25){\circle{0.08}}
\put(1.75,3.125){\circle{0.08}}
\put(2,3){\circle{0.08}}
\put(2.25,2.875){\circle{0.08}}
\put(2.5,2.75){\circle{0.08}}
\put(2.75,2.625){\circle{0.08}}
\put(3,2.5){\circle{0.08}}
\put(3.25,2.375){\circle{0.08}}

\put(3.5,2.25){\circle{0.08}}
\put(3.75,2.125){\circle{0.08}}
\put(4,2){\circle{0.08}}
\put(4.25,1.875){\circle{0.08}}
\put(4.5,1.75){\circle{0.08}}
\put(4.75,1.625){\circle{0.08}}
\put(5,1.5){\circle{0.08}}
\put(5.25,1.375){\circle{0.08}}
\put(5.5,1.25){\circle{0.08}}
\put(5.75,1.125){\circle{0.08}}
\put(6,1){\circle{0.08}}
\put(6.25,0.875){\circle{0.08}}
\put(6.5,0.75){\circle{0.08}}
\put(6.75,0.625){\circle{0.08}}
\put(7,0.5){\circle{0.08}}
\put(7.25,0.375){\circle{0.08}}
\put(7.5,0.25){\circle{0.08}}
\put(7.75,0.125){\circle{0.08}}
\put(8,0){\circle{0.08}}
\put(8,0){\circle{0.12}}

\put(8.25,0.125){\circle{0.08}}
\put(8.5,0.25){\circle{0.08}}
\put(8.75,0.375){\circle{0.08}}
\put(9,0.5){\circle{0.08}}
\put(9.25,0.625){\circle{0.08}}
\put(9.5,0.75){\circle{0.08}}
\put(9.75,0.875){\circle{0.08}}
\put(10,1){\circle{0.08}}
\put(10.25,1.125){\circle{0.08}}
\put(10.5,1.25){\circle{0.08}}
\put(10.75,1.375){\circle{0.08}}
\put(11,1.5){\circle{0.08}}
\put(11.25,1.625){\circle{0.08}}
\put(11.5,1.75){\circle{0.08}}
\put(11.75,1.875){\circle{0.08}}
\put(12,2){\circle{0.08}}
\put(12,2){\circle{0.12}}

\put(0,0){\vector(1,0){14}}
\put(0,0){\vector(0,1){5.5}}

\put(0,0){\circle{0.12}}
\put(0,0){\circle{0.08}}

\put(5,1){\circle{0.12}}
\put(5,1){\circle{0.08}}

\put(0,4){\circle{0.12}}
\put(0,4){\circle{0.08}}

\put(12,4.5){\circle{0.12}}
\put(12,4.5){\circle{0.08}}

{\tiny

\put(-6.2,4){$P_{1}'=(\log m,\nu _{q}(a_{m}))$}

\put(6,-0.85){$P_{2}'=(\log j_{2},0)$}

\put(-1.5,1.22){$P_{2}=(\log j_{1},\nu _{p}(a_{j_{1}}))$}

\put(12.3,4.4){$P_{3}=(\log n,\nu _{p}(a_{n}))$}

\put(12.3,1.9){$P_{3}'=(\log n,\nu _{q}(a_{n}))$}

\put(-3,-0.85){$P_{1}=(\log m,0)$}

\put(11,-0.85){$(\log n,0)$}

\put(-10.9,-0.85){{\bf Figure 10.}}
}

\end{picture}
\end{center}
\bigskip

In our case $\nu _{q}(a_{n})=\nu _{p}(a_{j_{1}})=\nu _{q}(a_{m})=1$ and $\nu _{p}(a_{n})=2$, so by i), ii) and Lemma \ref{puncte}, each one of the edges $P_{1}P_{2}$, $P_{2}P_{3}$, $P_{1}'P_{2}'$ and $P_{2}'P_{3}'$ contains a single segment. Moreover, the slope of $P_1P_2$ is smaller than the slope of $P_2P_3$, as $\sqrt {mn}<j_{1}<n$, so in particular the log-integral points $P_{1},P_{2}$ and $P_{3}$ are not collinear. 
The intersection
\[
\left\{ \frac{j_{1}}{m},\frac{n}{j_{1}}\right\} \bigcap \left\{ \frac{j_{2}}{m},\frac{n}{j_{2}}\right\} ,
\]
is again empty, showing that $f$ must be irreducible. 
\end{proof}

\begin{theorem}\label{nnnp}
Let $f(s)=\frac{a_{m}}{m^{s}}+\cdots +\frac{a_{n}}{n^{s}}$ be an algebraically primitive Dirichlet polynomial with coefficients in a unique factorization domain $R$, with $a_{m}a_{n}\neq 0$ and let
$p,q$ be two prime elements of $R$. Assume there exist indices $j_{1}\neq j_{2}$ with $m<j_{1}<\sqrt{mn}$, $m<j_{2}<n$ and $j_{1}j_{2}\neq mn$ such that:

i) \thinspace $p^{2}\mid a_{i}$ for $m\leq i<j_{1}$, $p\mid a_{i}$ for $j_{1}\leq i<n$,  $p\nmid a_{n}$, $p^{2}\nmid a_{j_{1}}$ and $p^{3}\nmid a_{m}$;

ii) \thinspace $q\mid a_{i}$ for $i\neq j_{2}$, $q\nmid a_{j_{2}}$, $q^{2}\nmid a_{m}$ and  $q^{2}\nmid a_{n}$.

\noindent Then $f$ is irreducible over $Q(R)$.
\end{theorem}
\begin{proof}\ Figure 11 shows the Newton log-polygons with respect to two prime elements $p$ and $q$, one consisting of two edges having negative slopes, and the other consisting of two edges having slopes of different sign, as in Figure 5.f. 

\begin{center}

\setlength{\unitlength}{4.65mm}
\begin{picture}(12,6)
\linethickness{0.075mm}

\put(12,0){\line(0,1){4.8}}   
\thicklines

\put(0,4.5){\line(2,-1){7}}   

\put(7,1){\line(5,-1){5}}   

\linethickness{0.15mm}

\put(12,4){\circle{0.12}}
\put(12,4){\circle{0.08}}
\put(11.75,3.875){\circle{0.08}}
\put(11.5,3.75){\circle{0.08}}
\put(11.25,3.625){\circle{0.08}}
\put(11,3.5){\circle{0.08}}
\put(10.75,3.375){\circle{0.08}}
\put(10.5,3.25){\circle{0.08}}
\put(10.25,3.125){\circle{0.08}}
\put(10,3){\circle{0.08}}
\put(9.75,2.875){\circle{0.08}}
\put(9.5,2.75){\circle{0.08}}
\put(9.25,2.625){\circle{0.08}}
\put(9,2.5){\circle{0.08}}
\put(8.75,2.375){\circle{0.08}}

\put(8.5,2.25){\circle{0.08}}
\put(8.25,2.125){\circle{0.08}}
\put(8,2){\circle{0.08}}
\put(7.75,1.875){\circle{0.08}}
\put(7.5,1.75){\circle{0.08}}
\put(7.25,1.625){\circle{0.08}}
\put(7,1.5){\circle{0.08}}
\put(6.75,1.375){\circle{0.08}}
\put(6.5,1.25){\circle{0.08}}
\put(6.25,1.125){\circle{0.08}}
\put(6,1){\circle{0.08}}
\put(5.75,0.875){\circle{0.08}}
\put(5.5,0.75){\circle{0.08}}
\put(5.25,0.625){\circle{0.08}}
\put(5,0.5){\circle{0.08}}
\put(4.75,0.375){\circle{0.08}}
\put(4.5,0.25){\circle{0.08}}
\put(4.25,0.125){\circle{0.08}}
\put(4,0){\circle{0.08}}
\put(4,0){\circle{0.12}}

\put(3.75,0.125){\circle{0.08}}
\put(3.5,0.25){\circle{0.08}}
\put(3.25,0.375){\circle{0.08}}
\put(3,0.5){\circle{0.08}}
\put(2.75,0.625){\circle{0.08}}
\put(2.5,0.75){\circle{0.08}}
\put(2.25,0.875){\circle{0.08}}
\put(2,1){\circle{0.08}}
\put(1.75,1.125){\circle{0.08}}
\put(1.5,1.25){\circle{0.08}}
\put(1.25,1.375){\circle{0.08}}
\put(1,1.5){\circle{0.08}}
\put(0.75,1.625){\circle{0.08}}
\put(0.5,1.75){\circle{0.08}}
\put(0.25,1.875){\circle{0.08}}
\put(0,2){\circle{0.08}}
\put(0,2){\circle{0.12}}

\put(0,0){\vector(1,0){14}}
\put(0,0){\vector(0,1){5.5}}

\put(0,0){\circle{0.12}}
\put(0,0){\circle{0.08}}

\put(7,1){\circle{0.12}}
\put(7,1){\circle{0.08}}

\put(0,4.5){\circle{0.12}}
\put(0,4.5){\circle{0.08}}

\put(12,0){\circle{0.12}}
\put(12,0){\circle{0.08}}

{\tiny

\put(-6.2,4.4){$P_{1}=(\log m,\nu _{p}(a_{m}))$}

\put(7.2,1.1){$P_{2}=(\log j_{1},\nu _{p}(a_{j_{1}}))$}

\put(2.5,-0.85){$P_{2}'=(\log j_{2},0)$}

\put(-6.2,1.85){$P_{1}'=(\log m,\nu _{q}(a_{m}))$}

\put(12.3,4){$P_{3}'=(\log n,\nu _{q}(a_{n}))$}

\put(-1.5,-0.85){$(\log m,0)$}

\put(10,-0.85){$P_{3}=(\log n,0)$}

\put(-10.9,-0.85){{\bf Figure 11.}}
}

\end{picture}
\end{center}
\bigskip

In our case $\nu _{q}(a_{n})=\nu _{p}(a_{j_{1}})=\nu _{q}(a_{m})=1$ and $\nu _{p}(a_{m})=2$, so by i), ii) and Lemma \ref{puncte}, each one of the edges $P_{1}P_{2}$, $P_{2}P_{3}$, $P_{1}'P_{2}'$ and $P_{2}'P_{3}'$ contains a single segment. The slope of $P_1P_2$ is smaller than the slope of $P_2P_3$, this time because $m<j_{1}<\sqrt{mn}$, so in particular the log-integral points $P_{1},P_{2}$ and $P_{3}$ are not collinear.
The intersection
\[
\left\{ \frac{j_{1}}{m},\frac{n}{j_{1}}\right\} \bigcap \left\{ \frac{j_{2}}{m},\frac{n}{j_{2}}\right\} ,
\]
is again empty, which shows that $f$ is irreducible. 
\end{proof}

\begin{theorem}\label{zpnp}
Let $f(s)=\frac{a_{m}}{m^{s}}+\cdots +\frac{a_{n}}{n^{s}}$ be an algebraically primitive Dirichlet polynomial with coefficients in a unique factorization domain $R$, with $a_{m}a_{n}\neq 0$ and let
$p,q$ be two prime elements of $R$. Assume there exists an index $j\in \{ m+2,\dots ,n-1\} $ with $j\neq mn/(m+1)$ such that:

i) \ $p\mid a_{i}$ for all $i\geq m+2$, $p\nmid a_{m}a_{m+1}$, $p^{2}\nmid a_{n}$;

ii) \thinspace $q\mid a_{i}$ for all $i\neq j$, $q\nmid a_{j}$, $q^{2}\nmid a_{m}$ and $q^{2}\nmid a_{n}$;

\noindent Then $f$ is irreducible over $Q(R)$.
\end{theorem}
\begin{proof}\ Let us represent in Figure 12 the Newton log-polygons with respect to two prime elements $p$ and $q$, the first one consisting of an horizontal edge followed by an edge with positive slope, and the second one consisting of two edges having slopes with different signs, as in Figure 5.g. 

\begin{center}

\setlength{\unitlength}{4.65mm}
\begin{picture}(12,6)
\linethickness{0.075mm}

\put(12,0){\line(0,1){4.5}}   
\thicklines

\put(0,0){\line(1,0){4}}   

\put(4,0){\line(2,1){8}}   

\linethickness{0.15mm}

\put(0,4){\circle{0.12}}
\put(0,4){\circle{0.08}}
\put(0.25,3.875){\circle{0.08}}
\put(0.5,3.75){\circle{0.08}}
\put(0.75,3.625){\circle{0.08}}
\put(1,3.5){\circle{0.08}}
\put(1.25,3.375){\circle{0.08}}
\put(1.5,3.25){\circle{0.08}}
\put(1.75,3.125){\circle{0.08}}
\put(2,3){\circle{0.08}}
\put(2.25,2.875){\circle{0.08}}
\put(2.5,2.75){\circle{0.08}}
\put(2.75,2.625){\circle{0.08}}
\put(3,2.5){\circle{0.08}}
\put(3.25,2.375){\circle{0.08}}

\put(3.5,2.25){\circle{0.08}}
\put(3.75,2.125){\circle{0.08}}
\put(4,2){\circle{0.08}}
\put(4.25,1.875){\circle{0.08}}
\put(4.5,1.75){\circle{0.08}}
\put(4.75,1.625){\circle{0.08}}
\put(5,1.5){\circle{0.08}}
\put(5.25,1.375){\circle{0.08}}
\put(5.5,1.25){\circle{0.08}}
\put(5.75,1.125){\circle{0.08}}
\put(6,1){\circle{0.08}}
\put(6.25,0.875){\circle{0.08}}
\put(6.5,0.75){\circle{0.08}}
\put(6.75,0.625){\circle{0.08}}
\put(7,0.5){\circle{0.08}}
\put(7.25,0.375){\circle{0.08}}
\put(7.5,0.25){\circle{0.08}}
\put(7.75,0.125){\circle{0.08}}
\put(8,0){\circle{0.08}}
\put(8,0){\circle{0.12}}

\put(8.25,0.125){\circle{0.08}}
\put(8.5,0.25){\circle{0.08}}
\put(8.75,0.375){\circle{0.08}}
\put(9,0.5){\circle{0.08}}
\put(9.25,0.625){\circle{0.08}}
\put(9.5,0.75){\circle{0.08}}
\put(9.75,0.875){\circle{0.08}}
\put(10,1){\circle{0.08}}
\put(10.25,1.125){\circle{0.08}}
\put(10.5,1.25){\circle{0.08}}
\put(10.75,1.375){\circle{0.08}}
\put(11,1.5){\circle{0.08}}
\put(11.25,1.625){\circle{0.08}}
\put(11.5,1.75){\circle{0.08}}
\put(11.75,1.875){\circle{0.08}}
\put(12,2){\circle{0.08}}
\put(12,2){\circle{0.12}}

\put(0,0){\vector(1,0){14}}
\put(0,0){\vector(0,1){5.5}}

\put(0,0){\circle{0.12}}
\put(0,0){\circle{0.08}}

\put(4,0){\circle{0.12}}
\put(4,0){\circle{0.08}}

\put(0,4){\circle{0.12}}
\put(0,4){\circle{0.08}}

\put(12,4){\circle{0.12}}
\put(12,4){\circle{0.08}}

{\tiny

\put(-6.2,3.9){$P_{1}'=(\log m,\nu _{q}(a_{m}))$}

\put(6.3,-0.85){$P_{2}'=(\log j_{2},0)$}

\put(1.5,-0.85){$P_{2}=(\log j_{1},0)$}

\put(12.3,3.9){$P_{3}=(\log n,\nu _{p}(a_{n}))$}

\put(12.3,1.9){$P_{3}'=(\log n,\nu _{q}(a_{n}))$}

\put(-3.2,-0.85){$P_{1}=(\log m,0)$}

\put(11.35,-0.85){$(\log n,0)$}

\put(-10.9,-0.85){{\bf Figure 12.}}
}

\end{picture}
\end{center}
\bigskip

In our case $j_{1}=m+1$, $j_{2}$ is denoted by $j$, $\nu _{p}(a_{n})=\nu _{q}(a_{m})=\nu _{q}(a_{n})=1$, so by i), ii) and Lemma \ref{puncte}, each of the edges $P_{1}P_{2}$, $P_{2}P_{3}$, $P_{1}'P_{2}'$ and $P_{2}'P_{3}'$ contains a single segment. Again, let $g(s)=\frac{b_{c}}{c^{s}}+\cdots +\frac{b_{d}}{d^{s}}$ be a nonconstant factor of $f$. If $d$ would be less than $n$, then by Theorem \ref{DumDir}, $\frac{d}{c}$ would belong to
\[
\left\{ \frac{m+1}{m},\frac{n}{m+1}\right\} \bigcap \left\{ \frac{j}{m},\frac{n}{j}\right\} .
\]
Since this intersection is empty (as $j\neq m+1$ and $j\neq mn/(m+1)$), $d$ must be equal to $n$ (and $c$ to $m$), proving that $f$ is irreducible. 
\end{proof}

\begin{theorem}\label{zppp}
Let $f(s)=\frac{a_{m}}{m^{s}}+\cdots +\frac{a_{n}}{n^{s}}$ be an algebraically primitive Dirichlet polynomial with coefficients in a unique factorization domain $R$, with $a_{m}a_{n}\neq 0$, and let $p$ and $q$ be two prime elements of $R$. Assume that there exists an index $j\in \{ m+2,\dots ,n-1\} $ with $j\neq mn/(m+1)$ and $j>\sqrt{mn}$ such that:

i) \ $p\mid a_{i}$ for all $i\geq m+2$, $p\nmid a_{m}a_{m+1}$, $p^{2}\nmid a_{n}$;

ii) \thinspace $q\nmid a_{m}$, $q\mid a_{i}$ for $m<i\leq j$, $q^{2}\nmid a_{j}$, $q^{2}\mid a_{i}$ for $j<i\leq n$ and $q^{3}\nmid a_{n}$.

\noindent Then $f$ is irreducible over $Q(R)$.
\end{theorem}
\begin{proof}\ Let us represent in Figure 13 the Newton log-polygons with respect to two prime elements $p,q$, the first one consisting of an horizontal edge followed by an edge with positive slope, and the second one consisting of two edges having positive slopes, as in Figure 5.h. 

\begin{center}
\setlength{\unitlength}{4.65mm}
\begin{picture}(12,6)
\linethickness{0.075mm}

\put(12,0){\line(0,1){5}}   
\thicklines

\put(0,0){\line(1,0){4}}   

\put(4,0){\line(2,1){8}}   

\linethickness{0.15mm}

\put(4,0){\circle{0.08}}
\put(4,0){\circle{0.12}}

\put(12,4){\circle{0.12}}
\put(12,4){\circle{0.08}}

\put(0,0){\circle{0.12}}
\put(0,0){\circle{0.08}}
\put(0.25,0.025){\circle{0.08}}
\put(0.5,0.05){\circle{0.08}}
\put(0.75,0.075){\circle{0.08}}
\put(1,0.1){\circle{0.08}}
\put(1.25,0.125){\circle{0.08}}
\put(1.5,0.15){\circle{0.08}}
\put(1.75,0.175){\circle{0.08}}
\put(2,0.2){\circle{0.08}}
\put(2.25,0.225){\circle{0.08}}
\put(2.5,0.25){\circle{0.08}}
\put(2.75,0.275){\circle{0.08}}
\put(3,0.3){\circle{0.08}}
\put(3.25,0.325){\circle{0.08}}
\put(3.5,0.35){\circle{0.08}}
\put(3.75,0.375){\circle{0.08}}
\put(4,0.4){\circle{0.08}}
\put(4.25,0.425){\circle{0.08}}
\put(4.5,0.45){\circle{0.08}}
\put(4.75,0.475){\circle{0.08}}
\put(5,0.5){\circle{0.08}}
\put(5.25,0.525){\circle{0.08}}
\put(5.5,0.55){\circle{0.08}}
\put(5.75,0.575){\circle{0.08}}
\put(6,0.6){\circle{0.08}}
\put(6.25,0.625){\circle{0.08}}
\put(6.5,0.65){\circle{0.08}}
\put(6.75,0.675){\circle{0.08}}
\put(7,0.7){\circle{0.08}}
\put(7.25,0.725){\circle{0.08}}
\put(7.5,0.75){\circle{0.08}}
\put(7.75,0.775){\circle{0.08}}
\put(8,0.8){\circle{0.08}}
\put(8,0.8){\circle{0.12}}
\put(8.25,0.9){\circle{0.08}}
\put(8.5,1){\circle{0.08}}
\put(8.75,1.1){\circle{0.08}}
\put(9,1.2){\circle{0.08}}
\put(9.25,1.3){\circle{0.08}}
\put(9.5,1.4){\circle{0.08}}
\put(9.75,1.5){\circle{0.08}}
\put(10,1.6){\circle{0.08}}
\put(10.25,1.7){\circle{0.08}}
\put(10.5,1.8){\circle{0.08}}
\put(10.75,1.9){\circle{0.08}}
\put(11,2){\circle{0.08}}
\put(11.25,2.1){\circle{0.08}}
\put(11.5,2.2){\circle{0.08}}
\put(11.75,2.3){\circle{0.08}}
\put(12,2.4){\circle{0.08}}
\put(12,2.4){\circle{0.12}}

\put(0,0){\vector(1,0){14.9}}
\put(0,0){\vector(0,1){5.5}}

{\tiny 
\put(3.5,-0.85){$P_{2}=(\log j_{1},0)$}

\put(8.7,0.37){$P_{2}'=(\log j_{2},\nu _{q}(a_{j_{2}}))$}

\put(12.3,2.2){$P_{3}'=(\log n,\nu _{q}(a_{n}))$}

\put(12.3,3.8){$P_{3}=(\log n,\nu _{p}(a_{n}))$}

\put(-3.5,-0.85){$P_{1}=P_{1}'=(\log m,0)$}

\put(11,-0.85){$(\log n,0)$}

\put(-10.9,-0.85){{\bf Figure 13.}}
}

\end{picture}
\end{center}
\bigskip

In our particular case $j_{1}=m+1$, $j_{2}$ is denoted by $j$, $\nu _{p}(a_{n})=\nu _{q}(a_{j})=1$ and $\nu _{q}(a_{n})=2$, so by i), ii) and Lemma \ref{puncte}, each one of the edges $P_{1}P_{2}$, $P_{2}P_{3}$, $P_{1}'P_{2}'$ and $P_{2}'P_{3}'$ contains a single segment. Also, the slope of $P_2'P_3'$ is greater than the slope of $P_1'P_2'$, as $j>\sqrt{mn}$, so in particular the log-integral points $P_{1}',P_{2}'$ and $P_{3}'$ are not collinear. Again, let $g(s)=\frac{b_{c}}{c^{s}}+\cdots +\frac{b_{d}}{d^{s}}$ be a nonconstant factor of $f$. If $d$ would be less than $n$, by Theorem \ref{DumDir}, $\frac{d}{c}$ would belong to
\[
\left\{ \frac{m+1}{m},\frac{n}{m+1}\right\} \bigcap \left\{ \frac{j}{m},\frac{n}{j}\right\} .
\]
Since this intersection is empty (as $j\neq m+1$ and $j\neq mn/(m+1)$), $d$ must be equal to $n$ (and $c$ to $m$), so $f$ must be irreducible. 
\end{proof}

\begin{theorem}\label{zpnn}
Let $f(s)=\frac{a_{m}}{m^{s}}+\cdots +\frac{a_{n}}{n^{s}}$ be an algebraically primitive Dirichlet polynomial with coefficients in a unique factorization domain $R$, with $a_{m}a_{n}\neq 0$ and let
$p,q$ be two prime elements of $R$. Assume there exists an index $j$ with $m+1<j<\sqrt{mn}$ such that:

i) \ $p\mid a_{i}$ for all $i\geq m+2$, $p\nmid a_{m}a_{m+1}$, $p^{2}\nmid a_{n}$;

ii) $q^{2}\mid a_{i}$ for $m\leq i<j$, $q\mid a_{i}$ for $j\leq i<n$,  $q\nmid a_{n}$, $q^{2}\nmid a_{j}$ and $q^{3}\nmid a_{m}$;

\noindent Then $f$ is irreducible over $Q(R)$.
\end{theorem}
\begin{proof}\ Figure 14 displays the Newton log-polygons with respect to two prime elements $p$ and $q$, the first one consisting of a horizontal edge followed by an edge with positive slope, and the second one consisting of two edges having negative slopes, as in Figure 5.i. 

\begin{center}
\setlength{\unitlength}{4.65mm}
\begin{picture}(12,6)
\linethickness{0.075mm}

\put(12,0){\line(0,1){4.5}}   
\thicklines

\put(0,0){\line(1,0){4}}   

\put(4,0){\line(2,1){8}}   

\linethickness{0.15mm}

\put(0,4){\circle{0.12}}
\put(0,4){\circle{0.08}}
\put(0.25,3.875){\circle{0.08}}
\put(0.5,3.75){\circle{0.08}}
\put(0.75,3.625){\circle{0.08}}
\put(1,3.5){\circle{0.08}}
\put(1.25,3.375){\circle{0.08}}
\put(1.5,3.25){\circle{0.08}}
\put(1.75,3.125){\circle{0.08}}
\put(2,3){\circle{0.08}}
\put(2.25,2.875){\circle{0.08}}
\put(2.5,2.75){\circle{0.08}}
\put(2.75,2.625){\circle{0.08}}
\put(3,2.5){\circle{0.08}}
\put(3.25,2.375){\circle{0.08}}

\put(3.5,2.25){\circle{0.08}}
\put(3.75,2.125){\circle{0.08}}
\put(4,2){\circle{0.08}}
\put(4.25,1.875){\circle{0.08}}
\put(4.5,1.75){\circle{0.08}}
\put(4.75,1.625){\circle{0.08}}
\put(5,1.5){\circle{0.08}}
\put(5.25,1.375){\circle{0.08}}
\put(5.5,1.25){\circle{0.08}}
\put(5.75,1.125){\circle{0.08}}
\put(6,1){\circle{0.08}}
\put(6.25,0.875){\circle{0.08}}
\put(6.5,0.75){\circle{0.08}}
\put(6.75,0.625){\circle{0.08}}

\put(7,0.5){\circle{0.12}}
\put(7,0.5){\circle{0.08}}
\put(7.25,0.475){\circle{0.08}}
\put(7.5,0.45){\circle{0.08}}
\put(7.75,0.425){\circle{0.08}}
\put(8,0.4){\circle{0.08}}

\put(8.25,0.375){\circle{0.08}}
\put(8.5,0.35){\circle{0.08}}
\put(8.75,0.325){\circle{0.08}}
\put(9,0.3){\circle{0.08}}
\put(9.25,0.275){\circle{0.08}}
\put(9.5,0.25){\circle{0.08}}
\put(9.75,0.225){\circle{0.08}}
\put(10,0.2){\circle{0.08}}
\put(10.25,0.175){\circle{0.08}}
\put(10.5,0.15){\circle{0.08}}
\put(10.75,0.125){\circle{0.08}}
\put(11,0.1){\circle{0.08}}
\put(11.25,0.075){\circle{0.08}}
\put(11.5,0.05){\circle{0.08}}
\put(11.75,0.025){\circle{0.08}}
\put(12,0){\circle{0.08}}
\put(12,0){\circle{0.12}}

\put(0,0){\vector(1,0){14}}
\put(0,0){\vector(0,1){5.5}}

\put(0,0){\circle{0.12}}
\put(0,0){\circle{0.08}}

\put(4,0){\circle{0.12}}
\put(4,0){\circle{0.08}}

\put(0,4){\circle{0.12}}
\put(0,4){\circle{0.08}}

\put(12,4){\circle{0.12}}
\put(12,4){\circle{0.08}}

{\tiny

\put(-6.2,3.85){$P_{1}'=(\log m,\nu _{q}(a_{m}))$}

\put(7.1,0.85){$P_{2}'=(\log j_{2},\nu _{q}(a_{j_{2}}))$}

\put(3,-0.85){$P_{2}=(\log j_{1},0)$}

\put(12.3,3.85){$P_{3}=(\log n,\nu _{p}(a_{n}))$}

\put(-2,-0.85){$P_{1}=(\log m,0)$}

\put(10.2,-0.85){$P_{3}'=(\log n,0)$}

\put(-10.9,-0.85){{\bf Figure 14.}}
}

\end{picture}
\end{center}
\bigskip

In our particular case $j_{1}=m+1$, $j_{2}$ is denoted by $j$, $\nu _{p}(a_{n})=\nu _{q}(a_{j})=1$ and $\nu _{q}(a_{m})=2$, so by i), ii) and Lemma \ref{puncte}, each one of the edges $P_{1}P_{2}$, $P_{2}P_{3}$, $P_{1}'P_{2}'$ and $P_{2}'P_{3}'$ contains a single segment. Moreover, the slope of $P_2'P_3'$ is greater than the slope of $P_1'P_2'$, as $j<\sqrt{mn}$, so in particular the log-integral points $P_{1}',P_{2}'$ and $P_{3}'$ are not collinear. Now, if $g(s)=\frac{b_{c}}{c^{s}}+\cdots +\frac{b_{d}}{d^{s}}$ is a nonconstant factor of $f$ with $d<n$, then by Theorem \ref{DumDir}, $\frac{d}{c}$ must belong to
\[
\left\{ \frac{m+1}{m},\frac{n}{m+1}\right\} \bigcap \left\{ \frac{j}{m},\frac{n}{j}\right\} .
\]
The existence of an index $j$ with $m+1<j<\sqrt{mn}$ requires that $\sqrt{mn}>m+2$, and also implies that $\frac{mn}{m+1}>\sqrt{mn}$, so our $j$ cannot be equal to $\frac{mn}{m+1}$. Thus, the above intersection is empty, so $d$ must be equal to $n$, and $c$ to $m$, which proves the irreducibility of $f$. 
\end{proof}

\begin{theorem}\label{pnp}
Let $f(s)=\frac{a_{m}}{m^{s}}+\cdots +\frac{a_{n}}{n^{s}}$ be an algebraically primitive Dirichlet polynomial with coefficients in a unique factorization domain $R$, $a_{m}a_{n}\neq 0$, let $p,q$ be prime elements of $R$, and assume that $q_{1},\dots ,q_{k}$ are all the prime factors of $m\cdot n$. Let $j$ be an index with $m<j<n$ and let $\delta :=\gcd (\nu _{p}(a_{n})-\nu _{p}(a_{m}),\nu _{q_{1}}(n)-\nu _{q_{1}}(m),\dots ,\nu _{q_{k}}(n)-\nu _{q_{k}}(m))$. If
\smallskip

i) \ $\nu _{p}(a_{m})\neq \nu _{p}(a_{n})$\ \ and \ \ 
$\left ( \frac{n}{i}\right ) ^{\nu _{p}(a_{i})-\nu _{p}(a_{m})}>\left ( \frac{m}{i}\right ) ^{\nu _{p}(a_{i})-\nu _{p}(a_{n})}$ for $m<i<n$,

ii) \thinspace $q\mid a_{i}$ for all $i\neq j$, $q\nmid a_{j}$, $q^{2}\nmid a_{m}$ and $q^{2}\nmid a_{n}$;

iii) $j\not\in \{ m^{1-\frac{i}{\delta }}n^{\frac{i}{\delta }}\ :\ 0<i<\delta \} $,
\smallskip

\noindent then $f$ is irreducible over $Q(R)$.
\end{theorem}
\begin{proof} \ Figure 15 displays the Newton log-polygons with respect to two prime elements $p$ and $q$, the first one consisting of a single edge, and the second one consisting of two edges with slopes of different signs, as in Figure 5.j. 

\begin{center}

\setlength{\unitlength}{4.65mm}
\begin{picture}(12,6)
\linethickness{0.075mm}

\put(12,0){\line(0,1){4.5}}   
\thicklines

\put(0,1){\line(4,1){12}}   

\linethickness{0.15mm}

\put(0,4){\circle{0.12}}
\put(0,4){\circle{0.08}}
\put(0.25,3.875){\circle{0.08}}
\put(0.5,3.75){\circle{0.08}}
\put(0.75,3.625){\circle{0.08}}
\put(1,3.5){\circle{0.08}}
\put(1.25,3.375){\circle{0.08}}
\put(1.5,3.25){\circle{0.08}}
\put(1.75,3.125){\circle{0.08}}
\put(2,3){\circle{0.08}}
\put(2.25,2.875){\circle{0.08}}
\put(2.5,2.75){\circle{0.08}}
\put(2.75,2.625){\circle{0.08}}
\put(3,2.5){\circle{0.08}}
\put(3.25,2.375){\circle{0.08}}

\put(3.5,2.25){\circle{0.08}}
\put(3.75,2.125){\circle{0.08}}
\put(4,2){\circle{0.08}}
\put(4.25,1.875){\circle{0.08}}
\put(4.5,1.75){\circle{0.08}}
\put(4.75,1.625){\circle{0.08}}
\put(5,1.5){\circle{0.08}}
\put(5.25,1.375){\circle{0.08}}
\put(5.5,1.25){\circle{0.08}}
\put(5.75,1.125){\circle{0.08}}
\put(6,1){\circle{0.08}}
\put(6.25,0.875){\circle{0.08}}
\put(6.5,0.75){\circle{0.08}}
\put(6.75,0.625){\circle{0.08}}
\put(7,0.5){\circle{0.08}}
\put(7.25,0.375){\circle{0.08}}
\put(7.5,0.25){\circle{0.08}}
\put(7.75,0.125){\circle{0.08}}
\put(8,0){\circle{0.08}}
\put(8,0){\circle{0.12}}

\put(8.25,0.125){\circle{0.08}}
\put(8.5,0.25){\circle{0.08}}
\put(8.75,0.375){\circle{0.08}}
\put(9,0.5){\circle{0.08}}
\put(9.25,0.625){\circle{0.08}}
\put(9.5,0.75){\circle{0.08}}
\put(9.75,0.875){\circle{0.08}}
\put(10,1){\circle{0.08}}
\put(10.25,1.125){\circle{0.08}}
\put(10.5,1.25){\circle{0.08}}
\put(10.75,1.375){\circle{0.08}}
\put(11,1.5){\circle{0.08}}
\put(11.25,1.625){\circle{0.08}}
\put(11.5,1.75){\circle{0.08}}
\put(11.75,1.875){\circle{0.08}}
\put(12,2){\circle{0.08}}
\put(12,2){\circle{0.12}}

\put(0,0){\vector(1,0){14}}
\put(0,0){\vector(0,1){5.5}}

\put(0,1){\circle{0.12}}
\put(0,1){\circle{0.08}}

\put(12,4){\circle{0.12}}
\put(12,4){\circle{0.08}}

{\tiny

\put(-6.2,3.9){$P_{1}'=(\log m,\nu _{q}(a_{m}))$}

\put(6.1,-0.85){$P_{2}'=(\log j,0)$}

\put(12.3,3.9){$P_{2}=(\log n,\nu _{p}(a_{n}))$}

\put(12.3,1.9){$P_{3}'=(\log n,\nu _{q}(a_{n}))$}

\put(-6.2,0.9){$P_{1}=(\log m,\nu _{p}(a_{m}))$}

\put(-1.3,-0.85){$(\log m,0)$}

\put(10.8,-0.85){$(\log n,0)$}

\put(-10.9,-0.85){{\bf Figure 15.}}
}

\end{picture}
\end{center}
\bigskip

In our particular case $\nu _{q}(a_{m})=\nu _{q}(a_{n})=1$, so by i), ii) and Lemma \ref{puncte}, each one of the edges $P_{1}'P_{2}'$ and $P_{2}'P_{3}'$ contains a single segment, while $P_{1}P_{2}$ contains precisely $\delta $ segments.
If $\delta =1$, then $f$ is irreducible by Theorem \ref{calaDumas}, so we may assume that $\delta \geq 2$. If $g(s)=\frac{b_{c}}{c^{s}}+\cdots +\frac{b_{d}}{d^{s}}$ is a nonconstant factor of $f$ with $d<n$, then by Theorem \ref{DumDir}, $\frac{d}{c}$ must belong to
\[
\left\{ \frac{j}{m},\frac{n}{j}\right\} \bigcap \left\{ \left( \frac{n}{m}\right) ^{\frac{i}{\delta }}\ :\ 0<i<\delta \right\} .
\]
In view of iii), this intersection is empty, so $d$ must be equal to $n$, and $c$ to $m$, which proves the irreducibility of $f$. 
\end{proof}

\begin{theorem}\label{pzp}
Let $f(s)=\frac{a_{m}}{m^{s}}+\cdots +\frac{a_{n}}{n^{s}}$ be an algebraically primitive Dirichlet polynomial with coefficients in a unique factorization domain $R$, $a_{m}a_{n}\neq 0$, and let $p,q$ be prime elements of $R$. If
\smallskip

i) \ $\nu _{p}(a_{m})\neq \nu _{p}(a_{n})$\ \ and \ \ 
$\left ( \frac{n}{i}\right ) ^{\nu _{p}(a_{i})-\nu _{p}(a_{m})}>\left ( \frac{m}{i}\right ) ^{\nu _{p}(a_{i})-\nu _{p}(a_{n})}$ for $m<i<n$,

ii) \thinspace $q\nmid a_{m}a_{m+1}$, $q\mid a_{i}$ for all $i>m+1$, and $q^{2}\nmid a_{n}$,
\smallskip

\noindent then $f$ is irreducible over $Q(R)$.
\end{theorem}
\begin{proof}\ In Figure 16 we present the Newton log-polygons with respect to two prime elements $p$ and $q$, the first one consisting of a single edge, and the second one consisting of a horizontal edge, followed by an edge with positive slope, as in Figure 5.k. 

\begin{center}
\setlength{\unitlength}{4.65mm}
\begin{picture}(12,6)
\linethickness{0.075mm}

\put(12,0){\line(0,1){4}}   

\thicklines

\put(0,0){\line(1,0){6.6666}}   
\put(6.6666,0){\line(3,1){5.3334}}   
\linethickness{0.15mm}

\put(0,0){\circle{0.12}}
\put(0,0){\circle{0.08}}

\put(0,1){\circle{0.12}}
\put(0,1){\circle{0.08}}
\put(0.25,1.05){\circle{0.08}}
\put(0.5,1.1){\circle{0.08}}
\put(0.75,1.15){\circle{0.08}}
\put(1,1.2){\circle{0.08}}
\put(1.25,1.25){\circle{0.08}}
\put(1.5,1.3){\circle{0.08}}
\put(1.75,1.35){\circle{0.08}}
\put(2,1.4){\circle{0.08}}
\put(2.25,1.45){\circle{0.08}}
\put(2.5,1.5){\circle{0.08}}
\put(2.75,1.55){\circle{0.08}}
\put(3,1.6){\circle{0.08}}
\put(3.25,1.65){\circle{0.08}}
\put(3.5,1.7){\circle{0.08}}
\put(3.75,1.75){\circle{0.08}}
\put(4,1.8){\circle{0.08}}
\put(4.25,1.85){\circle{0.08}}
\put(4.5,1.9){\circle{0.08}}
\put(4.75,1.95){\circle{0.08}}
\put(5,2){\circle{0.08}}
\put(5.25,2.05){\circle{0.08}}
\put(5.5,2.1){\circle{0.08}}
\put(5.75,2.15){\circle{0.08}}
\put(6,2.2){\circle{0.08}}
\put(6.25,2.25){\circle{0.08}}
\put(6.5,2.3){\circle{0.08}}
\put(6.75,2.35){\circle{0.08}}
\put(7,2.4){\circle{0.08}}
\put(7.25,2.45){\circle{0.08}}
\put(7.5,2.5){\circle{0.08}}
\put(7.75,2.55){\circle{0.08}}
\put(8,2.6){\circle{0.08}}
\put(8.25,2.65){\circle{0.08}}
\put(8.5,2.7){\circle{0.08}}
\put(8.75,2.75){\circle{0.08}}
\put(9,2.8){\circle{0.08}}
\put(9.25,2.85){\circle{0.08}}
\put(9.5,2.9){\circle{0.08}}
\put(9.75,2.95){\circle{0.08}}
\put(10,3){\circle{0.08}}
\put(10.25,3.05){\circle{0.08}}
\put(10.5,3.1){\circle{0.08}}
\put(10.75,3.15){\circle{0.08}}
\put(11,3.2){\circle{0.08}}
\put(11.25,3.25){\circle{0.08}}
\put(11.5,3.3){\circle{0.08}}
\put(11.75,3.35){\circle{0.08}}
\put(12,3.4){\circle{0.08}}
\put(12,3.4){\circle{0.12}}

\put(0,0){\vector(1,0){14}}
\put(0,0){\vector(0,1){5.5}}

\put(6.6666,0){\circle{0.12}}

\put(12,1.7778){\circle{0.12}}
\put(12,1.7778){\circle{0.08}}

{\tiny

\put(4.8,-0.85){$P_{2}=(\log j,0)$}

\put(12.3,1.5){$P_{3}=(\log n,\nu _{q}(a_{n}))$}

\put(-6.3,0.9){$P_{1}'=(\log m,\nu _{p}(a_{m}))$}

\put(12.3,3.2){$P_{2}'=(\log n,\nu _{p}(a_{n}))$}

\put(-2.2,-0.85){$P_{1}=(\log m,0)$}

\put(10.8,-0.85){$(\log n,0)$}

\put(-10.9,-0.85){{\bf Figure 16.}}
}
\end{picture}
\end{center}
\bigskip

In our particular case $j=m+1$ and $\nu _{q}(a_{n})=1$, so by i), ii) and Lemma \ref{puncte}, each one of the edges $P_{1}P_{2}$ and $P_{2}P_{3}$ contains a single segment, while $P_{1}P_{2}$ contains precisely $\delta $ segments, with 
\[
\delta :=\gcd (\nu _{p}(a_{n})-\nu _{p}(a_{m}),\nu _{q_{1}}(n)-\nu _{q_{1}}(m),\dots ,\nu _{q_{k}}(n)-\nu _{q_{k}}(m)),
\]
where $q_{1},\dots ,q_{k}$ are all the prime factors of $m\cdot n$. If $\delta =1$, the irreducibility of $f$ follows by Theorem \ref{calaDumas}, so we may assume that $\delta \geq 2$. Now, if $g(s)=\frac{b_{c}}{c^{s}}+\cdots +\frac{b_{d}}{d^{s}}$ is a nonconstant factor of $f$ with $d<n$, then by Theorem \ref{DumDir}, $\frac{d}{c}$ must belong to
\[
\left\{ \frac{m+1}{m},\frac{n}{m+1}\right\} \bigcap \left\{ \left( \frac{n}{m}\right) ^{\frac{i}{\delta }}\ :\ 0<i<\delta \right\} .
\]
To prove that this intersection is empty, we observe that an equality of the form $\frac{m+1}{m}=(\frac{n}{m})^{\frac{i}{\delta }}$ with $i\in \{ 1,\dots ,\delta -1\} $ leads to $(m+1)^{\delta }=n^{i}m^{\delta -i}$, and an equality of the form $\frac{n}{m+1}=(\frac{n}{m})^{\frac{i}{\delta }}$ with $i\in \{ 1,\dots ,\delta -1\} $ leads to $(m+1)^{\delta }=n^{\delta -i}m^{i}$. As $m$ and $m+1$ are coprime, and $i$ is at least $1$, none of these equalities can hold, so $d$ must be equal to $n$, and $c$ to $m$, proving that $f$ is irreducible. 
\end{proof}

\begin{theorem}\label{ppp}
Let $f(s)=\frac{a_{m}}{m^{s}}+\cdots +\frac{a_{n}}{n^{s}}$ be an algebraically primitive Dirichlet polynomial with coefficients in a unique factorization domain $R$, $a_{m}a_{n}\neq 0$, let $p,q$ be prime elements of $R$, and assume that $q_{1},\dots ,q_{k}$ are all the prime factors of $m\cdot n$. Let $j$ be an index, $\sqrt{mn}<j<n$ and let $\delta :=\gcd (\nu _{p}(a_{n})-\nu _{p}(a_{m}),\nu _{q_{1}}(n)-\nu _{q_{1}}(m),\dots ,\nu _{q_{k}}(n)-\nu _{q_{k}}(m))$. If
\smallskip

i) \ $\nu _{p}(a_{m})\neq \nu _{p}(a_{n})$\ \ and \ \ 
$\left ( \frac{n}{i}\right ) ^{\nu _{p}(a_{i})-\nu _{p}(a_{m})}>\left ( \frac{m}{i}\right ) ^{\nu _{p}(a_{i})-\nu _{p}(a_{n})}$ for $m<i<n$,

ii) \thinspace $q\nmid a_{m}$, $q\mid a_{i}$ for $m<i\leq j$, $q^{2}\nmid a_{j}$, $q^{2}\mid a_{i}$ for $j<i\leq n$ and $q^{3}\nmid a_{n}$,

iii) $j\not\in \{ m^{1-\frac{i}{\delta }}n^{\frac{i}{\delta }}\ :\ 0<i<\delta \} $,
\smallskip

\noindent then $f$ is irreducible over $Q(R)$.
\end{theorem}
\begin{proof}\ Figure 17 displays the Newton log-polygons with respect to two prime elements $p$ and $q$, the first one consisting of a single edge, and the second one consisting of two edges with positive slopes, as in Figure 5.l. 

\begin{center}
\setlength{\unitlength}{4.65mm}
\begin{picture}(12,6)
\linethickness{0.075mm}

\put(12,0){\line(0,1){5}}   
\thicklines

\put(0,1){\line(4,1){12}}   

\linethickness{0.15mm}

\put(0,1){\circle{0.08}}
\put(0,1){\circle{0.12}}

\put(12,4){\circle{0.12}}
\put(12,4){\circle{0.08}}

\put(0,0){\circle{0.12}}
\put(0,0){\circle{0.08}}
\put(0.25,0.025){\circle{0.08}}
\put(0.5,0.05){\circle{0.08}}
\put(0.75,0.075){\circle{0.08}}
\put(1,0.1){\circle{0.08}}
\put(1.25,0.125){\circle{0.08}}
\put(1.5,0.15){\circle{0.08}}
\put(1.75,0.175){\circle{0.08}}
\put(2,0.2){\circle{0.08}}
\put(2.25,0.225){\circle{0.08}}
\put(2.5,0.25){\circle{0.08}}
\put(2.75,0.275){\circle{0.08}}
\put(3,0.3){\circle{0.08}}
\put(3.25,0.325){\circle{0.08}}
\put(3.5,0.35){\circle{0.08}}
\put(3.75,0.375){\circle{0.08}}
\put(4,0.4){\circle{0.08}}
\put(4.25,0.425){\circle{0.08}}
\put(4.5,0.45){\circle{0.08}}
\put(4.75,0.475){\circle{0.08}}
\put(5,0.5){\circle{0.08}}
\put(5.25,0.525){\circle{0.08}}
\put(5.5,0.55){\circle{0.08}}
\put(5.75,0.575){\circle{0.08}}
\put(6,0.6){\circle{0.08}}
\put(6.25,0.625){\circle{0.08}}
\put(6.5,0.65){\circle{0.08}}
\put(6.75,0.675){\circle{0.08}}
\put(7,0.7){\circle{0.08}}
\put(7.25,0.725){\circle{0.08}}
\put(7.5,0.75){\circle{0.08}}
\put(7.75,0.775){\circle{0.08}}
\put(8,0.8){\circle{0.08}}
\put(8,0.8){\circle{0.12}}
\put(8.25,0.9){\circle{0.08}}
\put(8.5,1){\circle{0.08}}
\put(8.75,1.1){\circle{0.08}}
\put(9,1.2){\circle{0.08}}
\put(9.25,1.3){\circle{0.08}}
\put(9.5,1.4){\circle{0.08}}
\put(9.75,1.5){\circle{0.08}}
\put(10,1.6){\circle{0.08}}
\put(10.25,1.7){\circle{0.08}}
\put(10.5,1.8){\circle{0.08}}
\put(10.75,1.9){\circle{0.08}}
\put(11,2){\circle{0.08}}
\put(11.25,2.1){\circle{0.08}}
\put(11.5,2.2){\circle{0.08}}
\put(11.75,2.3){\circle{0.08}}
\put(12,2.4){\circle{0.08}}
\put(12,2.4){\circle{0.12}}

\put(0,0){\vector(1,0){14.9}}
\put(0,0){\vector(0,1){5.5}}

{\tiny 

\put(8.75,0.37){$P_{2}'=(\log j,\nu _{q}(a_{j}))$}

\put(12.3,2.2){$P_{3}'=(\log n,\nu _{q}(a_{n}))$}

\put(12.3,3.8){$P_{2}=(\log n,\nu _{p}(a_{n}))$}

\put(-6.3,0.9){$P_{1}=(\log m,\nu _{p}(a_{m}))$}

\put(-2.1,-0.85){$P_{1}'=(\log m,0)$}

\put(10.8,-0.85){$(\log n,0)$}

\put(-10.9,-0.85){{\bf Figure 17.}}
}

\end{picture}
\end{center}
\bigskip

In our particular case $\nu _{q}(a_{j})=1$ and $\nu _{q}(a_{n})=2$, so by i), ii) and Lemma \ref{puncte}, each one of the edges $P_{1}'P_{2}'$ and $P_{2}'P_{3}'$ contains a single segment, while $P_{1}P_{2}$ contains precisely $\delta $ segments. Also, the slope of $P_2'P_3'$ is greater than the slope of $P_1'P_2'$, as $\sqrt{mn}<j<n$. Again, if $g(s)=\frac{b_{c}}{c^{s}}+\cdots +\frac{b_{d}}{d^{s}}$ is a nonconstant factor of $f$ with $d<n$, then by Theorem \ref{DumDir}, $\frac{d}{c}$ must belong to
\[
\left\{ \frac{j}{m},\frac{n}{j}\right\} \bigcap \left\{ \left( \frac{n}{m}\right) ^{\frac{i}{\delta }}\ :\ 0<i<\delta \right\} .
\]
On the other hand, according to iii), this intersection must be empty, so $d$ must be equal to $n$, and $c$ to $m$, which proves the irreducibility of $f$. 
\end{proof}

Our last result in this section is the following one.

\begin{theorem}\label{pnn}
Let $f(s)=\frac{a_{m}}{m^{s}}+\cdots +\frac{a_{n}}{n^{s}}$ be an algebraically primitive Dirichlet polynomial with coefficients in a unique factorization domain $R$, $a_{m}a_{n}\neq 0$, let $p,q$ be prime elements of $R$, and assume that $q_{1},\dots ,q_{k}$ are all the prime factors of $m\cdot n$. Let $j$ be an index, $m<j<\sqrt{mn}$ and let $\delta :=\gcd (\nu _{p}(a_{n})-\nu _{p}(a_{m}),\nu _{q_{1}}(n)-\nu _{q_{1}}(m),\dots ,\nu _{q_{k}}(n)-\nu _{q_{k}}(m))$. If
\smallskip

i) \ $\nu _{p}(a_{m})\neq \nu _{p}(a_{n})$\ \ and \ \ 
$\left ( \frac{n}{i}\right ) ^{\nu _{p}(a_{i})-\nu _{p}(a_{m})}>\left ( \frac{m}{i}\right ) ^{\nu _{p}(a_{i})-\nu _{p}(a_{n})}$ for $m<i<n$,

ii) $q^{2}\mid a_{i}$ for $m\leq i<j$, $q\mid a_{i}$ for $j\leq i<n$,  $q\nmid a_{n}$, $q^{2}\nmid a_{j}$ and $q^{3}\nmid a_{m}$,

iii) $j\not\in \{ m^{1-\frac{i}{\delta }}n^{\frac{i}{\delta }}\ :\ 0<i<\delta \} $,
\smallskip

\noindent then $f$ is irreducible over $Q(R)$.
\end{theorem}
\begin{proof}\ In Figure 18 we plot the Newton log-polygons with respect to two prime elements $p$ and $q$, the first one consisting of a single edge, and the second one consisting of two edges with negative slopes, as in Figure 5.m. 

\begin{center}
\setlength{\unitlength}{4.65mm}
\begin{picture}(12,6)
\linethickness{0.075mm}

\put(12,0){\line(0,1){4.5}}   
\thicklines

\put(0,1){\line(4,1){12}}   

\linethickness{0.15mm}

\put(0,1){\circle{0.12}}
\put(0,1){\circle{0.08}}
\put(0.25,3.875){\circle{0.08}}
\put(0.5,3.75){\circle{0.08}}
\put(0.75,3.625){\circle{0.08}}
\put(1,3.5){\circle{0.08}}
\put(1.25,3.375){\circle{0.08}}
\put(1.5,3.25){\circle{0.08}}
\put(1.75,3.125){\circle{0.08}}
\put(2,3){\circle{0.08}}
\put(2.25,2.875){\circle{0.08}}
\put(2.5,2.75){\circle{0.08}}
\put(2.75,2.625){\circle{0.08}}
\put(3,2.5){\circle{0.08}}
\put(3.25,2.375){\circle{0.08}}

\put(3.5,2.25){\circle{0.08}}
\put(3.75,2.125){\circle{0.08}}
\put(4,2){\circle{0.08}}
\put(4.25,1.875){\circle{0.08}}
\put(4.5,1.75){\circle{0.08}}
\put(4.75,1.625){\circle{0.08}}
\put(5,1.5){\circle{0.08}}
\put(5.25,1.375){\circle{0.08}}
\put(5.5,1.25){\circle{0.08}}
\put(5.75,1.125){\circle{0.08}}
\put(6,1){\circle{0.08}}
\put(6.25,0.875){\circle{0.08}}
\put(6.5,0.75){\circle{0.08}}
\put(6.75,0.625){\circle{0.08}}

\put(7,0.5){\circle{0.12}}
\put(7,0.5){\circle{0.08}}
\put(7.25,0.475){\circle{0.08}}
\put(7.5,0.45){\circle{0.08}}
\put(7.75,0.425){\circle{0.08}}
\put(8,0.4){\circle{0.08}}

\put(8.25,0.375){\circle{0.08}}
\put(8.5,0.35){\circle{0.08}}
\put(8.75,0.325){\circle{0.08}}
\put(9,0.3){\circle{0.08}}
\put(9.25,0.275){\circle{0.08}}
\put(9.5,0.25){\circle{0.08}}
\put(9.75,0.225){\circle{0.08}}
\put(10,0.2){\circle{0.08}}
\put(10.25,0.175){\circle{0.08}}
\put(10.5,0.15){\circle{0.08}}
\put(10.75,0.125){\circle{0.08}}
\put(11,0.1){\circle{0.08}}
\put(11.25,0.075){\circle{0.08}}
\put(11.5,0.05){\circle{0.08}}
\put(11.75,0.025){\circle{0.08}}
\put(12,0){\circle{0.08}}
\put(12,0){\circle{0.12}}

\put(0,0){\vector(1,0){14}}
\put(0,0){\vector(0,1){5.5}}

\put(0,4){\circle{0.12}}
\put(0,4){\circle{0.08}}

\put(12,4){\circle{0.12}}
\put(12,4){\circle{0.08}}

{\tiny

\put(-6.2,3.85){$P_{1}'=(\log m,\nu _{q}(a_{m}))$}

\put(7.1,0.85){$P_{2}'=(\log j,\nu _{q}(a_{j}))$}

\put(12.3,3.85){$P_{2}=(\log n,\nu _{p}(a_{n}))$}

\put(-6.3,0.85){$P_{1}=(\log m,\nu _{p}(a_{m}))$}

\put(-1.3,-0.85){$(\log m,0)$}

\put(10,-0.85){$P_{3}'=(\log n,0)$}

\put(-10.9,-0.85){{\bf Figure 18.}}
}

\end{picture}
\end{center}
\bigskip

In our particular case $\nu _{q}(a_{j})=1$ and $\nu _{q}(a_{m})=2$, so by i), ii) and Lemma \ref{puncte}, each one of the edges $P_{1}'P_{2}'$ and $P_{2}'P_{3}'$ contains a single segment, while $P_{1}P_{2}$ contains precisely $\delta $ segments. Also, the slope of $P_2'P_3'$ is greater than the slope of $P_1'P_2'$, as $m<j<\sqrt{mn}$. Again, if $g(s)=\frac{b_{c}}{c^{s}}+\cdots +\frac{b_{d}}{d^{s}}$ is a nonconstant factor of $f$ with $d<n$, then by Theorem \ref{DumDir}, $\frac{d}{c}$ must belong to
\[
\left\{ \frac{j}{m},\frac{n}{j}\right\} \bigcap \left\{ \left( \frac{n}{m}\right) ^{\frac{i}{\delta }}\ :\ 0<i<\delta \right\} .
\]
By iii), this intersection must be empty, so $d$ must be equal to $n$, and $c$ to $m$, proving the irreducibility of $f$. 
\end{proof}

\section{The irreducibility of Dirichlet polynomials that assume a prime value}\label{primevalues}

Some of the classical or more recent irreducibility criteria for integer polynomials rely on information on the prime factorizations of the values that they assume at some integral arguments. Among the first such results that appeared over the time, we mention here the following irreducibility criteria of St\"ackel \cite{Stackel}, Ore \cite{Ore4}, and Weisner \cite{Weisner}.
\medskip

{\bf Theorem} (St\"ackel, 1918)\ {\em Let $f(X)\in\mathbb{Z}[X]$ and let $A$ denote the maximum of the absolute values of the coefficients of $f(X)$. If there is an integer $t$ with $|t|>1+A$ for which $f(t)$ is a prime number, then $f(X)$ is irreducible in $\mathbb{Z}[X]$.}
\medskip

{\bf Theorem} (Ore, 1934)\ {\em Let $f\in\mathbb{Z}[X]$ be an arbitrary polynomial of degree $n\geq 2$. Assume that for an integer $t$ we have $|f(t)|=pq$ for some positive integers $p,q$ with $p$ prime. Assume also that all the roots of $f$ lie outside the disk $\{z:|z-t|\leq\rho\}$ for some real number $\rho\geq q$. Then $f$ is irreducible in $\mathbb{Q}[X]$.}
\medskip

{\bf Theorem} (Weisner, 1934)\ {\em Let $f\in\mathbb{Z}[X]$ be an arbitrary polynomial of degree $n\geq 2$ and leading coefficient $a_n$. Assume that all the roots of $f$ are in the disk $\{z:|z|<\rho\}$ for some real number $\rho\geq 1$. Assume also that for an integer $t$ we have $|f(t)|=p^{\ell}q$ and $p\nmid f'(t)^{\ell-1}$ for some positive integers $p,q,\ell$ with $p$ prime. If $|t|\geq q+\rho$ or $p^{\ell}\geq |a_n|(|t|+\rho)^{n-1}$, then $f$ is irreducible in $\mathbb{Q}[X]$.}
\medskip

Similar related results have been proved by G. P\'olya and G. Szeg\"o \cite{PolyaSzego}, Girstmair \cite{Girstmair} and Guersenzvaig \cite{Guersenzvaig}, to name just a few authors. Other famous such results are, for instance, Cohn's irreducibility criterion and its generalization to an arbitrary base obtained by Brillhart, Filaseta and Odlyzko \cite{Brillhart}, and to function fields over finite fields by Murty \cite{Murty}.

The aim of this section is to find irreducibility criteria for Dirichlet polynomials that assume a value divisible by a prime, or by a prime power. Some of the results that we will obtain are analogous to the above criteria of St\"ackel, Ore and Weisner. In the proof of our results we will need an inequality of Gelfond on the heights of the factors of a multivariate polynomial (see Mahler \cite{Mahler} for a proof). We will first need to fix some notations. For a polynomial
\[
f(z_1,\dots ,z_n)=\sum_{h_1=0}^{m_1}\dots \sum_{h_n=0}^{m_n}a_{h_1\dots h_n}z_1^{h_1}\cdots z_n^{h_n}
\]
in $n$ variables $z_1,\dots ,z_n$ with arbitrary real or complex coefficients, we let 
\[
H(f)=\max\limits _{h_i=0,\dots ,m_i,\ i=1,\dots ,n}|a_{h_1\dots h_n}|,
\]
the height of $f$, and denote by $\nu (f)$ the number of variables $z_1,\dots ,z_n$ that occur in $f$ at least to the first degree. With this notation we have the following result.
\medskip

{\bf Theorem} (Gelfond). {\em Let $f(z_1,\dots ,z_n)=\prod\limits _{i=1}^{\ell}f_i(z_1,\dots ,z_n)$ with $f_1,\dots ,f_{\ell }$ polynomials with arbitrary real or complex coefficients. Then we have the inequality}
\begin{equation}\label{GelfondInequality}
\prod\limits _{i=1}^{\ell}H(f_i)\leq 2^{m_1+\cdots +m_n-\nu(f)}\{(m_1+1)\cdots (m_n+1)\}^{\frac{1}{2}}H(f).
\end{equation}

This inequality will be particularly useful in the case of polynomials with integer coefficients, for which the right side provides an upper bound for the heights of the factors $f_1,\dots ,f_{\ell}$. 
\begin{definition}
For a Dirichlet polynomial $f(s)=\frac{a_{m}}{m^{s}}+\cdots +\frac{a_{n}}{n^{s}}$ with integer coefficients and $a_ma_n\neq 0$ we define $H(f)=\max\{ |a_m|,\dots ,|a_n|\} $ and call it the ``height'' of $f$.
\end{definition}
Before stating our first result in this section, we will mention that despite its somewhat unaestethic presence in inequality (\ref{camcomplicata}) below, the constant $1+\frac{\log _3(\log _212)}{2}$ plays a crucial role in obtaining conditions in closed form for $t$ in Corollaries \ref{Corolarsupportcardinality1}, \ref{Corolarsupportcardinality2} and \ref{Corolarsupportcardinality3}, and also in Corollaries \ref{CorolarsupportcardinalityPlaell1}, \ref{CorolarsupportcardinalityPlaell2} and \ref{CorolarsupportcardinalityPlaell3}.
\begin{theorem}\label{supportcardinality}
Let $f$ be a Dirichlet polynomial of composite degree $n$ with integer coefficients, and assume that $p$ is the least prime factor of $n$, and that $f$ has $k\geq 1$ relevant primes. If $|f(-t)|=Pq$ with $P$ a prime number, $q$ a positive integer, and $t$ an integer satisfying
\begin{equation}\label{camcomplicata}
t>\frac{n}{p}\log \biggl( \frac{nqH(f)}{p}\cdot\Bigl(\frac{n}{2}\Bigr)^{k\bigl(1+\frac{\log _3(\log _212)}{2}\bigr)}\biggr),
\end{equation}
then $f$ is irreducible.
\end{theorem}
\begin{proof}
Let us assume to the contrary that $f(s)=g(s)h(s)$ with $g,h$ nonconstant Dirichlet polynomials with integer coefficients. Since $|f(-t)|=Pq$ and $P$ is prime, we deduce from the equality $|f(-t)|=|g(-t)|\cdot |h(-t)|$ that one of the integers $|g(-t)|$ and $|h(-t)|$ must be divisible by $P$, so the other must divide $q$, say $|g(-t)|\mid q$. In particular we have
\begin{equation}\label{maimicdecatq}
|g(-t)|\leq q.
\end{equation}
Let  now $F(X_1,\dots ,X_k)\in \mathbb{Z}[X_1,\dots ,X_k]$ be the polynomial obtained from $f(s)$ by letting $X_i=p_i^{-s}$ for each of the $k$ relevant primes $p_1,\dots ,p_k$ of $f$. We obviously have $H(f)=H(F)$. Suppose now that $g(s)=\frac{b_{c}}{c^{s}}+\cdots +\frac{b_{d}}{d^{s}}$, with $d$ a proper divisor of $n$, and $c\leq d$. Since $g$ is a factor of $f$, we can use the same identification above and obtain a polynomial  $G(X_1,\dots ,X_k)\in \mathbb{Z}[X_1,\dots ,X_k]$, which is a factor of $F$ and satisfies $H(g)=H(G)$. Note that some of the indeterminates $X_1,\dots ,X_k$ might not effectively appear in the expression of $G$.
Let $\nu _{p_i}(\cdot )$ denote the usual $p_i$-adic valuation, for $i=1,\dots ,k$.   By (\ref{GelfondInequality}) we must have
\[
H(g)\leq 2^{m_1+\cdots +m_k-k}\{(m_1+1)\cdots (m_k+1)\}^{\frac{1}{2}}H(f),
\]
where for each $i=1,\dots ,k$, $m_i=\max\limits _{N\in Supp(f)}\nu _{p_i}(N)$, the greatest multiplicity of $p_i$ in the prime factorizations of the integers in $Supp(f)$. Since a uniform upper bound for $m_1, \dots ,m_k$ is obviously $\log_2n$, we deduce that
\begin{equation}\label{cevacumi}
H(g)\leq 2^{k\log _2\frac{n}{2}}(\log _22n)^{\frac{k}{2}}H(f)=\Bigl(\frac{n}{2}\Bigr)^k(\log _22n)^{\frac{k}{2}}H(f).
\end{equation}

We will first assume that $c=d$, so $g(-t)=b_dd^t$.
According to (\ref{maimicdecatq}), we have $|g(-t)|\leq q$, so in this case we reach a contradiction if $|d^t|>q$, that is for $t>\log _dq$, as $|b_d|\geq 1$. In particular, we obtain a contradiction for $t>\log _2q$, as $d\geq 2$. Since the right side in (\ref{camcomplicata}) exceeds $\log _2q$, we conclude that $g$ cannot consist of a single term, so $f$ must be algebraically primitive. We will therefore assume that $c<d$. By the triangle inequality and (\ref{cevacumi}) we successively deduce in this case that
\begin{eqnarray*}
q & \geq & |b_cc^t+\cdots +b_dd^t|\geq |b_dd^t|-H(g)(c^t+\cdots +(d-1)^t)\\
 & \geq & d^t-H(g)(1+\cdots +(d-1)^t)\\
 & \geq & d^t-\Bigl(\frac{n}{2}\Bigr)^k(\log _22n)^{\frac{k}{2}}H(f)(1+\cdots +(d-1)^t) \\
 & \geq & d^t-\Bigl(\frac{n}{2}\Bigr)^k(\log _22n)^{\frac{k}{2}}H(f)(d-1)^{t+1}.
\end{eqnarray*}
To reach a contradiction, it suffices to show that $d^t>(\frac{n}{2})^k(\log _22n)^{\frac{k}{2}}H(f)(d-1)^{t+1}+q$. We will actually prove that $(\frac{d}{d-1})^t>(\frac{n}{2})^k(\log _22n)^{\frac{k}{2}}H(f)(d-1)+q$, that is
\[
t>\log _{\frac{d}{d-1}}\left(\Bigl(\frac{n}{2}\Bigr)^k(\log _22n)^{\frac{k}{2}}H(f)(d-1)+q\right).
\]
Observing that the right side in last inequality is an increasing function of $d$, and $d\leq \frac{n}{p}$, it will be sufficient to prove that
\begin{equation}\label{ceacunminusp}
t>\log _{\frac{\frac{n}{p}}{\frac{n}{p}-1}}\Bigl(\Bigl(\frac{n}{2}\Bigr)^k(\log _22n)^{\frac{k}{2}}H(f)\Bigl(\frac{n}{p}-1\Bigr)+q\Bigr),
\end{equation}
or even more, to simplify computations, that
\begin{equation}\label{aproapegata}
t>\log _{\frac{\frac{n}{p}}{\frac{n}{p}-1}}\left(\Bigl(\frac{n}{2}\Bigr)^k(\log _22n)^{\frac{k}{2}}H(f)\cdot\frac{nq}{p}\right).
\end{equation}
Let us denote by $A$ the term under the logarithm in (\ref{aproapegata}). At this point, we can use the fact that $\bigl(\frac{\frac{n}{p}}{\frac{n}{p}-1}\bigr)^{\frac{n}{p}}>e$, the base of natural logarithm, to deduce that 
\[
\frac{n}{p}\log A>\log _{\frac{\frac{n}{p}}{\frac{n}{p}-1}}A,
\]
which shows that we can actually ask $t$ to satisfy the inequality
\[
t>\frac{n}{p}\left( k\log \frac{n}{2}+\frac{k}{2}\log (\log _2 2n) +\log \Bigl(\frac{nq}{p}H(f)\Bigr) \right).
\]
Since $n$ is composite, we have $n\geq 4$. We will study the case $n=4$ separately, so we will first assume that $n\geq 6$. We search for a real positive $\delta $ as small as possible such that $(\frac{n}{2})^{\delta }\geq \log _2 2n$ for $n\geq 6$. It is easy to check that $\delta =\log _3(\log _212)$. Thus, it suffices to ask $t$ to satisfy the inequality
\[
t>\frac{n}{p}\left( k\log \frac{n}{2}+k\frac{\log _3(\log _212)}{2}\log \frac{n}{2} +\log \Bigl(\frac{nq}{p}H(f)\Bigr) \right),
\]
or, equivalently,
\[
t>\frac{n}{p}\log \biggl(\Bigl(\frac{n}{2}\Bigr)^{k\bigl(1+\frac{\log _3(\log _212)}{2}\bigr)}\cdot \frac{nqH(f)}{p} \biggr),
\]
which is precisely condition (\ref{camcomplicata}). This completes the proof for $n\geq 6$.

We consider now the case $n=4$. Here we have $p=d=2$, and if $f$ would have a nonzero term $\frac{a_3}{3^s}$, its irreducibility would follow by Proposition \ref{prop2}, so we can assume that the only relevant prime of $f$ is $2$, meaning that $k=1$. Thus, we need to prove that the conclusion on the irreducibility of $f$ holds for $t$ satisfying
\begin{equation}\label{penultima}
t>2\log \left(2^{2+\frac{\log _3(\log _212)}{2}}\cdot qH(f) \right).
\end{equation}
On the other hand, condition (\ref{ceacunminusp}) reduces in this case to \begin{equation}\label{cuplusq}
t>\log _2(2\sqrt{3}H(f)+q).
\end{equation} 

We will first consider the case that $q\geq 2$. In this case (\ref{cuplusq}) obviously holds if $t$ satisfies (\ref{penultima}), since $\log _2(2\sqrt{3}qH(f))>\log _2(2\sqrt{3}H(f)+q)$, $2\bigl(2+\frac{\log _3(\log _212)}{2}\bigr)\log 2>3>\log_{2}(2\sqrt{3})$, and $2\log (qH(f))=2(\log 2)(\log _{2}(qH(f)))>\log _{2}(qH(f))$. 

For $q=1$ we need to show that 
\[
2\log \left(2^{2+\frac{\log _3(\log _212)}{2}}\cdot H(f) \right)>\log _2(2\sqrt{3}H(f)+1),
\]
or, equivalently, that
\begin{equation}\label{ultimulcaz}
\log _2\left(2^{2+\frac{\log _3(\log _212)}{2}}\cdot H(f) \right)^{2\log 2}>\log _2(2\sqrt{3}H(f)+1),
\end{equation}
which obviously holds for $H(f)\geq 2$, since the term under the logarithm in the left side is greater than $3.72H(f)^{1.38}$, which in turn is greater than $2\sqrt{3}H(f)+1$ for $H(f)\geq 2$.

However, inequality (\ref{ultimulcaz}) fails for $H(f)=1$, so in this case we need to analyse separately the Dirichlet polynomials of the form $f(s)=\frac{a}{1^s}+\frac{b}{2^s}+\frac{c}{4^s}$ with integers $a,b,c$ satisfying $|a|\leq 1$, $|b|\leq 1$ and $|c|=1$. If both $a$ and $b$ are zero, $|f(-t)|$ cannot be a prime number for positive integers $t$, so such an $f$ cannot satisfy the hypotheses of the theorem. If $a=0$ and $b\neq 0$, then $f$ must be of the form $\pm (\frac{1}{2^s}+\frac{1}{4^s})$, or of the form $\pm(\frac{-1}{2^s}+\frac{1}{4^s})$, which are not algebraically primitive. Note that even if we ignore the fact that $f$ must be algebraically primitive, both $\pm (\frac{1}{2^s}+\frac{1}{4^s})$ and $\pm(\frac{-1}{2^s}+\frac{1}{4^s})$ fail to satisfy the conditions in the statement. In the first case $|f(-t)|$ cannot be a prime number for positive integers $t$, while in the second case, the only prime value of the form $4^t-2^t$ is obtained for $t=1$, but $t=1$ does not satisfy condition (\ref{camcomplicata}), since the right side in (\ref{camcomplicata}) is in this case greater than $3$. 

Assume now that $b=0$ and $a\neq 0$, so $f$ is now of the form $\pm (\frac{1}{1^s}+\frac{1}{4^s})$, or of the form $\pm(\frac{-1}{1^s}+\frac{1}{4^s})$. In the first case, the irreducibility of $f$ follows by direct computation, without looking at its possible prime values. However, in this case for $|f(-t)|=4^t+1$ to be a prime, it must in fact be a Fermat prime, and $t=4$ and $t=8$ satisfy condition (\ref{camcomplicata}), with corresponding prime values the Fermat primes $F_3=257$ and $F_4=65537$. In our second case, the Dirichlet polynomials $f(s)=\pm(\frac{-1}{1^s}+\frac{1}{4^s})$ are obviously reducible, and the only prime value of the form $|f(-t)|=4^t-1$ is $3$, and is obtained for $t=1$, which fails to satisfy condition (\ref{camcomplicata}).

We are thus left with the case that none of $a,b$ and $c$ is zero, so $|a|=|b|=|c|=1$. It is easy to see that all the corresponding eight such Dirichlet polynomials are irreducible, since their hypothetical nonconstant factors would be forced to be of the form $\pm(\frac{1}{1^s}+\frac{1}{2^s})$, or of the form $\pm(\frac{-1}{1^s}+\frac{1}{2^s})$, and any product of two such factors would force $b$ to be zero or $\pm2$. 
This completes the proof.
\end{proof}
\begin{remark}
It should be noted that to improve (\ref{camcomplicata}), 
instead of bounding $1+\cdots +(d-1)^t$ from above by $(d-1)^{t+1}$, one might use Faulhaber's formula
\[
\sum\limits _{i=1}^n(d-1)^t=\frac{1}{t+1}\sum\limits _{i=0}^t\tbinom{t+1}{i} B_i(d-1)^{t-i+1},
\]
combined with majorants for the absolute values of the Bernoulli numbers $B_i$. Another way to improve (\ref{camcomplicata}) would be to search for a sharper bound on $H(g)$ than the right side in (\ref{GelfondInequality}),
but such attempts might produce much more involved formulae.
\end{remark}
In particular, we obtain from Theorem \ref{supportcardinality} the following result for the case that no information on the least prime factor of $f$ and on the number of relevant primes of $f$ is known.
\begin{corollary}\label{Corolarsupportcardinality1}
Let $f$ be a Dirichlet polynomial of degree $n\geq 2$ with integer coefficients. If $|f(-t)|=Pq$ with $P$ a prime number, $q$ a positive integer, and $t$ an integer satisfying
\[
t>n^2+\frac{n}{2}\log (qH(f)),
\]
then $f$ is irreducible.
\end{corollary}
\begin{proof}
If $n$ is prime, the irreducibility of $f$ follows by Proposition \ref{prop1}, without any additional condition. For a composite $n$ we may ask $t$ to satisfy (\ref{camcomplicata}) with $2$ instead of $p$, and with $\pi(n)$, the prime counting function, instead of $k$, since the number of relevant primes of $f$ is at most $\pi(n)$. In this case condition (\ref{camcomplicata}) becomes
\[
t>\left(1+\frac{\log _3(\log _212)}{2}\right)\frac{n\pi(n)}{2}\log \left(\frac{n}{2}\right)+\frac{n}{2}\log \left(\frac{nqH(f)}{2} \right).
\]
By \cite[(3.6)]{BarkleyRosserSchoenfeld} we have $\pi(n)<1.25506\frac{n}{\log n}$ for $n>1$, so we can ask $t$ to satisfy
\[
t>\frac{1.25506}{2}\cdot\left(1+\frac{\log _3(\log _212)}{2}\right)\cdot \frac{n^2}{\log n}\log \left(\frac{n}{2}\right)+\frac{n}{2}\log \left(\frac{nqH(f)}{2} \right).
\]
Rather surprisingly, the constant $\frac{1.25506}{2}\cdot\bigl(1+\frac{\log _3(\log _212)}{2}\bigr)$ is quite close to $1$, more precisely, it is equal to $0.99217077...$, so we can ask $t$ to satisfy
\[
t>\frac{n^2}{\log n}(\log n-\log 2)+\frac{n}{2}\log \left(\frac{n}{2}\right)+\frac{n}{2}\log (qH(f)),
\]
or, equivalently,
\[
t>n^2-\frac{n^2\log 2}{\log n}+\frac{n}{2}\log \left(\frac{n}{2}\right)+\frac{n}{2}\log (qH(f)).
\]
The proof finishes by observing that $\frac{n^2\log 2}{\log n}>\frac{n}{2}\log \bigl(\frac{n}{2}\bigr)$ for $n\geq 2$.
\end{proof}
In particular, we obtain for $q=1$ the following result for Dirichlet polynomials that assume a prime value, that may be regarded as an analogue of St\"ackel's irreducibility criterion.
\begin{corollary}\label{Corolarsupportcardinality2}
Let $f$ be a Dirichlet polynomial of degree $n\geq 2$ with integer coefficients. If $|f(-t)|$ is a prime number for some integer $t>n^2+\frac{n}{2}\log H(f)$, then $f$ is irreducible.
\end{corollary}
The lower bound for $t$ simplifies even more if all the coefficients of $f$ have absolute values at most 1, as in the following result.
\begin{corollary}\label{Corolarsupportcardinality3}
Let $f$ be a Dirichlet polynomial of degree $n\geq 2$ with integer coefficients $a_i\in \{-1,0,1\}$. If $|f(-t)|$ is a prime number for some integer $t>n^2$, then $f$ is irreducible.
\end{corollary}
We mention here that for $\ell>1$, Weisner's technical condition that $p\nmid f'(t)^{\ell -1}$ reduces to $p\nmid f'(t)$, and in case $f$ can be written as a product of two nonconstant factors $g$ and $h$, prevents $g(t)$ and $h(t)$ to be both divisible by $p$. As we shall see in the following result, in the case of a Dirichlet polynomial $f(s)$ that factors nontrivially as $g(s)\cdot h(s)$, to prevent $g(-t)$ and $h(-t)$ to be both divisible by a certain prime $P$, we will need a surprisingly different kind of condition on the derivative of $f$, namely an irrationality condition for the principal $P$th root of a certain expression depending on $f$.

\begin{theorem}\label{supportcardinalityPlaell}
Let $f$ be a Dirichlet polynomial of composite degree $n\geq 2$ with integer coefficients $a_i$, and assume that $p$ is the smallest prime factor of $n$, and that $f$ has $k\geq 1$ relevant primes. Assume that $|f(-t)|=P^{\ell}q$ with $P$ a prime number, $\ell$ and $q$ positive integers, $P\nmid q$, $\ell \geq 2$, and $t$ an integer satisfying
\[
t>\frac{n}{p}\log \left( \frac{nqH(f)}{p}\cdot\Bigl(\frac{n}{2}\Bigr)^{k\bigl(1+\frac{\log _3(\log _212)}{2}\bigr)}\right).
\]
If the principal $P$th root of \thinspace $\prod\limits _{i}i^{a_i\cdot i^t}$ is an irrational number, then $f$ is irreducible.
\end{theorem}
\begin{proof}
Assume that $f(s)=\frac{a_m}{m^s}+\cdots +\frac{a_n}{n^s}$, and that $f=g\cdot h$, with $g(s)=\frac{b_{m_1}}{m_1^s}+\cdots +\frac{b_{n_1}}{n_1^s}$ and $h(s)=\frac{c_{m_2}}{m_2^s}+\cdots +\frac{c_{n_2}}{n_2^s}$ nonconstant Dirichlet polynomials with integer coefficients. Since $|f(-t)|=P^{\ell}q$ and $f(-t)=g(-t)h(-t)$, we have $|g(-t)|=P^{a}q_1$ and $|h(-t)|=P^{b}q_2$ for two nonnegative integers $a,b$ with $a+b=\ell$ and two positive integers $q_1,q_2$ with $q_1q_2=q$ (hence none of $q_1$ and $q_2$ is divisible by $P$). We will prove now that one of the integers $a$ and $b$ must be zero. Let us assume to the contrary that both $a$ and $b$ are positive. Observe that
\begin{eqnarray*}
f'(-t) & = & -\sum\limits _{m\leq i\leq n}a_ii^t\log i\\
& = & -h(-t)\cdot\biggl(\sum\limits _{m_1\leq i\leq n_1}b_ii^t\log i\biggr) -g(-t)\cdot\biggl(\sum\limits _{m_2\leq i\leq n_2}c_ii^t\log i\biggr)\\
& = & -P^bq_2\biggl(\sum\limits _{m_1\leq i\leq n_1}b_ii^t\log i\biggr)-P^aq_1\biggl(\sum\limits _{m_2\leq i\leq n_2}c_ii^t\log i\biggr),
\end{eqnarray*}
from which we deduce that
\begin{equation}\label{Ppower}
\prod\limits _{i}i^{a_i\cdot i^t}=\Bigl(\prod\limits _{i}i^{b_i\cdot i^t}\Bigr)^{P^{b}q_2}\cdot \Bigl(\prod\limits _{i}i^{c_i\cdot i^t}\Bigr)^{P^{a}q_1}.
\end{equation}
Observe now that the right side in (\ref{Ppower}) is the $P$th power of a rational number, implying that the principal $P$th root of  $\prod\limits _{i}i^{a_i\cdot i^t}$ is a rational number, a contradiction. Therefore one of $a$ and $b$ must be zero, say $a=0$, implying that $|g(-t)|=q_1$. In particular, this shows that $|g(-t)|\leq q$. The proof continues now using precisely the same lines as in the proof of Theorem \ref{supportcardinality}.
\end{proof}
\begin{remark}\label{finitelymanyP}
We notice that for a given Dirichlet polynomial with integer coefficients $a_i$, if $\prod\limits _{i}i^{a_i\cdot i^t}\neq 1$ for some positive integer $t$, then the principal $P$th root of  $\prod\limits _{i}i^{a_i\cdot i^t}$ is an irrational number for all but finitely many prime numbers $P$.
\end{remark}
In particular, we obtain from Theorem \ref{supportcardinalityPlaell} the following result for the case that no information on the least prime factor of $f$ and on the number of relevant primes of $f$ is known.
\begin{corollary}\label{CorolarsupportcardinalityPlaell1}
Let $f$ be a Dirichlet polynomial of degree $n\geq 2$ with integer coefficients $a_i$. Assume that $|f(-t)|=P^{\ell}q$ with $P$ a prime number, $q$ a positive integer, $P\nmid q$, $\ell\geq 2$ an integer, and $t$ an integer satisfying
\[
t>n^2+\frac{n}{2}\log (qH(f)).
\]
If the principal $P$th root of \thinspace $\prod\limits _{i}i^{a_i\cdot i^t}$ is an irrational number, then $f$ is irreducible.
\end{corollary}
We note that Corollary \ref{Corolarsupportcardinality1} and Corollary \ref{CorolarsupportcardinalityPlaell1} can be viewed as analogous results for Dirichlet polynomials of the irreducibility criteria of Ore and Weissner.

As a particular case, we obtain for $q=1$ the following irreducibility conditions for Dirichlet polynomials that assume a prime power value.
\begin{corollary}\label{CorolarsupportcardinalityPlaell2}
Let $f$ be a Dirichlet polynomial of degree $n\geq 2$ with integer coefficients $a_i$. Assume that $|f(-t)|=P^{\ell}$ with $t>n^2+\frac{n}{2}\log H(f)$ an integer, $P$ a prime, and $\ell\geq 2$ an integer. If the principal $P$th root of \thinspace $\prod\limits _{i}i^{a_i\cdot i^t}$ is an irrational number, then $f$ is irreducible.
\end{corollary}
Here too, the simplest conditions on $t$ appear when all the coefficients $a_i$ are zero, or $\pm 1$.
\begin{corollary}\label{CorolarsupportcardinalityPlaell3}
Let $f$ be a Dirichlet polynomial of degree $n\geq 2$ with integer coefficients $a_i\in\{-1,0,1\}$. Assume that $|f(-t)|=P^{\ell}$ with $t,P$ and $\ell$ integers with $t>n^2$, $P$ prime, and $\ell\geq 2$. If the principal $P$th root of $\prod\limits _{i}i^{a_ii^t}$ is an irrational number, then $f$ is irreducible.
\end{corollary}

\section{The multiplicities of the irreducible factors of Dirichlet polynomials\\ and square-free criteria}\label{squarefree}
A direct consequence of the definition is the following square-free criterion.
\begin{proposition}\label{elementarysquarefree}
A primitive Dirichlet polynomial of square-free degree is square-free.
\end{proposition}
In the forthcoming results we will use the following definition.
\begin{definition}\label{Deltak(n)}
For any integer $n\geq 2$ we let $M(n)=\max\limits_{p\mid n} \nu _p(n)$, and for any positive integers $n$ and $k$ we let $\Delta _k(n)=\max\{d:d\mid n\ \text{and}\ d^k\mid n\}$, and also
\[
n_{<k} = \prod\limits_{p\hspace{0.8mm}\text{prime},\hspace{0.6mm} \nu_p(n)<k}p^{\nu _p(n)}\quad \text{and}\quad n_{\geq k}=\prod\limits_{p\hspace{0.8mm} \text{prime},\hspace{0.6mm} \nu_p(n)\geq k}p^{\nu _p(n)},
\]
with an empty product equal to $1$. Thus $n=n_{<k}\cdot n_{\geq k}$,
$\Delta_1(n)=n$, and for $k\geq 2$ we have $\Delta_k(n)=1$ if and only if $n$ is $k$-power-free, that is if and only if $n=n_{<k}$.

For a primitive Dirichlet polynomial $f(s)\in DP(R)[s]$ having the canonical decomposition $f(s)=c\cdot f_1(s)^{n_1}\cdots f_k(s)^{n_k}$, with $c\in R\setminus \{0\}$ and $f_i\in DP(R)[s]$ irreducible and pairwise nonassociated in divisibility, we let $M(f)=\max\{ n_1,\dots ,n_k\}$.
\end{definition}

We will first prove two $k$-power-freeness criteria for Dirichlet polynomials $f(s)\in DP(R)[s]$ whose degrees have at least one prime factor with multiplicity $\geq k$, which complement Proposition \ref{prop2} and Proposition \ref{prop3}.

\begin{theorem}\label{THsquareoftheleastpn}
Let $k\geq 2$ and $n$ be integers with $M(n)\geq k$, let $q_k$ be the smallest prime factor of $n$ with multiplicity $\geq k$, and $f$ a Dirichlet polynomial of degree $n$. If $f$ has a nonzero coefficient $a_i$ with $n-\min\{n_{<k},q_k\}q_{k}^{k-1}<i<n$, then $f$ is $k$-power-free.
\end{theorem}
\begin{proof}
Let us assume that
$n_{\geq k}=p_1^{n_1}\cdots p_r^{n_r}$, say, with $p_1,\dots ,p_r$ distinct prime numbers and $n_i\geq k$ for $i=1,\dots,r$. Then we have $\Delta_k(n)=p_1^{\lfloor \frac{n_1}{k}\rfloor }\cdots p_r^{\lfloor \frac{n_r}{k}\rfloor }$ and $q_k=\min\{ p_1,\dots ,p_r \}$. If one of 
$n_1,\dots ,n_r$ is at least $k+1$, say $n_1\geq k+1$, then $\frac{n}{\Delta_k(n)}\geq n_{<k}\cdot p_1^{n_1-\lfloor \frac{n_1}{k}\rfloor }\geq n_{<k}\cdot p_{1}^k\geq n_{<k}\cdot q_{k}^k$. If $r\geq 2$ and all the $n_i$'s are equal to $k$, then $\frac{n}{\Delta_k(n)}= n_{<k}\cdot p_1^{k-1}\cdots p_{r}^{k-1}>n_{<k}\cdot q_{k}^{2k-2}\geq n_{<k}\cdot q_{k}^k$. We are left with the case that $r=1$ and $n_1=k$, so $q_k=p_1$ and $n_{\geq k}=q_k^k$. In this case we have $\Delta_k(n)=q_k$ and $\frac{n}{\Delta_k(n)}=n_{<k}\cdot q_k^{k-1}$. Thus, we conclude that
\begin{equation}\label{n/Deltak}
\frac{n}{\Delta_k(n)}\geq n_{<k}\cdot q_k^{k-1}.
\end{equation}
Let $f(s)=\frac{a_1}{1^s}+\cdots +\frac{a_n}{n^s}$ and assume towards a contradiction that $f(s)=g(s)^kh(s)$ for some Dirichlet polynomials $g$ and $h$, say $g(s)=\frac{b_1}{1^s}+\cdots +\frac{b_t}{t^s}$ and $h(s)=\frac{c_1}{1^s}+\cdots +\frac{c_{\ell}}{\ell^s}$, with $t>1$. We may then write $g^k(s)=\frac{d_1}{1^s}+\cdots +\frac{d_{t^k}}{t^{ks}}$ with coefficients $d_i=\sum_{j_1\cdots j_k=i}b_{j_1}\cdots b_{j_k}$ for $i=1,\dots ,t^k$. Let now $i_{max}=\max\{i:i<n\ \text{and}\ a_i\neq 0\}$. By the multiplication rule, the coefficient $a_{i_{max}}$ is given by $a_{i_{max}}=\sum_{j\cdot j'=i_{max}}d_{j}c_{j'}$, and in this sum we must have $d_{j}c_{j'}\neq 0$ for at least one pair $(j,j')$ with $j\cdot j'=i_{max}$, as $a_{i_{max}}\neq 0$. Consider such a pair $(j,j')$. Since $t^k\cdot \ell=n$ and $i_{max}<n$, at least one of the inequalities $j\leq t^k$ and $j'\leq \ell$ must be a strict one (note that if $n=p^k$ for some prime $p$, then we must have $t=p=q_k$ and $\ell =1$, so for such integers $n$ the inequality $j'<\ell $ cannot hold). 

Assume first that $j<t^k$ and $j'\leq \ell$. Since $t^k\mid n$, we have $t\leq\Delta_k(n)$, and since $j$ is an index smaller than $t^k$ such that $d_j\neq 0$, we observe that $j\leq t^{k-1}(t-1)$, with equality if $b_{t-1}\neq 0$. Using also (\ref{n/Deltak}), we successively deduce that
\[
i_{max}=j\cdot j'\leq t^{k-1}(t-1)j'\leq (t^k-t^{k-1})\ell=n-\frac{n}{t}\leq n-\frac{n}{\Delta_k(n)}\leq n-n_{<k}\cdot q_k^{k-1}.
\]

Let us assume now that $j\leq t^k$ and $j'<\ell$. In this case we have
\[
i_{max}=j\cdot j'\leq t^k(\ell-1)=n-t^k\leq n-q_{k}^k,
\]
as $t>1$ is a divisor of $n$ with $t^k\mid n$, and hence $t\geq q_{k}$. By these inequalities we deduce that $i_{max}$ is at most $n-\min\{n_{<k},q_k\}q_{k}^{k-1}$. We conclude that if $f$ has a nonzero coefficient $a_i$ with $n-\min\{n_{<k},q_k\}q_{k}^{k-1}<i<n$, then $f$ must be $k$-power-free.
\end{proof}

Let us denote by $p_n$ the smallest prime factor of $n$. Observe that $p_n\leq q_k$, and that if $n_{<k}\neq 1$, then we also have $p_n\leq n_{<k}$, so we obtain as a corollary the following weaker, but simpler result.

\begin{corollary}\label{squareoftheleastpn}
Let $k\geq 2$ and $n$ be integers with $M(n)\geq k$, $p_n$ the smallest prime factor of $n$, and $f$ a primitive Dirichlet polynomial of degree $n$. Then $f$ is $k$-power-free in each of the following two cases:

i) \ $n_{<k}\neq 1$ and $f$ has a nonzero coefficient $a_{i}$ with $n-p_n^k<i<n$;

ii) $n_{<k}=1$ and $f$ has a nonzero coefficient $a_{i}$ with $n-p_n^{k-1}<i<n$.
\end{corollary}

\begin{theorem}\label{THsquareoftheleastpmpn}
Let $k\geq 2$ be an integer, $f=\frac{a_m}{m^s}+\cdots +\frac{a_n}{n^s}$ a primitive Dirichlet polynomial with $a_ma_n\neq 0$, $M(m)\geq k$ and $M(n)\geq k$. Let $q_k$ and $r_k$ be the smallest prime factors of $m$ and $n$ that have multiplicities $\geq k$, respectively. If $m>\frac{n}{r_k^k}$ and $f$ has a nonzero coefficient $a_i$ with $m<i<m+\min\{m_{<k},q_k\}\cdot q_k^{k-1}$, then $f$ is $k$-power-free.
\end{theorem}
\begin{proof}
We deduce as in the proof of Theorem \ref{THsquareoftheleastpn} that
\begin{equation}\label{m/deltanouCuk}
\frac{m}{\Delta_k(m)}\geq m_{<k}\cdot q_k^{k-1}.
\end{equation}

Let us assume again towards a contradiction that $f(s)=g(s)^kh(s)$ for some Dirichlet polynomials $g$ and $h$, say $g(s)=\frac{b_{u}}{u^s}+\cdots +\frac{b_t}{t^s}$ and $h(s)=\frac{c_v}{v^s}+\cdots +\frac{c_{\ell}}{\ell^s}$, with $t>1$. We may then write $g^k(s)=\frac{d_{u^k}}{u^{ks}}+\cdots +\frac{d_{t^k}}{t^{ks}}$ with coefficients $d_i=\sum_{j_1\cdots j_k=i}b_{j_1}\cdots b_{j_k}$ for $i=u^k,\dots ,t^k$. Let now $i_{min}=\min\{i:i>m\ \text{and}\ a_i\neq 0\}$. By the multiplication rule, the coefficient $a_{i_{min}}$ is given by $a_{i_{min}}=\sum_{j\cdot j'=i_{min}}d_{j}c_{j'}$, and in this sum we must have $d_{j}c_{j'}\neq 0$ for at least one pair $(j,j')$ with $j\cdot j'=i_{min}$, as $a_{i_{min}}\neq 0$. Let $(j,j')$ be such a pair. Since $u^k\cdot v=m$ and $i_{min}>m$, at least one of the inequalities $j\geq u^k$ and $j'\geq v$ must be a strict one (note that if $m=p^k$ for some prime $p$, then we must have $u=p=q_k$ and $v =1$, so for such $m$ the inequality $j'<v$ cannot hold).

Assume first that $j>u^k$ and $j'\geq v$. Since $u^k\mid m$, we have $u\leq\Delta_k(m)$, and since $j$ is an index greater than $u^k$ such that $d_j\neq 0$, we observe that $j\geq u^{k-1}(u+1)$, with equality if $b_{u+1}\neq 0$. In view of (\ref{m/deltanouCuk}), we deduce that
\[
i_{min}=j\cdot j'\geq u^{k-1}(u+1)j'\geq (u^k+u^{k-1})v=m+\frac{m}{u}\geq m+\frac{m}{\Delta_k(m)}\geq m+m_{<k}\cdot q_k^{k-1}.
\]

Assume now that $j\geq u^k$ and $j'>v$. In this case we have
\[
i_{min}=j\cdot j'\geq u^k(v+1)=m+u^k\geq m+q_{k}^k,
\]
provided we show that $u$ is a divisor of $m$ greater than $1$, and hence $u\geq q_{k}$, as $u^k\mid m$. To see this, assume to the contrary that $u=1$. This will force $v$ to be equal to $m$, implying that $\ell \geq m$, which leads to $n=t^k\ell\geq t^km\geq r_k^km$ (the inequality $t\geq r_k$ holds due to the fact that $t^k\mid n$), thus contradicting our assumption that $m>\frac{n}{r_k^k}$, as needed.

Since in both cases $i_{min}$ is at least $m+\min\{m_{<k},q_k\}\cdot q_k^{k-1}$, we conclude that if $f$ has a nonzero coefficient $a_i$ with $m<i<m+\min\{m_{<k},q_k\}\cdot q_k^{k-1}$, then $f$ must be $k$-power-free.
\end{proof}

If we denote by $p_m$ and $p_{n}$ the smallest prime factors of $m$ and $n$, respectively, we have $p_m\leq q_k$, and that if $m_{<k}\neq 1$, then we also have $p_m\leq m_{<k}$. Besides, $p_n\leq r_k$, so we have as a corollary the following simpler $k$-power-freeness criterion.

\begin{corollary}\label{squareoftheleastpmpn}
Let $k\geq 2$ be an integer and $f=\frac{a_m}{m^s}+\cdots +\frac{a_n}{n^s}$ a primitive Dirichlet polynomial with $a_ma_n\neq 0$, $M(m)\geq k$ and $M(n)\geq k$. Let $p_m$ and $p_{n}$ be the smallest prime factors of $m$ and $n$, respectively. If $m>\frac{n}{p_n^k}$, then $f$ is $k$-power-free in each of the following two cases:

i) \ $m_{<k}\neq 1$ and $f$ has a nonzero coefficient $a_{i}$ with $m<i<m+p_{m}^k$;

ii) $m_{<k}=1$ and $f$ has a nonzero coefficient $a_{i}$ with $m<i<m+p_{m}^{k-1}$.
\end{corollary}

In particular, from Corollary \ref{squareoftheleastpn} and Corollary \ref{squareoftheleastpmpn} we obtain for $k=2$ the following square-free criterion.
\begin{corollary}\label{corocoeficientszero}
Let $m<n$ be non square-free integers, each one having at least one prime factor with multiplicity $1$, and let $f(s)=\frac{a_{m}}{m^{s}}+\cdots +\frac{a_{n}}{n^{s}}$ be a Dirichlet polynomial with $a_{m}a_{n}\neq 0$. If one of the coefficients $a_{n-1}$, $a_{n-2}$ or $a_{n-3}$ is nonzero, or, if $m>n/4$ and one of the coefficients $a_{m+1}$, $a_{m+2}$ or $a_{m+3}$ is nonzero, then $f$ is square-free. 
\end{corollary}

From Theorem \ref{THsquareoftheleastpn} and Theorem \ref{THsquareoftheleastpmpn} one may obtain upper bounds for the multiplicities of the irreducible factors of $f$ that depend only on information on the prime factorizations of $m$ and $n$, and on $Supp(f)$, regardless of the values of the nonzero coefficients of $f$. These bounds are given in full generality in Corollary \ref{BoundMaxMult1} and Corollary \ref{BoundMaxMult2}, and in weaker, but more tractable form in Corollary \ref{BoundMaxMult1weak} and Corollary \ref{BoundMaxMult2weak}.
\begin{corollary}\label{BoundMaxMult1}
Let $n$ be a positive integer with $M(n)=M\geq 2$. For any integer $k$ with $2\leq k\leq M$ denote by $q_k$ the smallest prime factor of $n$ with multiplicity $\geq k$. Let $f(s)=\frac{a_1}{1^s}+\cdots +\frac{a_n}{n^s}$ be a primitive Dirichlet polynomial of degree $n$, and let $i_{max}=\max\{i:i<n\ {\rm and}\ a_i\neq 0\}$. If $i_{max}>n-\min\{n_{<M},q_M\}\cdot q_{M}^{M-1}$, then
\begin{equation}\label{primamajorare}
M(f)<\min\{k:\ 2\leq k\leq M\  {\rm and}\  i_{max}>n-\min\{n_{<k},q_k\}\cdot q_k^{k-1}\}.
\end{equation}
\end{corollary}
\begin{proof}
For $k_1<k_2$ we have $q_{k_1}\leq q_{k_2}$ and $n_{<k_1}\leq n_{<k_2}$, so 
\[
\min\{n_{<k_1},q_{k_1}\}\cdot q_{k_1}^{k_1-1}<\min\{n_{<k_2},q_{k_2}\}\cdot q_{k_2}^{k_2-1},
\]
which shows that the sequence $n-\min\{n_{<k},q_k\}\cdot q_k^{k-1}$ is a strictly decreasing one. Thus, if $i_{max}>n-\min\{n_{<M},q_M\}\cdot q_{M}^{M-1}$, there exists a smallest integer $k\in \{2,3,\dots ,M\}$ such that $n-\min\{n_{<k},q_k\}\cdot q_k^{k-1}<i_{max}$, say $k_{min}$. By Theorem \ref{THsquareoftheleastpn} we deduce that $f$ is $k_{min}$-power-free, so $M(f)$ satisfies (\ref{primamajorare}). 
\end{proof}
If we replace in (\ref{primamajorare}) the term $\min\{n_{<k},q_k\}\cdot q_k^{k-1}$ by $p_n^{k-1}$, with $p_n$ the smallest prime factor of $n$, we obtain a weaker, but more tractable bound for $M(f)$,  as in the following result.

\begin{corollary}\label{BoundMaxMult1weak}
Let $n$ be a positive integer with $M(n)=M\geq 2$, and let $p_n$ be the smallest prime factor of $n$. Let $f(s)=\frac{a_1}{1^s}+\cdots +\frac{a_n}{n^s}$ be a primitive Dirichlet polynomial of degree $n$, and let $i_{max}=\max\{i:i<n\ {\rm and}\ a_i\neq 0\}$. If $i_{max}>n-p_n^{M-1}$, then 
\begin{equation}\label{formulapentrukmin}
M(f)<\begin{cases}
             1+\lceil\log_{p_n}(n-i_{max})\rceil, & \text{if $i_{max}\not\in\{n-p_n,n-p_n^2,\dots ,n-p_n^{M-1}\}$},\\
             2+\log_{p_n}(n-i_{max}), & \text{otherwise}.
           \end{cases}
\end{equation}
\end{corollary}
\begin{proof}
If $i_{max}>n-p_n^{M-1}$, there exists a smallest $k\in \{2,3,\dots ,M\}$ such that $n-p_n^{k-1}<i_{max}$, say $k_{min}$, displayed in curly brace in (\ref{formulapentrukmin}). By Corollary \ref{BoundMaxMult1} we conclude that $M(f)<k_{min}$.
\end{proof}

\begin{corollary}\label{BoundMaxMult2}
Let $m$ and $n$ be positive integers with $m<n$ and such that $M(n)\geq M(m)=M\geq 2$, and let $f=\frac{a_m}{m^s}+\cdots +\frac{a_n}{n^s}$ be a primitive Dirichlet polynomial with $a_ma_n\neq 0$. For any integer $k\in\{2,3,\dots ,M\}$ let $q_k$ and $r_k$ be the smallest prime factors of $m$ and $n$ that have multiplicities $\geq k$, respectively. Let $i_{min}=\min\{i:i>m\ \text{and}\ a_i\neq 0\}$.
If $m\geq \frac{n}{r_M^M}$ and $i_{min}<m+\min\{m_{<M},q_M\}\cdot q_{M}^{M-1}$, then
\begin{equation}\label{adouamajorare}
M(f)<\min\{k:\ 2\leq k\leq M,\  m>\frac{n}{r_k^k}\ {\rm and}\  i_{min}<m+\min\{m_{<k},q_k\}\cdot q_k^{k-1}\}.
\end{equation}
\end{corollary}
\begin{proof}
For $k_1<k_2$ we have $r_{k_1}\leq r_{k_2}$, $q_{k_1}\leq q_{k_2}$ and $m_{<k_1}\leq m_{<k_2}$, so $\frac{n}{r_{k_1}^{k_1}}\geq \frac{n}{r_{k_2}^{k_2}}$ and
\[
\min\{m_{<k_1},q_{k_1}\}\cdot q_{k_1}^{k_1-1}<\min\{m_{<k_2},q_{k_2}\}\cdot q_{k_2}^{k_2-1},
\]
which shows that the sequence $\frac{n}{r_k^k}$ is a decreasing one, while the sequence $m+\min\{m_{<k},q_k\}\cdot q_k^{k-1}$ is a strictly increasing one. Thus, if $m>\frac{n}{r_M^M}$ and $i_{min}<m+\min\{m_{<M},q_M\}\cdot q_{M}^{M-1}$, there exists a smallest integer $k\in \{2,3,\dots ,M\}$ such that $m>\frac{n}{r_k^k}$ and $i_{min}<m+\min\{m_{<k},q_k\}\cdot q_k^{k-1}$, say $k_{min}$. By Theorem \ref{THsquareoftheleastpmpn} we deduce that $f$ is $k_{min}$-power-free, so $M(f)$ satisfies (\ref{adouamajorare}). 
\end{proof}
If we replace in (\ref{adouamajorare}) the term $\min\{m_{<k},q_k\}\cdot q_k^{k-1}$ by $p_m^{k-1}$, and $r_k$ by $p_n$, with $p_m$ and $p_n$ the smallest prime factors of $m$ and $n$, respectively, we obtain the following weaker, but more tractable bound for $M(f)$.
\begin{corollary}\label{BoundMaxMult2weak}
Let $m$ and $n$ be integers with $2\leq m<n$ and $M(n)\geq M(m)=M\geq 2$, and let $p_m$ and $p_n$ be the smallest prime factors of $m$ and $n$, respectively. Let $f=\frac{a_m}{m^s}+\cdots +\frac{a_n}{n^s}$ be a primitive Dirichlet polynomial with $a_ma_n\neq 0$, and let $i_{min}=\min\{i:i>m\ \text{and}\ a_i\neq 0\}$.
If $m\geq \frac{n}{p_n^M}$ and $i_{min}<m+p_m^{M-1}$, then $M(f)<k_{min}$ with
\begin{equation}\label{adouamajorareweak}
k_{min}=\min\left\{k:\ 2\leq k\leq M,\ k>\max \{\log_{p_n}(\frac{n}{m}),1+\log _{p_m}(i_{min}-m) \}\right\}.
\end{equation}
\end{corollary}
\begin{proof}
If $m\geq \frac{n}{p_m^M}$ and $i_{min}<m+p_m^{M-1}$, there exists a smallest $k\in \{2,3,\dots ,M\}$ such that $m\geq \frac{n}{p_m^k}$ and $i_{min}<m+p_m^{k-1}$, say $k_{min}$, given by (\ref{adouamajorareweak}). By Corollary \ref{BoundMaxMult2} we conclude that $M(f)<k_{min}$.
\end{proof}

In \cite{Alkan} we proved the following square-free criterion for polynomials over unique factorization domains, that uses no derivatives.
\medskip

{\bf Proposition}\ {\em Let $R$ be a unique factorization domain and $f\in R[X]$ a primitive
polynomial with $\deg (f)=n\geq 2$. Fix an arbitrarily chosen
integer $m\geq 2$. Then the following are equivalent: 

i) The polynomial $f$ is square-free; 

ii) For every $g\in R[X]$ with $\deg (g)=n-1$, $g^{m}$ is not divisible by $f$.}
\medskip

For unique factorization domains $R$ that include the real numbers (in particular logarithms of positive rationals), one may define the derivative and the higher derivatives with respect to $s$ of a Dirichlet polynomial $f\in DP(R)[s]$, and use them to study its repeated irreducible factors.
For Dirichlet polynomials over arbitrary unique factorization domains that do not necessarily include the real numbers, we will devise analogous square-free conditions that require no derivatives.

\begin{theorem}\label{squarefreecrit}
Let $R$ be a unique factorization domain and $f\in DP(R)[s]$ a primitive Dirichlet polynomial of degree $n\geq 2$, with $n$ not square-free. Fix an arbitrarily chosen integer $m\geq 2$. Then the following statements are equivalent: 

i) The Dirichlet polynomial $f$ is square-free; 

ii) For every prime factor $p$ of $n$ with multiplicity at least $2$, and every Dirichlet polynomial $g\in DP(R)[s]$ of degree at most $\frac{n}{p}$, $g^{m}$ is not divisible by $f$;

iii) For every prime factor $p$ of $n$ with multiplicity at least $2$, and every Dirichlet polynomial $g\in DP(R)[s]$ of degree $\frac{n}{p}$, $g^{m}$ is not divisible by $f$.
\end{theorem}
\begin{proof}\ $i)\Rightarrow ii)$ Let $f=f_{1}\cdots f_{k}$, with $f_{1},\ldots ,f_{k}\in DP(R)[s]$ irreducible and pairwise nonassociated in divisibility. Let $p$ be a prime factor of $n$ and assume there exists a Dirichlet polynomial $g\in DP(R)[s]$ with $\deg g\leq\frac{n}{p}$, such that $f$ divides $g^{m}$. Then each one of the $f_{i}$'s will divide $g$, and hence $f$ will divide $g$, which is impossible since $\deg g<n$.

$ii)\Rightarrow iii)$ The argument is trivial.

$iii)\Rightarrow i)$ This implication is a particular case of the following more general result on the multiplicities of the irreducible factors of Dirichlet polynomials.
\end{proof}
\begin{theorem}\label{THBoundMultiplicities}
Let $k\geq 2$ be an integer, $R$ a unique factorization domain and $f\in DP(R)[s]$ a primitive Dirichlet polynomial of degree $n$. Let us fix an arbitrarily chosen integer $m\geq k$. 

i) If for every prime factor $p$ of $n$ with $p^k\mid n$, and for every Dirichlet polynomial $g\in DP(R)[s]$ of degree $\frac{n}{p^{k-1}}$, $g^{m}$ is not divisible by $f$, then $f$ is $k$-power-free;

ii) If $f$ is $k$-power-free, then for every prime factor $p$ of $n$ with $p^k\mid n$, and for every Dirichlet polynomial $g\in DP(R)[s]$ of degree $\frac{n}{p^{\nu_p(n)-1}}$, $g^{m}$ is not divisible by $f$.
\end{theorem}
\begin{proof}\ i) The conclusion is trivial if all the prime factors of $n$ have multiplicities at most $k-1$, so we may assume that $n$ has at least one prime factor with multiplicity $\geq k$. Now let us assume to the contrary that $f=c\cdot f_{1}^{n_{1}}\cdots f_{r}^{n_{r}}$ with $c\in R\setminus \{0\}$ and $f_{1},\ldots ,f_{r}\in DP(R)[s]$ irreducible and pairwise nonassociated in divisibility, with at least one of the multiplicities $\geq k$, say $n_{1}\geq k$. Thus, the multiplicity of $f_{1}$ in the canonical decomposition of
$(\frac{f}{f_{1}^{k-1}})^{m}$ is $m(n_{1}-k+1)$, which is greater than or equal to $n_{1}$, as
$m\geq k$ and $k\geq \frac{n_{1}}{n_1-k+1}$. In particular, this shows that $f$ will divide $(\frac{f}{f_{1}^{k-1}})^{m}$. Let $p$ be a prime factor of $\deg f_1$, and observe that $p$ has multiplicity at least $k$ in the prime factorization of $n$, since $n_1\geq k$ and $(\deg f_1)^{n_1}\mid n$. Let now $q=(\frac{\deg f_1}{p})^{k-1}$. 
In case $q=1$ we take $g=\frac{f}{f_{1}^{k-1}}$, while if $q>1$ we multiply $\frac{f}{f_{1}^{k-1}}$ by a Dirichlet polynomial $h\in DP(R)[s]$ with $\deg h=q$, and take instead $g=\frac{fh}{f_{1}^{k-1}}$. In both situations we conclude that $g^{m}$ is divisible by $f$ and $\deg g=\frac{n}{p^{k-1}}$ with $p$ a prime factor of $n$ with $p^k\mid n$, a contradiction.

ii) Conversely, assume that $f=c\cdot f_{1}^{n_{1}}\cdots f_{r}^{n_{r}}$ with $c\in R\setminus \{0\}$ and $f_{1},\ldots ,f_{r}\in DP(R)[s]$ irreducible and pairwise nonassociated in divisibility, and with all the multiplicities $n_i$ at most $k-1$, and that there is a prime $p$ with $p^k\mid n$ and $g\in DP(R)[s]$ of degree $\frac{n}{p^{\nu_p(n)-1}}$ such that $f\mid g^m$. Then $f_1\cdots f_r$ must divide $g$, and since $n_i\leq k-1$ for each $i$, $f$ must in fact divide $g^{k-1}$, which in turn forces $n$ to divide $(\frac{n}{p^{\nu_p(n)-1}})^{k-1}$. In particular, this divisibility forces $\nu_p(n)$ to satisfy $\nu _p(n)\leq (k-1)(\nu _p(n)-(\nu_p(n)-1))=k-1$, which cannot hold, since $p^k\mid n$, and this completes the proof.
\end{proof}

In particular, by taking $m=k=M(n)$ in Theorem \ref{THBoundMultiplicities}, we obtain necessary and sufficient conditions for $M(f)$ to be smaller than $M(n)$.
\begin{corollary}\label{maxmultk}
Let $R$ be a unique factorization domain, $f\in DP(R)[s]$ a primitive Dirichlet polynomial of degree $n$, with $M(n)\geq 2$, and let us fix an arbitrarily chosen integer $m\geq M(n)$. Then $M(f)<M(n)$ if and only if for every prime factor $p$ of $n$ with multiplicity $M(n)$, and for every Dirichlet polynomial $g\in DP(R)[s]$ of degree $\frac{n}{p^{M(n)-1}}$, $g^{m}$ is not divisible by $f$.
\end{corollary}

The square-free conditions in Theorem \ref{squarefreecrit} simplify when the order of $f$ is a prime power, as follows.
\begin{corollary}\label{plamsquarefree}
Let $R$ be a unique factorization domain, and $f\in DP(R)[s]$ a primitive Dirichlet polynomial of degree $p^d$, with $p$ a prime number and $d\geq 2$. Fix an arbitrarily chosen integer $m\geq 2$. Then the following statements are equivalent: 

i) The Dirichlet polynomial $f$ is square-free; 

ii) For every $g\in DP(R)[s]$ of degree at most $p^{d-1}$, $g^{m}$ is not divisible by $f$;

iii) For every $g\in DP(R)[s]$ of degree $p^{d-1}$, $g^{m}$ is not divisible by $f$.
\end{corollary}

\begin{remark}\label{liniarizareGenerala}
When $R$ is a unique factorization domain with positive characteristic $\mathfrak{p}\geq k$, we may take the integer $m$ in Theorem \ref{THBoundMultiplicities} (and also in Theorem \ref{squarefreecrit} if $k=2$) to be equal to $\mathfrak{p}$, so that a condition like $f$ divides $g^{m}$ may be regarded as a homogeneous system of linear equations. Indeed, if $f(s)=\sum_{i=1}^{n}\frac{a_i}{i^s}$ and we assume that for a prime divisor $p$ of $n$ with $p^k\mid n$ and a Dirichlet polynomial $g$ of degree $\frac{n}{p^{k-1}}$, say $g(s)=\sum_{i=1}^{n/p^{k-1}}\frac{c_i}{i^s}$ we have \begin{equation}\label{fh=gmGenerala}
f(s)h(s)=g(s)^{\mathfrak{p}}
\end{equation}
for some Dirichlet polynomial $h$ of degree $t$, say $h(s)=\sum_{i=1}^{t}\frac{b_i}{i^s}$, we see that $t=\frac{n^{\mathfrak{p}-1}}{p^{(k-1)\mathfrak{p}}}$ and equality (\ref{fh=gmGenerala}) reads
\[
\sum_{i=1}^{n}\frac{a_i}{i^s}
\sum_{i=1}^{t}\frac{b_i}{i^s}=
\sum_{i=1}^{\frac{n}{p^{k-1}}}\frac{c_i^{\mathfrak{p}}}{i^{\mathfrak{p}s}}.  
\]
By equating the coefficients in this equality one obtains a homogeneous system of $(\frac{n}{p^{k-1}})^{\mathfrak{p}}$ linear equations in the unknowns $b_1,\dots ,b_t,c_1^{\mathfrak{p}},\dots ,c_{n/p^{k-1}}^{\mathfrak{p}}$, of matrix form
\begin{equation}\label{sistemGenerala}
\mathcal{A}_{\mathfrak{p},n,p,k}\cdot X=\mathbf{0},
\end{equation}
with $\mathcal{A}_{\mathfrak{p},n,p,k}$ the coefficient matrix depending on $\mathfrak{p},n,p$ and $k$, having $(\frac{n}{p^{k-1}})^{\mathfrak{p}}$ rows and $t+\frac{n}{p^{k-1}}$ columns, $X$ the column vector $(b_1,\dots ,b_t,c_1^{\mathfrak{p}},\dots ,c_{n/p^{k-1}}^{\mathfrak{p}})^T$, and $\mathbf{0}$ a zero column matrix.

One may easily check that the entries $\mathfrak{a}_{i,j}$ of $\mathcal{A}_{\mathfrak{p},n,p,k}$ are given by 
\begin{equation}\label{formulapentruaij}
\mathfrak{a}_{i,j}=\begin{cases}
             a_{\frac{i}{j}}, & \text{if $j\mid i$ and $j\in\{1,\dots , t\}$},\\
             0, & \text{if $j\nmid i$ and $j\in\{1,\dots ,t\}$},\\
            -1, & \text{if $i=(j-t)^{\mathfrak{p}}$ and $j\in\{t+1,\dots ,t+\frac{n}{p^{k-1}}\}$},\\
             0, & \text{if $i\neq (j-t)^{\mathfrak{p}}$ and $j\in\{t+1,\dots ,t+\frac{n}{p^{k-1}}\}$},
           \end{cases}
\end{equation}
with the convention that $a_{\mathfrak{i}}=0$ for $\mathfrak{i}>n$.
We notice that the system (\ref{sistemGenerala}) has more equations than unknowns, that is $(\frac{n}{p^{k-1}})^{\mathfrak{p}}>t+\frac{n}{p^{k-1}}$, which is equivalent to $n^{\mathfrak{p}-2}(n-1)>p^{(\mathfrak{p}-1)(k-1)}$. Indeed, since we may write $n=p^k\ell $ for some positive integer $\ell $, we need to check that $p^{\mathfrak{p}-k-1}\ell ^{\mathfrak{p}-2}(p^k\ell -1)>1$. Recall now that $\mathfrak{p}\geq k$. Last inequality obviously holds for $\mathfrak{p}\geq k+1$, while for $\mathfrak{p}=k$, it reduces to $\ell ^{k-2}(p^k\ell -1)>p$, which also holds, since $\ell ^{k-2}(p^k\ell -1)\geq p^2-1>p$, as $k\geq 2$, $\ell \geq 1$ and $p\geq 2$.
However, $\mathcal{A}_{\mathfrak{p},n,p,k}$ may have rows $r_i$ consisting only of zero entries, if for instance $i$ is not a $\mathfrak{p}$th power and $a_{\frac{i}{j}}=0$ for each divisor $j$ of $i$ with $1\leq j\leq t$. On the other hand, if $p^k$ is a proper divisor of $n$, that is $\ell>1$ (so that $t>1$ as well), some of the zero rows may simply occur due to some arithmetic constraints, regardless of the values of the coefficients of $f$. This is the case of the rows with index $i>n$ (implying that $a_i=0$ by our convention), with $i$ not a $\mathfrak{p}$th power, and not a multiple of any of the integers $2,3,\dots ,t$, and the removal of these rows does not affect the rank of $\mathcal{A}_{\mathfrak{p},n,p,k}$. Letting $\pi _t=\prod_{p\ \text{prime},\ p\leq t}p$, we see that these are the rows whose indices $i$ belong to the interval $(\max\{t,n\},(\frac{n}{p^{k-1}})^{\mathfrak{p}})$, are not $\mathfrak{p}$th powers, and are coprime to $\pi_t$. Let us denote by $\mathcal{S}$ the set of these indices, and by $\beta $ its cardinality. To count these indices, one may use the fact that given two positive integers $A$ and $B$, the number $\mathfrak{n}(A,B)$ of positive integers not greater than $A$ and which are coprime to $B$ is $\mathfrak{n}(A,B)=\sum_{d\mid B}\mu (d)\lfloor \frac{A}{d}\rfloor$, with $\mu$ the M\"obius function. Let $\mathcal{S}'$ be the set of integers in the interval $(\max\{t,n\},(\frac{n}{p^{k-1}})^{\mathfrak{p}})$ that are coprime to $\pi_t$, and $\gamma $ its cardinality. Then we have
\begin{equation}\label{GamaPrim}
\gamma =\mathfrak{n}((n/p^{k-1})^{\mathfrak{p}},\pi_t)-\mathfrak{n}(\max\{t,n\},\pi_t),
\end{equation}
as $t\mid (\frac{n}{p^{k-1}})^{\mathfrak{p}}$. We will prove now that $\mathfrak{p}\leq t$, hence $\mathfrak{p}\mid \pi_t$, implying that all the integers in $\mathcal{S}'$ are in fact coprime to $\mathfrak{p}$, so in particular none of them can be a $\mathfrak{p}$th power. To do this, we first notice that $t=p^{\mathfrak{p}-k}\ell^{\mathfrak{p}-1}$, and $t<n$ if and only if $p^{\mathfrak{p}-2k}<\ell^{2-\mathfrak{p}}$, which, since $\mathfrak{p}\geq 2$, is equivalent to $\mathfrak{p}<2k$. We distinguish the following three cases:

{\it Case 1:}\ \ $\mathfrak{p}=k$; \ in this case we have $\mathfrak{p}=k\leq 2^{k-1}\leq \ell^{\mathfrak{p}-1}=t$;

{\it Case 2:}\ \ $k+1\leq \mathfrak{p}<2k$; \ here we have $\mathfrak{p}<2k\leq 2\cdot 2^{k}\leq 2\cdot 2^{\mathfrak{p}-1}\leq 2\cdot \ell^{\mathfrak{p}-1}\leq p^{\mathfrak{p}-k}\ell^{\mathfrak{p}-1}=t$;

{\it Case 3:}\ \ $\mathfrak{p}>2k$; \ in this case we have $\mathfrak{p}<4\cdot 2^{\mathfrak{p}-1}\leq p^{k}\cdot 2^{\mathfrak{p}-1}\leq p^{k}\cdot \ell^{\mathfrak{p}-1}\leq p^{\mathfrak{p}-k}\ell^{\mathfrak{p}-1}=t$.

\noindent This shows us that $\mathcal{S}=\mathcal{S}'$, and hence $\beta =\gamma$.
We thus conclude that
\begin{equation}\label{InegalitateRankA}
{\rm rank}(\mathcal{A}_{\mathfrak{p},n,p,k})\leq \min\left\{\frac{n^{\mathfrak{p}-1}}{p^{(k-1)\mathfrak{p}}}+\frac{n}{p^{k-1}},\frac{n^{\mathfrak{p}}}{p^{(k-1)\mathfrak{p}}}-\gamma\right\},
\end{equation}
with $\gamma$ given by (\ref{GamaPrim}).

Note that if
\begin{equation}\label{rankGenerala}
{\rm rank}(\mathcal{A}_{\mathfrak{p},n,p,k})=\frac{n^{\mathfrak{p}-1}}{p^{(k-1)\mathfrak{p}}}+\frac{n}{p^{k-1}},
\end{equation}
then the system (\ref{sistemGenerala}) will have no nontrivial solutions.
Let us also note that the rightmost $\frac{n}{p^{k-1}}$ columns of $\mathcal{A}_{\mathfrak{p},n,p,k}$ contain a square diagonal submatrix of order $\frac{n}{p^{k-1}}$ whose diagonal entries are all equal to $-1$, corresponding to the rows of $\mathcal{A}_{\mathfrak{p},n,p,k}$ with indices $1,2^{\mathfrak{p}},\dots ,(\frac{n}{p^{k-1}})^{\mathfrak{p}}$. 

Let now $\mathcal{B}_{\mathfrak{p},n,p,k}$ be the matrix obtained from $\mathcal{A}_{\mathfrak{p},n,p,k}$ by removing its rightmost $\frac{n}{p^{k-1}}$ columns and its rows with indices $1,2^{\mathfrak{p}},\dots ,(\frac{n}{p^{k-1}})^{\mathfrak{p}}$. The matrix $\mathcal{B}_{\mathfrak{p},n,p,k}$ has $(\frac{n}{p^{k-1}})^{\mathfrak{p}}-\frac{n}{p^{k-1}}$ rows and $\frac{n^{\mathfrak{p}-1}}{p^{(k-1)\mathfrak{p}}}$ columns, and its entries $x_{i,j}$ can be described as follows. The set $S=\{1,2,\dots ,(\frac{n}{p^{k-1}})^{\mathfrak{p}}-\frac{n}{p^{k-1}}\}$ of indices of the rows of $\mathcal{B}_{\mathfrak{p},n,p,k}$ can be partitioned into $\frac{n}{p^{k-1}}-1$ disjoint subsets of consecutive integers
\[
S=\bigcup\limits _{\delta=1}^{\frac{n}{p^{k-1}}-1}S_{\delta}
\]
with $S_{\delta}=\{\delta^{\mathfrak{p}}-\delta+1,\delta^{\mathfrak{p}}-\delta+2,\dots ,(\delta+1)^{\mathfrak{p}}-(\delta+1) \}$ for $\delta=1,\dots ,\frac{n}{p^{k-1}}-1$. Thus, to any integer $i\in S$ we can associate a unique index $\delta_i\in\{1,\dots ,\frac{n}{p^{k-1}}-1\}$ such that $i\in S_{\delta_i}$, and an entry $x_{i,j}$ of the matrix $\mathcal{B}_{\mathfrak{p},n,p,k}$ is now easily seen to be
\[
x_{i,j}=\begin{cases}
             a_{\frac{i+\delta_i}{j}}, & \text{if $j\mid (i+\delta_i)$},\\
             0, & \text{otherwise},
           \end{cases}
\]
with the shift of $i$ with the corresponding $\delta_i$ appearing as a consequence of the removal of the rows of $\mathcal{A}_{\mathfrak{p},n,p,k}$ with indices $1,2^{\mathfrak{p}},\dots ,(\frac{n}{p^{k-1}})^{\mathfrak{p}}$. Here we use again the convention that $a_{\mathfrak{i}}=0$ for $\mathfrak{i}>n$.
Next, we observe that if instead of (\ref{rankGenerala}) we ask $\mathcal{B}_{\mathfrak{p},n,p,k}$ to have full rank, that is to satisfy
\begin{equation}\label{thesamerankGenerala}
{\rm rank}(\mathcal{B}_{\mathfrak{p},n,p,k})=\frac{n^{\mathfrak{p}-1}}{p^{(k-1)\mathfrak{p}}},
\end{equation}
then the system (\ref{sistemGenerala}) will have no nontrivial solutions.
We thus conclude that if (\ref{thesamerankGenerala}) holds for each prime factor $p$ of $n$ that has multiplicity at least $k$, then $f$ must be $k$-power-free. 

Note that if $f$ is as in Corollary \ref{plamsquarefree}, we need to test (\ref{thesamerankGenerala}) only for the prime $p$ in the statement, so we only need to check the equality ${\rm rank}(\mathcal{B}_{\mathfrak{p},n,p,2})=p^{(d-1)\mathfrak{p}-d}$. 
\end{remark}

By Theorem \ref{THBoundMultiplicities} and Remark \ref{liniarizareGenerala} one obtains the following $k$-power-free criterion.
\begin{theorem}\label{MatrixTHBoundMultiplicities}
Let $k\geq 2$ be an integer, $R$ a unique factorization domain with prime characteristic $\mathfrak{p}\geq k$ and $f\in DP(R)[s]$ a primitive Dirichlet polynomial of degree $n$. 
If for every prime factor $p$ of $n$ with $p^k\mid n$ we have
${\rm rank}(\mathcal{B}_{\mathfrak{p},n,p,k})=\frac{n^{\mathfrak{p}-1}}{p^{(k-1)\mathfrak{p}}}$, then $f$ is $k$-power-free.
\end{theorem}

We may prove now the following square-free criterion, that relies on Theorem \ref{squarefreecrit} and on the study of the rank of the matrix $\mathcal{B}_{\mathfrak{p},n,p,2}$.
\begin{theorem}\label{squarefreeinpositivecharacteristic} 
Let $R$ be a unique factorization domain with prime characteristic $\mathfrak{p}$, and $f\in DP(R)[s]$ a primitive Dirichlet polynomial of degree $n\geq 2$.

i) If ${\rm rank}(\mathcal{B}_{\mathfrak{p},n,p,2})=\frac{n^{\mathfrak{p}-1}}{p^{\mathfrak{p}}}$ for every prime factor $p$ of $n$ with $p^2\mid n$, then $f$ is square-free. 

ii) If the Frobenius map of $R$ is surjective, then $f$ is square-free if and only if the equality ${\rm rank}(\mathcal{B}_{\mathfrak{p},n,p,2})=\frac{n^{\mathfrak{p}-1}}{p^{\mathfrak{p}}}$ holds for every prime factor $p$ of $n$ with $p^2\mid n$.
\end{theorem}
\begin{proof}
i) We may obviously assume that $n$ is not square-free. If for every prime factor $p$ of $n$ with $p^2\mid n$ we have ${\rm rank}(\mathcal{B}_{\mathfrak{p},n,p,2})=\frac{n^{\mathfrak{p}-1}}{p^{\mathfrak{p}}}$, then all the corresponding systems of equations (\ref{sistemGenerala}) will have only trivial solutions, so no equality of the form (\ref{fh=gmGenerala}) with $g$ of degree $\frac{n}{p}$ will hold. By the equivalence i) $\Leftrightarrow$ iii) in Theorem \ref{squarefreecrit} we conclude that $f$ must be square-free.

ii) If we assume that the Frobenius map of $R$ is surjective, then we may replace $c_1^{\mathfrak{p}},\dots ,c_{n/p}^{\mathfrak{p}}$ by some arbitrary unknowns $d_1,\dots ,d_{n/p}$ and regard the system (\ref{sistemGenerala}) as a linear one. Therefore, if ${\rm rank}(\mathcal{B}_{\mathfrak{p},n,p,2})<\frac{n^{\mathfrak{p}-1}}{p^{\mathfrak{p}}}$, the system (\ref{sistemGenerala}) will have a nontrivial solution in $Q(R)$, the quotient field of $R$.  Multiplying by a common denominator if necessary, which is itself the $\mathfrak{p}$th power of an element in $R$, will provide a solution in $R$ having at least one nonzero $b_i$ or one nonzero $d_j$. We notice that in fact, such a nonzero solution must have at least one nonzero $b_i$ and at least one nonzero $d_j$ (and hence at least one nonzero $c_j$, as $c_j^{\mathfrak{p}}=d_j$). For such a nonzero solution we let $i_{max}=\max\{i:b_i\neq 0\}$ and $j_{max}=\max\{j:c_j\neq 0\}$ with  $i_{max}\leq t=\frac{n^{\mathfrak{p}-1}}{p^{\mathfrak{p}}}$, $j_{max}\leq \frac{n}{p}$, and $n\cdot i_{max}=j_{max}^{\mathfrak{p}}$, so equality (\ref{fh=gmGenerala}) holds for some $h(s)$ and $g(s)$ with $\deg h=i_{max}$ and $\deg g=j_{max}$. In particular, we found a Dirichlet polynomial $g\in DP(R)[s]$ of degree at most $\frac{n}{p}$ such that $g^{\mathfrak{p}}$ is divisible by $f$. By the equivalence i) $\Leftrightarrow$ ii) in Theorem \ref{squarefreecrit}, we conclude that $f$ cannot be square-free.
\end{proof}

\begin{corollary}\label{Corolarsquarefreeinpositivecharacteristic} 
Let $R$ be a unique factorization domain with prime characteristic $\mathfrak{p}$, and $f\in DP(R)[s]$ a primitive Dirichlet polynomial of degree $p^d$, with $p$ a prime number and $d\geq 2$. 

i) \ If ${\rm rank}(\mathcal{B}_{\mathfrak{p},n,p,2})=p^{(d-1)\mathfrak{p}-d}$, then $f$ is square-free. 

ii) If the Frobenius map of $R$ is surjective, then $f$ is square-free if and only if we have ${\rm rank}(\mathcal{B}_{\mathfrak{p},n,p,2})=p^{(d-1)\mathfrak{p}-d}$.
\end{corollary}

As in the case of polynomials, we can search the repeated irreducible factors of a Dirichlet polynomial with complex coefficients by studying its derivative. We will first study the common factors of two Dirichlet polynomials $f$ and $g$ by adapting some well-known techniques used for polynomials. In this respect, we have the following result.

\begin{proposition}\label{commonfactors}
Let $f(s)=\frac{a_1}{1^s}+\cdots +\frac{a_m}{m^s}$ and $g(s)=\frac{b_1}{1^s}+\cdots +\frac{b_n}{n^s}$ be two Dirichlet polynomials with coefficients in a unique factorization domain $R$, with $a_mb_n\neq 0$.  

i) If $f$ and $g$ have a common factor of degree $d$, then there exist two Dirichlet polynomials $u,v\in DP(R)[s]$ with $\deg u=\frac{n}{d}$, $\deg v=\frac{m}{d}$, and such that $f(s)u(s)+g(s)v(s)=0$. 

ii) Conversely, if $f(s)u(s)+g(s)v(s)=0$ for two Dirichlet polynomials $u,v\in DP(R)[s]$ with $\deg u=\frac{n}{d}$ and $\deg v=\frac{m}{d}$ for some divisor $d$ of $\gcd (m,n)$, then $\gcd(f,g)$ has degree a multiple of $d$, with $\deg \gcd(f,g)=d$ if $u$ and $v$ are relatively prime.
\end{proposition}
\begin{proof}
i) If $f$ and $g$ share a factor $h(s)$ of degree $d$, with $d\mid \gcd(m,n)$, then one may take $u(s)=\frac{g(s)}{h(s)}$ and $v(s)=-\frac{f(s)}{h(s)}$, and $u$ and $v$ have the desired properties.

ii) Assume now that $f(s)u(s)+g(s)v(s)=0$ for two Dirichlet polynomials $u,v\in DP(R)[s]$ with $\deg u=\frac{n}{d}$ and $\deg v=\frac{m}{d}$. Let $h(s)=\gcd(f(s),g(s))$, and let $\deg h=k$, say, with $1\leq k\leq \gcd(m,n)$. Writing $f=hf_1$ and $g=hg_1$ with $f_1$ and $g_1$ relatively prime yields $f_1(s)u(s)=-g_1(s)v(s)$, which further implies the divisibilities $f_1(s)\mid v(s)$ and $g_1(s)\mid u(s)$. In particular, we must have $\frac{m}{k}\mid \frac{m}{d}$ and $\frac{n}{k}\mid \frac{n}{d}$, so $k$ must be a multiple of $d$. 

Moreover, if $u$ and $v$ are relatively prime, the equality $f_1(s)u(s)=-g_1(s)v(s)$ forces $f_1$ and $v$ to be associated in divisibility, and also $f_2$ and $u$ to be associated in divisibility, so we must have the equalities $\deg f_1=\deg v$ and $\deg g_1=\deg u$. Thus $k$ must be equal to $d$, and in this case $f$ and $g$ have $h$ as a common factor of degree $d$. 
\end{proof}
Let now $u(s)=\frac{u_1}{1^s}+\cdots +\frac{u_{n/d}}{(n/d)^s}$ and $v(s)=\frac{v_1}{1^s}+\cdots +\frac{v_{m/d}}{(m/d)^s}$ in $DP(R)[s]$ be such that 
\begin{equation}\label{fu=gv}
f(s)u(s)+g(s)v(s)=0,
\end{equation}
with $f,g$ as in Proposition \ref{commonfactors}, and $d$ a divisor of $\gcd(m,n)$. By equating the coefficients, we may regard equality (\ref{fu=gv}) as a homogeneous system of $\frac{mn}{d}$ linear equations in the $\frac{n}{d}+\frac{m}{d}$ unknowns $u_1,\dots ,u_{n/d},v_1,\dots ,v_{m/d}$, written in matrix form as
\begin{equation}\label{matrixformresultant}
\mathcal{R}_d\cdot X=\mathbf{0},
\end{equation}
with $\mathcal{R}_d$ the coefficient matrix having $\frac{mn}{d}$ rows and $\frac{m+n}{d}$ columns, $X=(u_1,\dots ,u_{\frac{n}{d}},v_1,\dots ,v_{\frac{m}{d}})^T$, and $\mathbf{0}$ a zero column matrix. We notice that except for the case that $m=n=2$, the matrix $\mathcal{R}_d$ has more rows than columns. One may easily check that the entries $r_{i,j}$ of $\mathcal{R}_d$ are given by 
\begin{equation}\label{formulapentrurij}
r_{i,j}=\begin{cases}
             a_{\frac{i}{j}}, & \text{if $j\mid i$ and $j\in\{1,\dots , \frac{n}{d}\}$},\\
             0, & \text{if $j\nmid i$ and $j\in\{1,\dots ,\frac{n}{d}\}$},\\
             b_{\frac{i}{j-\frac{n}{d}}}, & \text{if $j-\frac{n}{d}\mid i$ and $j\in\{\frac{n}{d}+1,\dots ,\frac{n}{d}+\frac{m}{d}\}$},\\
             0, & \text{if $j-\frac{n}{d}\nmid i$ and $j\in\{\frac{n}{d}+1,\dots , \frac{n}{d}+\frac{m}{d}\}$},
           \end{cases}
\end{equation}
where by convention, we put $a_{\mathfrak{i}}=0$ for $\mathfrak{i}>m$ and $b_{\mathfrak{i}}=0$ for $\mathfrak{i}>n$. Here too, as in the case of the matrices $\mathcal{A}_{\mathfrak{p},n,p,k}$, there are rows of $\mathcal{R}_d$ that have only zero entries, no matter what coefficients we choose for $f$ and $g$. These are the rows whose indices $i$ exceed $\max\{m,n\}$ (so that $a_i=b_i=0$, by convention) and are not a multiple of any of the integers in $\{ 2,3,\dots,\frac{\max\{m,n\}}{d}\}$, that is the rows with indices $i$ in the closed interval $[\max\{m,n\}+1,\frac{mn}{d}]$, with $i$ not divisible by any prime number less than or equal to $\frac{\max\{m,n\}}{d}$. An account of these indices may be done by the method described in Remark \ref{liniarizareGenerala}.

We may now state the following criterion for two Dirichlet polynomials to be relatively prime.
\begin{theorem}\label{nocommonfactors} Let $f$ and $g$ be two Dirichlet polynomials with coefficients in a unique factorization domain $R$, of degrees $m$ and $n$, respectively, with $\gcd(m,n)>1$, and let $d$ be a divisor of $\gcd (m,n)$. Then $f$ and $g$ have no common factors of degree $\geq d$ if and only if ${\rm rank} (\mathcal{R}_d)=\frac{m+n}{d}$. In particular, $f$ and $g$ are relatively prime if and only if ${\rm rank} (\mathcal{R}_1)=m+n$.
\end{theorem}
\begin{proof}
If ${\rm rank} (\mathcal{R}_d)=\frac{m+n}{d}$ for some divisor $d$ of $\gcd (m,n)$, then the system (\ref{matrixformresultant}) has only the trivial solution, so any equality of the form (\ref{fu=gv}) with $1\leq \deg u\leq \frac{n}{d}$ and $1\leq \deg v\leq \frac{m}{d}$ must fail. In particular, will fail any equality of the form (\ref{fu=gv}) with $\deg u=\frac{n}{D}$ and $\deg v= \frac{m}{D}$ with $D$ a divisor of $n$ greater than or equal to $d$. By Proposition \ref{commonfactors} i) we deduce that $f$ and $g$ can have no common factors of degree $\geq d$. 

Assume now that ${\rm rank} (\mathcal{R}_d)<\frac{m+n}{d}$. Then the system (\ref{matrixformresultant}) will have a nonzero solution in the quotient field of $R$, which after multiplication by a common denominator will give a nonzero solution in $R$, that is a pair of nonzero $u$ and $v$ with $1\leq \deg u\leq \frac{n}{d}$ and $1\leq \deg v\leq \frac{m}{d}$ such that $f(s)u(s)=-g(s)v(s)$.
Let again $h(s)=\gcd(f(s),g(s))$, and assume that $\deg h=k$, say, with $1\leq k\leq \gcd(m,n)$. Writing again $f=hf_1$ and $g=hg_1$ with $f_1$ and $g_1$ relatively prime yields $f_1(s)u(s)=-g_1(s)v(s)$, so $f_1$ must divide $v$. In particular, we must have $\frac{m}{k}\mid \deg v$, and since $\deg v\leq\frac{m}{d}$, we conclude that $k\geq d$, so $f$ and $g$ have $h$ as a common factor of degree $\geq d$.

In particular, $f$ and $g$ are relatively prime if and only if ${\rm rank} (\mathcal{R}_1)=m+n$. This also shows that if $\mathcal{R}_1$ has full rank, then for each divisor $d$ of $\gcd (m,n)$, $\mathcal{R}_d$ must also have full rank.
\end{proof}

The matrix $\mathcal{R}_1$ may be regarded as an analogue of the Sylvester matrix associated to two univariate polynomials, and has the following shape, with $n$ columns with $a_i$'s, followed by $m$ columns with $b_j$'s, the $mn$ rows being filled out with zeros. Its entries are given by (\ref{formulapentrurij}) for $d=1$. Here we will only display $\mathcal{R}_1$ for the case that $m<n$. 

{\tiny
\[
\mathcal{R}_1=\left[
\begin{array}{ccccccccccccccccccccc}
a_1 &  &  &  &  & & & & & & & b_1 &  &  &  &  & & & & \\
a_2 & a_1  &  &  & & & & & & & & b_2 & b_1 &  &  &  & & & & \\
a_3 & 0 & a_1  &  & & & & & & & & b_3 & 0 & b_1  &  &  & & & \\
a_4 & a_2 & 0 & a_1 & & & & & & & & b_4 & b_2 & 0 & b_1 & & & & \\
a_5 & 0 & 0 & 0 & a_1 & & & & & & & b_5 & 0 & 0 & 0 & b_1 & & & \\
a_6 & a_3 & a_2 & 0 & 0 & a_1 & & & & & & b_6 & b_3 & b_2 & 0 & 0 & b_1 & \\
a_7 & 0 & 0 & 0 & 0 & 0 & a_1 & & & & & b_7 & 0 & 0 & 0 & 0 & 0 &  b_1 & \\
\vdots  & \vdots &  \vdots & \vdots & \vdots & \vdots & \vdots & \ddots & &  & & \vdots & \vdots  & \vdots  & \vdots & \vdots  & \vdots & \vdots  & \ddots & & \\
a_m  & \cdot & \cdot & \cdot & \cdot & \cdot & \cdot & \cdot & a_1 & & & b_m & \cdot &\cdot  &  \cdot & \cdot & \cdot &\cdot & \cdot & b_1 \\
   & \vdots &  \vdots  & \vdots & \vdots  & \vdots  & \vdots & \vdots  & \vdots  & \ddots & & \vdots & \vdots  & \vdots   & \vdots & \vdots  & \vdots & \vdots  & \vdots  & \vdots & \\
      &  & \cdot  & \cdot & \cdot& \cdot & \cdot & \cdot & \cdot  & \cdot & a_1 & \cdot & \cdot & \cdot & \cdot & \cdot & \cdot & \cdot & \cdot &  \cdot & \\
      &  & \vdots  & \vdots & \vdots & \vdots & \vdots & \vdots & \vdots & \vdots & \vdots & \vdots & \vdots  & \vdots & \vdots & \vdots & \vdots & \vdots & \vdots & \vdots & \\
 &  &  & \cdot & \cdot & \cdot & \cdot & \cdot & \cdot & \cdot & \cdot & b_n & \cdot & \cdot & \cdot & \cdot & \cdot & \cdot & \cdot & \cdot & \\
   & &  & a_m & \vdots & \vdots & \vdots & \vdots & \vdots &  \vdots & \vdots &  & \vdots & \vdots & \vdots & \vdots & \vdots & \vdots & \vdots & \vdots & \\
  &  &  &  &  & \vdots & \vdots & \vdots & \vdots & \vdots & \vdots &  & b_n & \vdots & \vdots & \vdots & \vdots & \vdots & \vdots & \vdots & \\
   &  &   &  &  & a_m & \vdots & \vdots & \vdots & \vdots & \vdots &  &  &  & \vdots & \vdots & \vdots & \vdots & \vdots & \vdots & \\
   &  &   &  &  &  &  &  &  &  & \vdots &  &  &  & b_n & & &  &  & \vdots & \\
   &  &   &  &  &  &  &  & &  & a_m & &  &  &  &  & &  &  & b_n & 
\end{array}
\hspace{-3.5mm}
\right] ,
\]
}

The lowest nonzero entry in each one of the first $n$ columns is $a_m$, and the lowest nonzero entry in each one of the following $m$ columns is $b_n$. We also note that in a column $c_j$ with $j\in \{1,\dots,n\}$ the number of zeros between $a_i$ and $a_{i+1}$ is equal to $j-1$, while in a column $c_j$ with $j\in \{n+1,\dots,n+m\}$ the number of zeros between $b_i$ and $b_{i+1}$ is equal to $j-n-1$.

In comparison to the case of univariate polynomials, where we have to compute the resultant of $f(X)$ and $g(X)$ to see if they are relatively prime, in the case of Dirichlet polynomials we have to compute the rank of $\mathcal{R}_1$, which seems to be a more involved task. 

For a Dirichlet polynomial $f$ of order $m$ with complex coefficients $a_i$, we may let $g$ to be $f^{(k)}$, one of the higher derivatives of $f$, and denote the corresponding matrices $\mathcal{R}_d$ (with $d\mid m$) by $\mathcal{D}_d^{(k)}$, with entries $d_{i,j}^{(k)}$ given by 
\begin{equation}\label{higherderivativescoefficients}
d_{i,j}^{(k)}=\begin{cases}
             a_{\frac{i}{j}}, & \text{if $j\mid i$ and $j\in\{1,\dots , \frac{m}{d}\}$},\\
             0, & \text{if $j\nmid i$ and $j\in\{1,\dots ,\frac{m}{d}\}$},\\
             (-1)^{k}a_{\frac{i}{j-\frac{m}{d}}}\log^{k}(\frac{i}{j-\frac{m}{d}}), & \text{if $j-\frac{m}{d}\mid i$ and $j\in\{\frac{m}{d}+1,\dots ,\frac{2m}{d}\}$},\\
             0, & \text{if $j-\frac{m}{d}\nmid i$ and $j\in\{\frac{m}{d}+1,\dots , \frac{2m}{d}\}$},
           \end{cases}
\end{equation}
with the convention that $a_{\mathfrak{i}}=0$ for $\mathfrak{i}>m$.
In comparison to the case of a univariate polynomial, where we need to compute its discriminant to see if it has repeated factors, for a Dirichlet polynomial we need to compute the rank of its associated matrix $\mathcal{D}_1^{(1)}$. In this respect, as an immediate consequence of Theorem \ref{nocommonfactors} we have the following square-free criterion.

\begin{corollary}\label{nosquaredfactors} Let $f(s)=\frac{a_1}{1^s}+\cdots +\frac{a_m}{m^s}$ be an algebraically primitive Dirichlet polynomial of degree $m$ with complex coefficients, with $m$ not square-free, and let $d$ be a divisor of $m$. Then $f$ and $f'$ have no common factors of degree $\geq d$ if and only if ${\rm rank} (\mathcal{D}_d^{(1)})=\frac{2m}{d}$. In particular, $f$ is square-free if and only if ${\rm rank} (\mathcal{D}_1^{(1)})=2m$.
\end{corollary}
\begin{proof}
It remains to prove that $f$ is square-free if and only if $f$ and $f'$ are relatively prime. Indeed, if $f=g^2h$ for some Dirichlet polynomials $g$ and $h$ with $g$ nonconstant, the $g$ will be also a factor of $f'$. Conversely, assume that $f$ and $f'$ share a nonconstant irreducible factor $g$, say $f=gh$ and $f'=g\ell$. This yields $g'h+gh'=g\ell$, so $g'h=g(\ell-h')$. Since $g$ is irreducible, to conclude that $g\mid h$, we must prove that $g$ cannot divide $g'$. Let $g(s)=\frac{b_1}{1^s}+\cdots +\frac{b_k}{k^s}$, say, with $k>1$ a divisor of $m$ and $b_k\neq 0$. If we assume to the contrary that $g'(s)=ag(s)$ for some complex number $a$, after equating the coefficients we deduce that $-b_i\log i=ab_i$ for each $i=1,\dots ,k$. These equalities can hold only if $a=-\log k$ and $b_1=\cdots =b_{k-1}=0$, so only if $g(s)=\frac{b_k}{k^s}$, and actually with $k$ prime, as $g$ was assumed to be irreducible. But this contradicts our assumption that $f$ is algebraically primitive, so $g$ must divide $h$, and hence $g^2\mid f$.
\end{proof}

More generally, the ranks of the matrices $\mathcal{D}_1^{(k)}$ provide information on the maximal multiplicity of the irreducible factors of a Dirichlet polynomial, as follows.

\begin{corollary}\label{maximalmultiplicity} Let $f(s)=\frac{a_1}{1^s}+\cdots +\frac{a_m}{m^s}$ be a Dirichlet polynomial of degree $m$ with complex coefficients, and assume that $M:=\max _{p\mid m}\nu_p(m)>1$. If ${\rm rank} (\mathcal{D}_1^{(k)})=2m$ for some positive integer $k<M$, then the maximal multiplicity of the irreducible factors of $f$ is at most $k$.
\end{corollary}

Computing the ranks of the matrices $\mathcal{D}_d^{(k)}$, with entries $d_{i,j}^{(k)}$ given by 
(\ref{higherderivativescoefficients}), requires testing the nonvanishing of some nonlinear forms in logarithms of natural numbers. This ultimately reduces to test the nonvanishing of some nonlinear forms in logarithms of prime numbers, which is by no means an easy task, not even in the simplest case when the coefficients $a_i$ are rational numbers.

\section{Logarithmic Newton polytopes and irreducibility criteria
for \\ multivariate Dirichlet polynomials}\label{logpolytope}

Many foundational results on multiple Dirichlet series appeared in the last decades, with applications in the study of moments of zeta and $L$-functions, and also in areas related to combinatorics and representation theory. For some of these results we refer the reader to Brubaker, Bump and Friedberg \cite{BrubakerBumpFriedberg1}, \cite{BrubakerBumpFriedberg2}, \cite{BrubakerBumpFriedberg3}, Diaconu, Goldfeld and Hoffstein \cite{DiaconuGoldfeldHoffstein}, and Chinta and Gunnels \cite{ChintaGunnels}, for instance. For some recent results on the residues of quadratic Weyl group multiple Dirichlet series we refer the reader to Diaconu, Ion, Pa\c sol and Popa \cite{DiaconuIonPasolPopa}, and for recent developments on the study of moments of quadratic $L$-functions to Bergstr\"om, Diaconu, Petersen and Westerland \cite{BergstromDiaconuPetersenWesterland}.

To study the factorization of multivariate Dirichlet polynomials, an intuitive approach is to search for connections with the intensively studied geometry of polytopes (see Ewald \cite{Ewald}, Gr\"unbaum \cite{Grunbaum}, Webster \cite{Webster}, Ziegler \cite{Ziegler} and Schneider \cite{Schneider} for some standard texts). 
In this section we will use a special type of polytopes that are most suitable to study the factorization of multivariate Dirichlet polynomials, namely logarithmic convex polytopes, that we will briefly call {\it log-polytopes}. Logarithmic convex hulls are used for instance in the study of logarithmically convex Reinhardt domains (Abhyankar \cite{Abhyankar}). For applications of logarithmic convexity the reader is referred to Boyd and Vandenberghe \cite{Boyd}.

Some notations and definitions are in order.

\begin{definition}\label{AlgPrimMultiv}
For a field $K$, we will denote by $DP(K)[s_1,\dots ,s_n]$ the ring of Dirichlet multivariate polynomials with coefficients in $K$ and indeterminates $s_{1},\dots ,s_{n}$. A nonzero such Dirichlet multivariate polynomial $f$ is a finite sum of the form
\begin{equation}\label{DirichletMultivariatePolynomial}
f(s_{1},\dots ,s_{n})=\sum\limits _{(i_{1},\dots ,i_{n})\in S_{f}} \frac{a_{i_{1},\dots ,i_{n}}}{i_{1}^{s_{1}}\cdots i_{n}^{s_{n}}}
\end{equation}
with $a_{i_{1},\dots ,i_{n}}\in K\setminus \{0\}$, and with the finite set $S_{f}\in {\mathbb{N}^{*}}^{n}$ called the {\it support} of $f$. 


The {\it (total) degree} of $f$, and the {\it degree of $f$ with respect to $s_j$} are defined as 
\[
\deg f=\max\limits _{(i_{1},\dots ,i_{n})\in S_{f}}i_1\cdots i_n\quad \text{and} \quad \deg _jf=\max\limits _{(i_{1},\dots ,i_{n})\in S_{f}}i_j,
\]
respectively, so we have the inequality $\deg f\leq \prod_{j=1}^n\deg _jf$.

We say that a multivariate Dirichlet polynomial $f$ is {\it algebraically primitive} if it has no nonconstant factor with only one term, that is no factor of the form $\frac{c}{i_{1}^{s_{1}}\cdots i_{n}^{s_{n}}}$ with $c\neq 0$ and at least one $i_{j}$ greater than $1$. 

\end{definition}
Unless otherwise specified, we will work over an algebraically closed field $K$. To simplify notations, we will use natural logarithms, but as in the univariate case, we might as well consider an arbitrary base $\mathfrak{b}>1$ for our logarithms, and this would not affect the study of the factorization properties of $f$. Since there is no risk of confusion with the definition of log-integral points in Section \ref{NewtonLogPol}, we will also use the term log-integral in a different setting. 
\begin{definition}\label{culogaritmi}
A point $v=(v_1,\dots ,v_n)\in\mathbb{R}^{n}$ is called {\it log-integral} if its coordinates $v_{i}$ are of the form $\log k_{i}$ for some positive integers $k_{1},\dots ,k_{n}$, respectively. A convex polytope in $\mathbb{R}^{n}$ is called {\it log-integral} if all of its vertices are log-integral. Thus a {\it log-integral} polytope will be the convex hull of a finite set $S$ of log-integral points in $\mathbb{R}^{n}$, so if $S=\{ x_{1},\dots ,x_{k}\} \subseteq \mathbb{R}^{n}$ with $x_{1},\dots ,x_{k}$ log-integral, then the log-integral polytope of $S$ is
\[
P^{log}(S)=\left\{ \sum\limits _{i=1}^{k}\lambda _{i}x_{i}: \lambda _{i}\in \left[0,1\right],\  \lambda _{1}+\cdots +\lambda _{k}=1\right\} .
\] 

Let again $f(s_{1},\dots ,s_{n})$ be as in (\ref{DirichletMultivariatePolynomial}). To each vector $(i_{1},\dots ,i_{n})\in S_{f}$ we will associate its {\it log-integral companion} $(\log i_{1},\dots ,\log i_{n})\in \mathbb{R}^{n}$, and the set of all these log-integral points will be denoted by $S_{f}^{log}$, and referred to as the {\it logarithmic support} of $f$.

The {\it log-integral Newton polytope} of $f$ (briefly {\it the Newton log-polytope} of $f$) is the convex hull of $S_{f}^{log}$ and will be denoted by $NP^{log}(f)$. Thus, if $S_{f}^{log}=\{ v_{1},\dots ,v_{m}\} $, say, then
\[
NP^{log}(f)=\left\{ \sum\limits _{i=1}^{m}\lambda _{i}v_{i}: \lambda _{i}\in \left[0,1\right],\  \lambda _{1}+\cdots +\lambda _{m}=1\right\} .
\] 
\end{definition}
The definition of the Minkowski sum of two polytopes applies to log-integral polytopes as well, so for two multivariate Dirichlet polynomials $g,h$ having the same indeterminates, one may consider the set $NP^{log}(g)+NP^{log}(h)=\{ x+y:x\in NP^{log}(g),\ y\in NP^{log}(h)\} $. With these notations we have the following analogue of Ostrowski's Theorem on Newton polytopes for multivariate polynomials \cite{Ostrowski1} (see also \cite{Ostrowski2}).
\begin{theorem}\label{OstroDir}
Let $f,g,h$ be multivariate Dirichlet polynomials such that $f=g\cdot h$. Then $NP^{log}(f)=NP^{log}(g)+NP^{log}(h)$.
\end{theorem}
\begin{proof}\ Our proof adapts the ideas in the proof of Ostrowski' s Theorem given in \cite{Gao}.
Let 
\begin{eqnarray*}
f(s_{1},\dots ,s_{n}) & = & \sum\limits _{(i_{1},\dots,i_{n})\in S_{f}} \frac{a_{i_{1},\dots ,i_{n}}}{i_{1}^{s_{1}}\cdots i_{n}^{s_{n}}}, \\
g(s_{1},\dots ,s_{n}) & = & \sum\limits _{(j_{1},\dots,j_{n})\in S_{g}} \frac{b_{j_{1},\dots ,j_{n}}}{j_{1}^{s_{1}}\cdots j_{n}^{s_{n}}},\ {\rm and}\\
h(s_{1},\dots ,s_{n}) & = & \sum\limits _{(k_{1},\dots,k_{n})\in S_{h}} \frac{c_{k_{1},\dots ,k_{n}}}{k_{1}^{s_{1}}\cdots k_{n}^{s_{n}}}
\end{eqnarray*}
be our multivariate Dirichlet polynomials with nonzero coefficients $a_{i_{1},\dots ,i_{n}}$, $b_{j_{1},\dots ,j_{n}}$, $c_{k_{1},\dots ,k_{n}}$, indeterminates $s_{1},\dots ,s_{n}$ and support sets $S_{f},S_{g}$ and $S_{h}$, respectively. Let us assume that $S_{f}^{log}=\{ u_{1},\dots ,u_{d_{1}}\} $, $S_{g}^{log}=\{ v_{1},\dots ,v_{d_{2}}\}$ and $S_{h}^{log}=\{ w_{1},\dots ,w_{d_{3}}\}$. A vector $x\in NP^{log}(f)$ may be written as 
\[
x=\sum\limits _{i=1}^{d_{1}}\lambda _{i}u_{i}\ \ {\rm with\ nonnegative}\ \lambda _{i}\ {\rm and}\ \sum\limits _{i=1}^{d_{1}}\lambda _{i}=1,
\]
where each $u_{i}$ is uniquely written as $u_{i}=(\log i_{1},\dots ,\log i_{n})$ for a certain $(i_{1},\dots ,i_{n})\in S_{f}$. By the multiplication rule for Dirichlet polynomials, our assumption that $f=g\cdot h$ shows that each term $\frac{1}{i_{1}^{s_{1}}\cdots i_{n}^{s_{n}}}$ appearing (multiplied by $a_{i_{1},\dots ,i_{n}}$) in the sum defining $f$ must be written as a product $(\frac{1}{j_{1}^{s_{1}}\cdots j_{n}^{s_{n}}})\cdot (\frac{1}{k_{1}^{s_{1}}\cdots k_{n}^{s_{n}}})$ for at least one pair of terms $(\frac{1}{j_{1}^{s_{1}}\cdots j_{n}^{s_{n}}},\frac{1}{k_{1}^{s_{1}}\cdots k_{n}^{s_{n}}})$
with $(j_{1},\dots ,j_{n})\in S_{g}$ and $(k_{1},\dots ,k_{n})\in S_{h}$, so $(\log i_{1},\dots ,\log i_{n})=(\log j_{1},\dots ,\log j_{n})+(\log k_{1},\dots ,\log k_{n})$. This shows that every convex combination of the $u_{i}$'s may be written as the sum of a convex combination of the $v_{j}$'s and a convex combination of the $w_{k}$'s, proving the inclusion $NP^{log}(f)\subseteq NP^{log}(g)+NP^{log}(h)$. 

To prove the other inclusion, let us consider a vertex $x$ of $NP^{log}(g)+NP^{log}(h)$, so $x=y+z$ for some $y\in NP^{log}(g)$ and $z\in NP^{log}(h)$. We will first prove that the pair $(y,z)$ is uniquely determined. Assume to the contrary that $x=y+z=y'+z'$ for some $y'\in NP^{log}(g)$ and $z'\in NP^{log}(h)$. Then we also have $x=\frac{1}{2}(y+z')+\frac{1}{2}(y'+z)$. Since this is a convex combination of the points $y+z'$ and $y'+z$, both belonging to $NP^{log}(g)+NP^{log}(h)$, and $x$ was assumed to be a vertex, we must actually have $y+z'=y'+z$, which forces $y-y'=y'-y$ and $z-z'=z'-z$, or equivalently $y=y'$ and $z=z'$. As $x$ is a vertex of $NP^{log}(g)+NP^{log}(h)$, we see now that the uniquely determined points $y$ and $z$ must be vertices of $NP^{log}(g)$ and $NP^{log}(h)$, respectively. Thus  
$y=(\log j_{1},\dots ,\log j_{n})$ and $z=(\log k_{1},\dots ,\log k_{n})$ for some $(j_{1},\dots ,j_{n})\in S_{g}$ and $(k_{1},\dots ,k_{n})\in S_{h}$. This further shows that the product $(\frac{1}{j_{1}^{s_{1}}\cdots j_{n}^{s_{n}}})\cdot (\frac{1}{k_{1}^{s_{1}}\cdots k_{n}^{s_{n}}})$ must appear (multiplied by a nonzero constant) in the sum defining $f$, so it must be of the form $\frac{1}{i_{1}^{s_{1}}\cdots i_{n}^{s_{n}}}$ for some $(i_{1},\dots ,i_{n})\in S_{f}$, which finally shows that $x\in S_{f}^{log}\subseteq NP^{log}(f)$. As any convex combination of the vertices of $NP^{log}(g)+NP^{log}(h)$ must also belong to $NP^{log}(f)$, we conclude that $NP^{log}(g)+NP^{log}(h)\subseteq NP^{log}(f)$, which completes the proof. 
\end{proof}

Theorem \ref{OstroDir} suggests the use of the following definition.
\begin{definition}\label{LID}
A log-integral polytope is {\it log-integrally decomposable} if it can be written as the sum of two log-integral polytopes each one containing at least two points, and {\it log-integrally indecomposable} otherwise. We also mention here that a multivariate Dirichlet polynomial is called absolutely irreducible, if it is irreducible over an algebraically closed field $K$.
\end{definition}
This definition allows one to transform Theorem \ref{OstroDir} into an absolute irreducibility criterion for multivariate Dirichlet polynomials, that may be regarded as an analogue of an absolute irreducibility criterion of Gao \cite[p. 507]{Gao} for multivariate polynomials:
\begin{theorem}\label{critOstroDir}
If $f$ is algebraically primitive and $NP^{log}(f)$ is log-integrally indecomposable, then $f$ is absolutely irreducible. 
\end{theorem}
\begin{proof}\ We first note that $f$ has no nonconstant factors with only one term, as it was assumed agebraically primitive. Thus, if we suppose now that $f$ factors as a product of two nonconstant multivariate Dirichlet polynomials $g$ and $h$, then each one of $g$ and $h$ should have at least two terms, which by Theorem \ref{OstroDir} would force $NP^{log}(f)$ to be log-integrally decomposable, a contradiction.
\end{proof}

We will prove now a result that may be regarded as a logarithmic analogue of Lemma 4.1 in \cite{Gao}, which will describe all the log-integral points belonging to the convex hull of a set of two different log-integral points in $\mathbb{R}^{n}$. We will first need some additional terminology. Let $\mathbf{vw}$ be any segment with log-integral endpoints $\mathbf{v}=(\log a_{1},\dots ,\log a_{n})$ and $\mathbf{w}=(\log b_{1},\dots ,\log b_{n})$ in $\mathbb{R}^{n}$, and for each $i=1,\dots ,n$ let $d_{i}=\gcd\{ v_{p}(b_{i})-v_{p}(a_{i}):p\ {\rm prime},\ p\mid a_{i}\cdot b_{i}\} $. Let us define 
\begin{equation}\label{dul}
\overline{\gcd}(\mathbf{vw})=\overline{\gcd}(\mathbf{w}-\mathbf{v})=\gcd(d_1,\dots ,d_n).
\end{equation}
Note that $\mathbf{w}-\mathbf{v}$ might no longer be a log-integral point in $\mathbb{R}^{n}$. Moreover, for any $k$ segments $\mathbf{v_1w_1},\dots ,\mathbf{v_kw_k}$ with log-integral endpoints, we define 
\[
\overline{\gcd}(\mathbf{v_1w_1},\dots ,\mathbf{v_kw_k})=\overline{\gcd}(\mathbf{w_1}-\mathbf{v_1},\dots ,\mathbf{w_k}-\mathbf{v_k})=\gcd(\overline{\gcd}(\mathbf{v_1w_1}),\dots ,\overline{\gcd}(\mathbf{v_kw_k})).
\]
 
\begin{lemma}\label{puncteR^n}
Let $\mathbf{v}=(\log a_{1},\dots ,\log a_{n})$ and $\mathbf{w}=(\log b_{1},\dots ,\log b_{n})$ be two different log-integral points in $\mathbb{R}^{n}$. Then the log-integral points belonging to the line segment $\mathbf{vw}$ are
\[
\mathbf{u}_{i}=\left( \log \thinspace (a_{1}^{1-\frac{i}{d}}b_{1}^{\frac{i}{d}}),\dots ,\log \thinspace (a_{n}^{1-\frac{i}{d}}b_{n}^{\frac{i}{d}})\right) ,\quad i=0,\dots ,d,
\]
with $d=\overline{\gcd}(\mathbf{vw})$. Moreover, if $\mathbf{u_i}$ is such a log-integral point, then $\frac{|\mathbf{u_i}-\mathbf{v}|}{|\mathbf{w}-\mathbf{v}|}=\frac{\overline{\gcd}(\mathbf{vu_i})}{\overline{\gcd}(\mathbf{vw})}=\frac{i}{d}$, where $|\mathbf{x}|$ denotes the Euclidean length of the vector $\mathbf{x}$.
\end{lemma}
\begin{proof}\ A point $\mathbf{u}$ belonging to the line segment $\mathbf{vw}$ has the form $(1-t)\mathbf{v}+t\mathbf{w}$ for some real number $t\in\left[0,1\right]$. 
For $\mathbf{u}$ to be log-integral we need $a_{i}^{1-t}b_{i}^{t}$
to be integer for each $i=1,\dots ,n$. Let us fix an index $i\in \{ 1,\dots ,n\} $ for which $a_i\neq b_i$, and assume that $p_{1},\dots ,p_{k}$ are all the prime factors of $a_{i}\cdot b_{i}$, so we may write
\[
a_{i}^{1-t}b_{i}^{t}=p_{1}^{(1-t)v_{p_{1}}(a_{i})+tv_{p_{1}}(b_{i})}\cdots p_{k}^{(1-t)v_{p_{k}}(a_{i})+tv_{p_{k}}(b_{i})},
\]
where for each index $j$ at most one of the multiplicities $v_{p_{j}}(a_{i})$ and $v_{p_{j}}(b_{i})$ might be zero.
Thus, for $a_{i}^{1-t}b_{i}^{t}$ to be an integer, we need $t\cdot (v_{p_{j}}(b_{i})-v_{p_{j}}(a_{i}))$ to be an integer for each $j=1,\dots ,k$ for which $v_{p_{j}}(a_{i})\neq v_{p_{j}}(b_{i})$, which shows in particular that $t$ must be rational, say $t=\frac{r}{s}\in$\thinspace $\left[0,1\right]$, with $r$ and $s$ coprime if $t\neq 0$ and $t\neq 1$. Thus, we need $s$ to divide $v_{p_{j}}(b_{i})-v_{p_{j}}(a_{i})$ for each $j=1,\dots ,k$, as $r$ and $s$ are coprime, or equivalently, we need $s$ to divide $d_{i}$. Since this must hold for each $i$, $s$ must be a divisor of $d$. The conclusion follows by observing that the set of all quotients $\frac{r}{s}\in\left[0,1\right]$ with $s$ a divisor of $d$ and $\gcd(r,s)=1$ if $r\neq 0$ and $r\neq s$, is in fact the set of all the quotients $\frac{i}{d}$ with $i\in \{0,\dots ,d\} $. 

For the second statement, observe that $\mathbf{u_i}=(1-\frac{i}{d})\mathbf{v}+\frac{i}{d}\mathbf{w}$, so $\mathbf{u_i}-\mathbf{v}=\frac{i}{d}(\mathbf{w}-\mathbf{v})$. Thus, it remains to prove that $\overline{\gcd}(\mathbf{vu_i})=i$. To do this, we will apply the definition of $\overline{\gcd}$ with $\mathbf{u_i}$ instead of $\mathbf{w}$. For each $j=1,\dots ,n$ let 
\begin{eqnarray*}
d'_{j} & = & \gcd\Bigl\{ v_{p}(a_{j}^{1-\frac{i}{d}}b_{j}^{\frac{i}{d}})-v_{p}(a_{j}):p\ {\rm prime},\ p\mid a_{j}\cdot b_{j}\Bigr\} \\
& = & \gcd \Bigl\{ \frac{i}{d}(\nu_p(b_{j})-\nu_{p}(a_{j})):p\ {\rm prime},\ p\mid a_{j}\cdot b_{j}\Bigr\} ,
\end{eqnarray*}
with $d=\overline{\gcd}(\mathbf{vw})$ given by (\ref{dul}). According to the definition, $\overline{\gcd}(\mathbf{vu_i})=\gcd(d'_1,\dots ,d'_n)$. Recall now that $d$ is the greatest common divisor of all the differences of the form $\nu_p(b_{j})-\nu_{p}(a_{j})$ obtained when $j$ runs in $\{1,\dots ,n\}$ and for a given $j$, $p$ runs in the set of all the primes dividing $a_j\cdot b_j$. Therefore, dividing all these differences by $d$ results in a set of coprime integers, which in view of the expressions of $d'_j$ above, shows that $\gcd(d'_1,\dots ,d'_n)=i$, as claimed.
\end{proof}

In particular, from Lemma \ref{puncteR^n} we obtain the following corollary.
\begin{corollary}\label{SegmentIndecompozabil}
Let $\mathbf{v}=(\log a_{1},\dots ,\log a_{n})$ and $\mathbf{w}=(\log b_{1},\dots ,\log b_{n})$ be two different log-integral points in $\mathbb{R}^{n}$, and let $d_{i}=\gcd\{ v_{p}(a_{i})-v_{p}(b_{i}):p\ {\rm prime},\ p\mid a_{i}\cdot b_{i}\} $ for $i=1,\dots ,n$. If $\gcd (d_{1},\dots ,d_{n})=1$, then the line segment $\mathbf{vw}$ is log-integrally indecomposable.
\end{corollary}
\begin{proof}\ By Lemma \ref{puncteR^n} we see that the only log-integral points belonging to the line segment $\mathbf{vw}$ are $\mathbf{v}$ and $\mathbf{w}$.
Let us assume to the contrary that $\mathbf{vw}=P^{log}(S_{1})+P^{log}(S_{2})$ for two finite sets of log-integral points $S_{1}, S_{2}\subset \mathbb{R}^{n}$, each one containing at least two points, say $X_{1},X_{2}\in S_{1}$ and $Y_{1},Y_{2}\in S_{2}$, with $X_{1}\neq X_{2}$ and $Y_{1}\neq Y_{2}$. As $X_{1}+Y_{1}$ and $X_{1}+Y_{2}$ are different log-integral points and both belong to $\{ \mathbf{v},\mathbf{w}\} $, we may assume without loss of generality that
\begin{equation}\label{primuta}
X_{1}+Y_{1}=\mathbf{v}\quad {\rm and}\quad  X_{1}+Y_{2}=\mathbf{w}.
\end{equation}
Now, since $X_{2}+Y_{1}$ and $X_{2}+Y_{2}$ are also different log-integral points belonging to $\{ \mathbf{v},\mathbf{w}\} $, we can neither have $X_{2}+Y_{1}=\mathbf{v}$ and $X_{2}+Y_{2}=\mathbf{w}$, as $X_{1}\neq X_{2}$, nor $X_{2}+Y_{1}=\mathbf{w}$ and $X_{2}+Y_{2}=\mathbf{v}$, for this will lead us after subtracting these two relations from the relations in (\ref{primuta}) to $X_{1}-X_{2}=\mathbf{v}-\mathbf{w}$ and $X_{1}-X_{2}=\mathbf{w}-\mathbf{v}$, which cannot hold simultaneously, as $\mathbf{v}\neq \mathbf{w}$. Thus the line segment $\mathbf{vw}$ must be log-integrally indecomposable.
\end{proof}

\medskip

Ostrowski \cite{Ostrowski2} obtained the following irreducibility criterion for multivariate polynomials:
\medskip

{\bf Theorem} (Ostrowski) {\em A two-term polynomial
\[
aX_{1}^{i_{1}}\cdots X_{k}^{i_{k}}+bX_{k+1}^{i_{k+1}}\cdots X_{n}^{i_{n}}\in K[X_{1},\dots ,X_{n}],\quad a,b\in K\setminus \{ 0\} ,
\]
is absolutely irreducible iff $\gcd(i_{1},\dots ,i_{n})=1$.}
\medskip

Combining now Theorem \ref{critOstroDir} and Corollary \ref{SegmentIndecompozabil}, we easily obtain a criterion of absolute irreducibility for multivariate Dirichlet polynomials consisting of two terms, that is similar to Ostrowski's irreducibility criterion above.
\begin{theorem}\label{SumaDubla}
Let $K$ be an algebraically closed field and $f(s_{1},\dots ,s_{n})=\frac{a}{a_{1}^{s_{1}}\cdots a_{n}^{s_{n}}}+\frac{b}{b_{1}^{s_{1}}\cdots b_{n}^{s_{n}}}$, with $a,b\in K$, $ab\neq 0$, and $a_{i}$ coprime to $b_{i}$ for each $i\in \{ 1,\dots ,n\} $. Then $f$ is absolutely irreducible if and only if the multiplicities of all the primes that appear in the prime factorizations of $a_{1},\dots ,a_{n},b_{1},\dots ,b_{n}$ are relatively prime.
\end{theorem}
\begin{proof}\ As $a_{i}$ is coprime to $b_{i}$ for each $i$, $f$ is algebraically primitive, and any term of the form $v_{p}(a_{i})-v_{p}(b_{i})$ with $p$ a prime dividing $a_{i}\cdot b_{i}$ will either be equal to $v_{p}(a_{i})$, or to $-v_{p}(b_{i})$. Our assumption that all these multiplicities are coprime leads via Corollary \ref{SegmentIndecompozabil} to the conclusion that the convex hull of $\{ (\log a_{1},\dots ,\log a_{n}),(\log b_{1},\dots ,\log b_{n})\} $, which is $NP^{log}(f)$, must be log-integrally indecomposable. This shows by Theorem \ref{critOstroDir} that $f$ must be absolutely irreducible. Conversely, let us assume that the greatest common divisor of the multiplicities of all the primes that appear in the prime factorizations of $a_{1},\dots ,a_{n},b_{1},\dots ,b_{n}$, say $d$, is greater than $1$. Then each of the integers $a_{1},\dots ,a_{n},b_{1},\dots ,b_{n}$ is a $d$-power of some positive integer, say $a_{i}=\alpha _{i}^d$ and $b_{i}=\beta _{i}^d$ for $i=1,\dots ,n$, and moreover, $\alpha _i$ and $\beta _i$ are coprime for each $i$. On the other hand, as $K$ is algebraically closed, there exist nonzero elements $\alpha ,\beta \in K$ such that $a=\alpha ^{d}$ and $-b=\beta ^{d}$. Then $f$ must be reducible, being divisible by $\alpha \cdot \alpha _{1}^{s_{1}}\cdots \alpha _{n}^{s_{n}}-\beta \cdot \beta _{1}^{s_{1}}\cdots \beta _{n}^{s_{n}}$ (which is a nonzero multivariate Dirichlet polynomial, as $\alpha _i$ and $\beta _i$ are coprime for each $i$). 
\end{proof}
Gao proved the following elegant criterion for integrally indecomposability of polytopes.
\medskip

{\bf Theorem \cite[Theorem 4.2]{Gao}}.\ {\em Let $Q$ be any integral polytope in $\mathbb{R}^n$ contained in a hyperplane $H$ and let $\mathbf{v}\in \mathbb{R}^n$ be an integral point lying outside of $H$. Suppose that $\mathbf{v_1},\dots ,\mathbf{v_k}$ are all the vertices of $Q$. Then the polytope ${\rm conv}(v,Q)$ is integrally indecomposable if and only if $\gcd(\mathbf{v}-\mathbf{v_1},\dots , \mathbf{v}-\mathbf{v_k})=1$.}
\medskip

A key ingredient in the proof of this result is the following lemma that counts integral points on a segment with integral endpoints.
\medskip

{\bf Lemma \cite[Lemma 4.1]{Gao}}.\ {\em Let $\mathbf{v_0}$ and $\mathbf{v_1}$ be two integral points in $\mathbb{R}^n$. Then the number of integral points on the line segment $\mathbf{v_0v_1}$, including $\mathbf{v_0}$ and $\mathbf{v_1}$, is equal to $\gcd(\mathbf{v_0}-\mathbf{v_1})+1$. Further, if $\mathbf{v_2}$ is any integral point on $\mathbf{v_0v_1}$, then $\frac{|\mathbf{v_2}-\mathbf{v_0}|}{|\mathbf{v_1}-\mathbf{v_0}|}=\frac{\gcd(\mathbf{v_2}-\mathbf{v_0})}{\gcd(\mathbf{v_1}-\mathbf{v_0})}$, where $|\mathbf{v}|$ denotes the Euclidean length of the vector $\mathbf{v}$.}

\medskip

A logarithmic analogue of Gao's Theorem is the following result.
\begin{theorem}\label{logarithmicGao}
Let $Q$ be any log-integral polytope in $\mathbb{R}^n$ contained in a hyperplane $H$, and let $\mathbf{v}\in \mathbb{R}^n$ be a log-integral point lying outside of $H$. Suppose that $\mathbf{v_1},\dots ,\mathbf{v_k}$ are all the vertices of the log-integral polytope $Q$. Then the polytope ${\rm conv}(v,Q)$ is log-integrally indecomposable if and only if \ $\overline{\gcd}(\mathbf{v}-\mathbf{v_1},\dots , \mathbf{v}-\mathbf{v_k})=1$.
\end{theorem}
\begin{proof}
The proof follows the same lines used in the proof of Gao's Theorem, with ``integral'' replaced by ``log-integral'', $\gcd$ replaced by $\overline{\gcd}$, and the above lemma replaced by Lemma \ref{puncteR^n}. The details will be omitted.
\end{proof}

The reader may combine Theorem \ref{critOstroDir} and Theorem \ref{logarithmicGao} to obtain criteria for absolute irreducibility for multivariate Dirichlet polynomials whose Newton log-integral polytopes have more than two vertices. Other criteria for log-integrally indecomposability may be obtained by adapting the remaining results of Gao in \cite{Gao} for multivariate polynomials, for instance. For an algorithm for testing indecomposability we refer the reader to Gao and Lauder \cite{GaoLauder}, and for absolute irreducibility modulo large prime numbers to Gao and Rodrigues \cite{GaoRodrigues}.

\section{Logarithmic upper Newton polygons and Stepanov-Schmidt type results} \label{Stepanov}

In the case of multivariate polynomials, apart from Newton polytopes, another geometrical object useful in their study is the {\it upper Newton polygon}, used by Stepanov and Schmidt \cite{Schmidt} (see also Gao \cite{Gao}) to prove the following absolute irreducibility criterion for bivariate polynomials. 
\smallskip

{\bf The Stepanov-Schmidt criterion.} {\em Let $K$ be a field and let $f\in K[X,Y]$ with degree $n$ in $X$. If the upper Newton polygon of $f$ with respect to $Y$ has only
one line segment from $(0,m)$ to $(n,0)$ and $\gcd(m,n)=1$, then $f$ is absolutely
irreducible over $K$.}
\smallskip

Here the upper Newton polygon of $f(X,Y)=a_0(Y)+a_1(Y)X+\cdots +a_{n}(X)X^n$ is the upper convex hull of the set of points $(0,\deg a_0), (1,\deg a_1),\dots , (n,\deg a_n)$.

In this section we will construct a logarithmic analogue for algebraically primitive multivariate Dirichlet polynomials of the upper Newton polygon for multivariate polynomials, which we will call the {\it logarithmic upper Newton polygon}, or briefly the {\it upper Newton log-polygon}.
We will fix again an arbitrarily chosen real number $\mathfrak{b}$ greater than $1$, which will be the base our logarithms will be considered to, and for a given base $\mathfrak{b}$, we will call {\it log-integral} any point having coordinates $(\log _{\mathfrak{b}}x,\ \log_{\mathfrak{b}}y)$ with $x,y$ positive integers. As we shall see later, in this case too the base $\mathfrak{b}$ will be in some sense irrelevant, so to simplify notations we will choose again to work with natural logarithms, keeping in mind that the construction of the upper Newton log-polygon below can be done using any base $\mathfrak{b}>1$.

Now let $K$ be a field, and $f(s_{1},\dots ,s_{t})$ be an algebraically primitive multivariate Dirichlet polynomial in $t\geq 2$ indeterminates $s_{1},\dots ,s_{t}$ and coefficients in $K$. Fix one of the indeterminates, say $s_{t}$, and denote it by $s$, then write $f$ as $f(s)=\sum\limits_{i=n'}^{n}\frac{a_{i}}{i^{s}}$ with $a_{i}$ multivariate Dirichlet polynomials with indeterminates $s_{1},\dots ,s_{t-1}$ and coefficients in $K$, $a_{n'}a_{n}\neq 0$. As $f$ is algebraically primitive, we have $1\leq n'<n$. Let us fix another indeterminate, say $s_{r}$ with $r\in \{1,\dots,t-1\} $, and then let us consider the degrees of the coefficients $a_{i}$ regarded as Dirichlet polynomials in indeterminate $s_{r}$ and coefficients Dirichlet polynomials in the remaining indeterminates. Denote these degrees by $\deg _{r}a_{i}$, and to each term $\frac{a_{i}}{i^{s}}$ of $f$ with $a_{i}\neq 0$ let us associate a point in plane with coordinates $(\log i,\log \deg_{r}a_{i})$. The {\it upper} convex hull of all of these points will be called the {\it upper Newton log-polygon} of $f$ with respect to indeterminate $s_{r}$.  

At this point we have to mention that in the case of our multivariate Dirichlet polynomials, when applied to the coefficients $a_i$, $-\log \deg_{r}(\cdot )$ has the properties of a valuation, as in the case of $\nu _p(\cdot)$, that was used in the construction of the Newton log-polygon in Section \ref{NewtonLogPol}. In this respect one might use the {\it lower} convex hull of the set of points $(\log i,-\log \deg_{r}a_{i})$, to derive results similar to those in Section \ref{NewtonLogPol}. However, to avoid carrying the minus sign, we will prefer to use the upper convex hull of the set of points $(\log i,\log \deg_{r}a_{i})$. More precisely, this upper convex hull is obtained as follows: 
Let $i_{1}=n'$, $P_{1}=(\log i_{1},\log \deg_{r} a_{i_{1}})=(\log n',\log \deg _{r}a_{n'})$, and let $P_{2}=(\log i_{2},\log \deg _{r}a_{i_{2}})$ with $i_{2}$ the largest integer with the property that there are no points $(\log i,\log \deg_{r} a_{i})$ lying above the line passing through $P_{1}$ and $P_{2}$. Then let $P_{3}=(\log i_{3},\log \deg _{r}a_{i_{3}})$, with $i_{3}$ the largest integer such that there are no points $(\log i,\log \deg_{r} a_{i})$ lying above the line passing through $P_{2}$ and $P_{3}$, and so on. The last line segment built in this way will be $P_{v-1}P_{v}$, with $P_{v}=(\log n,\log \deg _{r}a_{n})$.
The broken line $P_{1}\dots P_{v}$ thus constructed is the upper Newton log-polygon of $f$ with respect to the indeterminate $s_{r}$. We will refer to the log-integral points $P_{1},\dots ,P_{v}$ as the {\it vertices} of the upper Newton log-polygon of $f$.
We note that when we construct this upper convex hull, the missing terms in the sum defining $f$, that is the terms $\frac{a_{i}}{i^{s}}$ with $a_{i}=0$, are irrelevant, as we can formally consider them as having $-\infty $ for their $y$-coordinate (in view of the convention that the zero Dirichlet polynomial has degree $0$).

The line segments $P_{l}P_{l+1}$ will be called {\it edges}. An edge $P_{l}P_{l+1}$ might pass through some log-integral points, other than its endpoints $P_{l}$ and $P_{l+1}$. These points might have coordinates of the form $(\log i, \log y_{i})$ with $i,y_{i}\in \mathbb{Z_{+}}$ and $y_{i}$ not necessarily equal to $\deg _{r} a_{i}$. Let us denote all these intermediate points by $Q_{1},\dots ,Q_{k}$, say. Then $P_{l}Q_{1}$, $Q_{1}Q_{2}$, $\dots $, $Q_{k-1}Q_{k}$ and $Q_{k}P_{l+1}$ will be called {\it segments} of the upper Newton log-polygon. Thus, as in the case of Newton log-polygons, a segment will contain no log-integral points other than its endpoints, and an edge consists of a number of segments, possibly just one, in which case the edge itself will be called a segment.  Note that any two edges must have different slopes, while two segments that belong to the same edge will obviously have the same slope. 

The vectors $\overline{P_{i}P_{i+1}}$ will be called {\it vectors} of the edges of the upper Newton log-polygon,
and {\it the vector system} of the upper Newton log-polygon will be the union of all the vectors of its edges, repetitions allowed.

At the cost of a translation on the $x$-axis, this construction too may be also used for multivariate Dirichlet polynomials that are not algebraically primitive. The use of logarithms brings here too more geometric insight, while arithmetical conditions lead to effective results on the factorization of these multivariate Dirichlet polynomials. The construction of the vertices $P_{1},\dots ,P_{r}$ and the identification of all the intermediate log-integral points lying on the edges $P_{1}P_{2},\dots ,P_{v-1}P_{v}$ depends now on some more sophysticated ``log-diophantine" inequalities and equations. Thus, if we choose two indices $i$ and $j$ with $i<j$, we see that a log-integral point $(\log k,\log \deg _{r}a_{k})$ lies below or on the line passing through the log-integral points $(\log i,\log \deg _{r}a_{i})$ and $(\log j,\log \deg _{r}a_{j})$ if and only if
\begin{equation}\label{dedesubt}
\deg _{r}a_{k}\leq (\deg _{r}a_{i})^{\log _{\frac{j}{i}}\frac{j}{k}}(\deg _{r}a_{j})^{1-\log _{\frac{j}{i}}\frac{j}{k}},
\end{equation}
while a log-integral point $(\log k,\log y)$ with $y$ and $k$ integers with $i<k<j$, lies on this line if and only if
\begin{equation}\label{exactpedreapta}
y= (\deg _{r}a_{i})^{\log _{\frac{j}{i}}\frac{j}{k}}(\deg _{r}a_{j})^{1-\log _{\frac{j}{i}}\frac{j}{k}}.
\end{equation}
We recall that the log-integral points lying on a line segment with endpoints that are log-integral too are described in Lemma \ref{puncteR^n}, which will allow us to find explicit irreducibility conditions. We mention here that conditions (\ref{dedesubt}) and (\ref{exactpedreapta}) also show that the construction of the upper Newton log-polygon doesn't essentially depend on the base of the logarithms that we are using. So changing this base will not affect the intrinsic arithmetic properties of the upper Newton log-polygon that we will rely on when studying the factorization of a multivariate Dirichlet polynomial. 

We are now ready to state the following result, that is in some analogous to Theorem \ref{DumDir}.

\begin{theorem}\label{DumStepDir}
Let $f,g,h$ be nonconstant algebraically primitive multivariate Dirichlet polynomials with $f=g\cdot h$. Then the upper Newton log-polygon of $f$ with respect to an indeterminate $s_{r}$ is obtained by using translates of the edges of the upper Newton log-polygons of $g$ and $h$ with respect to $s_{r}$, using precisely one translate for each edge, so as to form a polygonal line with decreasing slopes, when considered from left to the right.
\end{theorem}

In other words, this theorem says that the vector system of the upper Newton log-polygon of $f$ is obtained by considering the union of the vector systems of the upper Newton log-polygons of $g$ and $h$, repetitions allowed, and then replacing any two vectors having the same orientation by their sum.
\medskip

\begin{proof}\ The proof is similar to that of Theorem \ref{DumDir}, with almost the same phrasing, and only some minor modifications, so some of the details will be omitted.

Let $f(s)=\sum\limits _{i=n'}^{n}\frac{a_{i}}{i^{s}}$, $g(s)=\sum\limits _{j=d'}^{d}\frac{b_{j}}{j^{s}}$ and $h(s)=\sum\limits _{k=n'/d'}^{n/d}\frac{c_{k}}{k^{s}}$ with $a_{i},b_{j}$ and $c_{k}$ multivariate Dirichlet polynomials in indeterminates $s_{1},\dots ,s_{r}$. Let us consider an edge of the upper Newton log-polygon of $f$, say $P_{l}P_{l+1}$ (that may consist of more than a single segment), and let us assume that the points $P_{l}$ and $P_{l+1}$ have coordinates $(\log i_{m},\log \deg _{r}a_{i_{m}})$, and $(\log i_{M},\log \deg _{r}a_{i_{M}})$, respectively, with $i_{m}<i_{M}$. The slope of $P_{l}P_{l+1}$ is
\[
S=\frac{\log \deg _{r}a_{i_{M}}-\log \deg _{r}a_{i_{m}}}{\log i_{M}-\log i_{m}}.
\]
Let now $\log \deg _{r}a_{i_{M}}-\log \deg _{r}a_{i_{m}}=A$ and $\log i_{M}-\log i_{m}=I$, so $S=\frac{A}{I}$. The edge $P_{l}P_{l+1}$ of the Schmidt-Stepanov log-polygon of $f$ belongs to the line $Iy -Ax=F$, where
\[
F=I\log \deg _{r}a_{i_{M}}-A\log i_{M}=I\log \deg _{r}a_{i_{m}}-A\log i_{m}.
\] 
According to the definition, all the points $(\log i,\log \deg _{r}a_{i})$, $i=n',\dots ,n$, either lie on this line, or below it, so $I\log \deg _{r}a_{i}-A\log i\leq F$ for all $i$, with a strict inequality for $i<i_{m}$ and $i>i_{M}$, and we have equality in the endpoints corresponding to $i_{m}$ and $i_{M}$ (and for the points $(\log i, \log \deg _{r}a_{i})$ with $i_{m}<i<i_{M}$, if any, any of the cases ``$<$" and ``=" being possible).

To each term $\frac{a_{i}}{i^{s}}$ we will associate a real number
\[
w\Bigl(\frac{a_{i}}{i^{s}}\Bigr):=I\log \deg _{r}a_{i}-A\log i,
\]
called its {\it weight} with respect to the edge $P_{l}P_{l+1}$. We may then regard $i_{m}$ and $i_{M}$ as the least and the largest $i$ such that the corresponding term $\frac{a_{i}}{i^{s}}$ of $f$ has maximum weight (which is $F$) with respect to $P_{l}P_{l+1}$, so $i_{m}$ and $i_{M}$ are uniquely determined.
For $g$ we define
\[
G:=\underset{d'\leq j\leq d}{\rm max}\{I\log \deg _{r}b_{j}-A\log j\} 
\]
and we will also define $j_{m}$ and $j_{M}$ as the least and the largest index, respectively, such that
\begin{equation}\label{GDirichletStepanov}
G=I\log \deg _{r}b_{j_{m}}-A\log j_{m}=I\log \deg _{r}b_{j_{M}}-A\log j_{M}.
\end{equation}  
We notice that $j_{m}\leq j_{M}$. Similarly, for $h$ we define
\[
H:=\underset{n'/d'\leq k\leq n/d}{\rm max}\{I\log \deg _{r}c_{k}-A\log k\} 
\]
and we denote by $k_{m}$ and $k_{M}$ the least and the largest index, respectively, such that
\begin{equation}\label{HDirichletStepanov}
H=I\log \deg _{r}c_{k_{m}}-A\log k_{m}=I\log \deg _{r}c_{k_{M}}-A\log k_{M}.
\end{equation}  
Here too, we have $k_{m}\leq k_{M}$. By the multiplication rule for Dirichlet series, we have
\begin{equation}\label{regulaDirichletStepanov}
\frac{a_{j_{m}k_{m}}}{(j_{m}k_{m})^{s}}=\sum\limits _{j\cdot k=j_{m}\cdot k_{m}}\frac{b_{j}}{j^{s}}\cdot\frac{c_{k}}{k^{s}}.
\end{equation}  
For two terms that are multiplied in (\ref{regulaDirichletStepanov}) we have $w(\frac{b_{j}}{j^{s}}\cdot\frac{c_{k}}{k^{s}})=w(\frac{b_{j}}{j^{s}})+w(\frac{c_{k}}{k^{s}})$, so the weight of the summand with $j=j_{m}$ and $k=k_{m}$ in (\ref{regulaDirichletStepanov}) is $G+H$, while the weights of the rest of the summands are all less than $G+H$, since for these ones we either have $j<j_{m}$, or $k<k_{m}$. Besides, when $j\cdot k$ is constant, the weight of a product $\frac{b_{j}}{j^{s}}\cdot\frac{c_{k}}{k^{s}}$ regarded as a function on $\log \deg _{r}b_{j}+\log \deg _{r}c_{k}$ is a strictly increasing function, as $I>0$. Therefore, since in (\ref{regulaDirichletStepanov}) we have $j\cdot k=j_{m}\cdot k_{m}$, hence $j\cdot k$ is constant, the sum $\log \deg _{r}b_{j}+\log \deg _{r}c_{k}$ (and hence $\deg _{r}b_{j}c_{k}$) has maximum value only for $j=j_{m}$ and $k=k_{m}$, for any other pair $(j,k)$ being strictly smaller. Thus $\deg _{r}a_{j_{m}k_{m}}=\deg _{r}b_{j_{m}}c_{k_{m}}$, so
\begin{equation}\label{exactDirichletStepanov}
w\Bigl(\frac{a_{j_{m}k_{m}}}{(j_{m}k_{m})^{s}}\Bigr)=G+H.
\end{equation}
Reasoning similarly with (\ref{regulaDirichletStepanov}) replaced by $\frac{a_{i}}{i^{s}}=\sum\limits _{j\cdot k=i}\frac{b_{j}}{j^{s}}\cdot \frac{c_{k}}{k^{s}}$, one may prove that
\begin{eqnarray}
w\Bigl(\frac{a_{i}}{i^{s}}\Bigr) & < & G+H\qquad
{\rm for}\  i<j_{m}\cdot k_{m},\label{caz1DirichletStepanov}\\
w\Bigl(\frac{a_{i}}{i^{s}}\Bigr) & \leq & G+H\qquad
{\rm for}\  i\geq j_{m}\cdot k_{m}.\label{caz2DirichletStepanov}
\end{eqnarray}
By (\ref{exactDirichletStepanov}), (\ref{caz1DirichletStepanov}) and (\ref{caz2DirichletStepanov}), we must have
$G+H=F$ and $j_{m}\cdot k_{m}=i_{m}$.
Similarly, we also obtain $j_{M}\cdot k_{M}=i_{M}$, leading to $\log i_{M}-\log i_{m}=(\log j_{M}-\log j_{m})+(\log k_{M}-\log k_{m})$, with at least one of $\log j_{M}-\log j_{m}$ and $\log k_{M}-\log k_{m}$ positive. If both $\log j_{M}-\log j_{m}$ and $\log k_{M}-\log k_{m}$ are nonzero, then the geometric segment with endpoints $(\log j_{m},\log \deg _{r}b_{j_{m}})$ and $(\log j_{M},\log \deg _{r}b_{j_{M}})$ must be an edge of the upper Newton log-polygon of $g$, and the geometric segment with endpoints $(\log k_{m},\log \deg _{r}c_{k_{m}})$ and $(\log k_{M},\log \deg _{r}c_{k_{M}})$ must be an edge of the upper Newton log-polygon of $h$, the slopes of these two edges being $S=\frac{A}{I}$, since from (\ref{GDirichletStepanov}) and (\ref{HDirichletStepanov}) one obtains
\begin{equation*}
\frac{\log \deg _{r}b_{j_{M}}-\log \deg _{r}b_{j_{m}}}{\log j_{M}-\log j_{m}}=\frac{A}{I}=\frac{\log \deg _{r}c_{k_{M}}-\log \deg _{r}c_{k_{m}}}{\log k_{M}-\log k_{m}}.
\end{equation*}
Thus the sum of the lengths of the edges with slope $S$ in the upper Newton log-polygons of $g$ and $h$ is precisely the length of the edge with slope $S$ in the upper Newton log-polygon of $f$, and the vector of the edge with slope $S$ in the upper Newton log-polygon of $f$ is the sum of the vectors of the edges with slope $S$ in the upper Newton polygons of $g$ and $h$. Finally, one may prove that if one of $g$ and $h$ has an edge with slope $S$ in its upper Newton log-polygon, then $S$ must appear among the slopes of the edges in the upper Newton log-polygon of $f$. Also, the upper Newton log-polygons of $g$ and $h$ cannot have edges with slopes other than the ones appearing in the upper Newton log-polygon of $f$. This completes the proof. 
\end{proof}

In particular, Theorem \ref{DumStepDir} shows that each nonconstant factor of a multivariate Dirichlet polynomial $f$ must contribute at least one edge (that may consist of a single segment) to the upper Newton log-polygon of $f$. Therefore, we have the following analogue for multivariate Dirichlet polynomials of the Stepanov-Schmidt criterion for bivariate polynomials:
\begin{corollary}\label{coroDumasStepanov}
If with respect to an indeterminate $s_{r}$, the upper Newton log-polygon of an algebraically primitive multivariate Dirichlet polynomial $f$ consists of a single edge, and this edge contains no log-integral points other than its endpoints, then $f$ must be irreducible.
\end{corollary}
Here too it is crucial to make distinction between edges and segments. It is not sufficient to ask the upper Newton log-polygon of $f$ to have a single edge, for an edge might in principle consist of multiple segments coming from the upper Newton log-polygons of two or more of the nonconstant factors of $f$. That's why in Corollary \ref{coroDumasStepanov} we had to ask the edge to have no log-integral points other than its endpoints.
In order to obtain an effective instance of this corollary, we will make use of Lemma \ref{puncteR^n}, as follows.
\begin{theorem}\label{EffectiveStepSchmidt}
Let $K$ be a field, $m<n$ two positive integers and $a_m(t),\dots ,a_n(t)$ Dirichlet polynomials in the indeterminate $t$ over $K$, with $a_ma_n\neq 0$ and $\deg a_m\neq \deg a_n$. Let also $f(s,t)=\sum _{i=m}^{n}\frac{a_i(t)}{i^{s}}$, assume that $f$ is algebraically primitive, and let
\begin{eqnarray*} 
d_{1} & = & \gcd\{ \nu_{p}(m)-\nu_{p}(n):p\ {\rm prime},\ p\mid m\cdot n\}\ and \\
d_{2} & = & \gcd\{ \nu_{p}(\deg a_m)-\nu_{p}(\deg a_n):p\ {\rm prime},\ p\mid \deg a_m\cdot\deg a_n\} .
\end{eqnarray*}
If $\deg a_{i}< (\deg a_{m})^{\log _{\frac{n}{m}}\frac{n}{i}}(\deg a_{n})^{1-\log _{\frac{n}{m}}\frac{n}{i}}$ for $m<i<n$ and  $\gcd(d_1,d_2)=1$, then $f$ is irreducible.
\end{theorem}
\begin{proof}\ In view of (\ref{dedesubt}), our hypothesis that $\deg a_{i}< (\deg a_{m})^{\log _{\frac{n}{m}}\frac{n}{i}}(\deg a_{n})^{1-\log _{\frac{n}{m}}\frac{n}{i}}$ for $m<i<n$ shows that all the log-integral points $(\log i,\log \deg a_i)$ with $m<i<n$ lie below the line passing through the log-integral points $A=(\log m,\log \deg a_m)$ and $B=(\log n,\log \deg a_n)$, so the upper Newton polygon of $f$ consists of a single edge joining $A$ and $B$. Conditions $\deg a_m\neq \deg a_n$ and $\gcd(d_1,d_2)=1$ show that the only log-integral points on the edge $AB$ are its endpoints $A$ and $B$, so the edge $AB$ consists of a single segment. By Corollary \ref{coroDumasStepanov}, $f$ must be irreducible. 
\end{proof}

We mention here that if the coefficients $a_i$ are multivariate Dirichlet polynomials, one may simultaneously use information on the upper Newton log-polygons with respect to different indeterminates, and obtain results similar to those in Section \ref{moreprimes}, but with considerably more involved statements.

\section{Examples}\label{Examples}

1) For every integer $n>1$, the Dirichlet polynomials $\zeta_n(s)=1+\frac{1}{2^s}+\cdots +\frac{1}{n^s}$ that approximate the Riemann zeta-function are irreducible. The claim follows by Proposition \ref{prop1} for prime $n$, and by Proposition \ref{prop2} for composite $n$.

2) For every nonzero integers $a,b,c,d$, the Dirichlet polynomial $f(s)=\frac{a}{10^s}+\frac{b}{11^s}+\frac{c}{14^s}+\frac{d}{16^s}$ is irreducible. Here one may apply Proposition \ref{prop3} with $m=10$, $n=16$ and $i=11$.

3) Let $f(s)$ be a Dirichlet polynomial with integer coefficients and support $\{ i_1,i_2,\dots ,i_k\}$ with $i_1<i_2<\cdots <i_k$, and let $a$ be a nonzero integer and $n>i_k$ an integer with $\gcd (i_1,\dots ,i_k,n)=1$. Then the Dirichlet polynomial $p\cdot f(s)+\frac{a}{n^s}$ is irreducible over $\mathbb{Q}$ for all but finitely many prime numbers $p$.

Indeed, assume that $f(s)=\frac{a_1}{i_1^s}+\cdots +\frac{a_k}{i_k^s}$ with $a_1,\dots ,a_k\in\mathbb{Z}$, $a_1\cdots a_k\neq 0$. To apply Theorem \ref{naiveEisenstein} i) to the Dirichlet polynomial $\frac{p\cdot a_1}{i_1^s}+\cdots + \frac{p\cdot a_k}{i_k^s}+\frac{a}{n^s}$, which is algebraically primitive (as $\gcd (i_1,\dots ,i_k,n)=1$), it suffices to choose any prime $p$ such that $p\nmid a_1a$.

4) Let $f(s)$ be a Dirichlet polynomial with integer coefficients and support $\{ i_1,i_2,\dots ,i_k\}$ with $1<i_1<i_2<\cdots <i_k$, and let $a$ be a nonzero integer and $m<i_1$ a positive integer with $\gcd (m,i_1,\dots ,i_k)=1$. Then the Dirichlet polynomial $\frac{a}{m^s}+p\cdot f(s)$ is irreducible over $\mathbb{Q}$ for all but finitely many prime numbers $p$.

To see this, assume again that $f(s)=\frac{a_1}{i_1^s}+\cdots +\frac{a_k}{i_k^s}$ with $a_1,\dots ,a_k\in\mathbb{Z}$, $a_1\cdots a_k\neq 0$. Here one may choose $m=1$, for instance. To apply Theorem \ref{naiveEisenstein} ii) to the Dirichlet polynomial $\frac{a}{m^s}+\frac{p\cdot a_1}{i_1^s}+\cdots + \frac{p\cdot a_k}{i_k^s}$, which is also algebraically primitive, it suffices to choose any prime $p$ such that $p\nmid a_ka$.

5) Let $m<i<n$ be integers with $m$ coprime to $n$, and let $f(s)=\frac{a_{m}}{m^{s}}+\frac{a_{i}}{i^{s}} +\frac{a_{n}}{n^{s}}$ be a Dirichlet polynomial with integer coefficients, $a_{m}a_ia_{n}\neq 0$. Assume that $q_{1},\dots ,q_{k}$ are all the prime factors of $m\cdot n$, and that their multiplicities in the prime decomposition of $mn$ are all odd. If $a_n$ is odd, $4\mid a_m$, $8\mid a_i$ but $8\nmid a_m$, then $f$ is irreducible over $\mathbb{Q}$.

Note first that $f$ is algebraically primitive, as $m$ an $n$ are coprime. The conclusion follows from Theorem \ref{calaDumas} with $p=2$. Indeed, we have $\nu _2(a_i)>2=\nu _2(a_m)>\nu _2(a_n)=0$. Condition ii) is obviously satisfied, since $\frac{n}{i}>1$ while $\frac{m}{i}<1$ and the exponents $\nu _{p}(a_{i})-\nu _{p}(a_{m})$ and $\nu _{p}(a_{i})-\nu _{p}(a_{n})$ are both positive. Condition iii) also holds, as $\nu _{2}(a_{n})-\nu _{2}(a_{m})=-2$, and for each $i=1,\dots ,k$, $\nu _{q_{i}}(n)-\nu _{q_{i}}(m)$ is odd, being either equal to $\nu _{q_{i}}(n)$, or to $-\nu _{q_{i}}(m)$, as $m$ and $n$ are coprime.

6) Let $m<n$ be coprime positive integers, at least one of them being not a square. Let also $i$ be an integer with $m<i<\sqrt{mn}$. Then for every prime number $p$, the Dirichlet polynomial $f(s)=\frac{1}{m^s}+\frac{p}{i^s}+\frac{p^2}{n^s}$ is irreducible. Indeed, $f$ is algebraically primitive, and since at least one of $m$ and $n$ is not a square, we have $\rho(m,n)<\sqrt{\frac{n}{m}}$. We may now apply Theorem \ref{coroFilaseta1}, since condition (\ref{pantamica}) reduces to the inequalities $\rho(m,n)<\sqrt{\frac{n}{m}}$ and $\rho(m,n)<\frac{n}{i}$, the second one being satisfied, as $i<\sqrt{mn}$. Moreover, for any nonzero integers $a,b,c$ with $p\nmid ac$, the Dirichlet polynomial $\frac{a}{m^s}+\frac{pb}{i^s}+\frac{p^2c}{n^s}$ is also irreducible over $\mathbb{Q}$.

7) Let $m<n$ be coprime positive integers, with at least one of them not a square, and let $i$ be an integer with $\sqrt{mn}<i<n$. Then for every prime number $p$, the Dirichlet polynomial $f(s)=\frac{p^2}{m^s}+\frac{p}{i^s}+\frac{1}{n^s}$ is irreducible. Again, $f$ is algebraically primitive, and since at least one of $m$ and $n$ is not a square, we have $\rho(m,n)<\sqrt{\frac{n}{m}}$. Here we will apply Theorem \ref{coroFilaseta2}, since condition (\ref{pantamare}) reduces to the inequalities $\rho(m,n)<\sqrt{\frac{n}{m}}$ and $\rho(m,n)<\frac{i}{m}$, the second one being satisfied, as $\sqrt{mn}<i$. Besides, for any nonzero integers $a,b,c$ with $p\nmid ac$, the Dirichlet polynomial $\frac{ap^2}{m^s}+\frac{pb}{i^s}+\frac{c}{n^s}$ is irreducible over $\mathbb{Q}$.

8) For every distinct prime numbers $p$ and $q$, the Dirichlet polynomial $f(s)=\frac{pq}{10^{s}}+\frac{q}{15^s} +\frac{p}{18^s}+\frac{pq}{38^{s}}$
is irreducible over $\mathbb{Q}$. Note that $f$ is algebraically primitive. Its irreducibility follows by Theorem \ref{npnp} with $j_1=15$ and $j_2=18$.

9) For any prime numbers $p,q,r$ with $p\neq q$, and any positive integer $n$, the Dirichlet polynomial $p(1+\frac{1}{r^s})^n+q(2+\frac{1}{r^s})^n$ is irreducible over $\mathbb{Q}$. 
To prove this, one may apply Corollary \ref{CorolarpqSchone} to the pair $(F,G)$ with $F(s)=1+\frac{1}{r^s}$ and $G(s)=2+\frac{1}{r^s}$. 

10) For every integers $m,n$ with $3\nmid m$ and $n\geq 2$, and every integers $b,c$ with $b\not\equiv c\pmod 3$, the Dirichlet polynomial $m(-1+\frac{1}{2^s})^n+3(-b+\frac{c}{2^s})$ is irreducible over $\mathbb{Q}$. Here one may apply Corollary \ref{CorolarGradePrime} with $k=1$, $F_1(s)=-1+\frac{1}{2^s}$, $p=3$, $G(s)=-b+\frac{c}{2^s}$ and $a=-2$.

11) Let $\{p_1,p_2,\dots ,p_{2n}\}$ and $\{q_1,q_2,\dots ,q_{2n}\}$ be two sets of prime numbers, not necessarily disjoint. Then for every $a,b\in \mathbb{C}\setminus \{ 0\} $, the multivariate Dirichlet polynomial $f(s_{1},\dots ,s_{n})=\frac{a}{(p_{1}^{q_1})^{s_{1}}\cdots (p_{n}^{q_n})^{s_{n}}}+\frac{b}{(p_{n+1}^{q_{n+1}})^{s_{1}}\cdots (p_{2n}^{q_{2n}})^{s_{n}}}$ is absolutely irreducible. Here we may apply Theorem \ref{SumaDubla}, since the multiplicities $q_1,q_2,\dots ,q_{2n}$ are obviously relatively prime.

12) Let $f(s,t)=\frac{a}{(p_1^{q_1})^s(p_2^{q_2})^t}+\frac{b}{(p_3^{q_3})^s(p_4^{q_4})^t}+\frac{c}{(p_5^{q_5})^s(p_6^{q_6})^t}$ with $p_1,\dots ,p_6,q_1,\dots,q_6$ distinct prime numbers, and $a,b,c\in\mathbb{C}\setminus\{0\}$. Let us denote $p_i^{q_i}$ by $r_i$ for $i=1,\dots ,6$, and assume without loss of generality that $r_1<r_3<r_5$. If
\[
\log _{\frac{r_4}{r_2}} \Bigl(\frac{r_3}{r_1}\Bigr)\neq \log _{\frac{r_6}{r_4}} \Bigl(\frac{r_5}{r_3}\Bigr),
\]
then $f$ is absolutely irreducible. Observe first that $f$ is algebraically primitive. We then see that $S_f^{log}=(\mathbf{v_1},\mathbf{v_2},\mathbf{v_3})$ with $\mathbf{v_1}=(\log r_1,\log r_2)$, $\mathbf{v_2}=(\log r_3,\log r_4)$ and $\mathbf{v_3}=(\log r_5,\log r_6)$, so condition $\overline{\gcd}(\mathbf{v_3}-\mathbf{v_1},\mathbf{v_3}-\mathbf{v_2})=1$ in Theorem \ref{logarithmicGao} is obviously satisfied. Since the condition above involving logarithms prevents $\mathbf{v_1},\mathbf{v_2}$ and $\mathbf{v_3}$ to be collinear, the conclusion follows by Theorem \ref{logarithmicGao} and Theorem \ref{critOstroDir}.

13) Let $f(s,t)=1+(1+\frac{1}{2^t})\frac{1}{8^s}+(1+\frac{1}{32^t})\frac{1}{16^s}$. We will conclude that $f$ is irreducible by applying Theorem \ref{EffectiveStepSchmidt}. We may write $f(s,t)=\frac{a_1(t)}{1^s}+\frac{a_2(t)}{8^s}+\frac{a_3(t)}{16^s}$ with $a_1=1$, $a_8=1+\frac{1}{2^t}$ and $a_{16}=1+\frac{1}{32^t}$. In this case we have $d_1=\gcd(\nu_2(16)-\nu _2(1))=4$ and $d_2=\gcd(\nu_2(32)-\nu _2(1))=5$, so $\gcd(d_1,d_2)=1$, and condition $\deg a_{i}< (\deg a_{m})^{\log _{\frac{n}{m}}\frac{n}{i}}(\deg a_{n})^{1-\log _{\frac{n}{m}}\frac{n}{i}}$ for $m=1$, $i=8$ and $n=16$ reduces to $2<1\cdot 32^{\frac{3}{4}}=13.4543...$, which obviously holds. Note that the same conclusion holds if we replace $a_8(t)$ by any Dirichlet polynomial in $t$ of degree at most $13$.

\setcounter{section}{1} 
\renewcommand{\thesection}{\Alph{section}}

\section*{Appendix A. Variations on a theme by Sch\"onemann}
\label{Schone}

Sch\"onemann \cite{Schonemann} proved the following irreducibility criterion for polynomials with integer coefficients, that admits a straightforward generalization to unique factorization domains.
\medskip

{\bf  Irreducibility criterion of Sch\" onemann} {\em  
Suppose that a polynomial $f(X)\in \mathbb{Z}[X]$ has the form $f(X)=\phi (X)^e+pM(X)$, where $p$ is a prime number, $\phi(X)$ is an irreducible polynomial modulo $p$, and $M(X)$ is a polynomial relatively prime to $\phi (X)$ modulo $p$, with $\deg M<\deg f$. Then $f$ is irreducible over $\mathbb{Q}$.
}

\medskip

For some generalizations of this result we refer the reader to Panaitopol and \c Stef\u anescu \cite{PanaitopolStefanescu1}, \cite{PanaitopolStefanescu2}, and Jakhar and Sangwan \cite{Jakhar3}, for instance.

In this section we will provide several irreducibility criteria for some classes of linear combinations of Dirichlet polynomials, that may be regarded as analogous results of Sch\"onemann's criterion and its variations. Our first result that relies on Sch\"onemann's ideea is the following.

\begin{theorem}\label{calaSchonemann}
Let $f(s)=qF(s)^{n}+pG(s)$ where $n$ is a positive integer, $F,G$ are Dirichlet polynomials with coefficients in a unique factorization domain $R$, $p$ is a prime element of $R$, and $q$ a nonzero element of $R$. If the leading coefficient of $f$ is not divisible by $p$, $F$ is irreducible modulo $p$, and $F$ and $G$ are relatively prime modulo $p$, then $f$ is irreducible over $Q(R)$, the quotient field of $R$.
\end{theorem}
\begin{proof}\ The proof is similar to that of Sch\"onemann's irreducibility criterion. For an arbitrary Dirichlet polynomial $h(s)$ we will denote by $\overline{h}(s)$ the Dirichlet polynomial obtained from $h$ by reducing its coefficients modulo $p$.
Let us first note that our assumption that the leading coefficient of $f$ is not divisible by $p$ forces the equality $\deg G\leq (\deg F)^{n}$ and also the fact that $q$ is not divisible by $p$. We also note that the leading coefficient of $F$ is allowed to be divisible by $p$, but as $F$ is irreducible modulo $p$, $f$ cannot have all the coefficients divisible by $p$. By Proposition \ref{prop1} we may assume that the degree of $f$ is composite. Now let us assume to the contrary that $f(s)=F_{1}(s)F_{2}(s)$, with $F_{1},F_{2}$ Dirichlet polynomials of degrees $r\geq 2$ and $t\geq 2$, respectively.  By reducing modulo $p$ the relation $f(s)=qF(s)^{n}+pG(s)=F_{1}(s)F_{2}(s)$, one obtains 
\begin{equation}\label{descompun}
\overline{f}(s)=\overline{q}\cdot \overline{F}(s)^{n}=\overline{F_{1}}(s)\overline{F_{2}}(s).
\end{equation}

Since the leading coefficient of $f$ is not divisible by $p$, the same will hold for the leading coefficients of the hypothetical factors $F_{1}$ and $F_{2}$, hence $\deg \overline{F}_{1}=\deg F_{1}\geq 2$ and $\deg \overline{F}_{2}=\deg F_{2}\geq 2$. Since $\overline{F}$ is irreducible modulo $p$, we deduce by (\ref{descompun}) that
\begin{equation}\label{F1F2}
\overline{F_{1}}(s)=\overline{c}\cdot \overline{F}(s)^{a}\quad {\rm and}\quad \overline{F_{2}}(s)=\overline{d}\cdot \overline{F}(s)^{b}
\end{equation}
for two integers $a,b\geq 0$ with $a+b=n$, and two elements $c,d\in R$, with $\overline{c}\overline{d}=\overline{q}$. 
We note here that in view of (\ref{F1F2}), none of $a$ and $b$ can be zero, as $\deg \overline{F}_{1}\geq 2$ and $\deg \overline{F}_{2}\geq 2$, so for $n=1$ we already have a contradiction. Therefore, in what follows we may assume that $a\geq 1$, $b\geq 1$ and $n\geq 2$.
As we shall see, the fact that both $a$ and $b$ are positive will be of crucial importance for the remaining part of the proof. Next, we see that (\ref{F1F2}) implies the existence of two Dirichlet polynomials $h_{1},h_{2}$ such that
\[
F_{1}(s)=cF(s)^{a}+ph_{1}(s)\quad {\rm and}\quad F_{2}(s)=dF(s)^{b}+ph_{2}(s),
\] 
so $qF^{n}+pG=(cF^{a}+ph_{1})(dF^{b}+ph_{2})=cdF^{n}+pch_{2}F^{a}+pdh_{1}F^{b}
+p^{2}h_{1}h_{2}$, that is
\begin{equation}\label{semifinal1}
(q-cd)F^{n}+pG=pch_{2}F^{a}+pdh_{1}F^{b}
+p^{2}h_{1}h_{2}.
\end{equation}
Recalling now that $\overline{c}\overline{d}=\overline{q}$, there exists an element $e\in R$ such that
\begin{equation}\label{Bezout}
cd+pe=q.
\end{equation}
Combining (\ref{semifinal1}) and (\ref{Bezout}), we deduce after division by $p$ that
\begin{equation}\label{semifinal2}
G(s)=ch_{2}(s)F(s)^{a}+dh_{1}(s)F(s)^{b}-eF(s)^{n}
+ph_{1}(s)h_{2}(s).
\end{equation}
Reducing now (\ref{semifinal2}) modulo $p$, one obtains
\[
\overline{G}=\overline{c}\cdot \overline{h}_{2}\cdot \overline{F}^{a}
+\overline{d}\cdot \overline{h}_{1}\cdot \overline{F}^{b}
-\overline{e}\cdot \overline{F}^{n},
\]
with the right term being divisible by $\overline{F}$, as both $a$ and $b$ are positive. This is obviously a contradiction, since $F$ and $G$ are supposed to be relatively prime modulo $p$, and this completes the proof of the theorem. 
\end{proof}

In particular, as a corollary of the proof of Theorem \ref{calaSchonemann}, we have the following result.

\begin{corollary}\label{CorolarcalaSchonemann}
Let $f(s)=qF(s)^{n}+pG(s)$ where $n$ is a positive integer, $F,G$ are Dirichlet polynomials with integer coefficients, $p$ a prime number, and $q$ a nonzero integer. If the leading coefficient of $f$ is not divisible by $p$, $F$ is irreducible modulo $p$, and there exists an integer $m$ with $p\mid F(m)$ and $p\nmid G(m)$, then $f$ is irreducible over $\mathbb{Q}$.
\end{corollary}
\begin{proof}\ The proof follows exactly the same lines as in the case of Theorem \ref{calaSchonemann}, until we reach (\ref{semifinal2}). The contradiction follows now by taking $s=m$ in equation (\ref{semifinal2}).
\end{proof}

We will also prove here a similar result that provides symmetric irreducibility conditions with respect to two non-associated prime elements of $R$.
\begin{theorem}\label{pqSchone}
Let $f(s)=pF(s)^{n}+qG(s)^{n}$ with $n$ a positive integer, $F,G$ monic nonconstant Dirichlet polynomials with coefficients in a unique factorization domain $R$ with $\deg F=\deg G$, and $p,q$ non-associated prime elements of $R$. If $F$ is irreducible modulo $q$, $G$ is irreducible modulo $p$, and 
$F\not\equiv G\ (\mbox{\em mod}\ pq)$, then $f$ is irreducible over $Q(R)$.
\end{theorem}
\begin{proof}\  For a Dirichlet polynomial $h$ we will denote by $\tilde{h}$ and $\overline{h}$ the Dirichlet polynomials obtained by reducing the coefficients of $h$ modulo $p$ and modulo $q$, respectively. Here the leading coefficient of $f$ is $p+q$, which is coprime to both $p$ and $q$, and $\deg f=n\deg F$, which is a prime number only if $n=1$ and $\deg F$ is prime. Assume again that $\deg f$ is composite and that $f(s)=F_{1}(s)F_{2}(s)$, with $F_{1},F_{2}$ Dirichlet polynomials of degrees $r\geq 2$ and $t\geq 2$, respectively.  By reducing modulo $q$ the relation $f(s)=pF(s)^{n}+qG(s)^{n}=F_{1}(s)F_{2}(s)$, one obtains 
\begin{equation}\label{descompun2}
\overline{f}(s)=\overline{p}\cdot \overline{F}(s)^{n}=\overline{F_{1}}(s)\overline{F_{2}}(s).
\end{equation}
Since the leading coefficient of $f$ is not divisible by $p$, the same will hold for the leading coefficients of the hypothetical factors $F_{1}$ and $F_{2}$, hence $\deg \overline{F}_{1}=\deg F_{1}\geq 2$ and $\deg \overline{F}_{2}=\deg F_{2}\geq 2$. As before, since $\overline{F}$ is irreducible modulo $q$, we deduce by (\ref{descompun2}) that
\begin{equation}\label{F1F2nou}
\overline{F_{1}}(s)=\overline{c}\cdot \overline{F}(s)^{a}\quad {\rm and}\quad \overline{F_{2}}(s)=\overline{d}\cdot \overline{F}(s)^{b}
\end{equation}
for two integers $a,b\geq 0$ with $a+b=n$, and two elements $c,d\in R$, with $\overline{c}\overline{d}=\overline{p}$. 
By (\ref{F1F2nou}) we see that in this case too, none of $a$ and $b$ can be zero, for $\deg \overline{F}_{1}\geq 2$ and $\deg \overline{F}_{2}\geq 2$, so for $n=1$ we already have a contradiction. Therefore, in what follows we will assume that $a\geq 1$, $b\geq 1$ and $n\geq 2$.
By (\ref{F1F2nou}) there exist two Dirichlet polynomials $h_{1},h_{2}$ such that
\[
F_{1}(s)=cF(s)^{a}+qh_{1}(s)\quad {\rm and}\quad F_{2}(s)=dF(s)^{b}+qh_{2}(s),
\] 
so $pF^{n}+qG^{n}=(cF^{a}+qh_{1})(dF^{b}+qh_{2})=cdF^{n}+qch_{2}F^{a}+qdh_{1}F^{b}
+q^{2}h_{1}h_{2}$, which yields
\begin{equation}\label{semifinal1nou}
(p-cd)F^{n}+qG^{n}=qch_{2}F^{a}+qdh_{1}F^{b}
+q^{2}h_{1}h_{2}.
\end{equation}
As before, since $\overline{c}\overline{d}=\overline{p}$, there exists an element $e\in R$ such that $p-cd=qe$, so by (\ref{semifinal1nou}) we deduce after division by $q$ that
\begin{equation}\label{semifinal2nou}
G^{n}=ch_{2}F^{a}+dh_{1}F^{b}-eF^{n}
+qh_{1}h_{2}.
\end{equation}
Reducing now (\ref{semifinal2nou}) modulo $q$, one obtains
\[
\overline{G}^{n}=\overline{c}\cdot \overline{h}_{2}\cdot \overline{F}^{a}
+\overline{d}\cdot \overline{h}_{1}\cdot \overline{F}^{b}
-\overline{e}\cdot \overline{F}^{n},
\]
with the right side being divisible by $\overline{F}$, as both $a$ and $b$ are positive. This shows that 
$\overline{G}^{n}$ must be divisible by $\overline{F}$. Recalling that $F$ is irreducible modulo $q$, we see that $\overline{G}$ must be divisible by $\overline{F}$. Now, since $\deg \overline{F}=\deg \overline{G}$, we actually have $\overline{G}=\overline{r}\cdot\overline{F}$ for some element $r\in R$ not divisible by $q$. As both $F$ and $G$ are monic, we conclude that $\overline{r}=\overline{1}$, so
$\overline{F}=\overline{G}$. Absolutely similar, by reducing the coefficients modulo $p$, one proves that $\tilde{F}=\tilde{G}$. This obviously cannot hold, as $F\not\equiv G\ (\mbox{\rm mod}\ pq)$, and this completes the proof. 
\end{proof}

In particular, we have the following irreducibility criterion for Dirichlet polynomials of prime power degree.
\begin{corollary}\label{CorolarpqSchone}
Let $f(s)=pF(s)^{n}+qG(s)^{n}$ with $p,q$ different positive primes, $n$ a positive integer, and $F,G$ monic Dirichlet polynomials with integer coefficients of equal degree $r$, with $r$ a prime number. If 
$F\not\equiv G\ (\mbox{\em mod}\ pq)$, then $f$ is irreducible over $\mathbb{Q}$.
\end{corollary}
\begin{proof}\ By Proposition \ref{prop1}, $F$ is irreducible modulo $q$ and $G$ is irreducible modulo $p$, as both $F$ and $G$ have prime degree.  
\end{proof}

One may find similar symmetric irreducibility conditions that do not require $F$ and $G$ to be monic, with slightly more involved statements.
We will end this section with a result for Dirichlet polynomials with integer coefficients, that uses information on the values that such series take at a certain negative integral argument. 

\begin{theorem}\label{gradeprime}
Let $f(s)=mF_{1}(s)^{n_{1}}\cdots F_{k}(s)^{n_{k}}+pG(s)$ with $F_{1},\dots ,F_{k},G$ Dirichlet polynomials with integer coefficients, $m$ a nonzero integer, $p$ a prime number, $n_{1},\dots , n_{k}$ positive integers, and assume that $F_{1},\dots ,F_{k}$ are irreducible modulo $p$. If the leading coefficient of $f$ is not divisible by $p$, and there exists a negative integer $a$ such that $F_{1}(a),\dots , F_{k}(a)$ are all divisible by $p$, but $G(a)$ is not divisible by $p$, then $f(s)$ is irreducible over $\mathbb{Q}$.
\end{theorem}
\begin{proof}\ Our assumption that the leading coefficient of $f$ is not divisible by $p$ forces $\deg G$ to be less than or equal to $\deg F_{1}^{n_{1}}\cdots F_{k}^{n_{k}}$, and prevents $m$ to be divisible by $p$. We note that the leading coefficients of $F_{1},\dots ,F_{k}$ are allowed to be divisible by $p$, but $f\not\equiv 0 \pmod p$, as $F_{1},\dots ,F_{k}$ are irreducible modulo $p$. As in previous results, we may assume that $f$ has composite degree. Assume to the contrary that $f$ is reducible, say $f(s)=f_{1}(s)f_{2}(s)$ with $\deg f_{1}\geq 2$ and $\deg f_{2}\geq 2$. Reducing modulo $p$, we obtain $\overline{m}\overline{F_{1}}^{n_{1}}\cdots \overline{F_{k}}^{n_{k}}=\overline{f_{1}}\cdot \overline{f_{2}}$. We note that the leading coefficients of $f_1$ and $f_2$ are not divisible by $p$, so $\deg \overline{f_{1}}=\deg f_{1}$ and $\deg \overline{f_{2}}=\deg f_{2}$. Therefore there exist integers $m_{1},m_{2}$ with
$\overline{m}=\overline{m_{1}}\cdot \overline{m_{2}}$ and nonnegative integers $a_{1},\dots a_{k}$, $b_{1},\dots ,b_{k}$ with $a_{i}+b_{i}=n_{i}$ for $i=1,\dots ,k$ such that
\[
\overline{f_{1}}=\overline{m_{1}}\overline{F_{1}}^{a_{1}}\cdots \overline{F_{k}}^{a_{k}}\quad {\rm and}\quad 
\overline{f_{2}}=\overline{m_{2}}\overline{F_{1}}^{b_{1}}\cdots \overline{F_{k}}^{b_{k}}.
\]  
Observe now that we can neither have $a_{1}=\cdots =a_{k}=0$, nor $b_{1}=\cdots =b_{k}=0$, as $\deg \overline{f_{1}}\geq 2$ and $\deg \overline{f_{2}}\geq 2$. The two equalities above imply the existence of two Dirichlet polynomials $g_{1},g_{2}$ such that $f_{1}=m_{1}F_{1}^{a_{1}}\cdots F_{k}^{a_{k}}+pg_{1}$ and $f_{2}=m_{2}F_{1}^{b_{1}}\cdots F_{k}^{b_{k}}+pg_{2}$, so we have
\begin{equation}\label{paranteze}
mF_{1}^{n_{1}}\cdots F_{k}^{n_{k}}+pG=
(m_{1}F_{1}^{a_{1}}\cdots F_{k}^{a_{k}}+pg_{1})
(m_{2}F_{1}^{b_{1}}\cdots F_{k}^{b_{k}}+pg_{2}).
\end{equation}
Since $m-m_{1}m_{2}=pe$ for a certain integer $e$, we obtain by (\ref{paranteze}) after divison by $p$ that
\begin{eqnarray*}\label{parantezedesfacute}
G(s) & = & m_{1}g_{2}(s)F_{1}(s)^{a_{1}}\cdots F_{k}(s)^{a_{k}}+m_{2}g_{1}(s)F_{1}(s)^{b_{1}}\cdots F_{k}(s)^{b_{k}}\\
& & -eF_{1}(s)^{n_{1}}\cdots F_{k}(s)^{n_{k}}+pg_{1}(s)g_{2}(s).
\end{eqnarray*}
Recall now our assumption that for some integer $a$, all of $F_{1}(a),\dots , F_{k}(a)$ are divisible by $p$, while $G(a)$ is not a multiple of $p$. The desired contradiction is obtained by letting $s=a$ in our previous equality. Therefore $f$ must be irreducible over $\mathbb{Q}$. \end{proof}

\begin{corollary}\label{CorolarGradePrime}
Let $f(s)=mF_{1}(s)^{n_{1}}\cdots F_{k}(s)^{n_{k}}+pG(s)$ with $F_{1},\dots ,F_{k},G$ Dirichlet polynomials with integer coefficients, $F_{1},\dots ,F_{k}$ monic and of prime degree, $m$ a nonzero integer, $p$ a prime number, and $n_{1},\dots , n_{k}$ positive integers. If the leading coefficient of $f$ is not divisible by $p$, and there exists a negative integer $a$ such that $F_{1}(a),\dots , F_{k}(a)$ are all divisible by $p$, but $G(a)$ is not divisible by $p$, then $f(s)$ is irreducible over $\mathbb{Q}$.
\end{corollary}
\begin{proof}\ By Proposition \ref{prop1}, $F_{1},\dots ,F_{k}$ are all irreducible modulo $p$. 
\end{proof}


\medskip

{\bf Acknowledgements} This work was done in the frame of the GDRI ECO-Math. The author is grateful to M.D. Staic and C.M. Bonciocat for helpful and insightful comments ans suggestions.


\begin{thebibliography}{99}
\bibitem{Abhyankar} S. Abhyankar, {\it Local Analytic Geometry}, World Scientific Publishing, Singapore, 2001.

\bibitem{Alkan} E. Alkan, A.I. Bonciocat, N.C. Bonciocat, A. Zaharescu, {\it Square-free criteria for polynomials using no derivatives}, Proc. Amer. Math. Soc. 135 (2007), no. 3, 677--687.

\bibitem{BarkleyRosserSchoenfeld} J. Barkley Rosser, L. Schoenfeld, {\it Approximate formulas for some functions of prime numbers}, Illinois J. Math. 6 (1962), 64--94.

\bibitem{BergstromDiaconuPetersenWesterland} J. Bergstr\"om, A. Diaconu, D. Petersen and C. Westerland, {\it Hyperelliptic curves, the scanning map, and moments of families of quadratic $L$-functions}, Preprint, arXiv:2302.07664 v2 [math.NT] 8 Feb 2024.

\bibitem{Besicovitch} A. Besicovitch, {\it Almost periodic functions}, Cambridge University Press, 1932.

\bibitem{Bouniakovsky} V. Bouniakovsky, {\it Nouveaux th\'eor\`emes relatifs \`a la distinction des nombres premiers et \`a la d\'ecomposition des entiers en facteurs}, M\'em. Acad. Sc. St. P\'etersburg 6 (1897), 305--329.

\bibitem{Bauer} M. Bauer, {\it Verallgemeinerung eines Satzes von Sch\" onemann}, J. Reine Angew. Math. 128 (1905), 87--89.

\bibitem{Bodin1} A. Bodin, P. D\`ebes, S. Najib, {\it The Schinzel hypothesis for polynomials}, Trans. Amer. Math. Soc. 373 (12) (2020), 8339--8364.

\bibitem{Bodin2} A. Bodin, P. D\`ebes, J. K\"onig and S. Najib, {\it The Hilbert-Schinzel specialization property}, J. Reine Angew. Math. 785 (2022), 55--78.

\bibitem{Bombieri} E. Bombieri, J.B. Friedlander, {\it Dirichlet polynomial approximations to zeta functions}, Annali della Scuola Normale Superiore di Pisa, Classe di Scienze $4^e$ s\' erie, tome 22 (3) (1995), 517--544.

\bibitem{BoncioMathNach} A.I. Bonciocat, N.C. Bonciocat, {\it A Capelli type theorem for multiplicative convolutions of polynomials}, Math. Nachr. 281 (9) (2008), 1240--1253. 

\bibitem{BoncioResMat} A.I. Bonciocat, N.C. Bonciocat, {\it Irreducibility criteria for the sum of two relatively prime multivariate polynomials}, Results Math. 74:65  (2019).

\bibitem{BoncioIndagationes} A.I. Bonciocat, N.C. Bonciocat, Y. Bugeaud, M. Cipu, {\it Apollonius circles and irreducibility criteria for polynomials}, Indag. Math. (N.S.) 33 (2022), 421--439. 

\bibitem{BoncioAnStUnivOvid} A.I. Bonciocat, N.C. Bonciocat, M. Cipu, {\it Irreducibility criteria for compositions and multiplicative convolutions of polynomials with integer coefficients}, An. \c Stiin\c t. Univ. ``Ovidius'' Constan\c ta Ser. Mat. 22 (2014), no. 1, 73--84.

\bibitem{BoncioPMD1} A.I. Bonciocat, N.C. Bonciocat, Y. Bugeaud, M. Cipu, M. Mignotte, {\it Ireducibility criteria for compositions of multivariate polynomials}, Publ. Math. Debr. 97/3-4 (2020), 321 -- 337.

\bibitem{BoncioBullSSMR} C.M. Bonciocat, N.C. Bonciocat, Y. Bugeaud, M. Cipu, M. Mignotte, {\it Irreducibility criteria for some classes of compositions of polynomials with integer coefficients},  Bull. Math.  Soc. Sci. Math. Roum. 65 (113)  (2022), 149--180.

\bibitem{BoncioAMH} N.C. Bonciocat, {\it Irreducibility criteria for compositions of multivariate polynomials}, Acta Math. Hungar. 156 (2018), no. 1, 172--181.

\bibitem{BoncioCommAlg} N.C. Bonciocat, {\it Sch\"onemann-Eisenstein-Dumas-type irreducibility conditions that use arbitrarily many prime numbers}, Comm. Algebra 43 (2015), no. 8, 3102--3122. 

\bibitem{BoncioMonatshefte} N.C. Bonciocat, Y. Bugeaud, M. Cipu, M. Mignotte, {\it Irreducibility criteria for compositions of polynomials with integer coefficients}, Monatsh. Math. 182 (2017), no. 3, 499--512.

\bibitem{BoncioPMD2} N.C. Bonciocat, Y. Bugeaud, M. Cipu, M. Mignotte, {\it Irreducibility criteria for sums of two relatively prime multivariate polynomials}, Publ. Math. Debr. 87 (2015), no. 3-4, 255--267.

\bibitem{BoncioIJNT} N.C. Bonciocat, Y. Bugeaud, M. Cipu, M. Mignotte, {\it Irreducibility criteria for sums of two relatively prime polynomials}, Int. J. Number Theory 9 (2013), no. 6, 1529--1539.

\bibitem{Bourgain} J. Bourgain, {\it On large values estimates for Dirichlet polynomials and the density hypothesis for the Riemann zeta function}, Int. Math. Res. Not., (2000), no. 2, 133--146.

\bibitem{Boyd} S. Boyd, L. Vandenberghe, {\it Convex Optimization}, Cambridge University Press, 2004.

\bibitem{Brillhart}  J. Brillhart, M. Filaseta, and A. Odlyzko, \textit{On an irreducibility theorem of A. Cohn}, Canad. J. Math. 33 (1981), no. 5, 1055--1059.

\bibitem{RBrown} R. Brown, {\it Roots of generalized Sch\" onemann polynomials in Henselian extension fields}, Indian J. Pure Appl. Math. 39 (2008), no 5, 403--410.

\bibitem{Brown} K.S. Brown, {\it The coset poset and probabilistic zeta function of a finite group}, J. Algebra 225 (2) (2000), 989--1012.

\bibitem{BrubakerBumpFriedberg1} B. Brubaker, D. Bump, S. Friedberg, {\it Weyl Group Multiple Dirichlet Series ii. the stable case}, Invent. Math. 165 (2) (2006), 325--355.

\bibitem{BrubakerBumpFriedberg2} B. Brubaker, D. Bump, S. Friedberg, {\it Weyl Group Multiple Dirichlet Series, Eisenstein series and crystal bases}, Ann. of Math. 173 (2) (2011), 1081--1120.

\bibitem{BrubakerBumpFriedberg3} B. Brubaker, D. Bump, S. Friedberg, {\it Weyl Group Multiple Dirichlet Series: Type A Combinatorial Theory}, vol. 175 of Annals of Mathematics Studies, Princeton University Press, Princeton, NJ, 2011.

\bibitem{BushHajir} M.R. Bush, F. Hajir, {\it An irreducibility lemma} , J. Ramanujan Math. Soc. 23 (2008) no. 1, 33--41.

\bibitem{Cashwell} E.D. Cashwell, C.J. Everett, {\it The ring of number-theoretic functions}, Pacific J. Math. 9 (4) (1959), 975--985.

\bibitem{Cavachi}  M. Cavachi, {\it On a special case of Hilbert's irreducibility theorem}, J. Number Theory 82 (2000), 96--99.

\bibitem{CVZ1}  M. Cavachi, M. V\^aj\^aitu, A. Zaharescu, {\it A class of irreducible polynomials}, J. Ramanujan Math. Soc. 17 (2002), 161--172.

\bibitem{CVZ2} M. Cavachi, M. V\^aj\^aitu, A. Zaharescu, {\it An irreducibility criterion for polynomials in several variables}, Acta Math. Univ. Ostrav. {\bf 12} (2004),
no. 1, 13--18.

\bibitem{ChintaGunnels} G. Chinta, P. Gunnels, {\it Constructing Weyl group multiple Dirichlet series}, J. Amer. Math. Soc. 23 (2010), 189--215.

\bibitem{Cox} D.A. Cox, {\it Why Eisenstein proved the Eisenstein criterion and why Sch\" onemann discovered it first}, Amer. Math. Monthly 118 (2011), no. 1, 3--21.

\bibitem{Detomi} E. Detomi, A. Lucchini, {\it Crowns and factorization of the probabilistic zeta function of a finite group}, J. Algebra 265 (2) (2003), 651--668.

\bibitem{Damian1} E. Damian, A. Lucchini, {\it The Dirichlet polynomial of a finite group and the subgroups of prime power index}, in Advances in Group Theory 2002, Aracne, Rome, 2003, 209--221.

\bibitem{Damian2} E. Damian, A. Lucchini, {\it On the Dirichlet polynomial of finite groups of Lie type}, Rendiconti del Seminario Matematico della Universita di Padova 115 (2006), 51--69.

\bibitem{Damian3} E. Damian, A. Lucchini, F. Morini {\it Some properties of the probabilistic zeta function of finite simple groups}, Pacific J. Math. 215 (1) (2004), 1--14.

\bibitem{DiaconuIonPasolPopa} A. Diaconu, B. Ion, V. Pa\c sol, A. Popa, {\it Residues of quadratic Weyl group multiple Dirichlet series}, Adv. Math. 476, August 2025, 110359. https://doi.org/10.1016/j.aim.2025.110359

\bibitem{DiaconuGoldfeldHoffstein} A. Diaconu, D. Goldfeld, J. Hoffstein, {\it Multiple Dirichlet series and moments of zeta and $L$-functions}, Compos. Math. 139 (3) (2003), 297--360.

\bibitem{Dickson} D.G. Dickson, {\it Zeros of Exponential Sums}, Proc. Am. Math. Soc. 16 (1) (1965), 84--89.

\bibitem{Dumas} G. Dumas, {\it Sur quelques cas d'irreductibilit\' e des polyn\^ omes \` a coefficients rationnels}, Journal de Math. Pure et Appl. 2 (1906), 191--258.

\bibitem{Eisenstein} G. Eisenstein, {\it \" Uber die Irreductibilit\" at und einige andere Eigenschaften der Gleichung, von welcher die Theilung der ganzen Lemniscate abh\" angt}, J. Reine Angew. Math. 39 (1850), 160--179.

\bibitem{Ewald} G. Ewald, {\it Combinatorial Convexity and Algebraic Geometry},  Graduate Texts in Mathematics, Vol. 168, Springer-Verlag, Berlin-New York, 1996.

\bibitem{Filaseta1} M. Filaseta, {\it On the irreducibility of almost all Bessel Polynomials}, Acta Math. 174 (1995), no. 2, 383--397.

\bibitem{Filaseta-Finch-Leidy} M. Filaseta, C. Finch, J.R. Leidy, {\it T.N. Shorey's influence in the theory of irreducible polynomials}, Diophantine Equations (ed. N. Saradha), Narosa Publ. House, New Delhi, 2005, pp. 77--102.

\bibitem{FilasetaTrifonov}  M. Filaseta and O. Trifonov, {\it The irreducibility of the Bessel Polynomials}, J. Reine Angew. Math. 550 (2002), 125--140.

\bibitem{Gao} S. Gao, {\it Absolute Irreducibility of Polynomials via Newton Polytopes}, J. Algebra 237 (2001), 501--520. 

\bibitem{GaoLauder} S. Gao, A.G.B. Lauder, {\it Decomposition of polytopes and polynomials}, Discrete Comput. Geom. 26 (1)(2001), 89--104.

\bibitem{GaoRodrigues} S. Gao, V.M. Rodrigues, {\it Irreducibility of polynomials modulo $p$ via Newton polytopes}, J. Number Theory 101 (1) (2003), 32--47.

\bibitem{Grunbaum} B. Gr\"unbaum, {\it Convex Polytopes}, Interscience, London-New York-Sydney, 1967.

\bibitem{Girstmair}  K. Girstmair, \textit{On an Irreducibility Criterion of M. Ram Murty}, Amer. Math. Monthly 112 (2005), no. 3, 269--270.

\bibitem{Guersenzvaig} N.H. Guersenzvaig, \textit{Simple arithmetical criteria for irreducibility of polynomials with integer coefficients}, Integers 13 (2013), 1--21.

\bibitem{GuthMaynard} L. Guth and J. Maynard, \textit{Large value estimates for Dirichlet polynomials}, arXiv:2405.20552v1 [math.NT] 31 May 2024 (to appear in Ann. of Math.)

\bibitem{Halasz} G. Halasz, {\it Uber die Mittelwerte multiplikativer zahlentheoretischer Funktionen}, Acta Math. Acad. Sci. Hungaricae 19 (1968), 365--403.

\bibitem{Heath-Brown} D.R. Heath-Brown, {\it A large values estimate for Dirichlet polynomials}, J. Lond.
Math. Soc. 2 (1) (1979), 8--18.

\bibitem{Hironaka} Y. Hironaka, {\it Zeta functions of finite groups by enumerating subgroups}, Comm. Algebra 45 (8) (2017), 3365--3376.

\bibitem{Huxley} M.N. Huxley, {\it Dirichlet Polynomials}, Elementary and analytic theory of numbers, Banach Center Publications Vol. 17, PWN Polish Scientific Publishers, Warsaw, 1985, 307--316.

\bibitem{Huxley2} M.N. Huxley, {\it On the difference between consecutive primes}, Invent. Math., 15 (1972), 164--170.

\bibitem{Ivic} A. Ivi\u c, {\it The Riemann zeta-function}, Wiley-Interscience, 1985.

\bibitem{Jakhar1} A. Jakhar, {\it A simple generalization of the Sch\"onemann-Eisenstein irreducibility criterion}, Arch. Math. 117 (4) (2021), 375--378.

\bibitem{Jakhar2} A. Jakhar, {\it On a mild generalization of the Sch\"onemann-Eisenstein-Dumas irreducibility criterion}, Comm. Algebra 46 (1) (2018), 114--118.

\bibitem{Jakhar3} A. Jakhar, N. Sangwan, {\it On a mild generalization of the Sch\"onemann irreducibility criterion}, Comm. Algebra 45 (4) (2017), 1757--1759.

\bibitem{Jutila} M. Jutila, {\it Zero-density estimates for L-functions}, Acta Arith., 32 (1977), 55--62.

\bibitem{Konigsberger} L. K\" onigsberger, {\it \" Uber den Eisensteinschen Satz von der Irreduzibilit\" at algebraischer Gleichungen}, J. Reine Angew. Math. 115 (1895), 53--78.

\bibitem{Kurschak} J. Kurschak, {\it Irreduzible Formen}, J. Reine Angew. Math.152 (1923), 180--191.

\bibitem{Lovasz} L. Lov\'asz, {\it On finite Dirichlet series}, Acta Math. Acad. Sci. Hung. 22 (1--2) (1971), 227--231.

\bibitem{Lucchini} A. Lucchini, {\it The $\chi$-Dirichlet polynomial of a finite group}, J. Group Theory 8 no. 2 (2005), 171--188.

\bibitem{MacLane} S. MacLane, {\it The Sch\" onemann-Eisenstein irreducibility criteria in terms of prime ideals}, Trans. Amer Math. Soc. 43 (1938), 226--239.

\bibitem{Mahler} K. Mahler, {\it On some inequalities for polynomials in several variables}, J. Lond. Math. Soc. 37 (1962), 341--344.

\bibitem{Montgomery} H.L. Montgomery, {\it Zeros of approximations to the zeta function}, Studies in Pure Mathematics, 497--506, Birkh\"auser, Basel 1983.

\bibitem{Montgomery2} H.L. Montgomery, {\it Mean and large values of Dirichlet polynomials}, Invent. Math. 8 (4) (1969), 334--345.

\bibitem{Mott} J. Mott, {\it  Eisenstein-type irreducibility criteria}, in: Zero-dimensional commutative rings (Knoxville, TN, 1994), Lecture Notes in Pure and Appl. Math., 171, Dekker, New York (1995), 307--329.

\bibitem{Murty} M. Ram Murty, {\it Prime numbers and irreducible polynomials}, Amer. Math. Monthly 109 (5) (2002), 452--458.

\bibitem{Netto} E. Netto, {\it \" Uber die Irreduzibilit\" at ganzahligen ganzer Funktionen}, Math. Ann. 48 (1896), 81--88.

\bibitem{Oliveira} W. D. Oliveira, {\it Zeros of Dirichlet polynomials via a density criterion}, J. Number Theory 203 (2019), 80--94.

\bibitem{Ore1} O. Ore, {\it Zur Theorie der Irreduzibilit\" atskriterien}, Math. Zeit. 18 (1923), 278--288.

\bibitem{Ore2} O. Ore, {\it Zur Theorie der Eisensteinschen Gleichungen}, Math. Zeit. 20 (1924), 267--279.

\bibitem{Ore3} O. Ore, {\it Zur Theorie der Algebraischen K\" orper}, Acta Math. 44 (1923), 219--314.

\bibitem{Ore4} O. Ore, {\it Einige Bemerkungen \"uber Irreduzibilit\"at}, Jahresbericht der Deutschen
Mathematiker-Vereinigung 44 (1934), 147--151.

\bibitem{Ostrowski1} A.M. Ostrowski, {\it On multiplication and factorization of polynomials, I. Lexicographic ordering and extreme aggregates of terms}, Aequationes Math. 13 (1975), 201--228.

\bibitem{Ostrowski2} A.M. Ostrowski, {\it On multiplication and factorization of polynomials, II. Irreducibility discussion}, Aequationes Math. 14 (1976), 1--32.

\bibitem{PanaitopolStefanescu1} L. Panitopol, D. \c Stef\u anescu, {\it Factorization of the Sch\" onemann polynomials}, Bull. Math. Soc. Sci. Math. Roumanie (N.S.) 32(80)(1988), no. 3, 259--262.

\bibitem{PanaitopolStefanescu2} L. Panaitopol, D. \c Stef\u anescu, {\it A resultant condition for the irreducibility of the polynomials}, J. Number Theory 25 (1987), 107--111.

\bibitem{Patassini1} M. Patassini, {\it On the irreducibility of the Dirichlet polynomial of an alternating group}, Trans. Amer. Math. Soc. 365 (8) (2013), 4041--4062.

\bibitem{Patassini2} M. Patassini, {\it On the irreducibility of the Dirichlet polynomial of a simple group of Lie type}, Isr. J. Math. 185 (2011), 477--507.

\bibitem{Patassini3} M. Patassini, {\it The probabilistic zeta function of $PSL(2,q)$, of the Suzuki groups $^2B_2(q)$ and of the Ree groups $^2G_2(q)$}, Pacific J. Math. 240 (1) (2009), 185--200.

\bibitem{Perron} O. Perron, {\it \"Uber eine Anwendung der Idealtheorie auf die Frage nach der Irreduzibilit\" at algebraischen Gleichengen}, Math. Ann. 60 (1905), 448--458.

\bibitem{PolyaSzego} G. P\'olya, G. Szeg\"o, {\it Problems and Theorems in Analysis I}, Springer, 1976.

\bibitem{Prasolov} V. Prasolov, {\it Polynomials}, Algorithms and Computation in Mathematics, Vol. 11, Springer, Berlin, 2010.

\bibitem{Rella} T. Rella, {\it Ordnungsbestimmungen in Integrit\" atsbereichen und Newtonsche Polygone}, J. Reine Angew. Math. 158 (1927), 33--48.

\bibitem{RoyVatwani} A. Roy, A. Vatwani, {\it Zeros of Dirichlet polynomials}, Trans. Am. Math. Soc. 374 (1) (2021), 643--661.

\bibitem{Schmidt} W.M. Schmidt, {\it Equations over Finite Fields: an Elementary Approach}, Lecture Notes
in Mathematics, Vol. 536, Springer-Verlag, Berlin--New York, 1976.

\bibitem{Schneider} R. Schneider, {\it Convex bodies: The Brunn-Minkowski theory}, Encyclopedia of Mathematics and its Applications, Vol. 44, Cambridge Univ. Press, Cambridge, UK, 1993.

\bibitem{Schonemann} T. Sch\" onemann, {\it Von denjenigen Moduln, welche Potenzen von Primzahlen sind}, J. Reine Angew. Math. 32 (1846), 93--105.

\bibitem{Skorobogatov} A.N. Skorobogatov, E. Sofos, {\it Schinzel Hypothesis on average and rational points}, Invent. Math. 231 (2023), 673--739.

\bibitem{Schinzel} A. Schinzel, W. Sierpi\'nski, {\it Sur certaines hypoth\`eses concernant les nombres premiers}, Acta Arith. 4 (3)(1958), 185--208.

\bibitem{Shala} B. Shala {\it The probabilistic zeta function of a finite lattice}, Rocky Mountain J. Math. 54 (5) (2004), 1511--1526.

\bibitem{Stackel} P. St\"ackel, {\it Arithmetischen Eigenschaften ganzer Funktionen}, Journal f\"ur Mathematik 148 (1918), 101--112.

\bibitem{Titchmarsh} E.C. Titchmarsh, {\it The theory of the Riemann zeta-function}, Clarendon Press, second edition, 1986.

\bibitem{Webster} R. Webster, {\it Convexity}, Oxford Univ. Press, Oxford, 1994.

\bibitem{Weintraub} S.H. Weintraub, {\it A mild generalization of Eisenstein's criterion}, Proc. Amer. Math. Soc. 141 (2013), no. 4, 1159--1160.

\bibitem{Weisner} L. Weisner, {\it Criteria for the irreducibility of polynomials}, Bull. Amer. Math. Soc. 40 (1934), 864--870.

\bibitem{Zhang} W. Zhang, P. Yuan and T. Zhou, {\it An irreducibility criterion for the sum of two relatively prime polynomials}, Publ. Math. Debr. 104 / 3-4 (12) (2024), 479--498.

\bibitem{Ziegler} G.M. Ziegler, {\it Lectures on Polytopes}, Graduate Texts in Mathematics, Vol. 152, Springer-Verlag, Berlin-New York, 1995.




\end{thebibliography}
\end{document}